%% file: convMA.tex
\documentclass{memo-l}

 \usepackage{graphicx,amssymb}
 \usepackage{amsmath}
\usepackage{url}

\newcommand\tri{\triangle}

\newcommand\rep{\mathrm{rep}}

\newcommand\Def{\mathrm{Def}}
\newcommand\CDef{\mathrm{CDef}}
\newcommand\SDef{\mathrm{SDef}}
\newcommand\hol{\mathrm{hol}}
\newcommand\torb{\widetilde{\mathcal{O}}}
\newcommand\orb{\mathcal{O}}
\newcommand\Ag{{\mathrm{Ag}}}
\newcommand\Pgl{{\mathrm{PGL}}(n+1, \bR)}
\newcommand\SLpm{{{\mathrm{SL}}_{\pm}(n+1, \bR)}}

\newcommand\DefSO{\Def^s_{\SI^{n}, \, E, \, {\mathcal{U}''},\, s'_{\mathcal{U}''} }(\mathcal{O})}  
\newcommand\DefEO{\Def^s_{E, \, {\mathcal{U}}, \, s_{\mathcal{U}}}(\mathcal{O})}

\newcommand\cR{{\mathcal{R}}}
\newcommand\cT{{\mathcal{T}}}
\newcommand\rrp{{\rm )}} 
\newcommand\rlp{{\rm (}}
\newcommand\Imm{{\mathsf{Im}}}
\newcommand\bN{{\mathbb N}}

\newtheorem{theorem}{Theorem}[chapter]
\newtheorem{lemma}[theorem]{Lemma}
\newtheorem{proposition}[theorem]{Proposition}
\newtheorem{corollary}[theorem]{Corollary}

\theoremstyle{definition}
\newtheorem{definition}[theorem]{Definition}
\newtheorem{example}[theorem]{Example}

\theoremstyle{remark}
\newtheorem{remark}[theorem]{Remark}

\numberwithin{section}{chapter}
\numberwithin{equation}{chapter}
\numberwithin{figure}{chapter}

\makeindex

\input{lstyles}

\begin{document}

\frontmatter

%\title{}
\title[The convex real projective manifolds and orbifolds with ends]{The convex real projective orbifolds with
radial or totally geodesic ends: The closedness and openness of deformations}

%    Remove any unused author tags.

%    author one information
\author{Suhyoung Choi} %\thanks{This work was supported by the National Research Foundation
%of Korea (NRF) grant funded by the Korea government (MEST)  (No. 2010-0000139).} }%(No. 2009-0057445).} }

\address{Department of Mathematical Sciences \\ KAIST \\
Daejeon 305-701, South Korea
             % \email{schoi@math.kaist.ac.kr}         %  \\
         }
         
\email{schoi@math.kaist.ac.kr}         
%\address{Department of Mathematics \\ KAIST \\
%Daejeon 305-701, South Korea}
%\email{schoi@math.kaist.ac.kr}

\thanks{This work was supported by the National Research Foundation
of Korea (NRF) grant funded by the Korea government (MEST)  (No. 2010-0000139) need to change later.} %(No. 2009-0057445).} }

%    author two information
%\author{}
%\address{}
%\curraddr{}
%\email{}
%\thanks{}

%    \date is required; it is the date received by the editor.
\date{\today}

%\subjclass[2010]{Primary }
%    Recognition of the 2010 edition of the Mathematics Subject
%    Classification requires a version of amsbook.cls from July 2009
%    or later.  If "2010" is not recognized, please upgrade.

%\keywords{}

\subjclass[2010]{Primary 57M50; Secondary 53A20, 53C10, 20C99, 20F65}
\keywords{geometric structures, real projective structures, relatively hyperbolic groups, $\PGL(n, \bR)$, $\SL(n, \bR)$, character of groups}

%\dedicatory{This is a test}

%\setcounter{lofdepth}{2}

\begin{abstract}
A real projective orbifold is an $n$-dimensional orbifold modeled on $\rp^n$ with the group $\Pgl$. 
We concentrate on an orbifold that contains a compact codimension $0$ 
submanifold whose complement is a union of neighborhoods of ends, diffeomorphic to closed $(n-1)$-dimensional orbifolds times intervals. 
A real projective orbifold has a {\em radial end} if a neighborhood of the end is foliated by projective geodesics that develop into
geodesics ending at a common point.
It has a {\em totally geodesic end} if the end can be completed to have  the totally geodesic boundary. 
% \marginpar{This is not yet defined}
The orbifold is said to be {\em convex} if any path can be homotopied to a projective geodesic with 
endpoints fixed. %The orbifold is {\em strictly convex} if the orbifold is covered by a convex domain $\Omega$ in $\rp^n$
%so that each line segment in $\Bd \Omega$ is in the closure of some 

A real projective structure sometimes admits deformations to parameters of real projective structures. 
We will prove the local homeomorphism 
between the deformation space of convex real projective structures on such an orbifold with radial or totally geodesic ends 
with various conditions with the $\Pgl$-character space of the 
fundamental group with corresponding conditions. We will use a Hessian argument to show that under a small deformation, 
a properly (resp. strictly) convex real projective orbifold with generalized admissible ends will remain properly 
and properly (resp. strictly) convex with generalized admissible ends. %provided so are the beginning real projective orbifold and its ends. 
%behave in a convex manner. 
%Here, each fundamental group of the ends is to be 
%isomorphic to a finite extension of a product of hyperbolic groups and abelian groups. 
%The theory here generalizes that of Koszul on the closed real projective 
%orbifolds or in other words, divisible actions of projective groups. 

%The understanding of the ends is not accomplished in this 
%paper as this forms an another subject. 

Lastly, we will prove the openness and closedness of the properly (resp. strictly) convex real projective structures on a class of orbifold with
generalized admissible ends, where we need the theory of Crampon-Marquis
and Cooper, Long and Tillmann on the Margulis lemma for 
convex real projective manifolds. 
The theory here partly generalizes that of Benoist on closed real projective orbifolds. 

\end{abstract}

\maketitle

\tableofcontents
\listoffigures

%\listoffigures
%\maketitle

%\tableofcontents

%\include{preface}
%\maketitle

%\mainmatter

%    Include unnumbered chapters (preface, acknowledgments, etc.) here.
%\include{preface}

\mainmatter

\include{convopenIV12}

%\appendix
%    Include appendix "chapters" here.
%\include{}

\backmatter
%    Bibliography styles amsplain or harvard are also acceptable.
\bibliographystyle{amsalpha}

%\bibliography{convBib}
%    See note above about multiple indexes.

\bibliographystyle{plain}

\printindex

\end{document}

%% file: lstyles.tex
\newcommand\bR{{\mathbb{R}}}

\newcommand\bZ{{\mathbb Z}}
\newcommand\bQ{{\mathbb Q}}

\newcommand\rp{\bR P}
\newcommand\GL{{\rm GL}}
\newcommand\SL{{\rm SL}}
\newcommand\PGL{{\rm PGL}}

\newcommand\PO{{\rm PO}}

\newcommand\Hom{{\rm Hom}}

\newcommand\dev{{\bf dev}}
\newcommand\SI{{\bf S}}

\newcommand\Bd{{\rm bd}}
\newcommand\clo{{\rm Cl}}
\newcommand\bdd{{\bf d}}

\newcommand\ra{\rightarrow}

\newcommand\emp{\emptyset}
\newcommand\eps{\epsilon}

\newcommand\Aff{{\rm Aff}}
\newcommand\ovl{\overline}

\newcommand\Idd{{\rm I}}

%% file: convopenIV12.tex
\chapter{Introduction} 

%%% July 25, 2013 8:23 start here....
%\marginpar{I should do everything on $\SI^n$ again. Define spherical length.} 
%\marginpar{ We must think about the change of types
%radial <-> horo <-> tot geoÉ
%An exposition is necessary.
%Also, I need strong irreducibility: i.e., no finite index subgroup is reducibleÉ
%}

%\marginpar{Bibliography}

%Begin remark

\section{History and motivations} 
Recently, there were many research papers on convex real projective structures on manifolds and orbifolds.
(See the work of Goldman \cite{Gconv}, Choi \cite{cdcr1}, \cite{cdcr2}, Benoist \cite{Ben0}, 
Kim \cite{ikim}, Cooper, Long, Thistlethwaite \cite{CLT1}, \cite{CLT2} and so on.)
One can see them as projectively flat torsion-free connections on manifolds. Topologists will view each of these as
a manifold with a structure given by 
 a maximal atlas of charts to $\bR P^n$ where transition maps are projective. 
Hyperbolic and many other geometric structures will induce canonical real projective structures. 
(See the numerous and beautiful examples in Sullivan-Thurston \cite{ST}.)
Sometimes, these can be deformed to real projective structures not arising from such obvious constructions. 
In general, the theory of the discrete group representations and their deformations form very much mysterious subjects still. 
%conclusion
We can use the results in studying linear representations of discrete groups.

%(See Cooper-Long-Thistlethwidth \cite{CLT}.)

\begin{figure}
\centerline{\includegraphics[height=4cm]{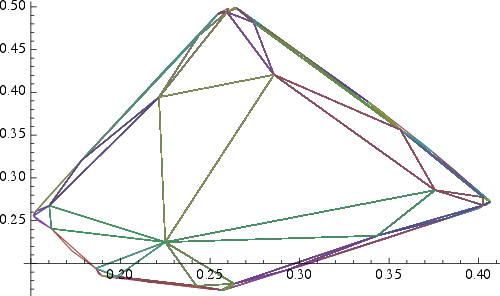}}
\caption{The developing images of convex $\bR P^n$-structures on $2$-orbifolds deformed from hyperbolic ones: $S^2(3,3,5)$.}
\label{fig:good2}
\end{figure}

\begin{figure}
\centerline{\includegraphics[height=4cm]{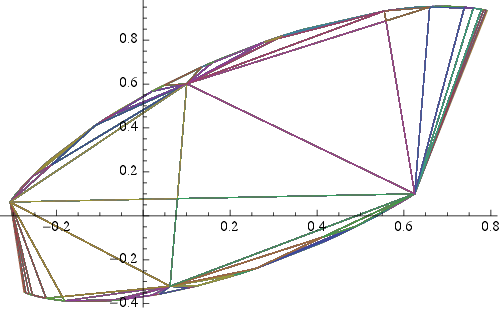}}
\caption{The developing images of convex $\bR P^n$-structures on $2$-orbifolds deformed from hyperbolic ones: 
$D^2(2, 7)$.}
\label{fig:good3}
\end{figure}

Since the examples are more easy to construct, we will be studying orbifolds, natural generalization of 
manifolds. 
The deforming a real projective structure on an orbifold to an unbounded situation results in the actions of 
the fundamental group on affine buildings which hopefully will lead us to some understanding of 
orbifolds and manifolds in particular in dimension three as indicated by
Cooper, Long, and Thistlethwaite. 
%Also, understanding when such deformations 
%exist seems to be an interesting question from the representation theory point of view. 

%% move to open ones
However, the manifolds studied are usually closed ones so far. %The work easily generalizes to closed orbifolds.
(See \cite{Ben4}, \cite{poly}, \cite{CLT1}, \cite{CLT2}, \cite{CHL}, \cite{CL}.)
%main sentence
We hope to generalize these theories to noncompact orbifolds with conditions on ends. 
In fact, we are trying the generalize the class of complete hyperbolic manifolds with finite volumes. 
%% explain this
These are $n$-orbifolds with compact suborbifolds 
whose complements are diffeomorphic to intervals times closed $(n-1)$-dimensional orbifolds. 
Such orbifolds are said to be {\em strongly tame} orbifolds. %For example, $E\times [0,1)$ for a compact orbifold $E$ is strongly tame. 
An {\em end neighborhood} is a component of a complement of a compact subset not contained in any compact subset of the orbifold. 
An {\em end} $E$ is an equivalence class of compatible sequences of end neighborhoods. 
Because of this, we can associate an $(n-1)$-orbifold at each end and we define 
the {\em end fundamental group} $\pi_1(E)$ as the fundamental group of the orbifold, a subgroup of the fundamental group $\pi_1(\orb)$. 
%The orbifold is said to be the {\em end orbifold}.
We also put the condition on end neighborhoods being foliated by radial lines or to have totally geodesic ideal boundary. 
Of course, this is not the only natural conditions, and we plan to explore the other conditions in some other occasions.
(We note that a strongly tame orbifold may have nonempty boundary that is compact.) 

%% prev work
We studied some such orbifolds of Coxeter type with ends  in \cite{poly}.
These have convex fundamental polytopes and are easier to understand. This paper generalizes the results there. 

\begin{figure}
\begin{center}$
\begin{array}{cccc}
\includegraphics[width=3.5in]{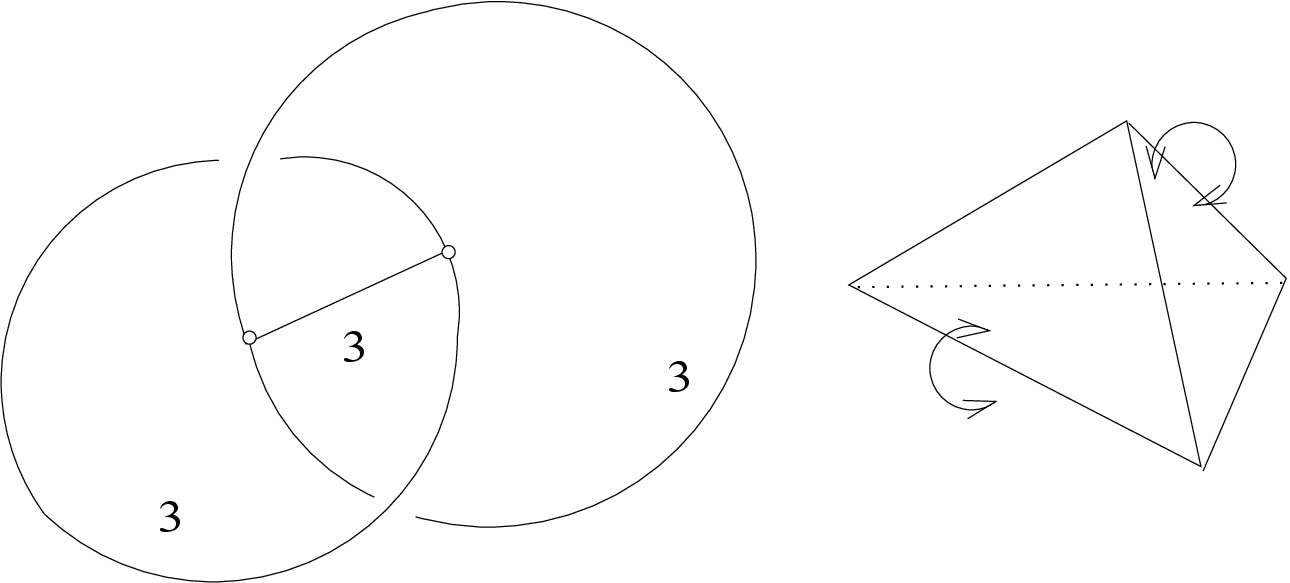} 
\end{array}$

\end{center} 

\caption{The handcuff graph and the construction of $3$-orbifold of Tillmann by pasting faces of a hyperbolic tetrahedron.}
\label{fig:cubs}
\end{figure}

S. Tillmann studied a complete hyperbolic $3$-orbifold obtained from gluing a complete hyperbolic 
tetrahedron. The one parameter family of deformations exists and can be solved explicitly. 
Later, Gye-Seon Lee and I computed more examples 
starting from hyperbolic Coxeter orbifolds (These are not published results. 
See Cooper-Long-Tillmann \cite{CLT3}, Heusener-Porti \cite{HP}, and  Ballas \cite{Ballas, Ballas2} for some computed $3$-manifold examples). 
However, the convexity of the results was the main 
question that arose. We will try to answer this. 
%They are also studying the same subject as us and 
%have claimed some very interesting analogous results but we have no idea of the precise nature of their work as it is 
%not available presently. 

Also, D. Cooper, D. Long, and S. Tillmann \cite{CLT3} 
and M. Crampon and L. Marquis \cite{CM} are studying these types of orbifolds as quotients of convex domains %\marginpar{update the refs..}
without deforming and hence generalizing the Kleinian group theory for complete hyperbolic manifolds. 
However, they only study the orbifolds with horospherical types of ends or equivalently finite volume orbifolds. 
Their work is in a sense dual to this work since
we start from orbifolds with real projective structures and deform. 
%(Include more Ballas...)
%However, the main objects that we are trying to generalize are complete hyperbolic manifolds. 
%\marginpar{References?}

%% Side remark 
In general, the theory of geometric structures on manifolds with ends is not studied very well. 
We should try to obtain more results here and find what are the appropriate conditions. 
This question seems to be also related to how to make sense of the topological structures of ends 
in many other geometric structures such as ones on modelled on symmetric spaces and so on. %There are not much theories here currently. 
%Also, we wish to obtain more examples of deformations of these structures and openness should 
%be of help in finding these examples as Yves Benoist had questioned. This might be one of 
%the central questions that we do not have answers yet. 
%% July 14th 4:46

\section{Some definitions}

%Main result
%% Define proj space groups.
Given a vector space $V$, we let $P(V)$ denote the space obtained by taking the quotient space of \index{$P(\cdot)$}
$\bR^{n+1} -\{O\}$ under the equivalence relation
\[v\sim w \hbox{ for } v, w \in \bR^{n+1} -\{O\}  \hbox{ iff } v = s w, \hbox{ for } s \in \bR -\{0\}.\] 
We let $[v]$ denote the equivalence class of $v \in \bR^{n+1} -\{O\}$. 
For a subspace $W$ of $V$, we denote by $P(W)$ the image of $W-\{O\}$ under the quotient map, also said to be a {\em subspace}. \index{subspace} 
Recall that the projective linear group $\PGL(n+1, \bR)$ acts on $\rp^n$, i.e.,
$P(\bR^{n+1})$, in a standard manner. % where the projective linear group $\PGL(n+1, \bR)$ acts on it in a standard manner. 
Let $\mathcal{O}$ be a noncompact strongly tame $n$-orbifold where the orbifold boundary is not necessarily empty. 
\begin{itemize} 
\item A {\em real projective orbifold} is an orbifold with a geometric structure modelled on $(\bR P^n, \PGL(n+1, \bR))$. \index{real projective orbifold} 
(See \cite{dgorb} and Chapter 6 of \cite{msj}.)
\item A real projective orbifold also has the notion of projective geodesics as given by local charts
and has a universal cover $\tilde{\mathcal{O}}$ where a deck transformation group $\pi_1({\mathcal{O}})$ acting on. 
\item The underlying space of $\mathcal{O}$ is homeomorphic to the quotient space $\tilde{\mathcal{O}}/\pi_1({\mathcal{O}})$.
\item A real projective structure on $\mathcal{O}$ gives us a so-called development pair  \index{real projective structure} 
$(\dev, h)$ where 
\begin{itemize} 
\item $\dev:\tilde{\mathcal{O}} \ra \bR P^n$ is an immersion, called the {\em developing map}, 
\item and $h:\pi_1(\mathcal{O}) \ra \PGL(n+1, \bR)$ is a homomorphism, called a {\em holonomy homomorphism}, 
satisfying 
\[\dev\circ \gamma= h(\gamma)\circ \dev \hbox{ for } \gamma \in \pi_1(\mathcal{O}).\] 
\item The pair $(\dev, h)$ is determined only up to the action 
\[g(\dev, h(\cdot)) = (g \circ \dev, g h(\cdot) g^{-1}) \hbox{ for } g \in \PGL(n+1, \bR)\]  
and any chart in the atlas extends to a developing map. (See Section 3.4 of \cite{Thbook}.) \index{developing map} \index{holonomy homorphism}
\end{itemize} 
\end{itemize} 
Let $\bR^{n+1 \star}$ denote the dual of $\bR^{n+1}$.  \index{$\bR^{n+1 \star}$}
Let $\bR P^{n \star}$ denote the dual projective space $P(\bR^{n+1 \star})$.  \index{$\bR P^{n \star}$}
$\PGL(n+1, \bR)$ acts on $\bR P^{n \star}$ by taking the inverse of the dual transformation. 
Then $h$ has a dual representation $h^* $ sending 
elements of $\pi_1(\orb)$ to the inverse of the dual transformation of $\bR^{n+1 \star}$.

For an element $g \in \PGL(n+1, \bR)$, we denote
\begin{align} 
g \cdot [w] &:= [\hat g(w)] & \hbox{ for } [w] \in \bR P^n \hbox{ or } \nonumber \\ 
  & := [(\hat g^{T})^{-1}(w)] & \hbox{ for } [w] \in \bR P^{n \star}
\end{align} 
where $\hat g$ is any element of $\SLpm$ mapping to $g$ and $\hat g^T$ the transpose of $\hat g$. \index{$\PGL(n+1, \bR)$}

The complement of a codimension-one subspace of $\bR P^n$ can be identified with an affine space 
$\bR^n$ where the geodesics are preserved. The group of affine transformations of $\bR^n$ are
is the restriction to $\bR^n$ of the group of projective transformations of $\bR P^n$ fixing the subspace. 
We call the complement an {\em affine subspace}. It has a geodesic structure of a standard affine space.  \index{affine subspace}
A {\em convex domain} in $\bR P^n$ is a convex subset of an affine subspace.  \index{convex domain}
A {\em properly convex domain} in $\bR P^n$ is a convex domain contained in a precompact subset of an affine subspace.  \index{convex domain! properly convex} 
 
The important class of real projective structures are so-called convex ones
where any arc in $\mathcal{O}$ can be homotopied with endpoints fixed to a straight geodesic where $\dev$ is injective to $\bR P^n$ except possibly at the end points. 
If the open orbifold has a convex structure, it is covered by a convex domain $\Omega$ in $\bR P^n$. 
Equivalently, this means that the image of the developing map $\dev(\tilde{\mathcal{O}})$ for the universal cover $\tilde{\mathcal{O}}$ of $\mathcal{O}$ 
is a convex domain. 
%Here we may assume $\dev(\tilde{\mathcal{O}})=\Omega$, and 
$\mathcal{O}$ is projectively diffeomorphic 
to $\dev(\tilde{\mathcal{O}})/h(\pi_1(\mathcal{O}))$. In our discussions, since $\dev$ often is an imbedding, 
$\tilde{\mathcal{O}}$ will be regarded as an open domain in $\bR P^n$ and $\pi_1(\mathcal{O})$ a subgroup of
$\PGL(n+1, \bR)$ in such cases. 
This simplifies our discussions. (See Chapter \ref{sec:prel}.)

We will have the following boundary deformability hypothesis for manageability. Otherwise the paper might become to large to handle. 
Let $\orb$ be a strongly tame real projective orbifold.
We assume that $\partial \orb$ is {\em strictly convex}; i.e.,  
each point of $\partial \orb$ has a neighborhood mapping to a convex ball with smooth strictly convex boundary under $\dev$.
Then each component of $\partial \orb$ can be deformed inward to a strictly convex boundary components by arbitrarily small amount
since one can find a smooth inward variation of $\partial \orb$. The hypersurface remains strictly convex for a short time. 
Moreover, we observe that 
 if the universal cover $\torb$ is a properly convex domain, the deformed $\torb$ is one also. 
(However, we will assume mostly that $\partial \orb=\emp$ in this paper for convenience and simplicity.)

%\marginpar{Unify our setting this way?}

%\end{frame} 

% April 13th 3:05pm

%We may always assume that the universal cover $\torb$ of $\orb$ always imbedded under $\dev$ and 
%hence is projectively diffeomorphic to a domain in $\bR P^n$. 

%% April 13 8:45pm
\subsection{Restricting the ends}

In this case, each end has a neighborhood diffeomorphic to a closed orbifold times an interval. 
This orbifold is independent of the choice of such a neighborhood, 
and it said to be the {\em end orbifold} associated with the end. 
The fundamental group of an end is isomorphic to the fundamental group of an end neighborhood. \index{end orbifold} \index{end fundamental group}
%, and the end orbifold 
%can be imbedded transversal to the radial foliation. 

A {\em lens} is a properly convex domain that is bounded by two smooth strictly convex open disks. \index{lens}
A {\em lens-orbifold}  is a compact quotient of a lens by a properly discontinuous action of a projective group.  \index{lens!orbifold}

%% Define ends
An end of a real projective orbifold $\orb$ is {\em totally geodesic} or {\em of type T} if  \index{end!totally geodesic} \index{end!type-T}
\begin{itemize}
\item the end has an end neighborhood that completes to a compact 
one homeomorphic to a closed $(n-1)$-dimensional orbifold times a half-open interval
and 
\item each point of the boundary has a neighborhood projectively diffeomorphic to 
an open set in an affine half-space. 
\end{itemize} 
The boundary component is called the totally geodesic {\em ideal boundary} ({\em component}) of \index{boundary!ideal}
the end. Such an ideal boundary may not be unique as there are two projectively inequivalent ways to add boundary components of 
elementary annuli (see Section 1.4 of \cite{cdcr2}). \index{elementary annulus} 
Two compactified end neighborhoods of an end are {\em equivalent} if they contain a common compactified end neighborhood.
We also require that 
\begin{description} 
\item[(Lens condition)] the ideal boundary be realized as a totally geodesic suborbifold in the interior of a lens-orbifold 
in a cover of a some ambient real projective orbifold of $\orb$ corresponding to the end fundamental group. \index{end!lens condition} 
\end{description}
We also define as follows: 
\begin{itemize}
\item The equivalence class of the chosen compactified end neighborhood is called a {\em marking } of the totally geodesic end. 
%In the universal cover $\torb$, it corresponds to a properly convex domain in a totally geodesic hypersurface. 
\item We will also call the ideal boundary the {\em end orbifold} of the end.
\end{itemize} 
(The reason for the lens condition here is to allow these ends to change to horospherical types.) 
We will call the totally geodesic ends with the above properties the ends of lens-type. 

An end of a real projective orbifold is {\em radial} or {\em of type R} if 
\begin{itemize}
\item the end has an end neighborhood foliated by properly imbedded projective geodesics and \index{end!radial} \index{end!type-R}
%that extend to concurrent geodesics for each chart and 
\item where nearby leaf-geodesics map under a developing map to open geodesics in $\bR P^n$ 
ending at the common point of concurrency. 
\end{itemize} 
Two such radial foliations of radial end neighborhoods of an end are {\em compatible} if they agree outside some compact subset of the orbifold. 
\begin{itemize} 
\item A {\em radial foliation marking} is a compatibility class of radial foliations. \index{end!radial foliation marking}
\item A {\em real projective orbifold with radial end marks} is a strongly tame 
orbifold with real projective structures and end neighborhoods with radial foliation markings. 
\end{itemize} 
%The radial foliation has a transversal real projective structure,  
The end orbifold has a unique induced real projective structure of one dimension lower
since the concurrent lines to a point form $\bR P^{n-1}$ and the real projective transformations fixing 
a point of $\bR P^n$ correspond to real projective transformations of $\bR P^{n-1}$. 
To summarize, an end orbifold is a well-defined closed real projective $(n-1)$-dimensional orbifold, 
which may depends on the choice of radial foliations. 
(Note that a real projective orbifold could have the same real projective structures and different radial foliation markings 
as Cooper pointed out.
Actually, the totally geodesic ends of lens-type are dual to radial ends and conversely. See Section \ref{subsec:totdual}.) %\marginpar{where to say this again}

%\marginpar{later?} 
The radial foliations and the ideal boundary components compactify $\orb$ to a compact orbifold $\bar \orb$, and the universal cover $\torb$ 
is contained in a completion $\torb'$ where the developing maps extend, probably not locally injectively. Note that $\bar \orb$ has the unique topology 
of a tame orbifold by attaching end orbifolds to $\orb$ to each ends. 
%One notes that it is possible for an end to be both totally geodesic and radial. In this case, we make an arbitrary choice
%and mark it accordiningly. 

%Convex but not properly convex ends might exist also. The relevant ones are classified in \cite{endclass}.

% An example
For example, a finite volume hyperbolic $n$-orbifold with cusps and
totally geodesic boundary components removed is an example. 
Let $\bR^{n+1}$ have standard coordinates $x_0, x_1, \dots, x_n$, and 
let $B$ be the subset in $\bR P^n$ corresponding to the cone given by 
\[x_0 > \sqrt{x_1^2 + \cdots + x_n^2}.\] 
The Klein model gives a hyperbolic space as $B \subset \bR P^n$ with 
the isometry group $\PO(1, n)$ a subgroup of $\Pgl$ acting on $B$. 
Thus, an above-type hyperbolic orbifold is projectively diffeomorphic to an open submanifold of $B/\Gamma$ for 
$\Gamma$ in $\PO(1, n)$. 
(Also, we could allow hyperideal ends by attaching radial ends.
See Section 3.1.1 in \cite{endclass}.)

An {\em ellipsoid} in $\bR P^n=P(\bR^{n+1})$ (resp. in $\SI^n=S(\bR^{n+1})$)  \index{ellipsoid}
is the projection $C -\{O\}$ of the null cone $C:=\{x \in \bR^{n+1}| b(x, x)=0\}$ for a nondegenerate 
bilinear form $b: \bR^{n+1} \times \bR^{n+1} \ra \bR$. Ellipsoids are always equivalent by projective automorphisms.
An {\em ellipsoid ball} is the closed contractible domain of in $\bR P^n$ (resp. $\SI^n$)  bounded by an ellipsoid. 
A {\em horoball} is an ellipsoid ball with a point $p$ of the boundary removed. \index{horoball}
An ellipsoid with a point $p$ on it removed is called a {\em horosphere}. The {\em vertex} of \index{horosphere}
the horosphere or the horoball is defined as $p$. 

Define $\Bd A$ for a subset $A$ of $\bR P^n$ (resp. in $\SI^n$) to be the {\em topological boundary} in $\bR P^n$ (resp. in $\SI^n$) \index{boundary!topological}
and define $\partial A$ for a manifold or orbifold $A$ to be the {\em manifold or orbifold boundary} \index{boundary!manifold}
and $A^o$ denote the manifold interior. \index{interior!manifold}
The closure $\clo(A)$ of a subset $A$ of $\bR P^n$ (resp. of $\SI^n$) is the topological closure in $\bR P^n$ (resp. in $\SI^n$). \index{$\clo(\cdot)$}

A real projective orbifold that is real projectively diffeomorphic to an orbifold
$U/\Gamma_p$ for a discrete subgroup $\Gamma_p \subset \PO(1, n)$ fixing 
a point $p \in \Bd B$ and a horoball $U$ with vertex $p$ is called a {\em horoball orbifold}. \index{horoball!orbifold}
A {\em horospherical end} is an end with an end neighborhood that is such an orbifold. 
(In our case, by strong tameness, the group contains an abelian group of maximal rank $n-1$ of finite index by Proposition \ref{prop:affinehoro}.) 

Given a real projective orbifold $\orb$, we add the restriction of the end to be a radial or a totally geodesic type. 
The end will be either assigned {\em $\cR$-type} or {\em $\cT$-type}. \index{$\cR$-type} \index{$\cT$-type}
\begin{itemize} 
\item An $\cR$-type end is required to be radial. 
\item A $\cT$-type end is required to have totally geodesic properly convex ideal boundary 
components or be horospherical. 
\end{itemize} 
A strongly-tame orbifold will always have such an assignment in this paper, 
and finite-covering maps will always respect the types. 

%We will always assume that our real projective orbifolds are usually strongly tame and with totally geodesic ends and 
%radial ends with markings already associated. 

\subsection{Definition of the deformation spaces with end marks} 
An {\em isotopy} $i: \orb \ra \orb$ is a self-diffeomorphism so that  \index{isotopy}
there exists a smooth parameter of self-diffeomorphism $i_t: \orb \ra \orb$, $t\in [0, 1]$, 
so that $i=i_1, i_0 = \Idd_{\orb}$. 
\begin{itemize}
\item Two real projective structures $\mu_0$ and $\mu_1$ are {\em isotopic} 
if there is an isotopy $i$ on $\mathcal{O}$ so that $i^*(\mu_1)=\mu_0$ where $i^*(\mu_1)$ is the induced structure from $\mu_1$ by $i$. 
\begin{itemize}
\item $i$ sends the radial end foliation for $\mu_0$ from an end neighborhood to the radial end foliation for $\mu_1$ in 
the another end neighborhood, and 
\item $i$ extends to the union of totally geodesic ideal boundary components as a diffeomorphism. 
\end{itemize}
\end{itemize} 
We define 
$\Def_E(\mathcal{O})$ as the deformation space of real projective structures on $\mathcal{O}$ with end marks; more precisely, 
this is the quotient space of the real projective structures $\mu$ on $\mathcal{O}$ satisfying the above conditions for
ends of type $\cR$ and $\cT$ 
under the isotopy equivalence relations.
The topology of such a space is  defined by the compact open $C^r$-topology for the space of developing maps $\dev$, $r \geq 2$.
%Unfortunately for noncompact orbifolds, the space is not metrizable.
We will discuss this more later on.  %\marginpar{Where exactly?} 
For noncompact orbifolds, these spaces can be very complicated especially if there are no end markings. 
(see \cite{dgorb}, \cite{Canary} and \cite{Goldman3} for more details. )
%
%\begin{definition}
%Given a strongly tame orbifold $\orb$, we assign each end to be of $\cR$-type or T-type. 
%We define $\Def_{E}(\mathcal{O})$ to be the quotient space of the space of 
%real projective structures with radial type ends for ends assigned to be $\cR$-type
%or totally geodesic or horospherical ends for ends assigned to be T-type 
%under the group of isotopies preserving the radial markings for $\cR$-type ones or horospherical ones 
%and extending to the ideal boundary smoothly for T-type ones. 
%The topology is given by $C^r$-topology, $ r \geq 2$, of the space of  maps  of form $\dev$ extended to the natural completion of $\torb$. 
%\end{definition} 
%Note that this is not merely a subspace of $\Def(\orb)$. 
%(Note that for $\cT$-type horospherical end, the isotopies need to preserve the radial foliations.)
%An {\em end fundamental group} $\pi_1(E)$ for an end $E$ is the fundamental group of the end neighborhood of $\mathcal{O}$ of product form. 
%It has many conjugates which are also end fundamental groups. 

%An {\em eigen-$1$-form} is a linear functional $\alpha$ in $\bR^{n+1}$ so that a linear transformation $\gamma$ 
%has the property $\alpha \circ \gamma = \lambda \alpha$ for some $\lambda \in \bR$. 

\begin{remark} \label{rem:Davis}
As suggested by Mike Davis, one can look at ends with holonomy groups of end fundamental groups acting on properly convex 
domains in totally geodesic subspaces of codimension between $2$ and $n-1$. While they are perfectly reasonable to occur, 
in particular for Coxeter type orbifolds, 
we shall avoid these types as they are not understandable yet
and we will hopefully study these in other papers. We will only be thinking of ends with holonomy groups of the end fundamental groups  acting on 
codimension $n$ or codimension $1$ subspaces. However, we think that the other types of the ends do not change the theory present here 
in an essential way. 
\end{remark} 

Let $\{g_1, \dots, g_m\}$ be the generators of $\pi_1(\orb)$. 
As usual $\Hom(\pi_1(\mathcal{O}), G)$ for a Lie group $G$ has an {\em algebraic topology} as a subspace 
of $G^m$. This topology is given by the notion of {\em algebraic convergence}
\[\{h_i\} \ra h \hbox{ if } h_i(g_j) \ra h(g_j) \in G \hbox{ for each } j, j=1, \dots, m.\] 

The {\em character space} $\rep(\pi_1(\mathcal{O}), \PGL(n+1,\bR))$ is the quotient space of \index{character space}
the homomorphism space \[\Hom(\pi_1(\mathcal{O}), \PGL(n+1,\bR))\] where $\PGL(n+1,\bR)$ acts by conjugation
\[h(\cdot) \mapsto g h(\cdot) g^{-1} \hbox{ for } g \in \PGL(n+1,\bR).\]
Each element is called a {\em character} in this paper. A {\em representation} is an element in the equivalence class of a character. 
A representation or a character is {\em stable} if the orbit of it or its representative is closed and the stabilizer is finite under 
the conjugation action in  $\Hom(\pi_1(\mathcal{O}), \PGL(n+1,\bR))$. (See \cite{JM} for more details.)
By Theorem 1.1 of \cite{JM}, a representation $\rho$ is stable if and only if any proper parabolic subgroup \index{character!stable} \index{representation!stable}
does not contain the image of $\rho$. % is not contained in 
%any proper parabolic subgroup. %It is easy to see that $\rho$ is irreducible since otherwise the stablizer is not compact. 
The stability and the irreducibility are open conditions in the Zariski topology. 
Also, if the image of $\rho$ is Zariski dense, then $\rho$ is stable by Lemma 3.5 of \cite{CL}.

A representation of a group $G$ into $\PGL(n+1, \bR)$ or $\SLpm$ is {\em strongly irreducible} \index{representation!strongly irreducible}
if the image of every finite index subgroup of $G$ (resp. every finite index subgroup of $G$) is irreducible. 
Actually, many of the orbifolds have strongly irreducible and stable holonomy homomorphisms by Theorem \ref{thm:sSPC}.

An {\em eigen-$1$-form} of a linear transformation $\gamma$ is a linear functional $\alpha$ in $\bR^{n+1}$ so that \index{eigen-$1$-form}
$\alpha \circ \gamma = \lambda \alpha$ for some $\lambda \in \bR$. 

We define
\begin{itemize} 
\item  \[\rep_E(\pi_1(\mathcal{O}), \PGL(n+1,\bR))\] to be the subspace of characters where 
\begin{description}
\item[The vertex condition] the restricted representation to each $\cR$-type p-end fundamental group has a nonzero common eigenvector \index{vertex condition}
and
\item[The lens condition] the restricted representation to each $\cT$-type p-end fundamental group acts properly \index{lens condition} 
discontinuously and cocompactly on a lens $L$, a properly convex domain with $L^o\cap P = L \cap P\ne \emp$ for a hyperplane $P$.
meeting $P$ or a horoball tangent to $P$. 
\end{description} 
\item We denote by 
 \[\rep_E^s(\pi_1(\mathcal{O}), \PGL(n+1,\bR))\]
 the subspace of stable and irreducible characters, and 
 \item \[\rep^s_{E, u}(\pi_1(\mathcal{O}), \PGL(n+1,\bR))\] the subspace of stable and irreducible characters where
 \begin{itemize}
\item  the restricted representation to each radial p-end fundamental group has a unique common eigenspace of dimension $1$ and 
\item each totally-geodesic end fundamental group has  a unique common space of eigen-$1$-forms of dimension $1$
meeting a lens with above properties. 
\end{itemize} 
\[\rep^s_{E, u}(\pi_1(\mathcal{O}), \PGL(n+1,\bR))  \subset \rep_{E}(\pi_1(\mathcal{O}), \PGL(n+1,\bR)).\] 
%We also require that for the hyperspace $P$ determined by the $1$-forms of the p-end, the holonomy group of 
%each $\cT$-type p-end acts on a horosphere tangent to $P$ with a vertex in $P$ 
%or a properly convex domain in $P$ properly discontinuously. 
%cocompactly on a complete affine subspace or a properly convex domain
%in $\bR P^{n-1 \ast}_p$ for each fixed dual point $p$. 
%(We call this condition the {\em end divisibility condition}.)
%\end{itemize} 
%(We just call this condition the {\em end divisibility condition}.)
%\item  \[\rep_{E*}(\pi_1(\mathcal{O}), \PGL(n+1,\bR))\] to be the subspace of representations where the restricted representation 
%to each totally-geodesic end fundamental group has nonzero common eigen-$1$-form. 
%\item \[\rep_{E*, u}(\pi_1(\mathcal{O}), \PGL(n+1,\bR))\] to be the subspace of representations where the restricted representation 
%to each totally-geodesic end fundamental group has unique common space of eigen-$1$-forms of dimension $1$.
\item We define 
\begin{align} 
&  \rep_{E, u}^s(\pi_1(\mathcal{O}), \PGL(n+1,\bR)) \nonumber \\ 
 & := \rep^s(\pi_1(\mathcal{O}), \PGL(n+1,\bR)) \cap \rep_{E, u}(\pi_1(\mathcal{O}), \PGL(n+1,\bR)).
  \end{align} 
%  \item We denote by
% \[\rep_E^i(\pi_1(\mathcal{O}), \PGL(n+1,\bR))\]
% the subspace of stable and irreducible representations, and define
%  \[\rep_{E, u}^i(\pi_1(\mathcal{O}), \PGL(n+1,\bR)) := \rep_E^i(\pi_1(\mathcal{O}), \PGL(n+1,\bR)) \cap \rep_{E, u}(\pi_1(\mathcal{O}), \PGL(n+1,\bR))\]
\end{itemize} 
%An example of hyperbolic 3-manifold with a cusp deforms to an incomplete hyperbolic 
%3-manifold with an end has invariant hyperplance of the holonomy groups of the end, but the end cannot be completed to have totally geodesic boundary. 
%Thus, we need the end divisibility condition.
%\marginpar{changed end div cond.. Need to change some proofs later}

%Also, the strong irreducibility is an algebraic condition since if a representation is irreducible virtually, then 
%its finite index normal subgroup act on a proper subspace $V$. 
%The end divisibility condition is needed to allow for totally geodesic ends to become horospherical ends and vice versa 
%(see Lemma \ref{lem:horob}).

Note that elements of $\Def_E(\orb)$ have holonomy characters in \[\rep_{E}(\pi_1(\mathcal{O}), \PGL(n+1,\bR)).\]
Denote by $\Def_{E, u}(\mathcal{O})$ the subspace of $\Def_{E}(\mathcal{O})$ of equivalence classes of real projective structures 
with holonomy characters in \[\rep_{E, u}(\pi_1(\mathcal{O}), \PGL(n+1,\bR)).\] 
Also, we denote by $\Def_{E}^s(\mathcal{O}) \subset \Def_{E}(\mathcal{O})$
and $\Def_{E, u}^s(\mathcal{O}) \subset \Def_{E, u}(\mathcal{O})$ the subspaces of equivalence classes of real projective 
structures with stable and irreducible characters. % Similarly, we define $\Def_{E}^s(\mathcal{O})$ and $\Def_{E, u}^s(\mathcal{O})$. 

%%% July 26 12:28
%An {\em end fundamental group condition} is explained in Section \ref{subsec:endf}. This condition is 
%one where if the restricted representation of the end fundamental group of an end fixes a point of $\bR P^n$, then it fixes a unique one
%or if the dual of the representation of the end fundamental group of an end fixes a point of $\bR P^{n \star}$, then it fixes a unique one. 
%This forces $\Def_{E}(\mathcal{O})$ to be same as $\Def_{E, u}(\mathcal{O})$. 

% We do caution the readers that these assumptions are not trivial and excludes some important representations.
 %For example, these spaces exclude some incomplete hyperbolic structures arising in Thurston's Dehn surgery constructions 
 %as they have at least two fixed points for the holonomy of the fundamental group of a toroidal end. However, we 
 %will present some methods to study these here as well in Section  \ref{subsec:endf}. 

\section{The  local homeomorphism and homeomorphism theorems} 
For technical reasons, we will be assuming $\partial \orb = \emp$ in most cases. In fact, a proper way to understand the boundary 
is through understanding the ends as in the hyperbolic manifold theory of Thurston. 
The following map $\hol$ is induced by sending $(\dev, h)$ to the conjugacy class of $h$ as 
isotopies preserve $h$:

\begin{theorem} \label{thm:A} 
Let $\mathcal{O}$ be a noncompact strongly tame real projective $n$-orbifold with radial ends or totally-geodesic ends of lens-type with 
markings and given types $\cR$ or $\cT$. Assume $\partial \orb =\emp$. 
%\begin{itemize} 
%\item 
Then the following map is a local homeomorphism\,{\rm :}  
\[\hol:\Def_{E, u}^s(\mathcal{O}) \ra \rep_{E, u}^s(\pi_1(O), \PGL(n+1,\bR)).\]
%\item Suppose that $\mathcal{O}$ has the end fundamental group conditions. 
%Then the following map is a local homeomorphism: 
%\[\hol:\Def_{E}^s(\mathcal{O}) \ra \rep_E^s(\pi_1(O), \PGL(n+1,\bR)).\] 
%\end{itemize} 
\end{theorem}
The restrictions of end types are necessary for this theorem to hold. 
This generalizes results for closed manifolds for many geometric structures starting from the classical results of Weil. 
(See Goldman \cite{Goldman3}, Canary-Epstein-Green \cite{Canary}, and \cite{dgorb} for many versions of similar results.)

We will present a more general ``section'' version Theorem \ref{thm:projective} in Chapter \ref{sec:loch}
giving this theorem as a corollary. (Also, see for related work by Cooper and Long \cite{CL2}.)

%%% April 29 2014 4:57pm 

\begin{definition} 
An {\em SPC-structure} or a {\em stable  irreducible properly-convex real projective structure} on an $n$-orbifold 
( {\em with radial or totally geodesic end of lens-type} ) is 
a real projective structure so that the orbifold is projectively diffeomorphic to a quotient orbifold of 
a properly convex domain in $\bR P^n$ by a discrete group
of projective automorphisms that is stable and  irreducible.
\end{definition}

\begin{definition}
Suppose that $\mathcal{O}$ has an SPC-structure. Let $\tilde U$ be 
the inverse image in $\tilde{\mathcal{O}}$ of the union $U$ of some choice of a collection of disjoint end neighborhoods of $\orb$.  \index{SPC-structure}
If every straight arc in the boundary of the domain $\tilde{\mathcal{O}}$ 
and every non-$C^1$-point  is contained in the closure of a component of $\tilde U$ for some choice of $U$, 
%the inverse image in $\tilde{\mathcal{O}}$ of the union of some collection of disjoint end neighborhoods of $\orb$,  
then $\mathcal{O}$ is said to be {\em strictly convex} with respect to the collection of the ends.  \index{convex!strictly}
And $\mathcal{O}$ is also said to have a {\em strict SPC-structure} with respect to the collection of ends. \index{SPC-structure!strict}
%or the some collection of the end neighborhoods.
\end{definition}
We will drop the respectiveness when it is obvious. 
Also $\mathcal{O}$ with its real projective structure is {\em strictly SPC} if it the structure is a strict SPC-structure. 
%The notion is independent of the choice of end neighborhoods when we will restrict to the types of ends. 
Later we will show that the word ``some" can be replaced by ``all"
if we restrict the end types. (See Corollary \ref{cor:strictconv}.)

%\marginpar{The lens here we need some conditions on boundary... We need more details..later?}
%A character or a representation is irreducible if the restriction to each finite-index subgroup is irreducible. 
%Let $\mathcal{O}$ be a tame $n$-orbifold where end fundamental groups are admissible. 

A {\em segment} is a connected arc in a one-dimensional subspace of $\bR P^n$. 
Given two points or subsets $A, B$ in an affine subspace $\bR^n$ of $\bR P^n$, we define the {\em join} $A \ast B$ \index{join} 
as the union of all segments in $\bR^n$ with end points in $A$ and $B$ 
respectively or its interior. More precisely, 
\[ A \ast B := \{[t v + (1-t) w]| t \in [0, 1], v \in C_A, w \in C_B \] 
where $C_A$ is the connected cone in $\bR^{n+1}$ mapping to $A$ and $C_B$ is one for $B$ 
in the open half space $H$ in $\bR^{n+1}$ corresponding to the affine subspace $\bR^n$. 
Since $A$ and $B$ are usually subsets of a convex domain in an affine space, the join is well-defined subset of the convex domain. 

%%% April 15 11:12pm 2014

For a topological manifold $A$, we denote by $\partial A$ the manifold boundary and \index{$\partial A$}
by $A^o$ the manifold interior. 
A {\em ray} is a segment starting from a point $v$ in $\bR P^n$ (resp. $\SI^n$) \index{ray}
that is oriented away from $v$. 
Two rays from $v$ are {\em equivalent} if the rays agree in a neighborhood of $v$. 
A {\em generalized lens} is a properly convex domain bounded by two open disks one of which is smooth and strictly convex \index{lens!generalized}
and the other boundary component allowed to be not strictly convex. 
A {\em generalized lens-cone} is a cone  $\{p\} \ast L -\{p\}$ over \index{lens-cone!generalized} 
a generalized lens $L$ so that every maximal ray from $p$ in $\{p\}\ast L$ meets 
each of the two boundary components of $L$ exactly once and the nonsmooth boundary component must be in the boundary of 
$\{p\} \ast L$. 
A lens or a generalized lens $L$ is a {\em strict lens} or {\em generalized strict lens} if $L - \partial L$ is nowhere dense in $\Bd L$. \index{lens!strict} \index{lens!generalized!strict}
A {\em cone over a lens with vertex $p$}, $p \not\in \clo(L)$, is defined as $\{p\} \ast L -\{p\}$ for a lens $L$ so that every maximal ray from 
$p$ in $\{p\}\ast L$ meets 
each of the two boundary components of $L$ exactly once. 
%A {\em generalized lens} is a properly convex domain bounded by two open disk inner one of which is smooth and strictly convex
%and the outer boundary component allowed to be not strictly convex. 
%A {\em generalized lens-cone} is a cone over a generalized lens with the above properties. 
Any subdomain of $\torb$ projectively diffeomorphic to the interior of the above is called by
the same name as if they are in $\bR P^n$. 
\begin{itemize} 
\item An $\cR$-type end is of {\em lens-type} if $\torb$ contains the lens-cone where the end fundamental group
acts on the lens fixing the vertex. 
An $\cR$-type end is of {\em generalized lens-type} if $\torb$ contains the interior of 
a generalized lens-cone  where the end fundamental group
acts on the generalized lens fixing the vertex. 
%\item A $\cT$-type end is of {\em lens-type} if there is a neighborhood of the ideal boundary component 
%in an ambient orbifold that is covered by a lens. 
\item An $\cR$-type end of a real projective orbifold is {\em admissible} if it is a radial end of lens-type or horospherical type and \index{end!admissible} 
a $\cT$-type end is {\em admissible } if it is totally geodesic (of lens-type) or horospherical type. 
\end{itemize} 
We can allow the $\cR$-type ends to be of generalized lens-type. 
Then we say that the ends are {\em admissible in the generalized sense}. \index{end!generalized admissible}
We require in both cases that the end fundamental group is a virtually a product of hyperbolic groups and abelian groups.
(See Section \ref{sub:admissible} for definitions.)
We define  \[\rep_{E, u, ce}(\pi_1(\mathcal{O}), \PGL(n+1, \bR))\] 
to be the subspace of 
\[\rep_{E, u}(\pi_1(\mathcal{O}), \PGL(n+1,\bR))\] 
where the holonomy group of each end fundamental group  is realized as that of an admissible end.
%(The realization is essentially unique by Remark \ref{rem:realization}.)
We define 
\begin{align} 
 & \rep_{E, u, ce}^s(\pi_1(\mathcal{O}), \PGL(n+1, \bR))    \nonumber \\  %\] to be the subspace of stable characters in 
 & \quad \quad := \rep_{E, u, ce}(\pi_1(\mathcal{O}), \PGL(n+1, \bR)) \cap \rep_E^s(\pi_1(\mathcal{O}), \PGL(n+1,\bR)).
 \end{align}
We will show that these are semi-algebraic sets or at least open subsets of such sets (See Section \ref{sub:semialg}.)

%A generalized lens-cone is a 

We are only using ``admissible'' end as this concept is equivalent to the principal boundary condition 
for two-dimensional real projective surfaces \cite{Gconv}. Also, in \cite{endclass}, we show that they are naturally
structurally stable, and definable by natural eigenvalue conditions.
(See Theorems 1.1, 1.2, 8.1, and 8.2 of \cite{endclass}.)
We remark that an example found by Ballas \cite{Ballas2} is not lens-type radial end. However, this end seems 
to be in the classification \cite{endclass}. We will attempt to understand these types of ends at later dates.  

%The orbifold $\mathcal{O}$ has no essential annulus if no element of the fundamental group one end of $\mathcal{O}$ 
%is homotopic to another element of the fundamental group of an end of $\mathcal{O}$ and the homotopy annulus 
%cannot be pushed into one end. (This is from the $3$-manifold topology).

%Given an orbifold $O$ with radial ends, each end has an end neighborhood that is diffeomorphic to the end orbifold times an interval.
%We denote their union by $U$. 
%An SPC-structure or irreducible properly convex real projective structure on $O$ is a convex real projective structure covered
%by a prop
%Let us choose $U$ disjoint from large compact submanifold of $O$ so that 
%the closure $\clo(\tilde U)$ of the inverse image $\tilde U$ in $\tilde O$ of $U$ is the union of the closures of each component of $\tilde U$. 

%A {\em convex end fundamental condition} is defined in Section \ref{subsec:endf}. Loosely speaking, this condition is 
%one where if a representation of an end of a convex orbifold with admissible ends fixes a point of $\bR P^n$, then it fixes a unique one. 
%This is more general than a mere end fundamental condition. 

%The end orbifold is said to be {\em admissible} in this case. 

We define 
\begin{itemize}
\item $\Def_{E, u, ce}(\mathcal{O})$ to 
be the subspace  $\Def_{E}(\mathcal{O})$ of classes of real projective structures with generalized admissible ends,
\item $\Def^s_{E, u, ce}(\mathcal{O})$ to be the deformation space of classes of real projective structures with generalized
admissible ends and stable and irreducible holonomy homomorphisms, 
\item $\CDef_{E, u, ce}(\mathcal{O})$ to be the deformation space of 
SPC-structures with generalized admissible ends and stable and irreducible holonomy homomorphisms, and 
\item $\SDef_{E, u, ce}(\mathcal{O})$ to be that of strict SPC-structures with 
admissible ends and stable and irreducible holonomy homomorphisms. 
\end{itemize} 
%We also define
%\begin{align} 
%\Def_{E, u, ce}(\mathcal{O}) &= \Def_{E, ce}(\mathcal{O}) \cap \Def_{E, u}(\mathcal{O}) \nonumber \\
%\Def^s_{E, u, ce}(\mathcal{O}) &= \Def^s_{E, ce}(\mathcal{O}) \cap \Def_{E, u}(\mathcal{O}) \nonumber \\
%\CDef^s_{E, u, ce}(\mathcal{O}) &=   \CDef^s_{E, ce}(\mathcal{O}) \cap \Def_{E, u}(\mathcal{O}) \nonumber \\
%\SDef^s_{E, u, ce}(\mathcal{O}) &=   \SDef^s_{E, ce}(\mathcal{O}) \cap \Def_{E, u}(\mathcal{O}) .
%\end{align}

%$ to be the deformation space of real projective structures with 
%admissible ends and irreducible holonomy, define $\CDef^s_{E, u, ce}(\mathcal{O})$ to be the deformation space of 
%SPC-structures with admissible ends, and define $\SDef^s_{E, ce}(\mathcal{O})$ to be that of strict SPC-structures with 
%admissible ends. 
We will allow for these structures that a radial lens-cone end could change to a horospherical type
and vice versa and a totally geodesic lens end could change to a horospherical one and vice versa.
But we will not allow a radial lens-cone end to change to a totally geodesic lens end.
%The $\cR$-type and $\cT$-type will be preserved by our restrictions as the horospherical ones are both 
%$\cT$-type and $\cR$-type but will be considered as having a single type according to our asssignment on $\orb$. 
%A non-horospherical $\cT$-type cannot change to a nonhorospherical $\cR$-type directly. 
%(The purpose of $\cR$-types and $\cT$-types are without them the local injective maps of maps $hol$ may not be locally
%injective.)  
%(See Section ????). 

By an {\em essential} annulus $A$, we mean a map $f: A \ra \mathcal{O}$ so that 
components of $\partial A$ are mapped into end neighborhoods and to a homotopy class of infinite order, and $f$ is not homotopic into 
an end neighborhood relative to $\partial A$. By an {\em essential torus} $T^2$, 
we mean a map $f: T^2 \ra \mathcal{O}$ so that the induced homomorphism $f_*: \pi_1(T^2) \ra \pi_1({\mathcal{O}})$ 
is injective to a free abelian group of rank two and where $f$ is not freely homotopic into an end of $\mathcal{O}$. 

%%% le 2 julie, 10:11 pm
For a strongly tame orbifold $\mathcal{O}$, 
\begin{itemize}
\item[(IE)] $\orb$ or $\pi_1(\orb)$ satisfies the {\em infinite-index end fundamental group condition} (IE) \index{end!condition!IE}
if $[\pi_1(E):\pi_1(\mathcal O)] = \infty$ for the end fundamental group $\pi_1(E)$ of each end $E$. 
%\item[(FA)] Let $\mathcal{E}$ denote the set of all conjugates of end fundamental group of $\pi_1(\orb)$. 
%Also, if $\bigcap_{\Gamma_E \in F} \Gamma_E = \{1\}$ for any infinite subset $F$ of $\mathcal{E}$, 
%we say that $\orb$ or $\pi_1(\orb)$ satisfies {\em finite immersed annuli condition} or (FA).
\item[(NA)] $\orb$ or $\pi_1(\orb)$ satisfies the {\em property-NA} \index{end!condition!NA} 
if $\orb$ has no essential annulus and %nor essential torus and 
$\pi_1(E)$ contains every element $g \in \pi_1(\orb)$ normalizing $\langle h \rangle$ for 
an infinite order $h \in \pi_1(E)$ for an end fundamental group $\pi_1(E)$ of an end $E$. 
\end{itemize} 
(NA) implies that $\orb$ contains no essential torus. These conditions are satisfied by complete hyperbolic manifolds with cusps,
the objects that we are trying to generalize. 
These are group theoretical properties with respect to the end groups. 

%%% July 26 1:01

\begin{theorem}\label{thm:B} %\label{cor:conv} 
%Suppose that $\Def_{Aff, E}$
Let $\mathcal{O}$ be a noncompact strongly tame $n$-orbifold with generalized admissible ends. 
Assume $\partial \orb =\emp$. 
Suppose that $\mathcal{O}$  satisfies {\rm (IE)} and {\rm (NA).}  % and the finite immersed annuli condition. 
Then 
\begin{itemize}
\item the subspace   of SPC-structures \[\CDef_{E, u, ce}(\mathcal{O}) \subset \Def^s_{E, u, ce}(\mathcal{O})\]  is open.
\item Suppose further that every finite-index subgroup of $\pi_1(\mathcal{O})$ contains no nontrivial infinite nilpotent normal subgroup
and $\partial \orb = \emp$. 
$\hol$ maps
$\CDef_{E, u, ce}(\mathcal{O})$  homeomorphically to a union of components of 
 \[\rep_{E, u, ce}(\pi_1(\mathcal{O}), \PGL(n+1, \bR)).\]
\end{itemize}
\end{theorem}
%Note that we droped superscript ``s'' at the end showing that the components consist of stable characters. 
Theorems \ref{thm:conv} and \ref{thm:closed1} and Corollary \ref{cor:closed1} prove this and following theorems.
(The more general ``section versions'' will not be proved.)
%\marginpar{I think I can show closedness}

%\marginpar{I need to define the terms}
%An {\em essential torus} is a torus that is not freely homotopic into an ends of $\mathcal{O}$. 
%\marginpar{ No essential tori condition same as peripheral abelian group condition.}

\begin{theorem} \label{thm:C} 
Let $\mathcal{O}$ be a strict SPC noncompact strongly tame $n$-dimensional orbifold with admissible ends 
and satisfies {\rm (IE)} and {\rm (NA)}. %the convex end fundamental group conditions
%the infinite-index end fundamental group condition and the finite immersed annuli and tori condition. \marginpar{These need to change.}
%Suppose that every abelian subgroup of $\pi_1(\mathcal{O})$ of rank $\geq 2$ is in one of the end fundamental groups. 
%Then%Furthermore, If the ends of $\mathcal{O}$ are permanently convex, then 
%$\hol$ maps the deformation space of SPC-structures on $\mathcal{O}$ homeomorphic to 
%a component of $\rep_{E}^s(\pi_1(\mathcal{O}), \PGL(n+1, \bR))$.
Assume $\partial \orb =\emp$. 
Then
\begin{itemize}
\item $\pi_1(\mathcal{O})$ is relatively hyperbolic with respect to its end fundamental group.
\item The subspace  $\SDef_{E, u, ce}(\mathcal{O}) \subset \Def^s_{E, u, ce}(\mathcal{O})$,
of strict SPC-structures with admissible ends is open. 
\item  Suppose further that every finite-index subgroup of $\pi_1(\mathcal{O})$ contains no nontrivial infinite nilpotent normal subgroup
and $\partial \orb =\emp$. 
Then $\hol$ maps the deformation space $\SDef_{E, u, ce}(\mathcal{O})$ of 
strict SPC-structures on $\mathcal{O}$ with admissible ends homeomorphically to 
a union of components of \[\rep_{E, u, ce}(\pi_1(\mathcal{O}), \PGL(n+1, \bR)).\]
%\item $\SDef^s_{E, u, ce}(\mathcal{O})$ is a union of components of $\Def^s_{E, u, ce}(\mathcal{O})$. 
\end{itemize}
\end{theorem}
%(Here the strictness is independent of the choice of the end neighborhoods. )
%We note that the second item with $\CDef^s_{E, u, ce}(\mathcal{O})$ is still a question. 

We will also show that an SPC-orbifold $\mathcal O$ with generalized admissible ends is strictly SPC with admissible ends iff 
$\pi_1(\mathcal{O})$ is relatively hyperbolic with respect to its end fundamental groups. 
(See Theorems \ref{thm:relhyp} and \ref{thm:converse}.) 

Finally we also remark that an example found by S. Ballas \cite{Ballas2} seems to obtain a parameters of radial ends not covered 
in this paper. This type is what we call quasi-joined ends in \cite{endclass} not included here. 
But the general philosophy of the paper is to treat these new types of ends as different strata to be added later. 

\section{Outline} 

The paper is divided into three parts: 
The part I is on the local homeomorphism property, i.e., Theorem \ref{thm:A}. 

In Chapter \ref{sec:prel}, we give elementary definitions of geometric structures, real projective 
structures, radial ends, totally geodesic ends and so on.
We give some well-known reducibility 
theorems for closed real projective orbifolds due to Koszul, Vey, and Benoist.
We discuss the admissible ends and their properties from \cite{endclass}. 
Also, the duality of the domains and the actions will be studied. 
Then we discuss the affine structures. 
We also study an affine suspension, a method to obtain an affine structure from a real projective structure.

In Chapter \ref{sec:loch}, we prove the local homeomorphism theorem, i.e., Theorem \ref{thm:A};
$\hol$ send the deformation space to the character space locally homeomorphically.
%In our paper, we need to consider radial end structures and corresponding holonomy conditions. 
%There are end fundamental group conditions, which arise naturally.
We discuss first the semialgebraic nature of the spaces. 
We will prove the theorem for the affine structure and change it to be applicable to real projective structures.
The methods are similar to what is in \cite{dgorb}.
Here, we need to have continuous sections of eigenvectors in the end holonomy groups. 
Finally, we transfer the theorem to the real projective cases using affine suspensions. 

In Part II, we discuss the convexity properties of the orbifolds and related these to the relative hyperbolicity of 
the fundamental groups of the orbifolds. 

In Chapter \ref{sec:conv}, we discuss convexity and define convex real projective structures on orbifolds.
%We discuss the decomposition results of Benoist on reducible closed real projective $(n-1)$-orbifolds.
We discuss horospherical ends and lens-shaped 
ends and their properties and various facts concerning their existence, stability, and examples and so on.
We also discuss the duality of $\cR$-type ends and $\cT$-type ends. 
If $\mathcal{O}$ satisfies (NA), then we define the boundary of the convex hulls of p-ends of $\torb$.
These results are mostly from \cite{endclass}.

In Chapter \ref{sec:SPC}, we define stable irreducible properly convex real projective structures or SPC-structures on orbifolds. 
%This is a convex projective structure whose developing image is properly convex and the holonomy 
%group is stable and irreducible. 
We define strict SPC-structures also.
%We show that this means that the orbifold has a relatively hyperbolic metric with end neighborhoods contracted.
We show using Bowditch's approach 
 that an SPC-orbifold has a relatively hyperbolic fundamental group with respect to 
its end fundamental groups if and only if the SPC-orbifold is strictly SPC. (Theorems \ref{thm:relhyp} and \ref{thm:converse}.)

In Part III, we discuss the openness and the closedness of the deformation spaces of convex real projective structures on orbifolds
and finally an example where our theory applies. 
A {\em deformation } means changing the real projective structures so that 
the developing maps change continuously in the $C^r$-topology, $ r \geq 1$, on every compact subset of $\torb$.

In Chapter \ref{sec:openness}, 
we prove that if ends of an orbifold are admissible in a generalized sense, 
then the deformations of (resp. strictly) SPC-structures will remain (resp. strictly) SPC-structures under 
irreducibility conditions, i.e., Theorem \ref{thm:B}.
The proof is divided into two: First, we show that there is a Hessian metric and under small perturbations of the real projective structures, 
we can still find a nearby Hessian metric. Basically, we find that the Koszul-Vinberg 
functions of the affine suspensions change by very small amounts.
Second, the Hessian metric and the boundary orbifold convexity assumption imply convexity. 
%We finish by showing that the convex structure is really a strict SPC-one when we start with strict IP

In Chapter \ref{sec:closed}, we show that the deformation space $\CDef_{E, u, ce}(\orb)$ or $\SDef_{E, u, ce}(\orb)$  
maps homeomorphic to the union of components of 
\[\rep_{E, u, ce}(\pi_1(\mathcal{O}), \PGL(n+1, \bR))\]
under appropriate assumptions.  
%the subspace of irreducible characters. 
%which is well-known to be an open dense subset of $\rep_E(\pi_1(\mathcal{O}), \PGL(n+1, \bR))$, i.e., Theorem C. 
%We can remove the irreducibilty condition in the character space. 
We will use the Margulis Lemma of Cooper, Long, and Tillmann for the thin subgroups of the discrete projective automorphism 
groups acting on properly convex domains. (See \cite{CLT3}.)

In Chapter \ref{sec:examples}, we describe some examples where the theory is applicable. 
These include the examples of S. Tillmann and a double of a tetrahedron reflection group of all order 3.
%Furthermore, for these examples, we can remove the irreducibilty condition of Theorem \ref{thm:C}
%since the tetrahedra in this examples can be controlled. 

%We also think that various other forms of the deformation theory of radial ends could exist. We hope to participate
%and work with anyone interested in this topic. Somehow, this theory should be related to other type of end theory 
%occuring in other types of geometric structures. 

%\begin{acknowledgements}

\subsection{Acknowledgment}
We thank Yves Benoist, Yves Carri\`ere, Daryle Cooper, Mika\"el Crampon, 
Kelly Delp, William Goldman, Ludovic Marquis,
Hyam Rubinstein, and Stephan Tillmann for many discussions. 
This paper began from a discussion the author had with  Tillmann on his construction of a parameter of 
real projective structures on small complete hyperbolic 3-orbifold. 
%Unfortunately, we could not continue the collaboration 
%due to difference in approaches. 
We thank D. Fried for helping with the algebraic nature of the character varieties
in Subsection \ref{sub:semialg}.
Many helpful discussions were carried out with C. Hodgson and we hope to publish the resulting 
examples in another papers. %We thank Y. Benoist for many of his results without which we could not have written 
%this paper, and helping me with the proof of Theorem \ref{thm:redtot} (i-3). 
We also thank B. Bowditch, C. Drutu, and M. Kapovich for many technical discussions and the valuable help
with the geometric group theory used here. In fact, C. Drutu supplied us with the proof of Theorem \ref{thm:drutu}. 

%\end{acknowledgements}

\part{The local homeomorphisms of the deformation spaces into the spaces of characters}

\chapter{Preliminary}\label{sec:prel}

% Define real proj structures , Affine structures
\section{Orbifolds with ends}
Let $H^n$ be the closed half-space of $\bR^n$. 
An $n$-dimensional {\em orbifold} is a second-countable Hausdorff space with an orbifold structure. 
An orbifold structure is given by a fine covering by open sets of form $\phi(U)$ where 
$(U, G, \phi)$ is a triple of an open subset of $H^n$ and $\phi:U \ra \phi(U)$ is a quotient map 
inducing a homeomorphism $U/G \ra \phi$ for a finite group $G$ acting on $U$. 
%quotient maps of open or half-open subsets of $\bR^n$ by finite groups of transformations
An inclusion map of an open set $\phi(U)$ with the model $(U, G, \phi)$ 
to another one $\psi(V)$ with the model $(V, H, \psi)$ induces an inclusion map $U \ra V$ equivariant 
with respect to an injection of the groups $G \ra H$ determined up to conjugations.  (See \cite{ALR} for details.)
An {\em orbifold structure} is a maximal fine covering. \index{orbifold}
We call $(U, G, \phi)$ the {\em model}. 
%The orbifolds enjoy much of the properties of manifolds. 
%As an example, a manifold is an orbifold in the sense that the manifold is covered by 
%open sets with trivial groups. 

%We also restrict to a {\em strongly tame $n$-orbifold} with finitely many ends; 
%i.e., the orbifold has a compact connected orbifold with finitely many boundary component  whose complement is homeomorphic to 
%the finite union of closed $(n-1)$-orbifolds times open intervals. Also, our orbifold can be compactified as an orbifold by 
%adding finitely many end orbifolds corresponding to the ends.

An orbifold $\mathcal{O}$ often has a simply-connected manifold as a covering space. 
In this case the orbifold is said to be {\em good}. We will assume this always for our orbifolds. \index{orbifold!good} 
(Orbifolds with geometric structures 
are always good by Thurston. See Chapter 13 of \cite{Thnote} and Chapter 6 of \cite{dgorb}.) 
Such a covering $\tilde{\mathcal{O}}$ is unique up to covering equivalences and 
is said to be the universal cover. There is a discrete group $\pi_1(\mathcal{O})$ acting on the universal 
cover so that we recover $\mathcal{O}$ as a quotient orbifold $\tilde{\mathcal{O}}/\pi_1(\mathcal{O})$, where 
$\pi_1(\mathcal{O})$ is said to be the (orbifold) fundamental group of $\mathcal{O}$. \index{orbifold!fundamental group}

The {\em local group} of a point of $\tilde{\mathcal{O}}$ is the inverse limit of the group acting on the model neighborhoods \index{orbifold!local group}
ordered by the lifts of the inclusion maps. It is well-defined. %up to local orbifold isotopy conjugations.

An {\em end neighborhood} of an orbifold is a component of the complement of a compact subset of an orbifold. \index{end!end neighborhood} 
The collection of the end neighborhoods is partially ordered by inclusion maps. 
\begin{itemize} 
\item An {\em end} is an equivalence class of sequences of end neighborhoods \index{end} 
\[U_i, i=1,2,..., U_i \supset U_{i+1} 
\hbox{ and } \bigcap_{i=1, 2,..} \clo(U_i)=\emp.\] 
\item Two such sequences $U_i$ and $V_j$ are {\em equivalent} if 
for each $i$, there exist $j, j'$ such that $U_i \supset V_j$ and $V_i \supset U_{j'}$. 
\end{itemize} 

Given an end $E$, we can associate the end fundamental group $\pi_1(E)$ since we can always find 
a sequence of proper end neighborhoods of product type in the end class. 
That is, $\pi_1(E)$ is defined as the inverse limit of $ \{\Imm(\pi_1(U_i)) \subset \pi_1(\orb)\}$ where maps are 
\[\Imm(\pi_1(U_i)) \ra \Imm(\pi_1(U_j)) \subset \pi_1(\orb),  i >  j.\]

Given the universal cover of $\torb$, a {\em proper pseudo-end neighborhood} is a component of an inverse image of  \index{end!p-end neighborhood} 
an end neighborhood. Two proper pseudo-end neighborhoods are {\em equivalent} if they intersect such a one. 
The equivalence class of a system of proper pseudo-end neighborhood is called a {\em pseudo-end} \index{end!p-end}
in $\torb$. If we require that the proper pseudo-end neighborhood is from one end $E$,  
there is a one-to-one correspondence 
\[ \{ \tilde E | \tilde E \hbox{ is a pseudo-end of } \torb \hbox{ corresponding to } E\} \Leftrightarrow \pi_1(\orb)/\pi_1(E).\]
The subgroup of deck transformation groups $\pi_1(\orb)$ preserving the class of a pseudo-end $\tilde E$
is called a {\em pseudo-end fundamental group} and is denoted by $\pi_1(\tilde E)$, obviously isomorphic to $\pi_1(E)$. 
Moreover, the deck transformation group $\pi_1(\orb)$ acts on the set of pseudo-ends and each orbit-equivalence class corresponds to an 
end of $\orb$. Hence, there is a natural map from the set of pseudo-ends of $\torb$ to the set of ends of $\orb$
induced by the covering map.

%For real projective orbifolds, 
%we study ones with {\em properly foliated ends}, which means that each end of the orbifold has an end neighborhood 
%foliated by properly imbedded real lines and local group associated with each point of a leaf is locally conjugate 
%by some leaf preserving flows. Here we require that the leaves develop to concurrent geodesics 
%and nearby leaf-geodesics map to infinite open geodesics in $\bR^n$ or open geodesics in $\bR P^n$ 
%ending at the common point of concurrency. Two foliations of end neighborhoods are {\em compatible} if 
%they agree on a smaller end neighborhood. The compatibility class of an end is  said to be a {\em foliation marking}. 
%We will also assume that our orbifolds always have foliation markings. 

%% Oct 31 10:38 PM
\section{Geometric structures on orbifolds} \label{sub:gso} 
Let $G$ be a Lie group acting on an $n$-dimensional manifold $X$.
For examples, we can let $X=\bR^n$ and $G=\Aff(\bR^n)$ for the affine group $\Aff(\bR^n)= \GL(n,\bR) \ltimes \bR^n$, \index{$\Aff(\bR^n)$}
i.e., the group of transformations of form $v \mapsto Av + b$ for $A \in \GL(n, \bR)$ and $b \in \bR^n$. 
Or we can let $X=\bR P^n$ and $G=\PGL(n+1, \bR)$, the group of projective transformations of 
$\bR P^n$.

The complement of $\bR P^n$ of a subspace of codimension-one can be identified with 
an affine subspace. 
We realize $\Aff(\bR^n)$ as a subgroup of transformations of $\PGL(n+1, \bR)$ fixing a subspace of codimension-one
as there is an isomorphism 
\[ (A, b) \mapsto \left[ \begin{array}{cc} A & b^T \\ 0  & 1 \end{array} \right], A \in \GL(n, \bR), b \in \bR^n \]
where $b^T$ is the transpose of $b$.

%% 4:45pm March 31, 2014
An {\em $(X,G)$-structure} on an orbifold $\mathcal{O}$ is an atlas of charts from open subsets of \index{$(X, G)$-structure} \index{geometric structure} 
$X$ with finite subgroups of $G$ acting on them, and the inclusions always lift to restrictions of 
elements of $G$ in open subsets of $X$. This is equivalent to saying that the orbifold $\mathcal{O}$ has 
a simply connected manifold cover $\tilde{\mathcal{O}}$ with an immersion 
$D:\tilde{\mathcal{O}} \ra X$ and the fundamental group $\pi_1(\mathcal{O})$ acts on $\tilde{\mathcal{O}}$ properly \index{developing map} \index{holonomy homomorphism} 
discontinuously so that $h: \pi_1(\orb) \ra G$ is a homomorphism satisfying $D\circ \gamma = h(\gamma)\circ D$ for 
each $\gamma \in \pi_1(\orb)$.  Here, $\pi_1(\orb)$ is allowed to have fixed points. 
(We shall use this second definition here.)
$(D, h(\cdot))$ is called a {\em development pair} and for a given $(X,G)$-structure, it is determined only up to an action \index{development pair} 
\[(D, h(\cdot)) \mapsto (k\circ D, kh(\cdot)k^{-1}) \hbox{ for } k\in G.\] 
Conversely, a development pair completely 
determines the $(X,G)$-structure. 

%When $G$ is 
Thurston showed that an orbifold with an $(X, G)$-structure is always good, i.e., covered \index{orbifold!good} 
by a manifold with an $(X, G)$-structure. %(See Theorem 6.1.1 of \cite{msj}.)

% Deformation spaces for rp and affine structures

An {\em isotopy} of an orbifold $\mathcal{O}$ is a map $f:\mathcal{O} \ra \mathcal{O}$  \index{isotopy}
with a map $F:\mathcal{O}\times I \ra \mathcal{O}$ 
so that 
\begin{itemize}
\item $F_t:\mathcal{O} \ra \mathcal{O}$ for $F_t(x) := F(x, t)$ every fixed $t$ is an orbifold 
diffeomorphism,  
\item $F_0$ is the identity, and 
\item $f=F_1$. 
\end{itemize}
Given an $(X,G)$-structure on another orbifold $\mathcal{O}'$, 
any orbifold diffeomorphism $f:\mathcal{O} \ra \mathcal{O}'$ induces an $(X,G)$-structure 
pulled back from $\mathcal{O}'$ which is given by using 
the local models of $\mathcal{O}'$ for preimages in $\mathcal{O}$.

Suppose that $\mathcal{O}$ is compact. 
We define the {\em isotopy-equivalence space} $\widetilde{\Def}_{X, G}(\orb)$ as  \index{isotopy-equivalence space} 
the space of development pairs $(\dev, h)$ quotient by the isotopies of the universal cover $\torb$ of $\orb$.
The {\em deformation space} $\Def_{X,G}(\mathcal{O})$ is given by the quotient of $\widetilde{\Def}_{X, G}(\orb')$
by the action of $G$: $g (\dev, h(\cdot)) = (g \circ \dev, g h(\cdot) g^{-1})$. (See \cite{dgorb} for details.) 
We can also interpret as follows: 
The deformation space $\Def_{X,G}(\mathcal{O})$ of the $(X,G)$-structures is \index{deformation space} 
the space of all $(X,G)$-structures on $\mathcal{O}$ quotient by the isotopy pullback actions.

This space can be thought of as the space of pairs $(D, h)$ with compact open $C^r$-topology for $r \geq 1$ \index{deformation space!topology} 
and the equivalence relation generated by the isotopy relation 
\begin{itemize}
\item $(D, h) \sim (D', h')$ if $D'=D\circ \iota$ and $h'=h$ for a lift $\iota$ of an isotopy
and 
\item $(D, h) \sim (D', h')$ if $D'=k\circ D$ and $h(\cdot)=kh(\cdot)k^{-1}$ for $k \in G$. 
\end{itemize} 
(See \cite{dgorb} or Chapter 6 of \cite{msj}.) 

For noncompact orbifolds with end structures, similar definitions hold except that we have to modify
the notion of isotopies to preserve the end structures.

\section{Oriented real projective structures} \label{subsec:orp} 

\begin{itemize}
\item Given a vector $v \in \bR^{n+1} -\{O\} $, we denote by $[v] \in \bR P^n$ the equivalence class. 
Let $\Pi: \bR^{n+1} -\{O\} \ra \bR P^n$ denote the projection. 
\item A {\em cone} in $\bR^{n+1}$ is a subset $C$ so that if $v \in C$, then $sv \in C$ for all $s \in \bR_+$. 
\item Given a connected subset $A$ of $\bR P^n$, a cone $C_A \subset \bR^{n+1} $ of $A$ is given 
as a cone in $\bR^{n+1}$ mapping onto $A$ under the projection $\Pi: \bR^{n+1} -\{O\} \ra \bR P^n$.
\item $C_A$ is unique up to the antipodal map $\mathcal{A}: \bR^{n+1} \ra \bR^{n+1}$ given by $v \ra -v$. 
\end{itemize}

We will be using the standard elliptic metric $\bdd$ on $\bR P^n$ (resp. in $\SI^n$) where
the set of geodesics coincides with the set of projective geodesics 
up to parameterizations. 
An open hemisphere of $\SI^n$ is called an {\em affine subspace} in $\SI^n$ \index{affine subspace} 
A {\em great segment} is a geodesic arc in $\SI^n$ with antipodal p-end vertices. A {\em convex segment} is an arc contained in \index{great segment} \index{convex segment} 
a great segment. 
A {\em convex subset} of $\SI^n$ is a subset $A$ where every pair of points of $A$ connected by a convex segment. \index{convex set} \index{$\bdd$} \index{elliptic metric}

% oriented projective structure... where to put?
Recall that $\SL_\pm,(n+1, \bR)$ is isomorphic to $\GL(n+1, \bR)/\bR^+$. 
%i.e, two linear maps are the same projective automorphism if and only if 
%they differ by a positive scalar. 
Then this group acts on $\SI^n$ to be seen as a quotient space of $\bR^{n+1}-\{O\}$ by the equivalence relation \index{$\SI^n$}
\[v \sim w, v, w \in \bR^{n+1}-\{O\} \hbox{ if } v = s w \hbox{ for } s \in \bR^+.\] 
Wet let $[v]$ denote the equivalence class of $v \in \bR^{n+1} -\{O\}$. 
Given a  $V \in \bR^{n+1}$, we denote by $S(V)$ the image of $V -\{O\}$ under the quotient map. \index{$S(\cdot)$}
The image is called a {\em subspace}. A set of antipodal points is a subspace of dimension $0$.  \index{subspace} 
There is a double covering map $\SI^n \ra \bR P^n$ with the deck transformation group generated by
$\mathcal{A}$. %$: [v] \ra [-v]$ for $v \in \bR^{n+1} -\{O\}$. 
This gives a projective structure on $\SI^n$. The group of projective automorphisms is identified with $\SL_\pm(n+1, \bR)$. \index{$\SL_\pm(n+1, \bR)$}
The notion of geodesics are defined as in the projective geometry: they correspond 
to arcs in great circles in $\SI^n$. 

A collection of subspaces $P(V_1), \dots, P(V_n)$ 
(resp. $S(V_1), \dots, S(V_n)$) are {\em independent} if the subspaces $V_1, \dots, V_n$ are independent. \index{independent subspaces}

An $(\SI^n, \SL_\pm(n+1, \bR))$-structure on $\mathcal{O}$ is said to be an {\em oriented real projective structure} on $\mathcal{O}$. \index{real projective structure!oriented} 
We define $\Def_{\SI^n}(\mathcal{O})$ as the deformation space of $(\SI^n, \SL_\pm(n+1, \bR))$-structures on $\mathcal{O}$.

%%%
%Suppose that $\mathcal O$ is a projective orbifold. 
%Given an end $E$ of an orbifold $\mathcal O$, we can find a neighborhood $U$. The inverse image 
%in the universal cover $\tilde \mathcal O$ is a union of open sets $U_1, U_2,...$. 
%We can choose $U$ so that a component , say $U'$, is a cell  and there is a subgroup of $\pi_1(\mathcal{O})$ acting on $U'$.
%By radial property, there is a concurrent 
%We denote this group of $\pi_1(E)$. 

%% Cone construction
%Recall that $\SI^n$ covers $\bR P^n$ by a double covering map $q$ sending the vectors to their equivalence classes
The group $\SL_{\pm}(n+1, \bR)$ of linear transformations of determinant $\pm 1$ 
maps to the projective group $\PGL(n+1, \bR)$ by a double covering homomorphism $\hat q$, 
and $\SL_{\pm}(n+1, \bR)$ acts on $\SI^n$ lifting the projective transformations. 
The elements are also {\em projective transformations}. 
%Given a development pair $(D, h)$ of a real projective orbifold $\mathcal{O}$, we can find a lift 
%$D':\tilde{\mathcal{O}} \ra \SI^n$ since $\tilde{\mathcal{O}}$ is a simply connected
%manifold and $h$ lifts to $h':\pi_1(\mathcal{O}) \ra \GL(n+1, \bR)$. 

We now discuss the standard lifting: 
 Given a real projective structure on $\mathcal{O}$, there is a development pair $(\dev, h)$ where 
$\dev: \tilde{\mathcal{O}} \ra \bR P^n$ is an immersion and $h: \pi_1(\mathcal{O}) \ra \PGL(n+1, \bR)$ is a homomorphism. \index{developing map!lifting} 
Since $\SI^n \ra \bR P^n$ is a covering map and $\tilde{\mathcal{O}}$ is a simply connected manifold,  \index{holonomy homomorphism!lifting} 
$\mathcal{O}$ being a good orbifold, 
there exists a lift $\dev': \tilde{\mathcal{O}} \ra \SI^n$ unique up to the action of $\{\Idd, -\Idd\}$.
This induces an oriented projective structure on $\tilde{\mathcal{O}}$ and $\dev'$ is a developing map for this geometric structure.
Given a deck transformation $\gamma:\tilde{\mathcal{O}} \ra \tilde{\mathcal{O}}$, the composition
$\dev' \circ \gamma$ is again a developing map for the geometric structure
and hence equals $h'(\gamma) \circ \dev'$ for $h'(\gamma) \in \SL_\pm(n+1, \bR)$. 
We verify that $h':\pi_1(\mathcal{O}) \ra \SL_\pm(n+1, \bR)$ is a homomorphism. 
Hence, $(\dev', h')$ gives us an oriented real projective structure, which 
induces the original real projective structure. 

Again, we can define the {\em radial end structures} and {\em totally geodesic ideal boundary} for \index{end!radial end structure} 
oriented real projective structures and also horospherical end neighborhoods in obvious ways. \index{end!totally geodesic!ideal boundary} 
They correspond in the direct way in the following theorem also. 

\begin{theorem}\label{thm:doubledef} 
There is a one-to-one correspondence between the space of real projective structures on an orbifold $\orb$ 
with the space of oriented real projective structures on $\orb$. 
Moreover, a real projective diffeomorphism
of real projective orbifolds is an $(\SI^n, \SL_\pm(n+1, \bR))$-diffeomorphism of oriented real projective orbifolds
and vice versa. 
\end{theorem} 
\begin{proof} 
Straightforward. See p. 143 of Thurston \cite{Thbook}.
\end{proof}

\begin{theorem}\label{thm:vgood} 
A real projective orbifold $S$ is covered finitely by a real projective manifold $M$
and $S$ is real projectively diffeomorphic to $M/G_1$ for a finite group $G_1$ of real projective automorphisms of $M$. 
An affine orbifold $S$ is covered finitely by an affine manifold $N$, and $S$ is affinely 
diffeomorphic to $N/G_2$ for a finite group $G_2$ of affine automorphisms of $N$. 
\end{theorem} 
\begin{proof} 
Since $\Aff(\bR^n)$ is a subgroup of a general linear group, Selberg's Lemma shows that there exists 
a torsion-free subgroup of the deck transformation group.  We can choose the group to be a normal subgroup
and the second item follows. 

A real projective structure induces an $(\SI^n, \SL_\pm(n+1, \bR))$-structure and vice versa. 
Also a real projective diffeomorphism of orbifolds is an $(\SI^n, \SL_\pm(n+1, \bR))$-diffeomorphism and vice versa. 
We regard the real projective structures on $S$ and $M$ as $(\SI^n, \SL_\pm(n+1, \bR))$-structures. 
We are done by Selberg's lemma that a finitely generated subgroup of a general linear group has a torsion-free subgroup of \index{Selberg's lemma} 
finite-index. 
%Since $\SL_\pm(n+1, \bR)$ is a subgroup of a general linear group, Selberg's lemma implies our result.
\end{proof}

%April 13 2014 10:22pm

\section{Metrics} \label{subsec:metrics}
\subsection{The Hausdorff metric}
Recall the standard elliptic metric $\bdd$ on $\bR P^n$ (resp. in $\SI^n$) where \index{Hausdorff metric} \index{elliptic metric} \index{$\bdd$}
%the set of geodesics concides with the set of projective geodesics 
%up to parameterizations. 
Given two sets $A$ and $B$ of $\bR P^n$ (resp of $\SI^n$), 
\[\bdd(A, B):=\inf \{\bdd(x, y)| x \in A, y \in B\}.\]
We can let $A$ or $B$ be points as well obviously. 

%A {\em segment of $\bdd$-length $\pi$} is actually an arc in a subspace of $\SI^n$ where end points are antipodal. \index{segment of $\bdd$-length $\pi$}

The {\em Hausdorff distance} between two convex subsets $K_1, K_2$ of $\bR P^n$ (resp. of $\SI^n$)
is defined by 
\[ \bdd^H(K_1, K_2) = \inf\{\eps\geq 0|  \clo(K_1) \subset N_\eps(\clo(K_2)), \clo(K_2) \subset N_\eps(\clo(K_1))\}\]
where $N_\eps(A)$ is the $\eps$-$\bdd$-neighborhood of $A$ under the standard metric $\bdd$ of $\bR P^n$ (of $\SI^n$) for $\eps >0$. 
$\bdd^H$ gives a compact Hausdorff topology on the set of all compact subsets of $\bR P^n$ (of $\SI^n$).
(See p. 281 of \cite{Munkres}.)
%We define the {\em Hausdorff distance} between two convex subsets $K_1, K_2$ of $\bR P^n$ (resp. in $\SI^n$) in the similar way
%giving us a compact Hausdorff topology on the set of all compact subsets of $\SI^n$. 

We say that a sequence of sets $\{K_i\}$ geometrically converges to a set $K$ if $\bdd^H(K_i, K) \ra 0$. \index{$\bdd^H$}
If $K$ is assumed to be closed, then the geometric limit is unique. 

Suppose that a sequence $\{K_i\}$  of compact convex domains geometrically converges to 
a compact convex domain $K$ in $\bR P^n$ (resp. in $\SI^n$),  i.e., $\bdd^H(K_i , K) \ra 0$. 
In this case we claim that
\begin{equation}\label{eqn:partialKi}
 \bdd^H(\partial K_i, \partial K) \ra 0:
 \end{equation}   
For every point $x$ of $\partial K$, and an $\eps$-$\bdd$-ball $B_x$, $B_x \cap \partial K_i \ne \emp$ 
for sufficiently large $i$ since $B_x$ must meet $K_i$ and be not contained in $K_i$ for sufficiently large $i$. 
There exists some sequence $x_i \in \partial K_i$ so that $x_i \ra x$.
Conversely, every convergence sequence $\{x_i\}$, $x_i \in \partial K_i$, must converge to $x \in \partial K$.

Given two sets $A$ and $B$ of $\torb$ or $\orb$ with a metric $d$, we define
\[d(A, B) := \inf \{ d(x, y)| x \in A, y \in B\}.\] 
The definition obviously extends to the cases when $A$ or $B$ are points. 
%We define a maximal distance
% \[d^{\mathbf m}(A, B):=\sup \{ d(x, y)| x \in A, y \in B\}.\]  %\marginpar{This is used? Not consistently used Check...} 

%%% April 7 10:51pm... I need to check above...

\subsection{The Hilbert metric}
Let $\Omega$ be a properly convex open domain. A line or a subspace of dimension-one 
in $\bR P^n$ has a two-dimensional homogenous coordinate system. 
Let $[o, s, q, p]$ denote the cross ratio of four points on a line as defined by 
\[ \frac{\bar o - \bar q}{\bar s - \bar q} \frac{\bar s - \bar p}{\bar o - \bar p} \] 
where $\bar o, \bar p, \bar q, \bar s$ denote respectively the first coordinates of the homogeneous coordinates \index{Hilbert metric} 
of $o, p, q , s$ so that the second coordinates equal $1$. 
Define a metric for $p, q \in \Omega$, 
$d_\Omega(p, q)= \log|[o,s,q,p]|$ where $o$ and $s$ are 
endpoints of the maximal segment in $\Omega$ containing $p, q$
where $o, q$ separates $p, s$. 
The metric is one given by a Finsler metric. (See \cite{Kobpaper}.)

Given an SPC-structure on ${\mathcal{O}}$, there is a Hilbert metric which we denote by $d_{\torb}$ \index{$d_{\torb}$}
on $\tilde{\mathcal{O}}$ and hence on $\tilde {\mathcal{O}}$. 
%To be more precise, we will expend $\orb$ by adding a lens and use the universal cover $\torb'$ containing $\torb$. 
%by Lemma \ref{lem:addlens}. 
Actually, we will make $\orb$ slightly small by inward perturbations of $\partial \orb$ preserving the strict convexity of $\partial \orb$. 
The Hilbert metric will be defined on original $\torb$. (We call this metric the {\em perturbed Hilbert metric}.)
This induces a metric on ${\mathcal{O}}$, including the boundary now. 
We will denote the metric by $d_{\orb}$. 
More precisely, 

Assume that $K_i \ra K$ geometrically for a sequence of properly convex domains $K_i$ and
a properly convex domain $K$.  
Suppose that two sequences of points $\{x_i| x_i \in K_i^o\}$ and $\{y_i| y_i \in K_i^o\}$ 
converge to $x, y \in K^o$ respectively. Since the end of a maximal segments always are in $\partial K_i$
and $\partial K_i \ra \partial K$, the above shows that 
\begin{equation} \label{eqn:HiHa}
d_{K_i^o}(x_i, y_i) \ra d_{K^o}(x, y)
\end{equation} 
holds.  We omit the details of the elementary proof. 
%\marginpar{Need references.}

\section{Convexity and convex domains}\label{subsec:conv}

%The results here are originally due to Vey \cite{Vey68}.
An affine manifold is {\em convex} if every path can be homotopied to an affine geodesic with endpoints fixed. \index{affine manifold} 
A {\em complete real line} in $\bR P^n$ is a $1$-dimensional affine subspace of $\bR^n$ with denote it by $\bR$. 
In $\SI^n$, a complete real line is defined as the interior of a great segment. 
An affine manifold is {\em properly convex} if there is no affine map from $\bR$ into it; i.e., there is no
complete affine line in its universal cover. 

\begin{proposition} %[Vey  \cite{Vey68}]
\label{prop:affconv} 
An affine manifold is convex if and only if a developing map sends 
the universal cover to a convex open domain in $\bR^n$. 
An affine manifold is properly convex if and only if the developing map sends 
the universal cover to a properly convex open domain in $\bR^n$. 
\end{proposition}
\begin{proof} 
The first part is Theorem 8.1 of Shima \cite{Shbook} or Theorem A.2 of \cite{psconv}.
The second part is Theorem 8.3 of \cite{Shbook} since the hyperbolicity of Kobayashi
is equivalent to the proper convexity. (See Kobayashi \cite{Kobpaper}.)

\end{proof} 

%\begin{lemma}\label{lem:affpc}
%A properly convex subset of an affine subspace is a convex subset of a compact 
%subset of possibly another affine subspace. 
%\end{lemma}
%\begin{proof} 
%This is similar to the $2$-dimensional situation in Section 1.3 of \cite{cdcr1}.
%\end{proof}

%% July 24, 4:19

%Talk about reducible proper convex.... 

%We redefine the notion of convexity in the introduction as Proposition \ref{prop:projconv} shows the equivalence. 
A complete real line in $\bR P^n$ is a $1$-dimensional subspace of $\bR P^n$ with one point removed.
%with denote it by $\bR$ sometimes. 
That is, it is the intersection of a $1$-dimensional subspace by an affine subspace. 
A {\em convex} projective geodesic is a projective geodesic in a real projective manifold which lifts to \index{convex}
a convex segment in $\SI^n$. 
A real projective manifold is {\em convex} if every path can be homotopied to a convex projective geodesic with endpoints fixed. 
It is {\em properly convex} if there is no projective map from the complete real line $\bR$. (See Chapters 2 and 3 of \cite{psconv} for more details.) \index{convex!properly convex}

\begin{proposition}[Vey]\label{prop:projconv}
\begin{itemize}
%\item An open or closed real projective orbifold is convex if and only if the developing map sends 
%the universal cover to  a convex domain in an affine subspace of $\bR P^n$ or else the universal cover 
%is projectively diffeomorphic to $\SI^n$. 
\item A strongly tame real projective orbifold is properly convex if and only if the developing map sends 
the universal cover to a properly convex open domain bounded in an affine subspace of $\bR P^n$. 
\item If a convex real projective orbifold is not properly convex, then
its holonomy homomorphism is virtually reducible.
\end{itemize}
\end{proposition}
\begin{proof} 
%The first part follows by affine suspension and Proposition \ref{prop:affconv}. 
For the first part, the affine suspension has a developing image to a properly convex subset of 
an affine subspace. (See Section \ref{sub:asusp}.)
For the final item, see \cite{ChCh}. (See also \cite{GV}.)
%a convex subset of $\bR P^n$ is a convex subset of an affine subspace $A^n$, isomorphic to an affine space. 
%A convex subset of $A^n$ that has a complete affine line must 
%contain a maximal complete affine subspace.  
%Two such complete maximal affine subspaces do not intersect since otherwise there is a larger complete affine subspace of 
%higher dimension by convexity. We showed in \cite{ChCh} that the maximal complete affine subspaces are mutually disjoint and 
%there exists the common boundary of the affine subspaces in $\bR P^n$, a lower dimensional subspace of $\bR P^n$. (See also \cite{GV}.) 
%The subspace is preserved under the group action. 
\end{proof} 

\begin{lemma} \label{lem:locconv}
Let $K$ be a compact domain with boundary $\partial K\ne \emp$ with a local homeomorphism ${\mathcal J}$ to $\SI^n$ so that 
%$\bR P^n$ {\rm (}resp. $\SI^n${\rm )} so that 
each point has a neighborhood in $J$ imbedding onto a convex domain. Then ${\mathcal J}$ maps $K$ homeomorphically to a convex domain 
in $\SI^n$ and give $K$ a convex real projective structure with possibly nonsmooth boundary. 
%in $\bR P^n$ {\rm (}resp. in $\SI^n${\rm ).} 
\end{lemma}
\begin{proof}
%We prove for $\SI^n$. 
${\mathcal J}$ induces a real projective structure on $K$. We can show that $K$ is now $1$-convex. 
Now, Theorem A.2 of \cite{psconv} proves this. %The $\bR P^n$-version follows from this. 
\end{proof}

%See \cite{Ben3} for the proofs of the following propositions.
\begin{proposition}[Corollary 2.13 of Benoist \cite{Ben3}]\label{prop:Benoist}  
Suppose that a discrete subgroup $\Gamma$ of $\PGL(n, \bR)$ \rlp resp. $\SL_{\pm}(n, \bR)$\rrp \,
 acts on a properly convex $(n-1)$-dimensional open domain $\Omega$ in $\bR P^{n-1}$ \rlp resp. $\SI^{n-1}$\rrp \, so 
that $\Omega/\Gamma$ is compact. Then the following statements are equivalent. 
\begin{itemize} 
\item Every subgroup of finite index of $\Gamma$ has a finite center. 
 \item Every subgroup of finite index of $\Gamma$ has a trivial center. 
\item Every subgroup of finite index of $\Gamma$ is irreducible in $\PGL(n, \bR)$ \rlp resp. $\SL_{\pm}(n, \bR)$\rrp. 
That is, $\Gamma$ is strongly irreducible. 
\item The Zariski closure of $\Gamma$ is semisimple. 
\item $\Gamma$ does not contain a normal infinite nilpotent subgroup. 
\item $\Gamma$ does not contain a normal infinite abelian subgroup.
\end{itemize}
\end{proposition}
The group with properties above is said to be the group with a {\em virtually center free group} or a {\em vcf-group}. \index{vertually center free group}
%We call the group such as above theorem  By above Proposition \ref{prop:Benoist}, 
%we see that every representation of the group acts irreducibly.

\begin{theorem}[Theorem 1.1 of Benoist \cite{Ben3}] \label{thm:Benoist} 
Let $\Gamma$ be a discrete subgroup of $\PGL(n, \bR)$ {\rm (}resp. $\SL_{\pm}(n, \bR)${\rm )} with a trivial virtual center. 
Suppose that a discrete subgroup $\Gamma$ of $\PGL(n, \bR)$ {\rm (}resp. $\SL_{\pm}(n, \bR)${\rm )} acts on 
a properly convex $(n-1)$-dimensional open domain $\Omega$ so 
that $\Omega/\Gamma$ is a compact orbifold.
Then every representation of a component of $\Hom(\Gamma, \PGL(n, \bR))$ {\rm (}resp. $\Hom(\Gamma, \SL_{\pm}(n, \bR)) ${\rm )} containing the inclusion 
representation also acts on a properly convex $(n-1)$-dimensional open domain cocompactly. 
\end{theorem}

%Let $V_1$ and $V_2$ be subspaces of $\bR P^n$ meeting at a unique point or are disjoint. 
%In general, a {\em join} of two convex sets $C_1 \subset V_1$ 
%and $C_2 \subset V_2$ is defined as 
%\[ \{[t v_1 + (1-t) v_2]| v_i \in C_{C_i}, i= 1, 2 \} \] 
%where $C_{C_i}$ is a cone in $\bR^{n+1}$ corresponding to $C_i$, $i=1, 2$. 
%The join is denoted by $C_1 * C_2$ in this paper. 
Given subspaces $V_1, \dots, V_m \subset \bR P^n$ (resp.  $\subset \SI^n$) where 
any two are mutually disjoint, 
and a subset $C_i \subset V_i$ for each $i$, 
we define a {\em strict join} of $n$ sets $C_1, \dots, C_m$
\[ C_1 * \cdots * C_m := \left\{ [ \sum_{i=1}^m t_i v_i ] | \sum_{i=1}^m t_i = 1, t_i \in [0, 1], v_i \in C_{C_i} \right\},\]
where $C_{C_i}$ is a cone in $\bR^{n+1}$ corresponding to $C_i$.  (Of course, this depends on the choices of $C_{C_i}$
up to $\mathcal{A}$.) %: \bR^{n+1} \ra \bR^{n+1}$.) 

A point $x$ of a strict join $C_1 \ast \cdots \ast C_j$ for convex sets $C_i$ has 
{\em join-coordinates} $[\lambda_1, \dots, \lambda_j]$ if $x =[ \sum_{i=1}^k \lambda_i \vec{v}_i]$ for  \index{join!strict} \index{join!coordinates} 
$\vec{v}_i $ a vector in the cone corresponding to $C_i$. 

A {\em cone-over} a strictly joined domain is one containing a strictly joined domain $A$ and 
is a union of segments from a cone-point $\not\in A$  to points of $A$.
%where the cone point is given by $V'\cap V''$.
%A {\em generalized lens-shaped domain} is a join of lens-domains. 

%The decompositions of type below will be called the {\em Vey decompositions}. 
\begin{proposition}[Theorem 1.1 of Benoist \cite{Ben3}] \label{prop:Ben2} Assume $n \geq 2$. 
Let $\Sigma$ be a closed $(n-1)$-dimensional properly convex projective orbifold 
and let $\Omega$ denote its universal cover in $\bR P^{n-1}$ \rlp resp. in $\SI^{n-1}$\rrp.
Then 
\begin{itemize}
\item[(i)] $\Omega$ is projectively diffeomorphic to the interior of 
a strict join $K_1 * \cdots * K_{l_0}$ where $K_i$ is a properly convex open domain of dimension $n_i \geq 0$
corresponding to a convex cone $C_i \subset \bR^{n_i+1}$. 
\item[(ii)] $\Omega$ is the image of the interior of $C_1 \oplus \cdots \oplus C_r$. 
\item[(iii)] The fundamental group 
$\pi_1(\Sigma)$ is virtually isomorphic to $\bZ^{l_0-1} \times \Gamma_1 \times \cdots \times \Gamma_{l_0}$ for 
$l_0 + \sum n_i = n$. Each $\Gamma_i$ has the property that each finite index subgroup has a trivial center. 
\item[(iv)] Each $\Gamma_j$ acts on $K_j$ cocompactly and the Zariski closure is 
a semi-simple Lie group in $\PGL(n_j+1, \bR)$ \rlp resp. in $\SL_\pm(n_j+1, \bR)$\rrp, 
and acts trivially on $K_m$ for $m \ne j$. 
\item[(v)] The subgroup corresponding to $\bZ^{l_0-1}$ acts trivially on each $K_j$
\end{itemize} 
\end{proposition} 
Supposing that $\pi_1(\Sigma)$ is admissible, the Zariski closure of $\Gamma_j$
is one of $O(n_j+1, 1)$, $\PGL(n_j+1, \bR)$, $\SL_\pm(n_j+1, \bR)$, or a union of their components. 

%A geodesic $l$ in $\Sigma$ is {\em forward recurrent} if for every $t \in \bR$ there exists a sequence $\{t_i\}$
%where  $t_i \ra \infty$ so that $d_{\Sigma}(l(t), l(t_i)) \ra 0$. 
%The geodesic $l$ is {\em reverse recurrent} if we can choose such sequence $\{t_i\}$ where $t_i \ra -\infty$.  
%
%\begin{corollary} \label{cor:rec} 
%Let $\Omega$ be a domain in $\bR P^n$ {\rm (}resp. $\SL_{\pm}(n, \bR)${\rm )} 
%covering a closed $(n-1)$-dimensional properly convex projective orbifold $\Sigma$. 
%Every point of $\Bd \Omega$ is an end point of 
%geodesic $l$ mapping to a forward recurrent or reverse recurrent geodesic in $\Sigma$. 
%\end{corollary} 
%\begin{proof} 
%We will prove for the case when $\Sigma$ is covered by a domain in $\bR P^{n-1}$. 
%If $\Gamma$ is hyperbolic, then this follows from the results in \cite{Ben1}. 
%If $\Gamma$ is a product group $\bZ^{m-1} \times \Gamma_1 \times \cdots \times \Gamma_m$, 
%then $\clo(\Omega) = K_1 \ast \cdots \ast K_m$ where $K_i$ 
%is a properly convex domain with hyperbolic or trivial group $\Gamma_i$ acting on it 
%for $i$, $i=1, \dots, m$. 
%
%Each point $x$ of $\Bd \Omega$ is on $K_{i_1} \ast \cdots \ast K_{i_j}$ for some collection $\{i_1, \dots, i_j\}$. 
%Let $l$ be a geodesic in $\Omega$ with end points in $x \in K_{i_1} \ast \cdots \ast K_{i_j}$ 
%and in $K_{i_{j+1}} \ast \cdots \ast K_{i_k}$ for the complement $\{i_{j+1}, \dots, i_k\}$ in $\{1,\dots, m\}$. 
%Since $\bZ^{m-1}$ is discrete cocompact in the diagonalizable matrix group $\bR^{m-1}$, 
%we obtained the desired $l$. 
%\end{proof}
%

A {\em convex hull} of a compact subset $A$ of $\bR P^n$ is defined 
as the smallest closed convex subset containing $A$  if $A$ is a bounded subset of an affine subspace of $\bR P^n$. \index{convex hull} 
This definition is independent of the choice of the affine subspace. 
A {\em convex hull} of a compact subset $A$ of $\SI^n$ is defined 
as the smallest closed convex subset containing $A$ in $\SI^n$. (A pair of antipodal points does not have a convex hull.)
By the compactness of $\bR P^n$ (resp. $\SI^n$), a convex hull of a compact set $A$ 
is a union of the set $S_1$ of $1$-simplices with endpoints in the closure of $A$ and 
the set $S_2$ of $2$-simplices with boundary edges in $S_1$ and $S_i$ of $i$-simplices with 
boundary sides in $S_{i-1}$ for $i=3, 4, \dots, n$. 
We denote it by $CH(A)$.

\section{Geometric convergence of convex real projective orbifolds} 

We say that a set $A$ {\em span} a subspace $S$ in $\bR P^n$ (resp, $\SI^n$) if $S$ is the smallest subspace containing $A$. \index{subspace!span} 
Now Proposition \ref{prop:permconv} covers the case of Corollary \ref{cor:smvar} when $\Gamma$ is virtually center-free or Benoist
since $K_t$ are always properly convex. 

\begin{corollary}\label{cor:smvar} 
Suppose that the fundamental group $\Gamma$ of a closed $(n-1)$-orbifold $\Sigma$ is 
a virtual product of hyperbolic groups and abelian groups. 
%Suppose that $\Gamma$ has a finite-index subgroup % where the central abelian subgroup has only positive 
%eigenvalues. 
We are given a path $\mu_t$, $t \in [0, 1]$, 
of convex $\bR P^{n-1}$-structures on $\Sigma$ equipped with $C^r$-topology, $r\geq 2$. 
Suppose that $\mu_0$ is properly convex or complete affine with abelian holonomy.
\begin{itemize} 
\item We can find a  family of developing maps $D_t$ to $\bR P^{n-1}$  \rlp resp. in $\SI^{n-1}$\rrp\, continuous in the $C^r$-topology
and a continuous family of holonomy homomorphisms $h_t: \Gamma \ra \Gamma_t$ 
so that $K_t := \clo(D_t(\tilde \Sigma))$ is a uniformly continuous family of convex domains in $\bR P^{n-1}$ \rlp resp. in $\SI^{n-1}$\rrp\, under 
the Hausdorff metric topology of the space of closed subsets of $\bR P^{n-1}$ \rlp resp. $\SI^n$\rrp.
\item In other words, 
given $0< \eps < 1/2$ and $t_0, t_1 \in [0, 1]$, we can find $\delta > 0$ such that if $|t_0 -t _1| < \delta$, then 
$K_{t_1} \subset N_\eps(K_{t_0})$ and $K_{t_0} \subset N_\eps(K_{t_1})$.
\item Also, given $0 < \eps < 1/2$ and $t_0, t_1 \in [0,1]$, we can find $\delta >0$
such that if $|t_0 -t _1| < \delta$, then 
$\partial K_{t_1} \subset N_\eps(\partial K_{t_0})$ and $\partial K_{t_0} \subset N_\eps(\partial K_{t_1})$.
\item Finally, $\mu_t$ is always virtually immediately deformable to a properly convex structure. 
\end{itemize} 
Alternatively, these are true whenever we choose $(D_t, h_t)$ so that $h_t$ is chosen to be a continuous path 
\[h_t: [0, 1] \ra \Hom(\Gamma, \PGL(n, \bR))\,  (\hbox{resp. } h_t: [0, 1] \ra \Hom(\Gamma, \SL_{\pm}(n, \bR))).\]
\end{corollary}
\begin{proof} 
%Since \[ \hol: \Def'(\Sigma) \ra \Hom(\Gamma, \PGL(n, \bR))\] is continuous. 
We will prove for $\SI^{n-1}$. The $\bR P^{n-1}$-version follows from this. 

We assume first that $\Gamma$ is a virtually-center-free group. 
First, for any sequence $t_i$, we can choose a subsequence $t_{i_j}$ so 
that $\{K_{t_{i_j}}\}$ converges to a compact set $K_\infty$ in the Hausdorff metric.
$h_{t_0}(\Gamma_0)$ acts on $K_\infty$ as in Choi-Goldman \cite{CG}.
By Benoist \cite{Ben3}, $K_\infty$ is a properly convex domain. 

Let us fix $K_\infty$.
Now, for any sequence $t_i$, suppose that a convergent subsequence $K_{t_{i_j}}$ 
converges to $K'_\infty$. Then we claim $K_\infty = K'_\infty$:
This follows since the set of attracting and repelling fixed points of elements of 
$h_{t_0}(\gamma)$ for every $\gamma \ne \Idd$ exist and is in $\partial K'_\infty$
and $\partial K_\infty$ by the $\Gamma$-invariance.
They are also dense by Theorem 1.1 of \cite{Ben} and the density of periodic orbits of Anosov flows and hence 
$\partial K'_\infty = \partial K_\infty$.
This implies $t \mapsto K_t$ is a continuous function by a well-known result in metric topology. 
This complete the proof in this case. 

The join $K_1^t \ast \cdots \ast K_k^t$ is properly convex always by Proposition \ref{prop:permconv}(i), 
The subspace $V_i^t$ spanned by $K_i^t$ depends continuously on $t$ since the holonomy of generators of $\Gamma_i$ determines $K_i^t$ uniquely. 
This completes the proof for the joined cases of this type. 

%The proof also follows for $\Gamma$ that is Benoist. 

%If $\Gamma$ is Benoist, then the argument is similar since the center is a group of diagonalizable matrices
%that can converge to such a group only by Theorem \ref{thm:permconv}. 

Suppose now that $\Gamma$ is virtually abelian. 
Then $\Omega_t$ is determined by the generators of the free abelian subgroup $\Gamma'$ of 
finite index with positive eigenvalues only by Lemma \ref{lem:realeign}. 
$\Gamma'$ determines the connected abelian Lie group $\Delta_t$ containing $h_t(\Gamma')$ 
and $\Omega_t$ is an orbit of $\Delta_t$ by Lemma \ref{lem:realeign}.  Now Lemma \ref{lem:gconv} implies the first item. 
Now to finish the proof, we suppose that $\Gamma$ has hyperbolic factors. 
By Proposition \ref{prop:permconv}(iii), $\Gamma$ acts on 
$d(\Delta_t) \ast K_1^t \ast \cdots \ast K_k^t$ for $t \in [0, 1]$
where $\Delta_0$ is a subgroup of the Zariski closure of the center of $h_0(\Gamma')$.
%$\clo(\Omega_t) = d(L_t
%Then each hyperbolic factor acts on $V_i^t$ as an irreducible semisimple group and acts divisibly 
%on the interior of a properly convex open domain $K_i^t \subset S(V_i^t)$. 
%Let $W^t$ be the direct sum $V_1^t \oplus \cdots \oplus V_k^t$. 
%The  finite-index free abelian part $\Gamma'$ in the center acts on $W^t$. 

%For $m_0$, $\Gamma$ acts on the interior of the join $d(\Delta_0) \ast K_1^0 \ast \cdots \ast K_k^0$
%where $\Delta_0$ is a subgroup of the Zariski closure of the center of $h_0(\Gamma')$.
%For sufficiently small $0 < \eps$, $h_t(\Gamma')$ acts on the interior of 
%$\hat K_t := d(\Delta_t) \ast K_1^t \ast \cdots \ast K_k^t$ for $t \in [0, \eps)$. 
%Since we can choose $K_t^o \cap \hat K_t^o \ne \emp$ for $t \in [0, \eps)$. 
%by considering the limit sets we obtain that $K_t =\hat K_t$ for $t \in [0, \eps)$. 

%Then given the finite-index free abelian part $\Gamma'$ in the center
%has a Zariski closure $\Delta_t$ as above. 
%The subspace $S(W^t)$ is disjoint from the subspace $J_t$ containing $d(\Delta_t)$. 
%This is true for $t=0$ since $\mu_0$ is virtually imediately deformable to a properly convex real projective structure. 
%We can show that $S(W^t) \cap J_t =\emp$ for $t \in [0, \eps]$ by Lemma \ref{lem:disj}. 
%Hence, the set of $t$ where $K_t = \hat K_t =  d(\Delta_t) \ast K_1^t \ast \cdots \ast K_k^t$ 
%is open and claosed. 
%By Proposition \ref{prop:permconv}(iii), we obtain that $\clo(\Omega)$ is the strict join 
%$d(\Delta_t) \ast K_1^t \ast \cdots K_k^t$ for the totally geodesic and the convex orbit $d(\Delta_t)$ of $\Delta_t$.
The result for $K_1^t \ast \cdots \ast K_k^t$ is done above. 
Our proof reduces to the case of $d(\Delta_t)$ and $\Delta_t$ only. This was already accomplished above. 

%This follows since $\Omega$ is the interior of the join $d(\Delta)\ast K_1 \ast \cdots \ast K_k$.%
%and by Proposition \ref{prop:vip}. 

The second item follows from the first one. The third one can be deduced by equation \eqref{eqn:partialKi}. 
The final item follows by Propositions \ref{prop:permconv}(i),(iii) and \ref{prop:vip}.

The final alternative formulation follows: Proposition \ref{prop:permconv}(iii)  determines the image of $D_t$ by the convexity of $\mu_t$. 
Each of the join part of the image of $D_t$ depends continuously on $h_t$ by Lemma \ref{lem:realeign} and by the result proved above in the proof
for properly convex domains denoted by as  $K_1^t \ast \cdots \ast K_k^t$. 
\end{proof}

%% May 1st 12:34 pm  Something not quite clear...

%\begin{example} \label{exmp:def}
%Corollary \ref{cor:smvar} is applicable to a compact affine $i$-manifold 
%with translations as the holonomy group. We can show that the affine $i$-manifold
%is deformable to a properly convex real projective manifold.  
%Basically, we start with a properly convex $1$-manifold $A_1$ and a parameter to a complete affine $1$-manifold. 
%This can be given by a parameter of open lines $I_t$ and a projective transformations $g_t$ that acts on $I_t$
%properly discontinuously. $\clo(I_t) $ converges to a $\bdd$-length $=\pi$ and $g_t$ converges to a unipotent element. 
%Next, we take a join of $I_t$ with a point $p_t$ and obtain a properly convex $2$-manifold $A_2$. 
%There is a parameter of $h_t$ fixing $I_t$ and $p_t$

%\end{example} 

%%% March 8th 9:04 pm 2014

%%% May 14 4:19pm

\section{ p-ends, p-end neighborhoods, and p-end groups } \label{sub:pend}

%Consider a convex real projective orbifold $\orb$. 

Let $\orb$ be a real projective orbifold with the universal cover $\torb$. 
We fix a developing map $\dev$ in this subsection. 
%a convex open domain in $\bR P^n$.
%Since $\orb$ is convex, we always put $\dev(\torb)$ to a convex domain in an affine subspace in $\bR P^n$. 
Given a radial end of $\orb$ and an end neighborhood $U$ of a product form $E \times [0, 1)$ with a radial foliation, \index{end!p-end} 
we take a component $U_1$ of $p^{-1}(U)$ and the lift of the radial foliation with leaves whose 
developing image end at a common point $x$ in $\bR P^n$. 
\begin{itemize}
\item We call $x$ the {\em pseudo-end vertex} of $\torb$. 
$x$ will be denoted by $v_{\tilde E}$ if its neighborhoods corresponds to a p-end $\tilde E$. \index{end!p-end vertex} 
\item Generalizing further, we call an open subset $U$ of $\torb$ containing a proper pseudo-end neighborhood
of $\tilde E$, where $\pi_1(\tilde E)$ acts on a {\em pseudo-end neighborhood.} \index{end!p-end neighborhood} 
A proper pseudo-end neighborhood is an example. 
(In the following``pseudo" will be shorted to ``p".)
\item Let $\SI^{n-1}_{v_{\tilde E}}$ denote the space of equivalence classes of rays from $v_{\tilde E}$ diffeomorphic to an $(n-1)$-sphere \index{$\SI^{n-1}_{v_{\tilde E}}$}
where $\pi_1({\tilde E})$ acts as a group of projective automorphisms.  
Here, $\pi_1({\tilde E})$ acts on $v_{\tilde E}$ and sends leaves to leaves in $U_1$. 
\item Given a p-end $\tilde E$ corresponding to $v_{\tilde E}$, 
we denote by $R_{v_{\tilde E}}(\torb)= S_{\tilde E}$ the space of directions of developed leaves under $\dev$  oriented away from $v_{\tilde E}$ in $\tilde{\mathcal{O}}$.
The space develops to $\SI^{n-1}_{\tilde E}$ by $\dev$ as an immersion. \index{$R_{v_{\tilde E}}(\torb)$} \index{$S_{\tilde E}$}
\item  Also, for a subset $K$ of $\orb$, we denote by $R_{v_{\tilde E}}(K)$ the space of directions of 
developed images of leaves in $\torb$ under $\dev$ mapping to rays oriented away from $v_{\tilde E}$ ending at $K$.  
 We have $R_{v_{\tilde E}}(\torb), R_{v_{\tilde E}}(K) \subset \SI^{n-1}_{v_{\tilde E}}$ if 
 $\torb$ is a convex domain in $\bR P^n$. 
%\item The space of radial leaves from $v_{\tilde E}$ form a convex domain $S_{\tilde E}$ in $\SI^{n-1}_{v_{\tilde E}}$ with an obvious real projective 
%structure induced from $\SI^{n-1}_{v_{\tilde E}}$. 
\item Recall that $S_{\tilde E}/\Gamma_{\tilde E}$ is projectively diffeomorphic to the {\em end orbifold} 
to be denoted by $\Sigma_E$. 
\end{itemize}

Given a totally geodesic end of $\orb$ and an end neighborhood $U$ of the product from $E \times [0, 1)$
with a compactification by a totally geodesic orbifold $E'$, 
we take a component $U_1$ of $p^{-1}(U)$ and a convex domain $S_{\tilde E}$, the ideal boundary component, 
developing to totally geodesic hypersurface under $\dev$.
Here $\tilde E$ is the p-end corresponding to $E$ and $U_1$.
There exists a subgroup $\Gamma_{\tilde E}$ acting on $S_{\tilde E}$. 
Again ${S_{\tilde E}}/\Gamma_{\tilde E}$ is projectively diffeomorphic to the {\em end orbifold} to be denote by $S_E$ again. 

\begin{itemize} 
\item We call $S_{\tilde E}$ the {\em p-ideal boundary component} of $\torb$. \index{end!p-ideal boundary component} 
\item The group $\Gamma_{\tilde E}$ is said to be a {\em p-end-fundamental group associated with $\tilde E$}. \index{end!p-end fundamental group}
\item We call $U_1$ a {\em proper p-end neighborhood of $\tilde E$}. 
\end{itemize} 
Generalizing further an open subset $U$ of $\torb$ containing a proper p-end-neighborhood of $\tilde E$, 
where $\pi_1(\tilde E)$ acts on, is said to be a {\em p-end neighborhood }. \index{end!p-end neighborhood}

\section{The admissible groups} \label{sub:gpadmissible} 

If every subgroup of finite index of a group $\Gamma$ has a finite center, $\Gamma$ is said to be a {\em virtual center-free group} or 
a {\em vcf-group}.
An {\em admissible group} is a finite extension of a finite product of  \index{admissible group}
$\bZ^{l} \times \Gamma_1 \times \cdots \times \Gamma_k$ 
for infinite hyperbolic groups $\Gamma_i$  where $l \geq k-1$ or $k=1$ holds and $l+ k \leq n$ holds. 
(See Section \ref{sub:ends} for details.
$l \geq k-1$ follows from the result of Benoist discussed there. We have $k=1$ and $l=0$ if and only if the end fundamental group is hyperbolic. 
For example, if our orbifold has a complete hyperbolic structure, then end fundamental groups are virtually free abelian.) 
%If we understand the geodesic ergodicity properties of vcf-groups, we would be able to expand 
%the class of groups of $\Gamma_i$. However, we are not able to understand the properties yet. 
 We will also say that an end is 
\begin{itemize} 
\item{\em hyperbolic} if the end fundamental group is hyperbolic, i.e., $k=1, l=0$ and 
\item is {\em Benoist} if $l+1=k \geq 1$ or $l= k  \geq 1$.
Benoist ends are said to be {\em permanently properly convex}. (See Proposition \ref{prop:permconv}(i).)  \index{convex!permanently properly convex} \index{Benoist group}
\end{itemize}
Hyperbolic ends are Benoist. 
(Of course, these definitions also apply to p-ends.)

%We also require 
%(See Drutu and Sapir \cite{DuSa}.)

\section{The admissible ends} \label{sub:admissible} 
%%% April 23, 2:57... Need to work...
Let $\orb$ a convex real projective orbifold with the universal cover 
$\torb$. %a convex open domain in $\bR P^n$ (resp. in $\SI^n$). 
%Define $\Bd A$ for a subset $A$ of $\bR P^n$ (resp. in $\SI^n$) to be the {\em topological boundary} in $\bR P^n$ (resp. in $\SI^n$)
%and define $\partial A$ for a manifold or orbifold $A$ to be the {\em manifold or orbifold boundary}. 
%The closure $\clo(A)$ of a subset $A$ of $\bR P^n$ (resp. of $\SI^n$) is the topological closure in $\bR P^n$ (resp. in $\SI^n$). 
\begin{itemize} 
\item A subdomain $K$ of $\bR P^n$ (resp. in $\SI^n$) is said to be {\em horospherical} if $\clo(K)$ is bounded by an ellipsoid \index{horosphere} 
and the boundary $\partial K$ is diffeomorphic to $\bR^{n-1}$ and $\Bd K - \partial K$ is a single point. 
\item $K$ is {\em lens-shaped} if it is a convex domain and the manifold-boundary $\partial K$ is a disjoint union of two smoothly embedded  \index{lens-shaped} 
$(n-1)$-cells $A_1$ and $A_2$ not containing any straight segment in them. %$\Bd K - \partial K$ is homeomorphic to a sphere of dimension $n-1$. 
$K$ is a {\em generalized lens} if we allow a component of $\partial K$ to be not necessarily smooth. 
\item A {\em cone} over a point $x$ and a set $A$ in an affine subspace of $\bR P^n$ (resp. in $\SI^n$), $x \not\in A$ is the set \index{cone} 
given by $x\ast A$ in $\bR P^n$ (resp. in $\SI^n$).
%where $v_x$ is a vector in direction of $x$ and $C_A$ is the cone in $\bR^{n+1}$ corresponding to $A$. 
(Here no vector in $C_A$ is same or antipodal to $v_a$.
Of course, this depends on the choice of $v_x$ and $C_A$ determined up to the inversion $v \ra -v$. )
\item A {\em lens-cone} is a cone $C:= \{x\} \ast L - \{x\}$ over a lens-domain $L$ and a point $x$, $x\not\in \clo(L)$, 
so that every maximal segment from $x$ in $C$ ends in one component $\partial_1 L$ of $\partial L$. 
A {\em generalized lens-cone} is a cone over a generalized lens-domain with the same properties
where a nonsmooth component has to be in the boundary of the cone. 
For two components $A_1$ and $A_2$ of $\partial L$, 
$A_1$ is called a {\em top hypersurface} if it faces the exterior of the join and $A_2$ is then called a {\em bottom hypersurface}.
\item A {\em lens} of a lens-cone $C$ is is the lens-shaped domain $A$ so that $C = \{ x \} \ast A -\{x\}$ for a point $x \not\in \clo(A)$. \index{lens} 
\item A {\em totally-geodesic subdomain} is a convex domain in a hyperspace. 
\item A {\em cone-over} a totally-geodesic domain $A$ is a cone over a point $x$ not in the hyperspace. \index{cone} 
\end{itemize} 
Any subset of $\torb$ developing diffeomorphic to the above sets under $\dev$ with $x$ being an end vertex 
is named by the same name. 
We will also call a real projective orbifold with boundary to be 
\begin{itemize}
\item a {\em horospherical} or 
\item a {\em lens-cone} or \index{lens-cone}
\item a {\em lens}, provided it is compact, \index{lens} 
\end{itemize} if it is covered by such domains and is homeomorphic to a closed $(n-1)$-orbifold times an interval.

%%% April 17 12:27 Consolidate JOIN stuffs.... 

%By an end of an orbifold, we will mean the end in the orbifold or the corresponding object in the universal cover. 
%We see that each end corresponds to a unique end vertex fixed by the holonomy group of the end. 

We introduce some relevant adjectives: 
Let $\Sigma_E$ be an $(n-1)$-dimensional end orbifold corresponding to a p-end $\tilde E$, and
let $\mu$ be a holonomy homomorphism 
\[\pi_1(\tilde E) \ra \PGL(n+1, \bR) \hbox{ (resp. } \SL_{\pm}(n_1, \bR) \hbox{)}  \] 
restricted from that of $\orb$. 
\begin{itemize} 
\item Suppose that $\mu(\pi_1(\tilde E))$
acts on a (generalized) lens-shaped domain $K$ in $\bR P^n$ (resp. in $\SI^n$) with boundary
a union of two open $(n-1)$-cells $A_1$ and $A_2$ and $\pi_1(\tilde E)$ acts properly on $A_1$ and $A_2$. 
Then $\mu$ is said to be a ({\em generalized}) {\em lens-shaped} representation for $\tilde E$ with respect to $x$.
%We define a {\em generalized lens-shaped} representation in the same manner. 
\item $\mu$ is a {\em totally-geodesic} representation if 
$\mu(\pi_1(\tilde E))$ acts on a totally-geodesic subdomain. 
\item If $\mu(\pi_1(\tilde E))$ acts on a horoball $K$, then 
$\mu$ is said to be a {\em horospherical} representation. 
In this case, $\Bd K - \partial K=\{v_{\tilde E} \}$ for the p-end vertex $v_{\tilde E}$ of $\tilde E$. 
\item If $\mu(\pi_1(\tilde E))$ acts on a strict joined domain,
then $\mu$ is said to be a {\em strict joined} representation. 
\end{itemize} 
%Since it is clear when a representation is a joined one, we simply use the term lens-shaped representation to denote 
%generalized lens-shaped one and just lens-shaped ones. 

%We will only study ends that are horospherical or lens-shaped for simplicity.
%Consider a properly convex end of an orbifold with radially foliated end neighborhood. 
%A radial end is {\em lens-shaped} or {\em totally geodesic} according to whether 
%it has a neighborhood isomorphic to the cone over such a domain with the cone vertex equal to the p-end vertex of 
%the leaves of the radial foliation.

%These are the main possibilities in an inductive sense that we can consider when 
%we are studying the orbifolds with end real projective orbifold structures that deform to ones 
%of the properly convex $(n-1)$-dimensional real projective orbifolds
%by Proposition \ref{prop:Benoist} (see Benoist \cite{Ben1}).
%(Unfortunately, we are unable to study other groups. For this, we need to classify 
%convex but not properly-convex real projective $n$-manifolds as occurring in \cite{ChCh}.
%Other possibilities to restrict the groups to be nilpotent but the same difficulty is there.)

A {\em concave p-end-neighborhood} is an imbedded p-end neighborhood of form
 $L - C'$ contained in a radial p-end neighborhood $L$ in $\tilde{\mathcal{O}}$ with end vertex $v_{\tilde E}$ \index{end!p-end neighborhood!concave} 
 where $\dev(L)$ is a generalized lens-cone $v_{\tilde E} \ast \dev(C')$ for a generalized lens $\dev(C')$. \index{end!generalized lens-cone} 
%where $L$ is projectively diffeomorphic to $\dev(C) \ast v_{\tilde E} \subset \bR P^n$ (resp. $\subset \SI^n$)
%where $\dev(C)$ is a generalized lens so that $v_{\tilde E} \not\in \clo(\dev(C))$
%and $C'$ is the subset of $L$ corresponding to $C$.

We redefine the definition given in the introduction. The equivalence is obvious by taking 
the closures in $\torb$ of lens p-end neighborhoods or concave p-end neighborhoods where 
the end-fundamental groups act on. 
\begin{definition}{(Admissible ends)} \label{defn:admissible} 
Let $\orb$ be a real projective orbifold with the universal cover $\torb$. % a convex open domain in $\bR P^n$ (in $\SI^n$).
Let $E$ be an end of $\orb$ and $\tilde E$ be the corresponding p-end with the admissible p-end fundamental group 
$\pi_1(\tilde E)$. 
\begin{itemize} 
\item We say that the radial p-end $E$ of $\orb$ is {\em horospherical} if $\tilde E$ has a p-end neighborhood that 
is a horoball in $\torb$.
\item We say that the radial end $E$ of $\orb$ is of {\em lens-type} if $\tilde E$ has a p-end neighborhood that 
is a lens-cone. $E$ is of {\em generalized lens-type} if $\tilde E$ has a concave p-end neighborhood. 
Equivalently, $\torb$ has the interior of a generalized lens giving us a p-end neighborhood of $\tilde E$. 
 in a generalized lens-cone
where $\pi_1(\tilde E)$ acts on. 
(We require that the cone-point has to be the end point of the radial lines 
for the given radial p-end for the two cases above.) 
%\item A totally geodesic p-end $E$ is of {\em lens-type} if the ideal boundary end orbifold $S_{\tilde E}$
%has a lens-neighborhood $L$  in an ambient orbifold containing $\torb$. 
%For a component $C_1$ of $L - S_{\tilde E}$ inside $\torb$, 
%$C_1 \cup S_{\tilde E}$ is said to be the {\em one-sided neighborhood} of $S_{\tilde E}$. 
\end{itemize} 
An end is {\em admissible} (resp. {\em admissible in a generalized sense}) 
if it is a radial horospherical or radial lens-type (resp. generalized lens-type) or totally geodesic lens p-end. \index{end!admissible} \index{end!generalized admissible} 
\end{definition} \index{end!radial lens-type} \index{end!horospherical} \index{end!totally geodesic!lens} 
%that the radial end $E$ of $\orb$ is of {\em lens-type} if $\tilde E$ has a p-end neighborhood that 
%is a lens-cone. $E$ is of {\em generalized lens-type} if $\tilde E$ has a concave p-end neighborhood. 
%Equivalently, $\torb$ has the interior of a generalized lens giving us a p-end neighborhood of $\tilde E$. 
Recall that a totally geodesic p-end $E$ of lens-type has the lens-condition that the ideal boundary end orbifold $S_{\tilde E}$
has a lens-neighborhood $L$  in an ambient orbifold containing $\torb$. 
For a component $C_1$ of $L - S_{\tilde E}$ inside $\torb$, 
$C_1 \cup S_{\tilde E}$ is said to be the {\em one-sided neighborhood} of $S_{\tilde E}$. \index{one-sided neighborhood} 

(Note that the notion of a totally geodesic radial end can be of lens-type but it is a different concept from that of a 
totally geodesic lens-type end. However, one can be converted to the other using some geometric operations of cutting and pasting.)

%(They are dual to lens-shaped ends. See Section 3.4 of \cite{endclass}.) 
%We will study lens-shaped ends in \cite{endclass}. \marginpar{Where precisely?} 
%and show that they are {\em strict lense-shaped}, i.e., the complement of the boundary is nowhere dense. 
%The lens-shaped ends are introduced as naturally from two-dimensional analogies. 
%We believe that they are essential in studying the radial ends and without this condition, and otherwise
%we might loose too much control. The condition generalizes that of Goldman \cite{Gconv} for 
%convex real projective surfaces with totally geodesic boundary. Without this condition, 
%we can study these surfaces as well as the author did (see \cite{cdcr1} and \cite{cdcr2}); however, there are some degeneracies. 

\begin{figure}
\centerline{\includegraphics[height=7cm]{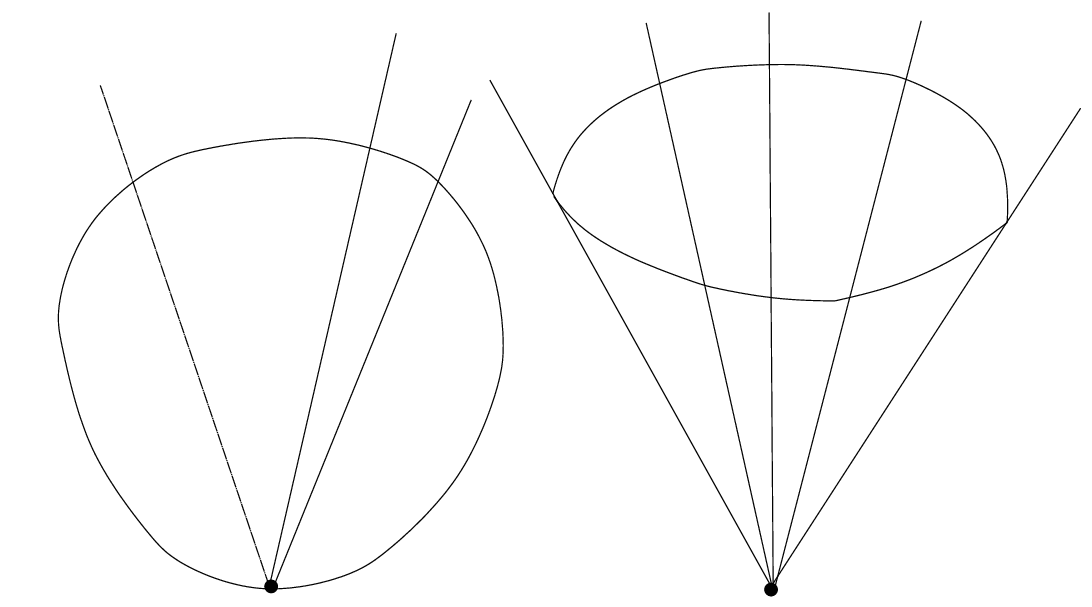}}
\caption{The universal covers of horospherical and lens shaped ends. The radial lines form cone-structures.}
\label{fig:lens}
\end{figure}

\section{The duality} \label{sub:duality} 
A {\em dilatation} is an affine automorphism $\bR^{n+1} \ra \bR^{n+1}$ of the affine space 
given by $v \ra s (v - w_0) + w_0$ for $s > 0, s \ne 1$ and an arbitrary point $w_0 \in \bR^{n+1}$. 
Here $s$ is the {\em expansion factor} of the dilatation and is uniquely determined by \index{dilatation} 
the dilatation, and $w_0$ is the fixed point.

We starts from linear duality. Let us choose the origin $O$ in $\bR^{n+1}$. 
Let $\Gamma$ be a group of linear transformations $\GL(n+1, \bR)$. 
Let $\Gamma^*$ be the {\em affine dual group} defined by $\{g^{\ast -1}| g \in \Gamma \}$ acting on 
the dual space $\bR^{n+1, \ast}$. 
Suppose that $\Gamma$ acts on a properly convex cone $C$ in $\bR^{n+1}$ with the vertex $O$.

An open convex cone  $C^*$ in $\bR^{n+1, *}$  is {\em dual} to an open convex cone $C $ in $\bR^{n+1}$  if 
$C^* \subset \bR^{n+1 \ast}$ is the set of linear transformations taking positive values on $\clo(C) -\{O\}$.
$C^*$ is a cone with vertex as the origin again. Note $(C^*)^* = C$. 

Now $\Gamma^*$ will acts on $C^*$. Also, if $\Gamma$ acts cocompactly on $C$ if and only if $\Gamma^*$ acts on $C^*$ cocompactly. 
A {\em central dilatation extension} $\Gamma'$ of $\Gamma$ by $\bZ$ is given by adding a dilatation by a scalar $s \in \bR_+ -\{1\}$ \index{central dilatation extension} 
with the fixed $O$.
The dual of $\Gamma'$ is a central dilatation extension of $\Gamma^*$. 

%Given an automorphism $g: \bR^{n+1} \ra \bR^{n+1}$, 
%Given a subgroup $\Gamma$ of $\SL_\pm(n+1, \bR)$, the {\em dual group} $\Gamma^*$ is 
%the group given by $\{ g^{-1 \ast}| g \in \Gamma\}$. 
 %$\Gamma$ and $\Gamma^*$ are isomorphic by the map $g \mapsto g^{* -1}$.
 Given a subgroup $\Gamma$ in $\Pgl$, an {\em affine lift} in $\GL(n+1, \bR)$ is any subgroup that maps to $\Gamma$ isomorphically
 under the projection. 
 Given a subgroup $\Gamma$ in $\Pgl$, the dual group $\Gamma^*$ is the image in $\Pgl$ of the dual of 
 any affine lift of $\Gamma$. For $\SLpm$, we define the dual groups as above. \index{dual group} 

A properly convex open domain $\Omega$ in $P(\bR^{n+1})$ (resp. in $S(\bR^{n+1})$) is {\em dual} to a properly convex open domain
$\Omega^*$ in $P(\bR^{n+1, \ast})$ (resp. in $S(\bR^{n+1, \ast})$)  if $\Omega$ corresponds to an open convex cone $C$ 
and $\Omega^*$ to its dual $C^*$. We say that $\Omega^*$ is dual to $\Omega$. 
We also have $(\Omega^*)^* = \Omega$ and $\Omega$ is properly convex if and only if so is $\Omega^*$.  \index{dual!domain}

We call $\Gamma$ a {\em divisible group} if a central dilatational extension acts cocompactly on $C$. \index{divisible group}
$\Gamma$ is divisible if and only if so is $\Gamma^*$. (See \cite{Ben1}).

%F%or a closed cone $C$, the closure of the set $f\in \bR^{n+1,\ast}$ taking nonnegative values on $C$
%is% defined as the dual $C^* \subset \bR^{n\ast}$ of $C$. 
%F%or an open properly convex subset $\Omega$ in $\SI^{n-1}$, the dual domain $\Omega^*$ is defined as the quotient 
%o%f the dual cone of the cone corresponding to $\Omega$. The dual set is also open and properly convex
%b%ut the dimension may not be change.  
%Again, we have $(\Omega^*)^* =\Omega$. 

Note that a hyperspace is an element of $\bR P^{n \ast}$ since it is represented as a $1$-form. 
And an element of $\bR P^n$ can be considered as a hyperspace in $\bR P^{n \ast}$. 
The following definition applies to $\Omega \subset \bR P^{n}$ (resp. $S(\bR^{n+1, \ast})$)  
and $\Omega^* \subset \bR P^{n \ast}$ (resp. $S(\bR^{n+1, \ast})$).
Given a properly convex domain $\Omega$, we define the {\em augmented boundary} of $\Omega$ \index{augmented boundary} 
\[\Bd^{\Ag} \Omega  := \{ (x, h)| x \in \Bd \Omega, h \hbox{ is a supporting hyperplane of } \Omega, 
h \ni x \} .\] Note that for each $x \in \Bd \Omega$, there exists at least one supporting hyperspace. 

\begin{remark} 
For open properly convex domains $\Omega_1$ and $\Omega_2$, 
we have 
\begin{equation}\label{eqn:dualinc}
\Omega_1 \subset \Omega_2 \hbox{ if and only if } \Omega_2^* \subset \Omega_1^*. 
\end{equation}
\end{remark}

We will call the homeomorphism below as the {\em duality map}. \index{duality map}
\begin{proposition} \label{prop:duality}
Suppose that $\Omega \subset \bR P^{n}$ \rlp resp. $S(\bR^{n+1, \ast})$\rrp and 
its dual $\Omega^* \subset \bR P^{n \ast}$ \rlp resp. $S(\bR^{n+1, \ast})$\rrp
are properly convex domains. 
\begin{itemize}  
\item[(i)] There is a proper quotient map $\Pi_{\Ag}: \Bd^{\Ag} \Omega \ra \Bd \Omega$
given by sending $(x, h)$ to $x$. 
\item[(ii)] Let $\Gamma$ act on properly discontinuously $\Omega$ if and only if so acts
$\Gamma^*$ on $\Omega^*$.
\item[(iii)] There exists a homeomorphism 
\[ {\mathcal{D}}: \Bd^{\Ag} \Omega \leftrightarrow \Bd^{\Ag} \Omega^* \] 
given by sending $(x, h)$ to $(h, x)$. 
\item[(iv)] Let $A \subset \Bd^{\Ag} \Omega$ be a subspace and $A^*\subset \Bd^{\Ag} \Omega^*$
be the corresponding dual subspace $\mathcal{D}(A)$. If a group $\Gamma$ acts properly discontinuously on $A$ 
%so that $A/\Gamma$ is a compact orbifold
if and only if $\Gamma^*$ so acts on $A^*$. % and $A^*/\Gamma^*$ is a compact orbifold.
\end{itemize} 
\end{proposition} 
\begin{proof} 
We will prove for $\bR P^n$. (The $\SI^n$-version has a similar proof.)
(i) Each fiber is a closed set of hyperplanes at a point forming a compact set.
The set of supporting hyperplanes at a compact subset of $\Bd \Omega$ is closed. 
The closed set of hyperplanes passing a compact subset of $\bR P^{n}$ is compact. 
Thus, $\Pi_{\Ag}$ is proper.  Clearly, $\Pi_{\Ag}$ is continuous since it is induced by a projection. 
% and it is an open map
%since $\Bd^{\Ag} \Omega$ is given the subspace topology from $\bR P^n \times \bR P^{n \ast}$
%with a box topology where $\Pi_{\Ag}$ extends to a projection.

(ii) Straightforward.

(iii) An element $(x, h)$ is $\Bd^{\Ag} \Omega$ if and only if $x \in \Bd \Omega$ and $h$ is represented 
by a $1$-form $\alpha_h$ so that $\alpha_h(y) > 0$ for all $y$ in the open cone $C$ corresponding to $\Omega$ and 
$\alpha_h(v_x) =0$ for a vector $v_x$ representing $x$. 

Since the dual cone $C^*$ consists of all nonzero $1$-form $\alpha$ so that $\alpha(y) > 0$ for all $y \in \clo(C) - \{O\}$. 
Thus, $\alpha(v_x) > 0$ for all $\alpha \in C^*$ and $\alpha_h(v_x) = 0$. 
$\alpha_h \not\in C^*$ since $v_x \in \clo( C)-\{O\}$.
But $h \in \Bd C^*$ as we can perturb $\alpha_h$ so that it is in $C^*$. 
Thus, $x$ is a supporting hyperspace at $h \in \Bd \Omega^*$.
Hence we obtain a continuous map ${\mathcal{D}}: \Bd^{\Ag} \Omega \ra \Bd^{\Ag} \Omega^*$. 
The inverse map is obtained in a similar way. 

(iv) The item follows from (ii) and (iii).  
\end{proof} 

%%%% July 14 8:33 pm  Next I need to go to proving parabolic nature... of fixed points.... Is this sufficient? 

%Given a convex domain $\Omega$, we denote by $R_p(\Omega)$ the space of equivalent open rays in $\Omega$ from $p$ in $\SI^{n-1}_p$.  

\begin{lemma}\label{lem:example}
Let $\Omega^*$ be the dual of a properly convex domain $\Omega$.
Then 
\begin{itemize}
\item[(i)] $\Bd \Omega$ is $C^1$ and strictly convex if and only $\Bd \Omega^*$  is $C^1$ and strictly convex.
\item[(ii)] $\Omega$ is a horospherical orbifold if and only if so is $\Omega^*$. 
\item[(iii)] Let $p \in \Bd \Omega$. Then
$\mathcal{D}$ sends  in a one-to-one and onto manner
\[\{(p, h)| h \hbox{ is a supporting hyperplane of $\Omega$ at $p$}\}\]
to $\{(h^*, p^*)| h^* \in D = p^* \cap \Bd \Omega^* \}$ where $D$ is a properly convex domain disjoint from $\Omega^{\ast o}$.   
%$\mathcal{D}$ sends the pairs of form $(p, h)$ of $p$ and an associated supporting hyperplane $h$ of $\Omega$ at $p$
 %to the pairs of form $(h^*, p^*)$ for the totally geodesic hyperplane $p^*$ containing $D = p^* \cap \Bd \Omega^*$ and points $h^*$ of $D$.  
\item[(iv)] $\Bd \Omega^*$ contains a properly convex domain $D = P \cap \Bd \Omega^*$ open in a totally geodesic hyperplane $P$ 
if and only if $\Bd \Omega$ contains 
a vertex $p$ with $R_p(\Omega)$ a properly convex domain. Moreover, $D$ and $R_p(\Omega)$ are properly
convex and are projectively diffeomorphic to dual domains. 
\end{itemize}
\end{lemma}
\begin{proof} 
These are straightforward. 
\end{proof} 
%We can generalize the second item to $D$ being an $i$-dimensional properly convex domain 
%if and only if $R_p(\Omega)$ is a convex open but not properly convex domain foliated by $n-1-i$-dimensional complete affine spaces. 
%(We omit the details.) 

\subsection{Affine orbifolds}
% Define generalized trianglulation for rp and affine

%Given an orbifold $\mathcal{O}$ with a compact suborbifold $K$ whose complement is 
%a disjoint union of interval product neighborhoods of ends, we give foliation marking to each end; 
%that is, we foliate an interval product neighborhood of an end by proper one-dimensional leaves. 

An {\em affine orbifold} is an orbifold with a geometric structure modelled on $(\bR^n, \Aff(\bR^n))$. \index{affine orbifold} 
%is equivalent to having a torsion-free flat affine connection on the tangent bundle of the orbifold. 
An affine orbifold has a notion of affine geodesics as given by local charts. Recall that a geodesic is {\em complete} in a direction
if the affine geodesic parameter is infinite in the direction. 
\begin{itemize} 
\item An affine orbifold has a {\em parallel end} if the corresponding end has end neighborhood foliated by properly imbedded affine geodesics \index{end!parallel}
parallel to each other in charts and each leaf is complete in the end direction. 
We assume that the affine geodesics are leaves assigned as above. 
\begin{itemize}
\item We obtain a smooth complete vector field $X_E$ in a neighborhood of $E$ for each end following the affine geodesics, \index{end!end vector field} 
which is affinely parallel in the flow; i.e., leaves have parallel tangent vectors. We call this 
an {\em end vector field}. 
\item We denote by $X_\mathcal{O}$ the vector field partially defined on $\mathcal{O}$ by taking a sum of 
vector fields defined on some mutually disjoint neighborhoods of the ends using the partition of unity. 
\item The oriented direction of 
the parallel end is uniquely determined in the developing image of each p-end neighborhood of the 
universal cover of $\mathcal{O}$.
\item Finally, we put a fixed complete Riemannian metric on $\mathcal{O}$ so that for each end there is an open 
neighborhood where the metric is invariant under the flow generated by $X_\mathcal{O}$.
Note that such a Riemannian metric always exists. 
\end{itemize} 
\item An affine orbifold has a {\em totally geodesic end} $E$ if each end can be completed by a totally geodesic affine hypersurface. \index{end!totally geodesic} 
That is, there exists a neighborhood of the end $E$ diffeomorphic to $E \times [0, 1)$ that compactifies to an orbifold
diffeomorphic to $E \times [0, 1]$, and each point of $E \times \{1\}$ has a neighborhood affinely diffeomorphic 
to a neighborhood of a point $p$ in $\partial H$ for a half-space $H$ of an affine space.  
This implies the fact that the corresponding p-end fundamental group $\pi_1(\tilde E)$ 
for a p-end $\tilde E$ going to $E$, 
$h(\pi_1(\tilde E))$ acts on a totally geodesic hyperplane $P$ corresponding to $E \times \{1\}$. 
\end{itemize} 

Recall that an orbifold is a topological space stratified by open manifolds. 
An affine or projective orbifold is {\em triangulated} if there is a smoothly imbedded $n$-cycle consisting of geodesic $n$-simplices 
on the compactified orbifold relative to ends by adding an ideal point to a radial end and an ideal boundary to each totally geodesic ends.
where the interiors of $i$-simplices in the cycle are mutually disjoint and are imbedded in
strata of the same or higher dimension. %and the local group
%at each point of the image is locally conjugate.

%%% Êle 18 Aug 11:55pm

\subsection{Affine suspension constructions} \label{sub:asusp}
The affine space $\bR^{n+1}$ is a dense open subset of $\bR P^{n+1}$ which is the complement of 
$(n+1)$-dimensional projective space $\bR P^{n+1}$.
Thus, an affine transformation is a restriction of a unique projective automorphism acting on $\bR^{n+1}$. 
The group of affine transformations $\Aff(\bR^{n+1})$ is isomorphic to the group of projective automorphisms 
acting on $\bR^{n+1}$ identified this way by the restriction homomorphism. 

An affine orbifold $\orb'$ is {\em radiant} if $h(\pi_1(\orb'))$ fixes a point in $\bR^{n}$ for the holonomy homomorphism $h: \pi_1(\orb) \ra \Aff(\bR^{n+1})$. 
A real projective orbifold $\mathcal{O}$ of dimension $n$ 
has a developing map $\dev': \torb \ra \SI^n$ and the holonomy homomorphism 
$h': \pi_1(\orb) \ra \SL_\pm(n+1, \bR)$. Here, $\SI^n$ is imbedded as a unit sphere in $\bR^{n+1}$. 
We obtain 
a radiant affine $(n+1)$-orbifold by taking $\tilde{\mathcal{O}}$ and $\dev'$ and $h'$: 
Define $D'':\tilde{\mathcal{O}} \times \bR^+ \ra \bR^{n+1}$ by sending $(x, t)$ to $t\dev'(x)$. 
For each element of $\gamma \in \pi_1(\mathcal{O})$, 
we define the transformation $\gamma'$ on $\tilde{\mathcal{O}} \times \bR^+$ by 
\begin{align} 
\gamma'(x,t) = & ( \gamma(x), \theta(\gamma)||h'(\gamma)(t\dev'(x))||) \nonumber \\ 
                         &\hbox{ for a homomorphism } \theta: \pi_1(\orb) \ra \bR_+.
\end{align}
 Also, there is a transformation $S_s:\tilde{\mathcal{O}} \times \bR^+ \ra \tilde{\mathcal{O}} \times \bR^+$ 
sending $(x, t)$ to $(x, st)$ for $s \in \bR^+$. 
Thus, \[\tilde{\mathcal{O}}\times \bR^+/\langle S_\rho,\pi_1(\mathcal{O}) \rangle, \rho\in \bR_+, \rho \ne 1\] 
is an affine orbifold with the fundamental group 
isomorphic to $\pi_1(\mathcal{O}) \times \bZ$ where the developing map is given by $D''$
the holonomy homomorphism is given by $h'$ and sending the generator of $\bZ$ to $S_\rho$. 
We call the result the {\em affine} suspension of $\mathcal{O}$, which of course is radiant. 
The representation of $\pi_1(\orb) \times \bZ$ with the center $\bZ$ mapped to a dilatation is called an {\em affine suspension} \index{affine suspension} 
of $h$. 
A {\em special affine suspension} is an affine suspension with $\theta \equiv 1$ identically. 
(See Sullivan-Thurston \cite{ST}, Barbot \cite{Bar} and Choi \cite{rdh} also.) 

\begin{definition}\label{den:affs}
We denote by $C(\torb)$ the manifold $\torb \times \bR$ with the structure 
given by $D''$, and say that $C(\torb)$ is the {\em affine suspension} of $\torb$. \index{$C(\cdot)$}
\end{definition} 

Let $S_t:\bR^{n+1} \ra \bR^{n+1}$, given by $\vec{v} \ra t\vec{v}$, 
$t \in \bR_+$, be a one-parameter family of dilations fixing a common point. 
A family of self-diffeomorphisms $\Psi_t$ on an affine orbifold $M $ \index{radiant flow diffeomorphism}
lifting $\hat \Psi_t: \widetilde M  \ra \widetilde M $
so that $D \circ \hat \Psi_t = S_t \circ D$ for $t \in \bR_+$
is called a group of {\em radiant flow diffeomorphisms}. 

%A one-dimensional group of dilatations $S_t: \bR^{n+1} \ra \bR^{n+1}$ given by $\vec{v} \ra t\vec{v}$, $t \in \bR_+$
% lifts to $\hat \Phi_t: \torb \times \bR_+ \ra \torb \times \bR_+$  given by $(x, s) \ra (x, s+t)$. 
%This action commutes with the action of $\pi_1(\orb) \times \bZ$ on $\torb \times \bR_+$. 

\begin{lemma}\label{lem:affsus}
Let $M$ be a strongly tame $n$-orbifold. 
\begin{itemize}
\item An affine suspension $\mathcal{O}'$ of a real projective orbifold $\orb$ always 
admits a group of radiant flow diffeomorphisms.
% one-parameter family of affine self-diffeomorphism $\Phi_t$, which we call {\em radiant flow diffeomorphisms}. 
Here, $\{\Phi_t\}$ is a circle and all flow lines are closed. 
\item Conversely, if there exists a group of radiant flow diffeomorphisms with closed orbits
on $M\times \SI^1$ with an affine structure, then  $M \times \SI^1$ is 
affinely diffeomorphic to one obtained by an affine suspension construction
from a real projective structure on $M$. 
\end{itemize} 
\end{lemma} 
\begin{proof} 
The only the second item is not shown. 
The generator of $\pi_1(S)$ factor goes to a dilatation. Thus, each closed curve along $\star \times \SI^1$ 
gives us a nontrivial homology. 
The homology direction of the flow equals $[[\ast \times \SI^1]] \in S(H_1(M \times \SI^1))$. 
By Theorem D of \cite{Friedflow}, there exists a connected cross-section homologous 
to $[M \times \ast] \in H_n(M\times \SI^1, V \times \SI^1))\cong H^1(M\times \SI^1)$
where $V$ is the union of the disjoint end neighborhoods of product form. 
By Theorem C of \cite{Friedflow}, any cross-section is isotopic to $M \times \ast$.  
The radial flow is transversal to the cross-section isotopic to 
$M\times \ast$ and hence $M$ admits a real projective 
structure. It follows easily now that $M \times \SI^1$ is an affine suspension. (See \cite{Bar} for examples.)
\end{proof} 

An affine suspension of a horospherical orbifold is called a {\em suspended horospherical orbifold}. \index{suspended horospherical orbifold} 
An end of an affine orbifold with an end neighborhood affinely diffeomorphic to this is said 
to be of {\em suspended horospherical type}. This has also a parallel end since the parallel direction is given by the fixed point in 
the boundary of $\bR^n$. 

Under the cone-construction, a real projective $n$-orbifold has radial, totally geodesic, or horospherical ends 
if and only if the affine $(n+1)$-orbifold affinely suspended from it 
has parallel,  totally geodesic, or suspended horospherical ends. 
%This can be seen by taking lines in the parallel class of the concurrence point
%at infinity $\bR P^{n}$ of $\bR^{n+1}$ in $\bR P^{n+1}$. 

%Also, if a projective $n$-orbifold is (resp. properly) triangulated, then the radially suspended affine $n$-orbifold is also (resp. properly) triangulated. 
%This can be seen by taking each simplex and taking a product with $\bR^+$ and taking a quotient by $S_2$. 
%Then such a space can be also cell decomposed into simplices times intervals. Taking the interior point of the cells, 
%we obtain triangulation. We now use the suspended map to triangulated the suspended $n$-orbifold.

%%% May 1 3:33pm

\chapter{The local homeomorphism theorems}\label{sec:loch}
%%%June 9, 2009 5:19

\section{ The semi-algebraic properties of $\rep^s(\pi_1(\mathcal{O}), \PGL(n+1, \bR))$ and related spaces }\label{sub:semialg}

Since $\mathcal{O}$ is the interior of a compact orbifold, 
there exists a finite set of generators $g_1, \dots, g_m$ with finitely many relators. 
First, $\Hom(\pi_1(\mathcal{O}), \PGL(n+1,\bR))$ can be identified with an algebraic subset of 
$\PGL(n+1,\bR)^m$ corresponding to the relators. 

Let $\Hom_E(\pi_1(\mathcal{O}), \PGL(n+1,\bR))$ denote the subspace of 
\[\Hom(\pi_1(\mathcal{O}), \PGL(n+1,\bR))\]
where the holonomy of each p-end fundamental group fixes a point of $\bR P^n$ or 
acts on a subspace $P$ of codimension-one and on a lens meeting $P$ satisfying the lens-condition or a horoball tangent to $P$.
Each end of $\orb$ is assigned to be an $\cR$-type end 
or a $\cT$-type end. 
%It is equivalent to requiring that the condition holds for one representative from each conjugacy classes of p-end fundamental groups.

Since a set of finitely many elements generates each end fundamental group, %that can be written as words in $g_1, \dots, g_m$, 
and the end fundamental groups are finite up to conjugacy, 
the conditions of having a common $1$-dimensional eigenspace 
for each of a finite collection of finitely generated subgroups is a semi-algebraic condition. 

Let $E$ denote an end of type $\cT$. 
Let \[\Hom_{E, TL}(\pi(E),  \PGL(n+1,\bR))\] denote 
the space of totally geodesic representations of $\pi_1(E)$ satisfying the lens-condition, again an open subset of the algebraic set. 
(This follows by the proof of Theorem 8.1 of \cite{endclass}.) %Also, $E$ satisfies the end divisibility condition.) 
If $\rho$ is of horospherical type, then $\pi_1(E)$ is virtually abelian by Theorem 1.1 of \cite{endclass}. 
Define  $\Hom_{E, p} (\pi_1(E), \PGL(n+1, \bR))$ to be the space of representations where an abelian group
of finite index goes into a parabolic subgroup in a copy of $\PO(n, 1)$. 
By Lemma \ref{lem:parab}, $\Hom_{E, p} (\pi_1(E), \PGL(n+1, \bR))$ is a closed algebraic set. 
%\end{itemize} 
 %%% 26 July 2013 9:06 pm I need to find above in endclass....

\begin{lemma} \label{lem:parab}
Let $G$ be a finite extension of a finitely generated free abelian group $\bZ^m$.
Then $\Hom_{E, p} (G, \PGL(n+1, \bR))$ is a closed algebraic set. 
\end{lemma}
\begin{proof}
Let $P$ be a maximal parabolic subgroup of a copy of $\PO(n+1, \bR)$ that fixes a point $x$. 
Then $\Hom(\bZ^m, P)$ is a closed algebraic set.
\[\Hom_{E, p}(\bZ^m, \PGL(n+1, \bR))\] equals a union 
\[\bigcup_{g\in  \PGL(n+1, \bR)} \Hom_p(\bZ^m, gPg^{-1}),\]
another closed algebraic set.  
Now $\Hom_{E, p}(G, \PGL(n+1, \bR))$ is a closed algebraic subset of 
\[\Hom_{E, p}(\bZ^m, \PGL(n+1, \bR)).\]
 \end{proof} 

Therefore we conclude that 
$\Hom_E(\pi_1(\mathcal{O}), \PGL(n+1,\bR))$ is an open subset of an algebraic set. 

% as observed by D. Fried. 
%Hence, $\Hom_E(\pi_1(\mathcal{O}), \PGL(n+1,\bR))$ is a semi-algebraic set 
%and so is
% \[\rep_E^s(\pi_1(\mathcal{O}), \PGL(n+1,\bR)).\]
A {\em parabolic subalgebra} $\mathfrak{p}$ is an algebra in a semi-simple Lie algebra $\mathfrak{g}$
whose complexification contains a maximal solvable subalgebra of $\mathfrak{g}$  (p. 279--288 of \cite{Var}).
%and the killing perpendicular $\mathfrak{p}^\perp$ is the nilradical of $\mathfrak{p}$ (p. 279--288 of \cite{Var}).
A {\em parabolic subgroup} $P$ of a semi-simple Lie group $G$ is the full normalizer of a parabolic subalgebra. \index{parabolic subgroup} 

Let $\Hom^s_E(\pi_1(\mathcal{O}), \PGL(n+1,\bR))$ denote the subspace of stable irreducible representations.
We first note:
\begin{itemize}
\item the subset of $\Hom_E(\pi_1(\mathcal{O}), \PGL(n+1,\bR))$ that holonomy groups acting on proper subspaces is a closed algebraic set,
the space of proper subspaces in $\bR^{n+1}$ is an algebraic set, and 
\item by Theorem 1.1 of \cite{JM}, the set of representations not in proper parabolic subgroups is 
open in $\Hom(\pi_1(\mathcal{O}), \PGL(n+1,\bR))$.
%\item The condition of acting on a lens is open condition by Koszul \cite{Koszul} and horospherical 
%\item $\Hom_{E, TL}(\pi(E),  \PGL(n+1,\bR))$
\end{itemize}

%Now it follows that 
%  \[\rep_{E, u}^s(\pi_1(\mathcal{O}), \PGL(n+1,\bR)) \]
%   is a topologically open subset of a semi-algebraic set.  
  
%  We consider   \[\rep_{E, u}^s(\pi_1(\mathcal{O}), \PGL(n+1,\bR)).\]
  
  Let $E$ be an end orbifold of $\orb$. 
 Given \[\rho \in \Hom_E(\pi_1(E), \PGL(n+1, \bR)),\] we have: 
\begin{itemize} 
\item If $\rho$ is of radial of lens-type, then each element of an open neighborhood is also radial
of  lens-type by Theorem 3.14 of \cite{endclass}. 
Let \[\Hom_{E, RL}(\pi(E),  \PGL(n+1,\bR))\] denote 
the space of radial lens-type representations of $\pi_1(E)$. Thus, it is an open subspace of the above algebraic set.
\item If $\rho$ is of totally geodesic of lens-type, then each element of an open neighborhood is also totally geodesic 
of  lens-type by Theorem 3.14 of \cite{endclass}. 
Let \[\Hom_{E, TL}(\pi(E),  \PGL(n+1,\bR))\] denote 
the space of totally geodesic lens-type representations of $\pi_1(E)$. Thus, it is an open subspace of the above algebraic set.
\end{itemize} 
 %%% 26 July 2013 9:06 pm I need to find above in endclass....

 Let 
\[ R_{E_i}: \Hom(\pi_1(\mathcal{O}), \PGL(n+1, \bR)) \ra \Hom(\pi_1(E_i), \PGL(n+1, \bR)) \] 
be the restriction map to the p-end fundamental group $\pi_1(E_i)$ corresponding to the end $E_i$ of $\mathcal{O}$. 
 
 %We now show that \[\rep_{E, u, ce}^s(\pi_1(\mathcal{O}), \PGL(n+1,\bR))\]
 %is an open subset of a semi-algebraic set. 
 Let $\mathcal{R}_{\orb}$ denote the set of radial ends of $\orb$,  
 and let $\mathcal{T}_{\orb}$ denote the set of totally geodesic ends of $\orb$. 
  We can identify \[\Hom_{E}^s(\pi_1(\mathcal{O}), \PGL(n+1,\bR))\] with the subset 
 {\small
 \begin{align*} 
  & \Hom^s(\pi_1(\mathcal{O}), \PGL(n+1,\bR)) \cap   \\
 &  (\bigcap_{E_i \in {\mathcal{R}_{\orb}}} R_{E_i}^{-1}( \Hom_{E_i}(\pi(E_i),  \PGL(n+1,\bR)))) \cap \\
  &  (\bigcap_{E_i \in {\mathcal{T}_{\orb}}} R_{E_i}^{-1}(\Hom_{E_i, p}(\pi(E_i),  \PGL(n+1,\bR)) \cup \Hom_{E_i, TL}(\pi(E_i),  \PGL(n+1,\bR)))).
%  \nonumber
\end{align*} 
}
Hence, this is an open subset of a semi-algebraic set. 

Let $\Hom^s_{E, u}(\pi_1(\mathcal{O}), \PGL(n+1,\bR))$ denote the subspace of 
\[\Hom^s_E(\pi_1(\mathcal{O}), \PGL(n+1,\bR))\]
where each end of $\cR$-type fixes a unique point of $\bR P^n$
and each end of $\cT$-type acts on a unique subspace of codimension-one satisfying the lens-condition or 
a horosphere tangent to it. 
%and satisfies the end divisibility condition. 
We obtain an open subset of a semi-algebraic set since we need to consider finitely many generators of 
the fundamental groups of the ends as pointed out by D. Fried. 

%The convex divibility condition on the action of the end fundamental group is an open condition by Koszul's openness
%in \cite{Kos}. 

 We can identify \[\Hom_{E, u, ce}^s(\pi_1(\mathcal{O}), \PGL(n+1,\bR))\] to be the subset 
 {\small
 \begin{align*} 
  & \Hom_{E, u}^s(\pi_1(\mathcal{O}), \PGL(n+1,\bR)) \cap   \\
  &  (\bigcap_{E_i \in {\mathcal{R}_{\orb}}} R_{E_i}^{-1}(\Hom_{E_i, p}(\pi(E_i),  \PGL(n+1,\bR)) \cup \Hom_{E_i, RL}(\pi(E_i),  \PGL(n+1,\bR)))) \\
  &\cap  (\bigcap_{E_i \in {\mathcal{T}_{\orb}}} R_{E_i}^{-1}(\Hom_{E_i, p}(\pi(E_i),  \PGL(n+1,\bR)) \cup \Hom_{E_i, TL}(\pi(E_i),  \PGL(n+1,\bR)))).
%  \nonumber
\end{align*} 
}

 Since  \[\rep_{E, u}^s(\pi_1(\mathcal{O}), \PGL(n+1,\bR))\]
is the Hausdorff quotient of the above set with the conjugation $\PGL(n+1, \bR)$-action, 
this is a topological open subset of a semi-algebraic set.
by Proposition 1.1 of \cite{JM}. 
Similarly, so is \[\rep_{E, u, ce}^s(\pi_1(\mathcal{O}), \PGL(n+1,\bR)).\]

%Similarly, we can show that  \[\rep_{E, u, ce}^s(\pi_1(\mathcal{O}), \PGL(n+1,\bR))\]
%is 
 
 %%% March 31 11:31pm 2014
\section{The end condition of affine structures}\label{subsec:endf}
A suspension of a legion in $\SI^n$ in $\bR^{n+1}$ is the inverse image of the region under 
the projection $\bR^{n+1} -\{O\} \ra \SI^n$. A {\em suspended lens} is a quotient affine manifold 
of a suspension of a lens. A {\em suspended horoball} is a quotient affine manifold of a suspension of 
a horoball. 

Given an affine orbifold $\mathcal{O}$ satisfying our end conditions,
and each end is given a parallel end type or a totally geodesic lens end type.  
Each end fundamental group of $\pi_1(\orb)$ will have a distinguished infinite cyclic group in the center. 
Each end of our orbifold $\orb$ is given an $\cR$-type or a $\cT$-type. 
\begin{itemize} 
\item An $\cR$-type end is allowed to be parallel always, and 
\item a $\cT$-type end is allowed to be totally geodesic with 
a suspended lens neighborhood in some cover of an ambient affine manifold corresponding 
to the end fundamental group or be parallel with a suspended horoball neighborhood. 
\end{itemize} 
Here the distinguished cyclic central subgroups are required to go to the groups of dilatations preserving the cones corresponding. 

Let us make a choice of conjugacy classes of 
the fundamental group $\pi_1(E)$ as a subgroup of $\pi_1(\mathcal{O})$ for every radial end 
$E$ as a subgroup of $\pi_1(\mathcal{O})$. 

We define a subspace $\Hom_E(\pi_1(\mathcal{O}), \Aff(\bR^{n+1}))$ of 
\[\Hom(\pi_1(\mathcal{O}), \Aff(\bR^{n+1}))\] to 
be the subspace where $h(\pi_1(E_i))$ for each end $E_i$  
consists of affine transformations 
\begin{itemize}
\item with linear parts with at least one common eigenvector if $E_i$ is of $\cR$-type or 
\item acting on an affine hyperspace $P$ and properly discontinuously and cocompactly 
\begin{itemize}
\item on a suspension $L$ of a lens meeting $P$ in its interior
\item or on a suspension of a horoball tangent to $P$ if $E_i$ is of $\cT$-type.
\end{itemize}
(Here we need to fix the generator of the boundary going to a dilatation for each $\cT$-ends.)
\end{itemize} 

%and each element of $h(\pi_1(U))$ fixes a unique point in $\bR^n$.
%This involves a choice of $\pi_1(U)$ in $\pi_1(\mathcal{O})$ as one can compose with a path in $\mathcal{O}$. 
%Given an end neighborhoods $U$, the collection of $\pi_1(U)$ is also form a directed set
%so that if $V \subset U$, then $\pi_1(V) \ra \pi_1(U)$. Let us take the inverse limit of $\pi_1(U)$
%and obtain a p-end group $\pi_1(E)$. 
%Actually, $\pi_1(E)$ will equal $\pi_1(S)$ for a suborbifold $S$ imbedded in $U$.
%The restriction $h|\pi_1(E)$ has an eigenvector or an eigen-$1$-form depending on whether 
%the end was assigned to be a radial end or a totally geodesic end. 

Let $e_1$ be the number of $\cR$-ends of $\mathcal O$. 
Let $\mathcal U$ be an open subspace of \[\Hom_E(\pi_1(\mathcal{O}), \Aff(\bR^{n+1}))\] invariant under 
the conjugation action so that one can choose a continuous section 
$s_{\mathcal U}^{(1)}: \mathcal U \ra (\bR^n - \{O\})^{e_1}$
sending a holonomy homomorphism $h$ to a common nonzero eigenvector of $h(\pi_1(E))$ for each R-type end $E$. 
%as a subgroup of $\pi_1(\mathcal{O})$ for every radial end 
%$E$ as a subgroup of $\pi_1(\mathcal{O})$ and
Here $s_{\mathcal U}^{(1)}$ satisfies 
 \[s_{\mathcal U}^{(1)}(g h(\cdot) g^{-1}) = g s_{\mathcal U}^{(1)}(h(\cdot)) \hbox{ for } g \in \Aff(\bR^n), h \in {\mathcal{U}}.\]
(The choice of the sections might not be canonical here.)  
We say that $s_{\mathcal U}$ is the {\em eigenvector-section} of $\mathcal U$. 

Let $AS(\bR^{n+1})$ denote the space of oriented affine hyperplanes in $\bR^{n+1}$. 
There is a standard action of $\Aff(\bR^{n+1})$ on $AS(\bR^{n+1})$. 
%Let $\mathcal U$ be an open subspace of \[\Hom^s_E(\pi_1(\mathcal{O}), \Aff(\bR^{n+1}))\] invariant under 
%the conjugation action so that 
One can choose a continuous section also
$s_{\mathcal U}^{(2)}: \mathcal U \ra AS(\bR^n)^{e_2}$
sending a holonomy homomorphism $h$ to an invariant hyperplane in $\bR^{n+1}$ of 
 $h(\pi_1(E))$ for each totally geodesic end $E$ of lens-type. Here
 $s_{\mathcal U}^{(2)} $ satisfies 
 \[s_{\mathcal U}^{(2)}(g h(\cdot) g^{-1}) = g (s_{\mathcal U}^{(2)}(h(\cdot))) \hbox{  for } g \in \Aff(\bR^{n+1}), h \in {\mathcal{U}}.\]
(The choice of the sections might not be canonical here.)  
We say that $s_{\mathcal U}^{(2)}$ is the {\em eigen-$1$-form section} of $\mathcal U$. 
%Let $\SI^{n\ast}_\infty$ denote the dual of the sphere $\SI^{n}_\infty$ of infinity of the affine space $\bR^{n+1}$.
We form the {\em eigensection} 
\[s_{\mathcal U}:= s_{\mathcal U}^{(1)}\times s_{\mathcal U}^{(2)}: 
\mathcal U \ra (\bR^{n+1} - \{O\})^{e_1}\times (AS(\bR^{n+1}))^{e_2}.\]
%determined by eigen-$1$-form of the end fundamental groups of {\mathcal T}-type. 
%provided 
%each $\cT$-type end fundamental group acts cocompactly on a complete affine subspace or a convex domain
%the holonomy group of each $\cT$-type p-end acts properly discontinuously on
%a cone $C'$ over a horosphere in $\SI^{n}_\infty$ tangent to $P$ containing a ray in $ \clo(C')- C'$
%or acts on a totally geodesic properly convex cone in 
%the hyperplane $P$ determined by $s_{\mathcal U}^{(2)}$. 
%in $\SI^{n-2 \ast}_p$ for each fixed dual point $p$ in $\SI^{n-1\ast}_\infty$ of the hyperplane determined by $s_{\mathcal U}^{(2)}$. 
%in $\bR P^{n-1 \ast}_p$ for each dual fixed point $p$ determined by $s_{\mathcal U}$. 
%(We again call this condition the {\em end divisibility condition}.)

%By Proposition \ref{prop:duality}, we can equivalently requre that acts cocompactly on a complete affine subspace or a convex domain
%in $\SI^{n-1 \ast}_p$ for each fixed dual point $p$ in $\SI^{n \ast}_\infty$ of the hyperplane determined by $s_{\mathcal U}^{(2)}$. 
%in $\bR P^{n \ast}_p$ for each dual fixed point $p$ determined by $s_{\mathcal U}$. 

We note that the affine structure with parallel and totally geodesic ends  also will determine 
a point of $(\bR^{n+1} - \{O\})^{e_1}\times (AS(\bR^{n+1}))^{e_2}$. 

%provided 
%\begin{itemize}
%\item each parallel end fundamental group acts on cocompactly on the affine space or 
%a properly convex domain in the affine space  of parallel lines as determined by $s_{\mathcal U}$
%and 
%\item each totally geodesic end fundamental group acts cocompactly on the affine space determined by it or a convex open domain in the affine
%space. 
%\end{itemize} 
%(We need this additional condition for $\cT$-type ends and just call this condition the end divisibility condition.)

One can identify $AS(\bR^{n+1})$ with an open subspace of $\SI^{n+1 \star}$ by 
sending the affine hyperspace to a hyperspace of $\SI^{n+1}$ and hence to a point of $\SI^{n+1 \star}$
by duality. In fact the open subspace is $\SI^{n+1 \star} -\{[\alpha], [-\alpha]\}$ where $\alpha$ is a $1$-form determining $\bR^{n+1}$. 

%We mention that often $\mathcal V$ might be the entire representation space or a component of it and 
%$s_{\mathcal V}$ can be defined by a natural choice and even uniquely if the end fundamental group conditions are 
%satisfied.

\begin{remark}[End fundamental group conditions] \index{end!end fundamental group condition}
There is also an important {\em end fundamental group condition}:
Let $P$ be some unspecified condition restricting the holonomy homomorphisms of ends. 
We say that $\mathcal U$ and $\pi_1(\mathcal{O})$ have the {\em unique fixed direction property with respect to $P$} for 
holonomy homomorphisms from the p-end fundamental group $\pi_1(\tilde E) \ra \Aff(\bR^n)$ arising from $\mathcal U$ 
\begin{itemize}
\item if for each parallel end $\tilde E$, the linear parts of holonomy elements of $\pi_1(\tilde E)$ have a nonzero eigenvector, then it is 
the nonzero common eigenvector unique up to scalar multiplications for $\mathcal U$ under the condition $P$, 
\item  if for each totally geodesic end $\tilde E$, the holonomy elements of $\pi_1(\tilde E)$ have  a common invariant affine hyperplane $H$, then 
$H$ is a unique invariant affine hyperplane under the condition $P$.
\end{itemize}
Of course, $P$ could be an empty condition.

%The subset $\mathcal U$  and the orbifold $\mathcal{O}$ will have the {\em end fundamental group condition} 
%if for the fundamental group of each of its ends, the holonomy restricted to it has the unique fixed direction property. 

More precisely, it is not a purely group condition but a geometric condition. In fact, it might be possible that such a condition holds 
for a component of character space but not for some other subsets of $\Hom^s_E(\pi_1(\mathcal{O}), \Aff(\bR^n))$.
In such cases, our results are valid for the components where the conditions hold. 

%We say that the orbifold $\mathcal{O}$ will have 
%the {\em end fundamental group condition} if the fundamental group of each of its end has the 
%unique fixed direction property for all representations in \[\Hom_E(\pi_1(\mathcal{O}), \Aff(\bR^n))\] 
%that arise as holonomy homomorphisms 
%of affine structures on $\mathcal{O}$ with radial ends. 

Finally, we say that the orbifold $\mathcal{O}$ will have the {\em convex end fundamental group condition} 
if the holonomy group of each of its radial ends has the unique fixed direction
and that of each of its totally geodesic end of lens-type has the unique fixed affine hyperplane
 for every holonomy homomorphism of  a convex affine structure on $\mathcal{O}$ with radial ends or totally geodesic ends in 
$\Hom^s_E(\pi_1(\mathcal{O}), \Aff(\bR^n)).$ 
%that arise as holonomy homomorphisms of  convex affine structures on $\mathcal{O}$ with radial ends.
\end{remark} 

\begin{example}\label{exmp:singul} 
For example, if each $\cR$-type end of $\mathcal{O}$ has singularity of dimension $1$ and
there are no $\cT$-type ends, 
then the end fundamental group condition holds:
If $\mathcal{O}$ is affine with parallel end, then the singularity line in the universal cover of 
$\mathcal{O}$ is in the parallel direction and determines the eigendirection.

%Suppose that each $\cR$-type end of $\mathcal{O}$ is irreducible.
%then  $S_{\tilde E}$ is a properly convex domain in $\SI^{n-1}_{v_{\tilde E}}$ for 
%the p-end vertex $v_{\tilde E}$. Then $S_{\tilde E}$ must be a join and 
%$\pi_1(\tilde E)$ must be reducible. 
%Therefore, the ends of $\mathcal{O}$ satisfies the convex end fundamental group condition.  
\end{example}
%% July 8, 12:22 pm
%If $\mathcal{O}$ is a $3$-or $4$-dimensional radiant affine manifold with parallel ends 
%and has two dimensional singularities in the ends, 
%then in the universal cover the singularities 
%are in the planes containing the origin and the planes both contain lines in parallel direction and hence they must meet 
%in a line, which has the unique eigendirection. 

%%% April 9th 4:58pm 2014

\section{The end condition for real projective structures}\label{subsec:endreal}
Now, we go over to real projective orbifolds:
We are given a real projective orbifold $\mathcal{O}$ with ends $E_1, \dots, E_{e_1}$ of $\cR$-type and $E_{e_1+1}, \dots, E_{e_1 + e_2}$ 
of $\cT$-type. Let us choose representative p-ends $\tilde E_1, \dots, \tilde E_{e_1}$ and $\tilde E_{e_1+1}, \dots, \tilde E_{e_1 + e_2}$.
Again, $e_1$ is the number of $\cR$-type ends,  
$e_2$ the number of $\cT$-type ends of $\orb$.

We define a subspace of 
$\Hom_E(\pi_1(\mathcal{O}), \PGL(n+1, \bR))$ to be as in Section \ref{sub:semialg}.
%the subspace where for each p-end $\tilde E_i$ 
%the holonomy homomorphism restrict in $\pi_1(\tilde E_i)$ to one fixing a common point in $\bR P^n$
%or in $\bR P^{n \star}$ depending on $\tilde E_i$ is of radial type or totally geodesic type. 

Let $\mathcal V$ be an open subset of 
\[\Hom^s_E(\pi_1(\mathcal{O}), \PGL(n+1,\bR))\] invariant under the conjugation action
so that one can choose a continuous section $s_{\mathcal V}^{(1)}: \mathcal V \ra (\bR P^{n})^{e_1}$
sending a holonomy homomorphism to a common fixed point of $h(\pi_1(\tilde E_i))$ for $i = 1, \dots, e_1$ and 
 $s_{\mathcal V}^{(1)}$ satisfies 
 \[s_{\mathcal V}^{(1)}(g h(\cdot) g^{-1})  = g \cdot s_{\mathcal V}^{(1)}(h(\cdot)) \hbox{ for } g \in \PGL(n+1, \bR).\] \index{section} 
There might be more than one choice of a section and the domain of definition. 
 $s_{\mathcal V}^{(1)}$ is said to be a {\em fixed-point section}.

Again we assume that for the open subset $\mathcal V$ of
\[\Hom^s_E(\pi_1(\mathcal{O}), \PGL(n+1,\bR))\] invariant under the conjugation action
suppose that one can choose a continuous section $s_{\mathcal V}^{(2)}: \mathcal V \ra (\bR P^{n\star})^{e_2}$
sending a holonomy homomorphism to a common dual fixed point of $\pi_1(\tilde E_i)$ for $i = e_1+1, \dots, e_2$, and 
 $s_{\mathcal V}^{(2)}$ satisfies $s_{\mathcal V}^{(2)}(g h(\cdot) g^{-1}) 
 = (g^*)^{-1}\circ s_{\mathcal V}^{(2)}(h(\cdot))$ for $g \in \PGL(n+1, \bR))$.
There might be more than one choice of section in certain cases. 
 $s_{\mathcal V}^{(2)}$ is said to be a {\em dual fixed-point} section.
 
 We define $s_{\mathcal V}: \mathcal V \ra (\bR P^n)^{e_1} \times (\bR P^{n\star})^{e_2}$
 as $ s_{\mathcal V}^{(1)} \times s_{\mathcal V}^{(2)}$ and call it a {\em fixed-section}
provided the holonomy group of 
each $\cT$-type p-end fundamental group $\tilde E_i$ acts on a horosphere tangent to $P$ determined by 
$s_{\mathcal V}^{(2)}$.
%and a vertex in $P$ or a properly convex domain in $P$ for  
%the hyperspace $P$ determined by $s_{\mathcal V}^{(2), i}$.
%in $\bR P^{n-1\ast}_{s_{\mathcal V}^{(2), i}}$, the space of lines through the $i$-th component $s_{\mathcal V}^{(2), i}$.
%We call this condition the {\em end divisibility condition}.

%  provided 
%  \begin{itemize}
%\item  each radial end fundamental group acts on cocompactly on the affine space or 
%a properly convex domain in $\bR P^{n-1}_p$ for the fixed point $p$ determined by $s_{\mathcal U}$
%and
%\item each totally geodesic end fundamental group acts cocompactly on 
%the hyperspace determined by the section or a convex open domain in the hyperspace. 
%\end{itemize} 
%(We again call this condition the end divisibility condition which is needed for $\cT$-type ends only.)

\begin{remark} 
Let $P$ be some condition on holonomies of end fundamental groups. 
We say that $\mathcal V$ and an end fundamental group $\pi_1(E)$ 
 have the {\em unique fixed point and dual fixed point property with respect to $P$} if  the holonomy homomorphism of 
an $\cR$-type p-end $\tilde E$ has a common fixed point, then it is the unique fixed point for $\mathcal V$ under the condition $P$
and if the holonomy homomorphism of the $\cT$-type end $\tilde E$ has a common dual fixed point, 
then it is the unique dual fixed point for $\mathcal V$ under $P$.
%(Of course, the condition could be an empty condition.)
%The subset $\mathcal V$  and the orbifold $\mathcal{O}$ will have the {\em end fundamental group condition} 
%if the fundamental group of each of its end for every holonomy homomorphism  in $\mathcal V$ 
%has the unique fixed point property. 

Finally we say that the orbifold $\mathcal{O}$ will have the {\em end fundamental group condition} \index{end!end fundamental condition} 
if the fundamental group of each of its p-end $\tilde E$ has the 
uniquely fixed point and dual fixed point property for all representations in $\Hom^s_E(\pi_1(\mathcal{O}), \PGL(n+1,\bR))$ 
according to the type of $\tilde E$
that arise as holonomy homomorphisms of real projective structures on $\mathcal{O}$. 
We say that the orbifold $\mathcal{O}$ will have the {\em convex end fundamental group condition} \index{end!convex end fundamental condition} 
if the fundamental group of each of its p-end has the 
unique fixed point and the dual fixed point property for all representations in \[\Hom^s_{E}(\pi_1(\mathcal{O}), \PGL(n+1,\bR))\] 
according to its type 
that arise as holonomy homomorphisms 
of convex real projective structures on $\mathcal{O}$ with radial ends or totally geodesic ends of lens-type.

%Actually, it is not a purely group condition but a geometric condition. In fact, it might be possible that such a condition holds
%for a subset of representation space but not for some other subsets of $\Hom_E(\pi_1(\mathcal{O}),\PGL(n+1, \bR))$.
%In such cases, our results are valid for the subsets where the conditions hold. 
\end{remark} 

\begin{example}
If $\mathcal{O}$ is real projective and has some singularity of dimension one in each end-neighborhood of an $\cR$-type end, 
then the universal cover of $\mathcal{O}$ has more than two  lines corresponding 
to singular loci. The developing image of the lines must meet at a point in $\bR P^n$, 
which is a fixed point of the holonomy group of an end. 
If $\mathcal{O}$ has dimension $3$, this is equivalent to requiring that the end orbifold has corner-reflectors or cone-points. 
\end{example} 

%\begin{example} 
%Also, let $\mathcal{O}$ be a real projective $3$-orbifold with $\cR$-type ends only. 
%If each end has an Euler characteristic $< 0$, then the 
%end fundamental group condition holds: Let $E$ be a radial end. 
%Each end $E$ has a $2$-dimensional real projective structure. Take a surface cover and the cover has 
%a real projective structure. By the convex decomposition theorem \cite{cdcr2}, the surface decomposes into 
%convex subsurfaces and $\pi$-annuli. The convex subsurface is properly convex. 
%An end neighborhood has a universal cover that decomposes into neighborhoods of the cones over
%the universal covers of properly convex subsurfaces and annuli. 
%The cone vertex of the universal cover of the properly convex 
%subsurface is the unique fixed point of the deck transformation group acting on it: 
%If there are more than two fixed points for the holonomy homomorphism of a surface, then 
%let us take a maximal subspace $S$ that is holonomy invariant. 
%The set of hyperspaces through $S$ is an invariant set and 
%hence we obtain a geodesic foliation on the end orbifold and the Euler characteristic of the end orbifold has to be zero,
%which is absurd. Hence, there is a unique fixed point. 
%\end{example} 

%If $E$ is a totally geodesic end, we can show that there exists a unique totally geodesic 
%$h(\pi_1(E))$-invariant subspace and a unique domain $\Omega$ in it where 
%$\Omega/h(\pi_1(E))$ is a closed orbifold projectively equivalent to $E$.  
%(We omit the proof.)

\begin{example} \label{example:Benoist} 
If $\mathcal{O}$ is a real projective $n$-orbifolds with $\cR$-type ends and virtually center-free end fundamental groups,
then the convex end fundamental group condition holds: Let $E$ be a radial end. 
The end fundamental group must be hyperbolic and irreducible. As in the above argument
since the end orbifold has a strictly convex real projective structure
with irreducible holonomy homomorphism by Benoist \cite{Ben1} and hence cannot preserve a foliation of
totally geodesic leaves of any dimension between $1$ and $n-1$. 

%Now, suppose that $E$ is a totally geodesic end. Then again there exists a unique hyperspace and 
%a convex open domain $\Omega$ so that $\Omega/h(\pi_1(E))$ is a closed orbifold projectively equivalent to $E$.  
% (See \cite{Ben5}.) (We omit the proof.)

%By a similar reasoning, if the end group is Benoist and $k = l-1 \geq 1$ case, then this is true also:
%Let $v_{\tilde E}$ be the p-end vertex of a radial p-end $\tilde E$ and $\pi_1(\tilde E)$ be the p-end fundamental group. 
%Suppose that invariant subspaces $V_1, \dots, V_k$ contain $v_{\tilde E}$ and $\pi_1(\tilde E)$ has a hyperbolic factor subgroup $\Gamma_i$ 
%acting irreducibly on $V_i$ for each $i=1, \dots, k$. 
%A $h(\pi_1(E))$-invariant proper subspace $V_{\tilde E}$ passing the p-end-vertex $v_{\tilde E}$
%must be inside on one of $V_i$ because of the diagonalizable $\bZ^k$-action. 
%But inside $V_i$, there exists the unique fixed point $v_{\tilde E}$ under hyperbolic $\Gamma_i$ as we deduced above. 
%But in particular if the end fundamental group has more abelian factors than hyperbolic factors, 
%then the end fundamental group conditions might be violated. 
%If $\tilde E$ is a $\cT$-type end, again the proof follows from \cite{Ben5} since an invariant hyperspace 
%contains every $V_i$ by irreducibility.

%By duality arguments, we can allow $\orb$ to have also $\cT$-type ends with the end fundamental group of Benoist type. 
%This argument works with $\cT$-type ends as well. 
\end{example} 

%Also if each end of $\mathcal O$ has a $1$-dimensional singularity in any end neighborhood, then  the 
%end fundamental group condition holds: 
%The endpoint of the developing image of the singularity would be the unique point fixed under the holonomy groups. 
%%%

%% Feb. 15 8:57pm
%%% July 26 12:28
%An {\em end fundamental group condition} is explained in Section \ref{subsec:endf}. This condition is 
%one where if the restricted representation of the end fundamental group of an end fixes a point of $\bR P^n$, then it fixes a unique one
%or if the dual of the representation of the end fundamental group of an end fixes a point of $\bR P^{n \star}$, then it fixes a unique one. 
%This forces $\Def_{E}(\mathcal{O})$ to be same as $\Def_{E, u}(\mathcal{O})$. 
\begin{example}[Cooper]
We do caution the readers that these assumptions are not trivial and exclude some important representations.
For example, these spaces exclude some incomplete hyperbolic structures arising in Thurston's Dehn surgery constructions 
as they have at least two fixed points for the holonomy homomorphism of the fundamental group of a toroidal end as was pointed out by Cooper. 
%In this case, the stability of lens condition is used. 
%In this case, we will always assume that 
%However, we 
%will present some methods to study these here as well in Section  \ref{subsec:endf}. 

%As said before, we can have holonomy homorphisms of the end fundamental groups that act on properly convex domains in
%a totally geodesic subspace of dimension $i$ for $0 < i < n-1$. 
%The strategy of this paper is to avoid these points. 
\end{example} 

\section{Perturbing horospherical ends}  %\marginpar{May be move this section ahead?}

The following concerns the deformations of $\bZ^n\ra \PGL(n+1, \bR)$ near 
horospherical representations. As long as we restrict to deformed representations
satisfying the lens-condition, there exist  $n$-dimensional properly convex domains
where the groups act on. (This answers a question of Tillman near 2006.
J. Porti also discussed with me on the parabolic representations in 2011.) 
%The lemma holds if the holonomy elements have norms 
%of eigenvalues equal to $1$, then this holds since the end has to be complete by \cite{Fried}. In particular, if holonomy elements are 
%for horospherical ends, this holds. 

Let $P$ be an oriented hyperspace of $\SI^n$ with a dual point $P^* \in \SI^{n\ast}$ represented by a $1$-form $w_P$. 
Let $\SI^{n-1 \ast}_{P^*}$ be the space of rays from $P^*$ corresponding to 
hyperspaces in $P$. Then the subspace $P$ is dual to $\SI^{n-1 \ast}_{P^*}$: 
each ray in $\SI^{n\ast}$ from $P^*$ define an oriented hyperspace $S'$ of $P$ as the set of common zeros of the $1$-forms in the ray.
The orientation of $S'$ is given by the open half-space where the $1$-forms near $w_P$ are positive.  
Conversely, a oriented pencil of hyperplanes determined by a hyperspace of $P$ is a ray in $\SI^{n-1 \ast}_{P^*}$ from $P^*$. 
(The obvious $\bR P^n$-version is omitted.)

%%% May 1, 6:29pm 2014... I might change the whole thing....

\begin{lemma}[Horospherical end perturbation] \label{lem:horob} 
Let $B$ be a horoball in $\bR P^n$ {\rm (}resp. in $\SI^n${\rm )} and $\Gamma_p$ be a group of projective automorphisms fixing $p$, 
$p \in \Bd B$, {\rm (}resp. $p \in \SI^n${\rm )}
so that $B/\Gamma_p$ is a horospherical orbifold.  
Let $P$ be a hyperplane in $\bR P^n$ {\rm (}resp. in $p \in \SI^n${\rm ).}
\begin{itemize}
\item Let $\Hom_{E, p, ce}(\Gamma_p, \Pgl)$ denote the space of representations $h$ fixing a common fixed point $p$. 
%where $h(\Gamma_p)$ acts on a lens-cone or a horoball with the vertex at $p$. 
Then there exists a sufficiently small neighborhood $K$ 
of the inclusion homomorphism of $\Gamma_p$ in $\Hom_{E, p, ce}(\Gamma_p, \Pgl)$ where for each $h \in K$, 
$h(\Gamma_p)$ acts on a properly convex domain $B_h$ so that $B_h/h(\Gamma_p)$ 
is homeomorphic to $B/\Gamma_p$ forming a radial end and fixes $p$.
%Moreover, the transverse real projective end orbifold is convex. 
\item Let $\Hom_{E, TL}(\Gamma_P, \Pgl)$ denote the space of representations 
where $h(\Gamma_P)$ for each element $h$ acts on $P$ satisfying the lens-condition. % or a horoball tangent to $P$. 
Then there exists a sufficiently small neighborhood $K$ 
of the inclusion homomorphism of $\Gamma_p$ in $\Hom_{E, TL}(\Gamma_p, \Pgl)$ where for each $h \in P$
$h(\Gamma_p)$ acts on a properly convex domain $B_h$ so that $B_h/h(\Gamma_p)$ 
is homeomorphic to $B/\Gamma_p$ and has a totally geodesic end of lens-type or horospherical end. 
\end{itemize} 
\end{lemma} 
\begin{proof} 
%Let $S_{\tilde E}$ be a real projective orbifold corresponding to $B/\Gamma_p$ transversal to radial foliation 
%and strictly convex. We can for example take a horosphere inside. 
%Then $S_{\tilde E}$ is closed real projective orbifold. For sufficiently small change,
% $B_h/h(\Gamma_p)$ has a real projective orbifold $S_{\tilde E, h}$ transversal to each radial leaf. 
% $S_{\tilde E, h}$ is still convex for sufficiently small change of $h$ since we can interpret the change 
% of holonomy as small change of connections in $C^r$ sense for $r \geq 2$ in a neighborhood of $S_{\tilde E}$. 
%Let $\Sigma_{\tilde E}$ denote the end orbifold of form $S_{\tilde E}/\Gamma_p$ 
%where $S_{\tilde E} \subset \SI^{n-1}_p$ is an affine space by the completeness. 
%We also know that $\clo(S_{\tilde E}) - S_{\tilde E}$ is a single point. (See \cite{endclass}.)
We will prove for the $\SI^n$-version, which implies the $\bR P^n$-version. 
%We prove the first item first. 
 Let us choose a smaller horoball $B'$ in $B$. 
 $B'/\Gamma_p$ has a boundary component $S'_{\tilde E}$ so 
$B'/\Gamma$ is diffeomorphic to $S_{\tilde E} \times [1, 0)$. 
$S'_{\tilde E}$ is strictly convex and transversal to the radial foliation. 
There exists a neighborhood $K$ in $\Hom_{E, p, ce}(\Gamma_p, \Pgl)$ corresponding to 
the connection on a fixed compact neighborhood $N$ of $S'_{\tilde E}$ changes only by $\eps$ in $C^2$-topology. 
(See the deformation theorem in \cite{Goldman3} which generalize to the compact orbifolds with boundary.) 
The universal cover $\tilde S'_{\tilde E}$ is a strictly convex codimension-one manifold
and it deforms to $\tilde S'_{\tilde E, h}$ that is still convex for sufficiently small $\eps$. 
Here, $\tilde S'_{\tilde E, h}$ may not be imbedded in $\bR P^n$ but is a submanifold of the deformed $n$-manifold $N_h$ from $N$
by the change of connections. 
Every ray from $p$ meets $\tilde S'_{\tilde E, h}$ transversally also by the $C^2$-condition.

%%% April 17, 2014 2:09pm lost here...
Let $v_x$ be a vector in direction of $x$ for $x \in \clo(B')$. 
We form a cone 
\[c(\tilde S'_{\tilde E, h}) := \{ [t v_p + (1-t) v_x]| t\in [0, 1], x \in \tilde S'_{\tilde E} \}. \]
Let $\tilde S_{\tilde E, h}$ denote the space of rays from $p$ ending at $\tilde S'_{\tilde E, h}$ in $c(\tilde S'_{\tilde E, h})$. 
Here $S_h := \tilde S_{\tilde E, h}/h(\Gamma_p)$ is a compact real projective orbifold of $(n-1)$-dimension.  

%\marginpar{Use Lemma \ref{lem:realeign}} 
A map $D_h: \tilde S_{\tilde E, h} \ra \SI^{n-1}_p$ sends the point $x$ to the image ray in $\SI^{n-1}_p$. 
By the assumption on $\Hom_{E, c}(\Gamma_p, \Pgl)$, the image is in

%we claim that $D_h$ is an imbedding to a properly 
%convex domain $\Omega_h$:
%Since $\Gamma_p$ is virtually abelian, a finite cover $S'_h$ of $S_h$ is octantisable by Proposition 2 of \cite{BenNil}. 
%$D_h$ maps to a union of orbits of an abelian Zariski closure $\Delta_h$ of a finite index abelian subgroup $H$ 
%of $h(\Gamma_p)$ in $\SI^{n-1}_p$. % \marginpar{I dont think correct refs.  Check refs...}
%Now orbits of $D_h$ are convex domains in $\SI^{n-1}_p$ by our assumption for $h$. 
%But $D_h$ cannot map to a more than one orbit of $\Delta_h$ since 
%$S'_h$ is $C^2$-close to a finite cover of  a complete closed affine orbifold 
%$S_{\tilde E}/\Gamma_p$ where $S_{\tilde E}$ was an orbit of a Zariski closure. 
%Let $J$ be an orbit were $D_h$ maps to then. Then $J/H$ is cocompact. 

It follows that $D_h$ is an imbedding to a domain $\Omega_h$ in $\SI^{n-1}_p$ where $h(\Gamma_p)$ acts properly discontinuously
and cocompactly. 
%The orbit $\Omega_h$ of $\Delta_h$ is a convex open domain
%since $\Gamma_p$ is virtually abelian. 
 %and hence each domain where $\Gamma_p$ acts properly discontinuously and cocompact 
%are projectively diffeomorphic. 

There is a one-to-one correspondence from $\tilde S'_{\tilde E, h}$ to $\tilde S_{\tilde E, h}:=\Omega_h$.
By convexity of $\tilde S_{\tilde E, h}$ and the strict convexity of $\tilde S'_{\tilde E, h}$, we obtain 
that $B_h$ is convex by Lemma \ref{lem:locconv}. 
The proper convexity of $B_h$ follows since $\tilde S'_{\tilde E, h}$ is strictly convex,
and hence $\clo(B_h)$ cannot contain a pair of antipodal points.

The second item is the dual of the first one.  If $h(\Gamma_P)$ acts on 
a horosphere tangent to $P$ with the vertex in $P$ properly discontinuously, 
then the dual group $h(\Gamma_P)^*$ acts on a horosphere with a vertex the point $P^*$ dual to $P$. 
Suppose that $h(\Gamma_P)$ acts on a convex domain $\Omega_P$ in $P$.
Then it acts on a convex domain in $\SI^{n-1 \ast}_{P^*}$ the space of rays from $P^*$ corresponding to 
hyperspaces in $P$ by Proposition \ref{prop:duality}. 
Therefore, we are reduced to the first item. 
\end{proof} 

An affine space $\bR^{n+1}$ has a great sphere $\SI^{n}_\infty$ as a boundary. 
We define $\SI^{n\ast}_\infty$ as the dual sphere where the dual group $G^*$ acts on
provided an affine group $G$ act on $\SI^{n}_\infty$.
Suppose that $G$ is an extension of an affine group by a cyclic group of dilatations that are in the center. 
We denote by $\Hom^S(G, \Aff(\bR^{n+1})$ the representation where the central cyclic group go to the group of dilatations. 

\begin{lemma}[Affine horospherical end perturbation] \label{lem:affhorob} 
Let $B$ be  a horoball and $\Gamma_p$ be a group of projective automorphisms fixing $p$ 
so that $B/\Gamma_p$ is a horospherical end orbifold.  
Let $C_B$ an affine cone corresponding to $B$ and $\Gamma'_p$ denote the affine transformation corresponding 
to $\Gamma_p$ centrally extended by an infinite cyclic dilatation group acting on $C_B$. 
\begin{itemize}
\item Let $\Hom_{E, p, ce}^S(\Gamma'_p, \Aff(\bR^{n+1}))$ 
denote the space of representations $h$ 
with linear parts with a common eigenvector $v_p$ 
so that $[v_p] = p$ where the restriction group of $h(\Gamma'_p)$ acts on a lens-cone or a horoball-cone
with the end parallel along $v_p$. 
%$\SI^{n}_\infty$ with vertex $p=[v_p]$. % and the central elements are mapped to nontrivial dilatations. 
Then there exists a sufficiently small neighborhood $P$ 
of the inclusion homomorphism of 
$\Gamma'_p$ in $\Hom_{E, p, ce}^S(\Gamma'_p, \Aff(\bR^{n+1}))$ where for each $h(\Gamma'_p)$ with $h \in P$
so that $h(\Gamma'_p)$ acts on a properly convex cone $C_{B_h}$ so that $C_{B_h}/h(\Gamma'_p)$ 
is homeomorphic to $C_B/\Gamma_p'$ and has a parallel end. 
\item Let $\Hom_{E, P, TL}^S(\Gamma'_p, \Aff(\bR^{n+1}))$ denote the space of representations 
$h$ acting on a hyperspace $P$  in $\bR^{n+1}$ where $h(\Gamma'_p)$ acts on a lens-cone $L$ properly discontinuously and cocompactly  
with $L^o \cap P = L\cap P \ne \emp$ or a horoball-cone 
tangent to $P$.
%where the restriction group of $h(\Gamma'_p)$ acts on a lens meeting $P'$ or a horosphere in $\SI^{n}_\infty$ tangent to $P'$
%for the boundary $P'$ of $P$. % and the central elements are mapped to nontrivial dilatations. 
%a convex domain in $P$  or a horosphere tangent to $P$ with 
%a vertex in $P$
%properly discontinuously. 
%of lines in 
%$\SI^{n-1\ast}_\infty$ passing $p=P^*$.  
%in the real projective space of $2$-planes in $\bR P^{n\ast}$ containing a line $L$ parallel to $\alpha_P$
%where $\alpha_P$ is the $1$-form defining $P$.  
Then there exists a sufficiently small neighborhood $K$ 
of the inclusion homomorphism of $\Gamma'_p$ in $\Hom_{E, P, ce}^S(\Gamma'_p, \Aff(\bR^{n+1}))$
 where for each $h(\Gamma'_p)$ with $h \in K$
so that $h(\Gamma'_p)$ acts on a properly convex cone $C_{B_h}$ so that $C_{B_h}/h(\Gamma'_p)$ 
is homeomorphic to $C_B/\Gamma_p'$ and has a totally geodesic or affine horospherical end. 
\end{itemize} 
\end{lemma} 
\begin{proof} 
This follows from Lemma \ref{lem:horob}. 
\end{proof}

%May 1 7:46pm 

\section{Local homeomorphism theorems}

Let $\orb$ be a noncompact $(n+1)$-orbifold of strongly tame type, and ends are assigned 
to be of $\cR$-type or $\cT$-type as is the convention in this paper. 
An affine manifold affinely diffeomorphic to the affine suspension of horospherical end neighborhood is said 
to be the {\em affinely suspended horoball neighborhood}. If an end has such a neighborhood, 
then we call the end {\em  affine horospherical type}. Since the projective automorphism group of 
a horosphere fixes a point, the fundamental group of the affine horospherical end 
preserves a direction. Thus, the end of an affine horospherical type is of parallel type. 
%An end is {\em degenerated} if the end holonomy of the end fundamental group acts on
%a unique hyperplane  and is a radial with the vertex at the hyperplane. 

Again there is a parallel foliation marking for each parallel end of $\orb$ and 
the ideal boundary components of totally geodesic ends of $\orb$ analogously defined. 

We define the end restricted deformation space $\Def_{A, E}(\mathcal{O})$ 
on $\mathcal{O}$ to be the quotient space of affine 
structures on $\mathcal{O}$ where 
\begin{itemize} 
\item each end is parallel if the end is of $\cR$-type or
\item is totally geodesic of suspended lens-type or suspended horospherical type if the end is of $\cT$-type.
\end{itemize}
under the  action of group of isotopies preserving the markings; i.e., 
preserves the radial foliation if the end is radial or horospherical or extends to a smooth diffeomorphism 
if the end is totally geodesic. (As above, each end has a distinguished infinite cyclic group in the center with holonomies in
dilatations in $\bR^{n+1}$.) 

We also define
\[\Hom_{E}(\pi_1(\mathcal{O}), \Aff(\bR^{n+1}))\] 
as the subspace of \[\Hom(\pi_1(\mathcal{O}), \Aff(\bR^{n+1}))\] of elements $h$ 
where 
\begin{itemize} 
\item $h(\pi_1(E))$ for each end $E$ has dilatations as images of the distinguished infinite cyclic groups. 
\item the elements of the representations $h(\pi_1(E))$ of the fundamental group of  each end $E$ has a common eigenvector if the end is of $\cR$-type or
\item $h(\pi_1(E))$ acts on a totally geodesic hyperspace $P$ with $C_L \cap P = C_L^o \cap P \ne \emp$ for a lens-cone $C_L$ 
or tangent to a horoball-cone $C_H$ where $C_L/h(\pi_1(E))$ or 
$C_H/h(\pi_1(E))$ is a compact orbifold. 
 if the end is of $\cT$-type.
\end{itemize}

We define $\Def_{A, E, \mathcal U, s_{\mathcal U}}(\mathcal{O})$ 
to be the subspace of $\Def_{A, E}(\orb)$ with the corresponding holonomy homomorphism in 
the open subset $\mathcal U$ of \[\Hom_{E}(\pi_1(\mathcal{O}), \Aff(\bR^{n+1}))\] 
invariant under the conjugation action 
and with affine structures so that the end direction is given by 
\[s_{\mathcal U}: \mathcal U \ra (\bR^{n+1}-\{O\})^{e_1} \times ({AS(\bR^{n+1})})^{e_2}\] 
where $\mathcal U$ is a conjugation-invariant subset of 
$\Hom(\pi_1(\mathcal{O}, \Aff(\bR^{n+1}))$ and
 \[s_{\mathcal U}(g h(\cdot) g^{-1}) = g \cdot s_{\mathcal U}(h(\cdot)) \hbox{ for } g \in \Aff(\bR^{n+1}).\]
(See Section \ref{subsec:endf}.)
%Given an open subset $\mathcal U$ of $\rep_E(\pi_1(\mathcal O)$ corresponding to $\Hom_{E}(\pi_1(\mathcal{O})$

We define the isotopy-equivalence space 
$\widetilde{\Def}_{A, E, \mathcal U, s_{\mathcal U}}(\mathcal{O})$ as the quotient space 
of all development pairs $\dev: \torb \ra \bR^{n+1}$ equivariant with holonomy homomorphisms 
$h:\pi_1(\orb) \ra \Aff(\bR^{n+1})$ corresponding to the elements of $\Def_{A, E, \mathcal U, s_{\mathcal U}}(\mathcal{O})$
under the isotopies of form $\iota: \torb \ra \torb$ preserving the parallel structures 
and the totally geodesic ideal boundary. The space has the compact open $C^1$-topology. 
Here $\Def_{A, E, \mathcal U, s_{\mathcal U}}(\mathcal{O})$ is the quotient space of $\widetilde{\Def}_{A, E, \mathcal U, s_{\mathcal U}}(\mathcal{O})$ 
under $\Aff(\bR^{n+1})$ acting by 
\[g (\dev, h(\cdot)) = (g\circ \dev, g h(\cdot) g^{-1}), g \in \Aff(\bR^{n+1}).\]
% We also note that 
% \[ \Def_{A, E, \mathcal U, s_{\mathcal U}}(\mathcal{O}= \widetilde{\Def}_{A, E, \mathcal U, s_{\mathcal U}}(\mathcal{O})/ \Aff(\bR^{n+1}).\]
(See \cite{dgorb} for details.)

Similarly, we define $\widetilde{\Def}_{A, E, u}(\mathcal{O})$ as the quotient space 
of development pairs corresponding to the elements of $\Def_{A, E, u}(\orb)$ under the isotopies of $\torb$ preserving the end structures. 
We also note that 
 \[ \Def_{A, E, u}(\mathcal{O})= \widetilde{\Def}_{A, E, u}(\mathcal{O})/ \Aff(\bR^{n+1}).\]

%%% EEE 709 6:47 move below..
%We can define the radial end in the same way as in the projective cases
%and define $\Def_{\SI^n, E}(\mathcal{O})$ as the deformation space of the $(\SI^n, \SL_\pm(n+1, \bR)$-structures with radial ends.
The rest of the proof of the first part of Theorem \ref{thm:affine} is similar to \cite{dgorb}. 
We cover $\mathcal{O}$ by open sets covering a codimension-$0$ compact orbifold $\mathcal{O}'$ and open sets which are end-parallel. 
%We order the end-parallel open sets to be prior to the open sets covering $\mathcal{O}'$. 

\begin{theorem}\label{thm:affine} 
Let $\mathcal{O}$ be a noncompact strongly tame affine $(n+1)$-orbifold with parallel ends and totally geodesic ends of lens-type
where the types of ends are assigned. Assume $\partial \orb =\emp$. 
Let $\mathcal U$ be a conjugation-invariant open subset of $\Hom_E(\pi_1(\mathcal{O}), \Aff(\bR^{n+1}))$
with an eigensection $s_{\mathcal U}$. 
%and $\mathcal U'$ be the image in $\rep_E(\pi_1(\mathcal{O}), \Aff(\bR^n))$. 
The map \[\hol: \widetilde{\Def}_{A, E, \mathcal U, s_{\mathcal U}}(\mathcal{O}) \ra \Hom_{E}(\pi_1(\mathcal{O}), \Aff(\bR^{n+1}))\]
sending affine structures determined by the eigen-section $s_{\mathcal U}$ to the conjugacy classes of holonomy homomorphisms
is a local homeomorphism on an open subset of $\mathcal U'$.
%\[\rep_{E}(\pi_1(\mathcal{O}), \Aff(\bR^n)).\]
\end{theorem}

%We also define the end restricted deformation space $\Def_{E}(\mathcal{O})$ 
%to be the subspace of the deformation space $\Def(\mathcal{O})$ of real projective structures on $\mathcal{O}$ where each end is radial
%if the end is given type ${\mathcal R}$ or totally geodesic of lens type or horospherical type if the end is given type ${\mathcal T}$ 
%according to a given assignment. 

Again $\Def^s_{E, \mathcal U, s_{\mathcal U}}(\mathcal{O})$ is defined to be the subspace of $\Def_E(\orb)$ 
with the stable irreducible holonomy homomorphisms in $\mathcal U$ 
and the end determined by $s_{\mathcal U}$, i.e., 
\begin{itemize} 
\item each $\cR$-type p-end has a p-end neighborhood 
foliated by geodesic leaves that go to rays from the fixed points as given by $s_{\mathcal U}$ under the developing map, or 
\item each $\cT$-type p-end is totally geodesic of lens-type satisfying the lens-condition or horospherical with hyperspace determined 
by $s_{\mathcal U}$. (See Section \ref{subsec:endreal}.)
\end{itemize}

\begin{theorem}\label{thm:projective} 
Let $\mathcal{O}$ be a noncompact strongly tame real projective $n$-orbifold with radial ends or totally geodesic ends of lens-type with types assigned and 
$\mathcal V$ a conjugation-invariant open subset of \[\Hom^s_E(\pi_1(\mathcal{O}), \PGL(n+1, \bR)),\] 
and $\mathcal V'$ the image in \[\rep^s_E(\pi_1(\mathcal{O}),  \PGL(n+1, \bR)).\]
Assume $\partial \orb =\emp$. 
Let $s_{\mathcal V}$ be the fixed-point section defined on $\mathcal V$ with images in $(\bR P^n)^{e_1} \times (\bR P^{n \star})^{e_2}$. 
Then the map \[hol:\Def^s_{E, \mathcal V, s_{\mathcal V}}(\mathcal{O}) \ra \rep^s_{E}(\pi_1(\mathcal{O}), \PGL(n+1, \bR))\]
sending the real projective structures with ends compatible with $s_{\mathcal V}$ to their conjugacy classes of holonomy homomorphisms
is a local homeomorphism  to an open subset of $\mathcal V'$.
\end{theorem}

Define
\[ \rep^s_{E, u}(\pi_1(\mathcal{O}), \Aff(\bR^n))\] as the subspace of stable irreducible holonomy 
where each $\cR$-type end holonomy group has a unique eigenvector 
and each $\cT$-type end holonomy group has a unique eigen-$1$-form.
Define $\Def^s_{A, E, u}(\mathcal{O})$ as the subspace of $\Def_{A, E}(\mathcal{O})$ mapping to 
the subspace under $\hol$.

The last part is proved by using affine suspension. This will prove Theorem \ref{thm:A} 
since the uniqueness of the fixed points of the end holonomy groups gives 
us the section $s_{\mathcal V}$ for $\mathcal V$ equal to the space of representations corresponding to 
\[\rep^s_{E, u}(\pi_1(\mathcal{O}),  \PGL(n+1, \bR)).\]

%Note that if we do not choose the sections, then the local homeomorphism must be replaced by 
%merely open maps. 

%If $\mathcal{O}$ has the end fundamental group conditions, then we set $\mathcal U$
%and $\mathcal V$ to be the corresponding character spaces. 
%We can show that the fixed section $s_{\mathcal U}$ is continuous for $\mathcal U$ the entire representation space 
%when the end fundamental group condition is satisfied since the finite generators of end fundamental groups 
%determine the section algebraically. The following Corollary proves Theorem \ref{thm:A}. 

\begin{corollary}\label{cor:affine} 
Suppose that $\mathcal{O}$ is a noncompact strongly tame $n$-orbifold. % with the end fundamental group conditions.
Assume $\partial \orb =\emp$. 
Then the map \[\hol: \widetilde{\Def}_{A, E, u}(\mathcal{O}) \ra \Hom_{E, u}(\pi_1(\mathcal{O}), \Aff(\bR^n))\]
sending affine structures to the conjugacy classes of their holonomy homomorphisms
is a local homeomorphism.
So is the map \[\hol:\Def^s_{E, u}(\mathcal{O}) \ra \rep^s_{E, u}(\pi_1(\mathcal{O}), \PGL(n+1,\bR)).\]
\end{corollary}
\begin{proof}
For each representation in \[ \Hom_{E, u}(\pi_1(\mathcal{O}), \Aff(\bR^n)),\] 
we find a unique set of vectors and $1$-forms corresponding to the ends. 
The continuity follows by considering sequences. 
We use the above paragraph and Theorems \ref{thm:affine} and \ref{thm:projective}. 
\end{proof} 

%% Nov 4 9:50 Continue...

%% May 14th 9:46pm 2014 

\section{The proof of Theorem \ref{thm:affine}.} 

We wish to now prove Theorem \ref{thm:affine} following the proof of Theorem 1 in Section 5 of \cite{dgorb}.
Then we will prove Theorem \ref{thm:projective} in Section \ref{sub:rps} using this chapter. 

Let $\mathcal{O}$ be an affine orbifold with the universal covering orbifold $\tilde{\mathcal{O}}$ 
with the covering map $p_\orb:\tilde{\mathcal{O}} \ra \mathcal{O}$
and let the fundamental group $\pi_1(\mathcal{O})$ act on it as an automorphism group.

Let $\mathcal U$ and $s_{\mathcal U}$ be as above. 
We will now define a map 
\[\hol:\widetilde{\Def}_{A, E , \mathcal U, s_{\mathcal U}}(\mathcal{O}) \ra \Hom_E(\pi_1(\mathcal{O}), \Aff(\bR^n))\] 
by sending the affine structure to the pair $(\dev, h)$ and to the conjugacy class of $h$ finally.
%The continuity of $\hol$ is easy to show in fact for any geometric structures. 
%We cover $\mathcal{O}$ by open sets as in \cite{dgorb}. 
There is a codimension-$0$ compact submanifold $\mathcal{O}'$ of $\mathcal{O}$ 
so that $\pi_1(\mathcal{O}') \ra \pi_1(\mathcal{O})$ is an isomorphism.
The holonomy homomorphism is determined on $\mathcal{O}'$. 
Since the deformation space has $C^r$-topology, $r \geq 1$,  induced by 
$\dev$, it follows that small changes of $\dev$ on compact domains in the $C^r$-topology imply sufficiently small changes in $h(g'_i)$ for 
generators $g'_i$ of $\pi_1(\mathcal{O}')$ and 
hence sufficiently small change of $h(g_i)$ for generators $g_i$ of $\pi_1(\mathcal{O})$. 
Therefore, $\hol$ is continuous.
(Actually for the continuity, we do not need any condition on ends.)

For the purpose of this paper, we use $r \geq 2$. We will use this fact a number of times.

A {\em $v$-parallel set} is a subset of $\bR^n$ which is invariant under the translation 
along positive multiples of a fixed nonzero vector $v$. 
That is, it should be a union of the images under translations along positive multiples of a nonzero vector.

An {\em end-parallel} subset of $\tilde{\mathcal{O}}$ or $\bR^n$ is a $v$-parallel set where 
$v$ is the eigenvector of the linear parts of the corresponding p-end.
%By the end-fundamental group condition, $v$ is uniquely determined for each end given a geometric structure or the holonomy group. 

%%% April 17 4:32pm

%\subsection{Affine structures}
To show the local homeomorphism property, we take an affine structure $(\dev, h)$ on $\mathcal{O}$ and the associated 
holonomy map $h$. We cover $\tilde{\mathcal{O}}'$ by small precompact open sets as in Section 5 of \cite{dgorb}. 
We cover $\mathcal{O}-\mathcal{O}^{\prime o}$ by end-parallel open sets. 
Consider Lemmas 3, 4, and 5 in \cite{dgorb}. We can generalize these to include the $v$-parallel sets
for an invariant direction of $v$ of the finite group $G_B$ where $v$ is an eigenvector of eigenvalue $1$ since $G_B$ is finite. 
We repeat them below. The proofs are very similar and use the commutativity of translation by eigenvectors with the action of $G_B$.

\begin{lemma}\label{lem:3} 
Let $G_B$ be a finite subgroup of $\Aff(\bR^n)$ acting on a $v_0$-parallel open subset $B$ of $\bR^n$ for 
an eigenvector $v_0$ of the linear part of $G_B$. Let
$h_t : G_B \ra \Aff(\bR^n),  t \in [0, \eps), \eps > 0$, be an analytic parameter of representations of $G_B$ so that
$h_0$ is the inclusion map. Let $v_t$ is a nonzero eigenvector of $h_t(G_B)$ for each $t$ and 
we assume that $t \mapsto v_t$ is continuous. 
Then for $0\leq t \leq \eps$, there exists a continuous family of diffeomorphisms
$f_t : B \ra B_t$ to a $v_t$-parallel open set $B_t$ in X so that $f_t$ conjugates the $h(G_B)$-action to
the $h_t(G_B)$-action; i.e., $ f_t h_t(g)f^{-1}_t= h_0(g)$ for each $g \in G_B$ and $t \in [0, \eps]$.
\end{lemma}
\begin{proof} 
We find the diffeomorphism $f_t$ in a transversal section of $B$ meeting every lines in $B$ parallel to $v_t$
and extend $f_t$ using the lines. 
\end{proof} 

Here $v_t$, $v_{h'}$ and $v_{h', t}$ below are of course determined by $s_{\mathcal U}$. 

Since $\Hom(G_B, \Aff(\bR^n))$ is a semialgebraic set, we obtain that each point has a cone-neighborhood, i.e., 
a topological neighborhood parameterized by $I \times S/\sim$ where $S$ is a semialgebraic set and $\sim$ is given by 
$(0, x) \sim (0, y), x, y \in S$. 

\begin{lemma}\label{lem:4} 
Let $G_B$ be a finite subgroup of $\Aff(\bR^n)$ acting on a $v_0$-parallel open subset $B$ of $\bR^n$ for 
an eigenvector $v_0$ of the linear part of $G_B$.  Suppose that $h$ is a
point of an algebraic set $V \subset \Hom(G_B, \Aff(\bR^n))$ for a finite group, and let $C$ be a cone neighborhood
of $h$. Suppose that $h$ is an inclusion map.
Suppose that $v_{h'}$ is the eigenvector of the linear part of $h'(G_B)$ for each $h'\in C$ and 
$h' \mapsto v_{h'}$ forms a continuous function $C \ra \bR^n$. 
Then for each $h' \in C$, there is a corresponding diffeomorphism
\[f_{h'} : B \ra B_{h'},  B_{h'} = f_{h'}(B)\]
so that $f_{h'}$ conjugates the $h({G_B})$-action on $B$ to the $h'(G_B)$-action on $B_{h'}$ {\rm ;} i.e.,
$f_{h'}^{-1} h'(g) f_{h'} = h(g)$ for each $g \in G_B$ where $B_{h'}$ is a $v_{h'}$-parallel open set. 
Moreover, the map $h' \mapsto f_{h'}$ is continuous from $C$
to the space $C^\infty(B, X)$ of smooth functions from $B$ to $X$.
\end{lemma}

Continuing to use the notation of Lemma \ref{lem:4}, we define a parameterization
$l : S\times [0, \eps] \ra C$ for a cone-neighborhood which is injective except at $ S\times \{0\}$ mapping  to $h$. 
(We fix $l$ although $C$ may become smaller and smaller). For $h' \in  S$, we denote by
$l(h'): [0, \eps] \ra C$ be a ray in $C$ so that $l(h')(0) =h$ and $l(h')(\eps) = h'$. Let the finite group
$G_B$ act on a $v_h$-parallel relatively compact 
submanifold $F$ of a $v_h$-parallel open set $B$ for an eigenvector $v_h$ of $h(G_B)$. 
Let $v_{h', t}$ be a nonzero eigenvector of $l(h')(t)$ for $h' \in S$ and $t \in [0, \eps]$ and 
we suppose that $S\times [0, \eps] \ra \bR^n$ given by $(h', t) \ra v_{h', t}$ is continuous.  

A $G_B$-equivariant isotopy $H: F \times [0, \eps] \ra \bR^n$ is a
map so that $H_t$ is an imbedding for each $t \in [0, \eps']$, with $0 < \eps' \leq \eps$, conjugating
the $G_B$-action on $F$ to the $l(h')(t)(G_B)$-action on $\bR^n$. 
Here $H_0$ is an inclusion map
$F \ra \bR^n$ where the image $H(F, t)$ is a $v_{h', t}$-parallel set for each $t$. 
Lemma \ref{lem:4} above says that for each $h'\in C$, there exists a $G_B$-equivariant
isotopy $H: B \times [0, \eps] \ra \bR^n$ so that the image $H(B, t)$ is a $v_{h', t}$-parallel open set for each $t$. 
We will denote by $H_{h', \eps'} : B \ra \bR^n$ the map obtained from
$H$ for $h'$ and $t = \eps'$. Note also by the similar proof, for each $h' \in S$, there exists a
$G_B$-equivariant isotopy $H : F \times [0, \eps''] \ra \bR^n$.

\begin{lemma}\label{lem:5}
Let $F$ be a $v_h$-parallel relatively compact 
submanifold of a $v_h$-parallel open set $B$ for an eigenvector $v_h$ of $h(G_B)$. 
Let $H: F \times [0, \eps] \times S \ra \bR^n$ be a map so that $H(h'): F \times [0, \eps'] \times S \ra \bR^n$ is a
$G_B$-equivariant isotopy of $F$ for each $h' \in S$ where $0 < \eps' \leq \eps$ for some $\eps > 0$. Then 
for a neighborhood $B'$ of $F$ in $B$, it follows that $H$
can be extended to $\hat H: B' \times [0, \eps'']\times S \ra \bR^n$ so that 
\[\hat H(h'): B' \times [0, \eps''] \ra \bR^n, 0 < \eps'' \leq \eps\] is a
$G_B$-equivariant isotopy of $B'$ for each $h' \in S$. %where $0 < \eps'' \leq \eps$. 
The image $\hat H(h')(t)(B')$ is a $v_{h', t}$-parallel open set for each $h', t$. 
\end{lemma}

%%% July 28, 7:13 pm To continue proving under new def...And then need to check "end"... 
\begin{proof}[Proof of Theorem \ref{thm:affine}]
%We are aiming to prove the local homeomorphism property of 
%\[\hol: \widetilde{\Def}_{A, E, u}(\mathcal{O}) \ra \Hom_{E, u}(\pi_1(\mathcal{O}), \Aff(\bR^n)).\]
%Let us choose a compact orbifold $\orb'$ so that $\orb - \orb'$ is a union of products of $(n-1)$-orbifolds with an interval. 
%First, $hol$ is continuous since we can restrict the map to the isotopy-equivalence space of 
%affine structures on a compact suborbifold $\orb'$ obtained from restricting the affine structures of $\orb$ to $\orb'$. 

%To simplify, we attach the ideal boundary that is an open subset of totally geodesic hyperspace
%to $\torb$ to each totally geodesic boundary. 
%For these boundary components of the new orbifold, we obtain that our section $s_{\mathcal U}$ gives 
%us the deformations of open sets as in the above lemmas. We will omit stating 
%these  obvious lemmas. 

To finish the proof, we define the local inverse map from a neighborhood in $\mathcal U$ of the image point.
Let $h$ be a representation corresponding to an element $h$ of it coming from an affine orbifold $\orb$ with radial or totally geodesic boundary. 
The task is to reassemble $\orb$ with new holonomy homomorphisms as we vary $h$ as in \cite{dgorb} following Thurston's approaches. 
\begin{itemize} 
\item For an $\cR$-type end, 
this is accomplished as in \cite{dgorb} for precompact open covering sets and 
for end-parallel open covering sets we use the above lemmas since we are working with finitely many open sets.
\item For a $\cT$-type end that has the totally geodesic ideal boundary, we first complete it with an open subset of 
a totally geodesic hyperspace. 
There exists an open subset where the corresponding p-end has totally geodesic hyperspace invariant 
under each holonomy group of the pseudo-end and is not horospherical. 
For a sufficiently small open set in $\mathcal U$, we can change 
each open neighborhood in the manner described in \cite{dgorb}. The totally geodesic ideal boundary 
does not present any difficulty here. 
\item For a $\cT$-type end $\tilde E$ that is a suspended horospherical end, 
we take an affinely suspended horospherical neighborhood projectively isomorphic
to $C_B/\Gamma_p$ where $\Gamma_p$ is the affine suspension group extended by a central infinite cyclic 
group generated by a dilatation. 
Lemma \ref{lem:affhorob} shows us how to obtain a totally geodesic end under small deformations 
of holonomy homomorphisms.  %\marginpar{This last s is incorrect.. Change.}

\end{itemize}

Since we can construct the end neighborhoods as above, we obtain the affine structures for points of $\mathcal U$ 
by using partition of unity and pasting the results as in \cite{dgorb}. To show that the local inverse is a continuous map 
for the compact open $C^r$-topology we only need to consider compact suborbifolds in $\orb$, and 
we use the fact that the conjugating maps of above Lemmas \ref{lem:3}, \ref{lem:4}, and \ref{lem:5}  
depend continuously on $\mathcal U$. 

Also, finally, we need to prove the local injectivity of $\hol$ as in the last step of the proof of Theorem \ref{thm:affine}. 
Given two structures $\mu_0$ and $\mu_1$ in a neighborhood of the deformation space, we show that if their 
holonomy homomorphisms are the same, then we can isotopy one in the neighborhood to the other using vector fields as in \cite{dgorb}.
%The proof is similar to to that of \cite{dgorb}
%Let each end $E$ of our orbifold $\mathcal{O}$ has a neighborhood $U$ with a smooth nonzero vector field $\chi_E$ 
%parallel to the eigendirection. We call the vector field the {\em end vector field}.

Because of the secton $s_{\mathcal U}$ defined on $\mathcal U$, given a holonomy $h: \pi_1(\orb) \ra \Aff(\bR^n)$, 
we have a direction of the parallel end that is unique for the holonomy homomorphism. 

First assume that $\orb$ has only $\cR$-type ends. 
Recall the compact suborbifold $\orb'$ so that $\orb - \orb'$ is homeomorphic to 
$E_i \times (0, 1)$ for each end orbifold $E_i$. 
%Here, our orbifolds are not necessarily compact but 
Now, $\orb$ has a Riemannian metric that is invariant under the flows generated by the end vector fields 
in the union of its end neighborhoods. % that is invariant under the flows generated by the end vector fields. 
On each compact suborbifold $\orb'$ of $\orb$ with $\partial \orb'$ transversal to the vector fields in the end neighborhood. 
these end vector fields will be uniformly $C^r$-bounded by a small uniform constant depending on how close 
the two structures $\mu_0$ and $\mu_1$ are in the $C^r$-topology in $\orb'$ of the universal cover. 
%We reparameterize the developing maps $\dev_0$ of $\mu_0$ and $\dev_1$ of $\mu_0$
%so that the flow lines so that the corresponding developing maps are at a unit speed along
%the directions. The flow lines agree since the directions agree by definition of the parallel ends
%and the holonomy. 
%The new structure $\mu'_i$ corresponding to new developing map is isotopic to $\mu_i$ for each $i= 1, 2$. 
%Hence, it is sufficient to estimate differences of maps in $\orb'$ and the estimates will 
%propagate to the end. We may need to reparameterize along the radial direction so that the differences of the maps are 
%uniformly bounded. 

Let $\dev_i$ be the developing map of $\mu_i$ for $i=1, 2$. 
Then the $C^r$-norm distance of $\dev_0$ and $\dev_1$ is bounded on each compact set $K \subset \torb$. 
Hence, we can isotopy $\mu_0$ to $\mu_1$ on a neighborhood of $K$ with some $C^r$-bounds $\eps > 0$. 
We can do this for some $\eps$ and $K$ mapping onto a suborbifold $\orb'$ where $\orb -\orb'$ is a product of 
intervals with closed orbifolds. We extend the isotopies using the parallel line extension parametrized by the Riemannian metric. 
Since the end-orbifolds are determined by their boundary orbifolds in $\orb$, 
we obtain an isotopy from $\mu_0$ to $\mu_1$ in 
an open neighborhood of the identity map. 
%Then the argument is similar to what is in \cite{dgorb}. 
(Here, we need to only check for compact suborbifolds since we define neighborhoods of the functions using 
the compact open $C^r$-topology. )

Suppose now that $\orb$ has some $\cT$-type ends. 
Suppose that $\mu_0$ and $\mu_1$ have totally geodesic ideal boundary corresponding to an end of $\orb$.
We attach the totally geodesic ideal boundary component for each end, 
and then we can argue as in \cite{dgorb}. 

Suppose that $\mu_0$ and $\mu_1$ have horospherical end neighborhoods corresponding to an end of $\orb$.
Then these are radial ends and the same argument as the above one for $\cR$-type ends will apply to show the injectivity. 
Finally, we cannot have the situation that $\mu_0$ have totally geodesic ideal boundary corresponding 
to an end while $\mu_1$ have a horoball end neighborhood for the same end. 
This follows since the end holonomy group acts on a properly convex domain in a totally geodesic hyperspace
and as such the end holonomy group elements have some norms of eigenvalues $> 1$. (See Proposition 1.1 of \cite{Ben5} for example.) 
\end{proof}

%%% Oct 27 2013 10:55 pm I need to get "rid"of useless "end fund grp conditions".

%%% April 17 6:23pm

\section{The proof of Theorem \ref{thm:projective}.} \label{sub:rps}

%\subsection{Real projective structures}\label{subsub:rps}

%%May 15th 12:58pm 2014

Suppose now that $\mathcal{O}$ is a real projective orbifold of dimension $n$. 
We assume that $\mathcal{O}$ have end that are assigned to be $\cT$-type or $\cR$-type ones. 
Let $\mathcal{O}' = \orb \times \SI^1$ be the affine suspension. $\pi_1(\orb')$ is isomorphic to $\pi_1(\orb) \times \bZ$. 
Each end has distinguished infinite cyclic group in the center given by the factor $\bZ$. 
 $\mathcal{O}'$ has a parallel end with the end direction determined 
by the radial ends of $\mathcal{O}$ and totally geodesic ends of lens type determined by that of $\orb$. 
Define $\Hom^{sS}_E(\pi_1(\mathcal{O}'), \Aff(\bR^{{n+1}}))$
to be the subspaces of the representation space $\Hom_E(\pi_1(\mathcal{O}'), \Aff(\bR^{{n+1}}))$ 
where 
\begin{itemize}
\item the $\bZ$-factor of $\pi_1(\mathcal{O}')= \pi_1(\mathcal{O}) \times \bZ$ always maps to a group of dilatations
and 
\item each of whose element $h$ has the stable irreducible linear part $\mathcal{L}(h)|\pi_1(\orb)$.
\end{itemize} 
We define as 
$\rep^{sS}_E(\pi_1(\mathcal{O}'), \Aff(\bR^{{n+1}}))$ the corresponding subspace of 
the character variety $\rep_E(\pi_1(\mathcal{O}'), \Aff(\bR^{{n+1}}))$.
Let $\mathcal U$ be the conjugation invariant subspace of 
$\Hom^{sS}_E(\pi_1(\mathcal{O}'), \Aff(\bR^{{n+1}}))$ and we are given the fixed section
\[s_{\mathcal U}: \mathcal U \ra (\bR^{n+1} - \{O\})^{e_1} \times (AS(\bR^{n+1}))^{e_2}.\]
%We define $\Def^{sS}_{A, E, \mathcal U, s_{\mathcal U}}(\mathcal{O}')$ 
%as the subset of elements of $\Def_{A, E, \mathcal U, s_{\mathcal U}}(\mathcal{O}')$ with 
%holonomy characters in $\rep^{sS}_E(\pi_1(\mathcal{O}'), \Aff(\bR^{{n+1}}))$.

For any element $\mu$ of $\Def_{A, E, \mathcal U, s_{\mathcal U}}(\mathcal{O}')$, $\orb'$ with $\mu$ 
a developing map pulls back a radiant vector field $\sum_{i=1}^{n+1} x_i \frac{\partial}{\partial x_i}$ on $\bR^{n+1}$. 
This gives us a radial flow on $\orb'$ with $\mu$. 
%The transversal holonomy is trivial since
%the radial flow
%Given a structure on $\orb$ corresponding to the element in the space, 
Each point $p$ of $\orb'$ has a neighborhood foliated by radial lines. Furthermore, the radial lines are always 
closed since a dilation from the central elements acts on each radial line giving us a closed orbit always. 
%The affine holonomy around the closed radial curve are all homotopic
%and they are not null-homologous. 
By Lemma \ref{lem:affsus}, $\orb' $ with $\mu$ is an affine suspension from $\orb$. 
Since $\orb$ can be imbedded transversal to the radial flow, it follows that $\orb'$ with $\mu$ 
gives us an $(\SI^{n}, \SL_\pm(n+1, \bR))$-structure on $\mathcal O$.

%\begin{lemma}\label{lem:afs}
%Let $\mu$ be an element of $\Def_{A, E, \mathcal U, s_{\mathcal U}}(\mathcal{O}\times \SI^1)$.
%Then $\mu$ is obtained as an affine suspension of a real projective structure on $\orb$.
%\end{lemma}
%\begin{proof} 
%
%\end{proof} 

%Given a structure on $\orb$ corresponding to the element in the space, 
%each point $p$ of $\orb$ has a neighborhoods foliated by radial lines. Furthermore, the radial lines are always 
%closed by compactness and the holonomy acting properly discontinuously on the space of radial 
%lines in $\torb$. The affine holonomy around the closed radial curve are all homotopic
%and they are not null-homologous. 

%There is a real cohomology class given by $\pi_1(\orb') \ra \bR$ by 
%$g \mapsto \log \det L(h(g))$ where $L(h(g))$ is the linear part of the holonomy of $g$.   
%Hence, there exists a cross-section of the radial flow on a finite covering of $\orb$ by Theorem E of Fried \cite{Friedflow}.
%Such an affine structure is always obtained by an affine suspension. (See Barbot \cite{Bar}, Fried \cite{FriedAffine}.)

%(Here the section plays no important role 
%in the suspension construction.) 

%By the affine suspension construction, $\pi_1(\mathcal{O}')$ is isomorphic to $\pi_1(\mathcal{O}) \times \bZ$, 
%and the generator $\zeta$ of the factor corresponding to $\bZ$ is always mapped by a holonomy $h$ to 
%a dilatation $h(\zeta)$, elements of $\pi_1(\mathcal{O})$ fix the unique fixed point $p_{h(\zeta)}$ of the dilatation. 
%We can choose the fixed point to be the origin by changing the developing map by a post-composition with 
%a translation. 

We define $\rep_E(\pi_1(\mathcal{O}), \GL({n+1}, \bR)) $ and $\rep_E(\pi_1(\mathcal{O}), \SL_\pm({n+1}, \bR))$ 
as the respective subsets where 
the the holonomy groups of the end fundamental groups of $\mathcal{O}$ have common eigenvectors.
By sending dilatations to the expansion factors, 
we obtain that \[\rep^{sS}_E(\pi_1(\mathcal{O}'), \Aff(\bR^{{n+1}}))\] is identical with 
\[\rep^s_E(\pi_1(\mathcal{O}), \GL({n+1}, \bR)) \times \bR_+\] which 
is the subspace of \[\rep^s(\pi_1(\mathcal{O}), \GL({n+1}, \bR)) \times \bR_+\] where the holonomy group of each p-end 
has an eigendirection or an eigen-$1$-form.
\[\rep^s_E(\pi_1(\mathcal{O}), \GL({n+1}, \bR))\times \bR_+\] can be identified with 
\[\rep^s_E(\pi_1(\mathcal{O}) , \SL_\pm({n+1}, \bR)) \times H^1(\pi_1(\mathcal{O}), \bR) \times \bR\] by using the isomorphism
\[\GL({n+1},\bR) \ra \SL_\pm({n+1},\bR) \times \bR\]
which is given by sending a matrix $L$ to $(L/|\det(L)|, \log(|\det(L)|))$. Let
 \[q_S: \rep^{sS}_E(\pi_1(\mathcal{O}'), \Aff(\bR^{{n+1}})) \ra \rep^s_E(\pi_1(\mathcal{O}) , \SL_\pm({n+1}, \bR))  \] 
 denote the obvious projection.

%An open subset of $\bR^{{n+1}} - \{O\}$ is {\em radial} if each radial flow leaf meets the open set in a connected interval only. 
Let $\mathcal U$ denote a conjugation invariant subset of $\Hom^{sS}_E(\pi_1(\mathcal{O}'), \Aff(\bR^{{n+1}}))$
with a section \[s_{\mathcal U}: \mathcal U  \ra (\bR^{n+1} -\{O\})^{e_1} \times (AS(\bR^{n+1}))^{e_2}.\]
By Theorem \ref{thm:affine}, and taking the Hausdorff quotients 
\begin{equation}\label{eqn:sS} 
\hol: \Def^{sS}_{A, E, \mathcal U,\, s_{\mathcal U}}(\mathcal{O}') \ra \rep^{sS}_E(\pi_1(\mathcal{O}'), \Aff(\bR^{{n+1}}))
\end{equation} 
is a local homeomorphism to its image since the former space is simply the inverse image of the second space. 

%Again each end of $\mathcal{O}$ is given an $\cR$-type or a $\cT$-types with the analogous meanings. 
We define $\Def_{\SI^{n}}(\mathcal{O})$ as the deformation space of $(\SI^{n}, \SL_\pm(n+1, \bR))$-structures on $\mathcal O$
and $\Def_{\SI^{n}, E}(\mathcal{O})$ as the quotient space of the space of 
$(\SI^{n}, \SL_\pm(n+1, \bR))$-structures on $\mathcal O$ with radial ends with radial marks and  
totally geodesic ends of lens-type with ideal boundary marks.  
Here the equivalence relation $\sim$ as before 
is given by the action of the group of isotopies preserving the radial end structures for radial ends 
and extending to totally geodesic ideal boundary for totally geodesic ends. 
Let $\mathcal U'$ denote a conjugation invariant subset of \[\Hom^s_E(\pi_1(\mathcal{O}), \SL_\pm(n+1, \bR)))\]
with a section \[s'_{\mathcal{U}'}:  \mathcal{U}' \ra (\SI^n)^{e_1} \times (\SI^{n \ast})^{e_2}.\] 
We define $\Def^s_{\SI^{n}, E, \, {\mathcal{U}'}, s'_{\mathcal{U}'}}(\mathcal{O})$  as the subspace of structures 
whose holonomy characters are in $\mathcal U'$ and the ends are compatible with $s'_{\mathcal U'}$, i.e., an end neighborhood of 
each structure is foliated by concurrent geodesics or by the totally geodesic hyperspace 
determined by $s'_{\mathcal U'}$.

\begin{proposition} \label{prop:openness1} 
Let $\mathcal{O}$ be a noncompact strongly tame $n$-orbifold where the types of ends are assigned. 
Let $\mathcal{U}'$ be a conjugation-invariant open subset of \[\Hom^s_E(\pi_1(\mathcal{O}), \SL_\pm(n+1, \bR))\]
with the section \[s'_{\mathcal{U}'}:  \mathcal{U}' \ra (\SI^n)^{e_1} \times (\SI^{n \ast})^{e_2}.\] 
%where $\mathcal U'$ is the image of $\mathcal{U}''$ in \[\rep_E(\pi_1(\mathcal{O}), \SL_\pm(n+1, \bR)).\]
The map \[\hol: \Def^s_{\SI^{n}, E, \mathcal{U}', s'_{\mathcal{U}'}}(\mathcal{O}) \ra \rep^s_{E}(\pi_1(\mathcal{O}), \SL_\pm(n+1, \bR))\]
sending  $(\SI^{n}, \SL_\pm(n+1, \bR))$-structures determined by the eigensection $s'_{\mathcal{U}'}$ to the conjugacy classes of holonomy homomorphisms
is a local homeomorphism to an open subset of $\mathcal{U}'$.
%
%Then \[\Def_{\SI^{n}, \, E, \, {\mathcal{U}'}, \, s'_{\mathcal{U}'}}(\mathcal{O}) \ra  \rep_E(\pi_1(\mathcal{O}), \SL_\pm(n+1, \bR))\] is 
%a local homeomorphism where $s'_{\mathcal{U}}$ is an arbitrary fixed section defined on $\mathcal{U}'$. 
\end{proposition} 
\begin{proof} 
Let $\mathcal O'$ be the product $\orb \times \SI^1$ as above.

%Let $b_1 := \mathrm{rank} H^1(\pi_1(\mathcal{O}), \bR)$.
%We note that there is an action of $\bR^{b_1 + 1}$ on  \[\rep^S_E(\pi_1(\mathcal{O}'), \Aff(\bR^{{n+1}}))\]
%by changing the holonomy of the generators by multiplying by dilatations fixing $p_\zeta$.

%We can identify $\bR^{n+1}$ with a suspace $\bR^{n+1} \times \{0\} \subset \bR^{n+2}$.
%$A^{{n+1} \star}$ with a subspace $\bR^{{n+1}} \times \{1\} \subset \bR^{n+2}$. 
%Given an open subset $\mathcal U$ in \[\Hom^S_E(\pi_1(\mathcal{O}'), \Aff(\bR^{{n+1}})),\] an eigen-map
%$s_{\mathcal U}: {\mathcal{U}}  \ra (\bR^{n+1} -\{O\})^{e_1} \times (AS(\bR^{n+1}))^{e_2}$ 
%induces a map $q\circ s_{\mathcal U}$ for the map $q = (q_1, q_2)$ where 
%$q_1: (\bR^{{n+1}}-\{O\})^e \ra (\SI^{n})^e$ is a quotient map 
%and $q_2: AS(\bR^{n+1})^{e_2}\ra \bR P^{n \ast}$ is an inclusion map.
%(We note that $q_2$ gives hyperplanes fixing the fixed point of the generator.) 
%We assume that $\mathcal U$ is an $\bR^{n+1}$-invariant open subset
%and the section $s_{\mathcal U}$ is also so. 
%Moreover, $\mathcal U$ corresponds to an open subset $\mathcal U'$ of $\Hom_E(\pi_1(\mathcal{O}), \SL_\pm({n+1}, \bR))$.

%Define $H^1(\mathcal{O}, \bR)$ as the space of homomorphisms $\pi_1(\mathcal{O}) \ra \bR$. 

%%% April 10 12:05am
Let $\mathcal U$ be the inverse image under $q_S$ of $\mathcal{U'}$ in \[\Hom^{sS}_E(\pi_1(\mathcal{O}'), \Aff(\bR^{{n+1}})).\]
Let \[q: (\bR^{n+1} -\{O\})^{e_1} \times (AS(\bR^{n+1}))^{e_2} \ra (\SI^n)^{e_1} \times (\SI^{n \ast})^{e_2}\]
be the obvious projections. 
Let the section $s_{\mathcal U}: {\mathcal{U}}  \ra (\bR^{n+1} -\{O\})^{e_1} \times (AS(\bR^{n+1}))^{e_2}$ 
be the one lifting $s'_{\mathcal U}$. (Here, each hyperplane in $\SI^n$ lifts to a hyperplane in $\bR^{n+1}$ through 
the fixed points of the holonomy groups of the center.)

%A radiant affine structure on $\mathcal{O}'$ can also be studied using its projection to $\SI^{n}$.
By Lemma \ref{lem:affsus}, an element of $\Def^{sS}_{A, E, \mathcal U,\, s_{\mathcal U}}(\mathcal{O}')$
gives us an element of $\Def_{\SI^{n}, E, \mathcal{U}', s'_{\mathcal{U}'}}(\mathcal{O})$: 
We can cover $\mathcal{O}'$ by radial cones with vertex at the origin and project to $\SI^{n}$. 
Each gluing of open radial cones becomes an element of $\SL_{\pm}(n+1, \bR)$ acting on $\SI^{n}$ 
with positive scalar factors forgotten.  
The parallel end structures and totally geodesic ideal boundary components for ends of $\orb'$ go to the radial end structures and 
the totally geodesic ideal boundary components of $\orb$. 
The isotopies in $\orb$ will give rise to isotopies in $\orb'$
suspending the vector fields on cross-sections preserving the parallel vector fields and the totally geodesic ideal boundary components.
%Conversely, isotopies in $\orb'$ give rise to isotopies in $\orb$ by taking averages under the circle action on $\orb'$ and projecting 
%the isotopies. 

Therefore, the following map $\mathcal{P}$ is defined:
\small
\begin{align} 
 \Def^{sS}_{A, \, E, \, \mathcal U, s_{\mathcal U}}(\mathcal{O}') & \quad \stackrel{\mathcal{P}}{\longrightarrow} & 
\Def^s_{\SI^{n}, \, E, \, {\mathcal{U}'}, \, q\circ s_{\mathcal U'}}(\mathcal{O}) \times H^1(\mathcal{O}, \bR) \times (\bR^+-\{1\}) \nonumber \\ 
\hol \downarrow  &   &  \downarrow \hol \quad \quad \nonumber \\
 \rep_E^{sS}(\pi_1(\mathcal{O}'), \Aff(\bR^{n+1}))  & \quad \longrightarrow & \rep^s_E(\pi_1(\mathcal{O}), \SL_\pm(n+1, \bR)) 
 \times H^1(\mathcal{O}, \bR) \times (\bR^+-\{1\}).
\end{align}
\normalsize
A section to $\mathcal{P}$ is defined by taking an affine suspension by the data in $H^1(\mathcal{O}, \bR) \times (\bR^+-\{1\})$
and the $(\SI^{n}, \SL_\pm(n+1, \bR))$-structures on $\mathcal O$ using the methods of Section \ref{sub:asusp}.
From this, we deduce that the horizontal maps are local homeomorphisms 
in the commutative diagram.
Since the left downarrow is a local homeomorphism, 
the result is proved. 

% \[\Def_{\SI^{n}, \, E, \, {\mathcal{U}'}, \, s_{\mathcal{U}'}}(\mathcal{O}) \ra  \rep_E(\pi_1(\mathcal{O}), \SL_\pm(n+1, \bR))\] is 
%a local homeomorphism where $s_{\mathcal{U}}$ is an arbitrary fixed section defined on $\mathcal{U}'$. 

\end{proof} 

%%% Oct 28 10:42 pm 
%The quotient map $\SL_\pm(n+1,\bR) \ra \PGL(n+1, \bR)$ is induced by the action of the group $\{\Idd, -\Idd\}$
%and is a double covering map.%
%The induced map  \[\hat q: \rep_E(\pi_1(\mathcal{O}), \SL_\pm(n+1, \bR)) \ra \rep_E(\pi_1(\mathcal{O}), \PGL(n+1, \bR))\] 
%is also induced by the same group action. 

% Oct 29, 2013 1:55pm 

The homomorphism $q:  \SL_\pm(n+1, \bR)) \ra \PGL(n+1, \bR)$ induces a continuous map
\[\hat q: \rep^s_E(\pi_1(\mathcal{O}), \SL_\pm(n+1, \bR)) \ra \rep^s_E(\pi_1(\mathcal{O}), \PGL(n+1, \bR)).\]
 
Let $\mathcal{U}$ be a conjugation invariant subset of $\Hom^s_E(\pi_1(\mathcal{O}), \PGL(n+1, \bR))$. 
 $s_{\mathcal{U}}$ is an arbitrary fixed section defined on $\mathcal{U}$. 
 Let $\mathcal{U}''$ denote the inverse image of $\mathcal{U}$ under $\hat q$. 
 We define $s'_{{\mathcal{U}''}}: {\mathcal{U}''} \ra (\SI^{n})^{e_1} \times (\SI^{n \ast})^{e_2}$ 
 be a continuous lift of $s_{\mathcal{U}}$. % which we obtain from the lifting as above. 
 The section is determined up to the action of $\{ \Idd, {\mathcal{A}}\}$ on each of $(e_1 + e_2)$-factors. 
 This gives us a section $\tilde s: \DefEO \ra \DefSO$ up to a choice of $s'_{{\mathcal{U}''}}$ by Theorem \ref{thm:lifting}.  
  The choice here is determined by the lifting of the development pair $(\dev, h)$. 
(For the lifting ideas, see p. 143 of Thurston \cite{Thbook}.)
 
 The map $\tilde q: \DefSO \ra \DefEO$ is induced by the action 
\[(\dev', h') \ra (q\circ \dev', \hat q \circ h').\] 

 It is easy to see that the section $\tilde s$ to $\tilde q$ is well-defined since the lifting $(\dev, h)$ give us development pairs 
 that are equivalent to up to $-\Idd \in \SL_\pm(n+1, \bR)$. 
The map $\tilde s$ is continuous since for a fixed compact subset of $\torb$ 
the $C^r$-closeness of the developing map to $\bR P^n$ means the $C^r$-closeness of the lifts
for $r \geq 1$.

Thus, we showed that
\begin{theorem}\label{thm:same}
Assume as in the above paragraphs. We obtain a homeomorphism 
\[\tilde q: \DefSO \ra \DefEO.\] \qed
\end{theorem}

\begin{corollary}\label{cor:same}
%Suppose that $\orb$ satisfies the end fundamental group condition. 
$\tilde q: \Def^s_{\SI^n, E, u}(\orb) \ra \Def^s_{E, u}(\orb)$ is a homeomorphism. 
\end{corollary}
\begin{proof}
In the unique eigenvector or eigen-$1$-form cases, the existence and the continuity of the sections are clear. 
\end{proof}

\begin{proof}[Proof of Theorem \ref{thm:projective}]
We have a commutative diagram: 
\begin{eqnarray} \label{eqn:com1}
\Def^s_{\SI^n, \, E, \, {\mathcal{U}}, \, s_{\mathcal{U}} }(\mathcal{O}) & \stackrel{ \tilde q}{\ra} &\DefEO \nonumber \\
\downarrow \hol \quad \quad &    &\quad \quad \downarrow \hol  \nonumber \\
\rep^s_E(\pi_1(\mathcal{O}), \SL_\pm(n+1, \bR)) & \stackrel{\hat q}{\ra} & \rep^s_E(\pi_1(\mathcal{O}), \PGL(n+1, \bR)).
\end{eqnarray}

First, we remark first that $\hat q$ maps onto the union of components with the associated Stiefel-Whitney number $0$.
%(See the commutative diagram equation \eqref{eqn:com1}.)
%(See Goldman's thesis \cite{Gthesis}.)
Since \[\SL_\pm(n+1, \bR) \ra \PGL(n+1, \bR)\] is a covering map, so is
\[\Hom^s_E(\pi_1(\mathcal{O}), \SL_\pm(n+1, \bR)) \ra \Hom^s_E(\pi_1(\mathcal{O}), \PGL(n+1, \bR))\] 
to the union of its image components (i.e., corresponding to ones with the corresponding 
Stiefel-Whitney classes equal to zero.) Each fiber is in one to one correspondence with $\Hom(\pi_1(\orb), \{\pm \Idd\})$. 
The induced map $\hat q$ is a local homeomorphism since the conjugation by $\SL_\pm(n+1, \bR)$  on the first space
is equivalent to one by $\PGL(n+1, \bR)$ since $\{\pm \Idd\}$ acts trivially. 

Since the left $\hol$ is locally onto and $\hat q$ is locally onto, 
so is the right $\hol$ by Theorem \ref{thm:same}. 

Given a neighborhood $V'$ in $\rep^s_E(\pi_1(\mathcal{O}), \PGL(n+1, \bR))$ that is in the image of the left $\hol$, 
we can find a local section to $\hat q$ as $\hat q$ is a local homeomorphism. Since the left $\hol$ is a local homeomorphism, 
and $\hat q$ is a local homeomorphism, 
there is a local section to the right $\hol$ by Theorem \ref{thm:same}.
\end{proof}

\section{A comment on lifting real projective structures} 

Let $\SL_{-}(n+1, \bR)$ denote the component of $\SL_{\pm}(n+1, \bR)$ not containing $\Idd$. 
A projective automorphism $g$ of $\SI^n$ is orientation preserving if and only if $g$ has a matrix in $\SL(n+1, \bR)$. 
For even $n$, the quotient map $ \SL(n+1, \bR) \ra \PGL(n+1, \bR)$ is an isomorphism and so is
the map $\SL_{-}(n+1, \bR) \ra \PGL(n+1, \bR)$ for the component 
of $\SL_\pm(n+1, \bR)$ with determinants equal to $-1$. For odd $n$, the quotient map 
$ \SL(n+1, \bR) \ra \PGL(n+1, \bR)$ 
is a $2$ to $1$ covering map onto its image component with deck transformations given by $A \ra \pm A$. 
Also, so is the map $\SL_{-}(n+1, \bR) \ra \PGL(n+1, \bR)$.

%Theorem \ref{thm:same} can be refined so that the holonomy lies in $\SL(n+1, \bR)$ and 
%the lifted structure is a $(\SI^n, \SL(n+1, \bR))$-structure. 

\begin{theorem}\label{thm:lifting} 
Let $M$ be a strongly tame orbifold.
Suppose that $h: \pi_1(M) \ra \PGL(n+1, \bR)$ is a holonomy homomorphism of real projective structure on $M$ with radial or  \index{real projective structure!lifting}
totally geodesic ends of lens-type. 
Then the image of $\hol$ for $\Def_E(\orb)$ in 
\[\rep^s(\pi_1(M), \PGL(n+1, \bR))\] is homeomorphic to 
that of $\hol$ for $\Def_{\SI^n, E}(\orb)$ in \[\rep^s(\pi_1(M), \SL(n+1, \bR)).\]  $h'$ is unique if $n$ is even. 
\begin{itemize}
\item Suppose that $M$ is orientable. 
We can lift to a homomorphism $h': \pi_1(M) \ra \SL(n+1, \bR)$, which is a holonomy homomorphism of 
the \hfil\break $(\SI^n, \SL_\pm(n+1, \bR))$-structure lifting the real projective structure. 
\item Suppose that $M$ is not orientable. Then we can lift $h$ to a homomorphism $h': \pi_1(M) \ra \SL_\pm(n+1, \bR)$
that is the holonomy homomorphism of the $(\SI^n, \SL_\pm(n+1, \bR))$-structure lifting the real projective structure
so that a deck transformation goes to a negative determinant matrix if and only if it reverses orientations. 
%$h'$ is unique if $n$ is even. 
%\item Finally, the map sending $[h]$ to $[h']$ in the conjugacy class in the small open set in the image of $\hol$
%in $\rep_E(\pi_1(\mathcal{O}), \PGL(n+1, \bR))$ to $\rep_E(\pi_1(\mathcal{O}), \SL_\pm(n+1, \bR))$
%is continuous and one-to-one. 
\end{itemize}
\end{theorem} 
\begin{proof}
Recall $\SL(n+1, \bR)$ is the group of orientation-preserving linear automorphisms of $\bR^{n+1}$ and hence is precisely 
the  group of orientation-preserving projective automorphisms of $\SI^n$. 
Since the deck transformations of the universal cover $\tilde M$ of the lifted $(\SI^n, \SL_\pm(n+1, \bR))$-orbifold 
are orientation-preserving, the holonomy of the lift are in $\SL(n+1, \bR)$. 
We use as $h'$ the holonomy homomorphism of the lifted structure. For even $n$, the uniqueness of $h'$ follows from 
the fact that $\SL(n+1,\bR)\ra\PGL(n+1,\bR)$ is a homeomorphism.

For the second part, we can double cover $M$ by an orientable orbifold $M'$ with an orientation-reversing 
$\bZ_2$-action of the projective automorphism group generated by $\phi:M'\ra M'$.
$\phi$ lifts to $\tilde \phi: \tilde M'\ra \tilde M'$ for the universal covering manifold $\tilde M'= \tilde M$  and hence 
$h(\tilde \phi) \circ \dev = \dev \circ \tilde \phi$ for  the developing map $\dev$ 
and the holonomy  \[h(\tilde \phi) \in \SL_-(n+1, \bR).\] 
Then it follows from the first item since $\dev$ preserves orientations for a given orientation of $\tilde M$. 
For even $n$, the uniqueness is 
the consequence of the uniqueness of the lift $h'$ in the orientable case
and the fact that $\SL_{-}(n+1,\bR)\ra\PGL(n+1,\bR)$ is a one-to-one homeomorphism also.
%The continuity of the lift follows from the mentioned one-to-one homeomorphisms. 
%Suppose that $[h_1], [h_2] \in V$ in a small open set $V$ are lifted to $h'_1$ and $h'_2$ in 
%$\Hom_E(\pi_1(\mathcal{O}), \SL_\pm(n+1, \bR)$ and 
%$h'_1(\gamma) = J h'_2(\gamma) J^{-1}$ for some $J \in \SL_\pm(n+1, \bR)$.
%However, the conjugation by $J$ do not change the orientation-preserving or reversing property of $\gamma$.
%Therefore, $h'_1 = h'_2$ and injectivity is proved. 
(See p. 143 of Thurston \cite{Thbook}.)
\end{proof}
In general, this proposition is used commonly but not written anywhere.

\part{Convexity of the orbifolds and the relative hyperbolic fundamental groups}

\chapter{Convexity}\label{sec:conv}
% other method

% Main result: Convexity while deforming

%%September 24 20:00 pm
%% Feb 17 11:41

%Sept 25 1:34pm
\section{Properties of ends}\label{sub:ends}

We will restate the results of \cite{endclass} here for general understanding need for what follows. 

\subsection{Properties of horospherical ends}

In \cite{endclass}, the horospherical ends were defined more generally but the definition was shown to be equivalent to ours. 
We will not repeat the definition. 

\begin{proposition}[Proposition 5.1 \cite{endclass}] \label{prop:affinehoro} 
Let $\mathcal O$ be a properly convex real projective $n$-orbifold with radial ends or totally geodesic ends of lens-type. 
Let $\tilde E$  be a horospherical p-end of its universal cover $\torb$ and $\Gamma_{\tilde E}$ denote the p-end fundamental group. 
Then the following statements hold: 
\begin{itemize}
\item[(i)] The space $S_{\tilde E} := R_{v_{\tilde E}}(\torb)$ \index{end!orbifold} 
of rays from the corresponding p-end point $v_{\tilde E}$ forms a complete affine subspace of dimension $n-1$. \index{$S_{\tilde E}$} \index{$R_{v_{\tilde E}}(\torb)$}
\item[(ii)] The norms of eigenvalues of $g \in \Gamma_{\tilde E}$
are all $1$.
\item[(iii)] A p-end point of a horospherical p-end cannot be on a segment in $\Bd \tilde{\mathcal{O}}$. 
\item[(iv)] For any compact set $K'$ inside a horospherical end-neighborhood, 
there exists a smooth convex smooth horospherical end-neighborhood disjoint from $K'$. 
%\item[(v)] $\pi_1(\tilde E)$ is virtually nilpotent. 
\item[(v)] $\pi_1({\tilde E})$ is virtually abelian and a finite index subgroup is in 
a conjugate of a parabolic subgroup of $\PO(n, 1)$ of rank $n-1$ in $\PGL(n+1, \bR)$ \rlp resp. $\SLpm$\rrp
that acts on an ellipsoid in $\clo(\torb) \subset \bR P^n$ \rlp resp. $\subset \SI^n$\rrp. %And hence $E$ has a cusp-type.
%\item Let $E$ be a complete end. %Suppose that $\pi_1(E)$ has holonomy with eigenvalues of absolute value $1$ only.
%Then $E$ is horospherical. 
\end{itemize}
\end{proposition}

The converse result is the following. 

\begin{theorem}[Theorem 5.2 \cite{endclass}]\label{thm:comphoro} 
Let $\orb$ be a properly convex $n$-orbifold with radial ends or totally geodesic ends of lens-type. 
Suppose that $E$ is a radial p-end of its universal cover $\torb$. Let $v_{\tilde E} \in \bR P^n$ 
be the p-end point and $\pi_1({\tilde E})$ be the p-end fundamental group corresponding to $E$. 
Suppose that the space $S_{\tilde E} := R_{v_{\tilde E}}(\torb)$ 
of rays from the corresponding p-end point $v_{\tilde E}$ forms a complete affine subspace of dimension $n-1$.
Then the following statements hold. 
\begin{itemize} 
\item[(i)] The eigenvalues of elements of  $h(\pi_1({\tilde E}))$ have unit norms only. 
\item[(ii)] A finite index subgroup of $h(\pi_1({\tilde E}))$ is contained in a unipotent group fixing $v_{\tilde E}$. 
\item[(iii)] $E$ is horospherical. %, i.e., cuspidal. 
%\item[(iv)] There exists a p-end neighborhood $U$ of $E$ where $\Bd U$ is a $C^\infty$-ellipsoid.
\end{itemize} 
\end{theorem}

%July 9 8:41pm Now I need to compare with endclass Check here. Then go to Hilbert metrics.
\subsection{The properties of lens-shaped ends} \label{subsec:lens}

%\begin{figure}
%\centerline{\includegraphics[height=5cm]{figlenslem}}
%\caption{The figure for Lemma \ref{lem:recurrent}. }
%\label{fig:lenslem}
%\end{figure}

%%% Aug 9 2013 5:06pm

%that is a component of a complement of a lens-shaped domain $D$ invariant under the p-end fundamental group 
%from the pseudo-end neighborhood that is the cone over $D$.
% An {\em open concave p-end neighborhood} is 
%the complement of a lens-part in a lens-shaped cone. 
%\marginpar{repeated!}

% and is an an imbedded end neighborhood contained in 
%a radial end neighborhood that is a component of a complement of the $(n-1)$-dimensional Fuchsian domain 
%when the end is a Fuchian cone 
%The image is an end neighborhood up to taking a finite cover of $\mathcal{O}$. 
%If we take a sufficiently large lens out, the image is an end neighborhood. 
%In this case, its image in $\mathcal{O}$ is also 
%said to be {\em concave end-neghborhood} of a corresponding end.
%(Note that the second statement is a stronger condition.)

A {\em trivial one-dimensional cone} is an open half space in $\bR^1$ given by $x > 0$ or $x < 0$.

Recall that if $\pi_1(E)$ is an admissible group, then $\pi_1(E)$ has  a finite index subgroup isomorphic to 
$\bZ^l \times \Gamma_1 \times \cdots \times \Gamma_k$ for some $l$ and $k$ for $l \geq k-1$
where each $\Gamma_i$ is hyperbolic. 
Here, we can identify $\tilde{\mathcal{O}}$ as a convex domain in $\bR P^n$ (resp. in $\SI^n$) for convenience.  %\marginpar{Should it be here.}
%We show that $k = n =l$ for $l$ the rank of the virtual free abelian center of $\pi_1(E)$.

Let us consider $E$ as a real projective $(n-1)$-orbifold and consider $\tilde E$ as a domain 
in $\SI^{n-1}$ and $h(\pi_1(E))$ induces $h': \pi_1(E) \ra \SL_\pm(n, \bR)$ acting on $\tilde E$.

%Sept 25 3:37pm
\begin{theorem}[Proposition 6.7 \cite{endclass}]\label{thm:lensclass}
Let $\mathcal{O}$ be a noncompact strongly tame properly convex real projective 
$n$-orbifold with radial ends or totally geodesic ends of lens-type and satisfies {\rm (IE)} and {\rm (NA)}.
%and there are no essential annuli or tori. 
Let $\tilde E$ be a generalized lens-shaped radial p-end of $\tilde{\mathcal{O}}$ in $\bR P^n$ \rlp resp. in $\SI^n$\rrp\,
associated with a p-end vertex $v_{\tilde E}$.
Assume that $\pi_1(\tilde E)$ is hyperbolic.  \index{end!radial!generalized lens}
%We denote by $h'(\pi_1(E))$ the quotient image of $h(\pi_1(E))$ in $\PGL(n, \bR)$, the projective automorphism group 
%of the space $\bR P^n_v$ of lines at $v$.
Then the following statements hold\,{\rm :}
\begin{itemize} 
\item[(i)] 
%The complement of the manifold boundary of the generalized lens-shaped domain $D$ is a nowhere dense set
%in $\Bd \clo(D)$ in $\bR P^n$ \rlp resp. in $\SI^n$\rrp. Moreover, $\Bd \clo(D) -\partial D$ is independent of the choice of $D$.
%That is, 
$D$ is strictly generalized lens-shaped. Moreover, each element $g \in \Gamma_{\tilde E}$ has an attracting 
fixed point in $\Bd \clo(D)$ in the ray from $v_{\tilde E}$ in the direction of $\Bd \Omega_{\tilde E}$. 
The set of attracting fixed points is dense in $\Bd \clo(D) - A - B$ for the top and the bottom hypersurfaces $A$ and $B$
forming the boundary of the lens $D$. 
%\item Suppose that $\pi_1(E)$ is Benoist. Then $\pi_1(E)$ is Gromov hyperbolic. 
\item[(ii)] %Let $v$ be a vertex of the end $E$ in $\rpn$. 
The closure in in $\bR P^n$ \rlp resp. in $\SI^n$\rrp  of a concave p-end-neighborhood of $v_{\tilde E}$ contains 
every segment $l$ in $\Bd \tilde{\mathcal{O}}$ meeting the closure of a concave p-end neighborhood of $v_{\tilde E}$ in $l^o$.
The set $S(v_{\tilde E})$ of maximal segments from $v_{\tilde E}$ in the closure of a p-end-neighborhood of $v_{\tilde E}$ is independent of 
the p-end-neighborhood, and $\bigcup S(v_{\tilde E})$ equals the closure of any p-end neighborhood of $v_{\tilde E}$
intersected with $\Bd \tilde{\mathcal{O}}$.
%\item[(iii)] Any end neighborhood $U'$ covering an end neighborhood in $\torb$ is a one-sided neighorhood of 
%$\bigcup S(v_{\tilde E})$: i.e., $U' \cap \bigcup S(v_{\tilde E})$ is an open neighborhood of $\bigcup S(v_{\tilde E})$.  
%\item[(iii)] One can find a concave end-neighborhood $U$ of $v_{\tilde E}$ mapping into an end-neighborhood of $\orb$. 
\item[(iii)] Any concave p-end neighborhood $U$ of $v_{\tilde E}$ under the covering map $p_{\orb}$
covers the p-end neighborhood of $E$ of form $U/\pi_1(E)$. That is, a concave p-end neighborhood is a proper p-end neighborhood. 
\item[(iv)] $\bigcup S(g(v_{\tilde E}))=g(\bigcup S(v_{\tilde E}))$ for $g \in \pi_1(E)$. Assume that $w$ is 
the p-end vertex of an irreducible hyperbolic p-end.  
$\bigcup S(v_{\tilde E})$ is an $(n-1)$-ball. 
Then $\bigcup S(v_{\tilde E})^o \cap \bigcup S(w) = \emp$ or $v_{\tilde E}=w$ for 
p-end vertices $v_{\tilde E}$ and $w$. %where $S^o(v_{\tilde E})$ is
%the relative interior of $\bigcup S(v_{\tilde E})$ in $\Bd \tilde{\mathcal{O}}$. 
\end{itemize}
\end{theorem}

%\begin{proposition}\label{prop:lenssupp} 
%Let $\mathcal{O}$ be a strongly tame $n$-orbifold with radial or totally geodesic ends and satisifies (IE) and (FA). 
%Let $E$ be a lens-shaped radial p-end of $\tilde{\mathcal{O}}$ associated with a p-end vertex $v_{\tilde E}$.
%Assume that $\pi_1(E)$ is hyperbolic. 
%Then let $D$ be the lens part of the lens-shaped domain of a lens-type pseduo-end $\tilde E$. 

%\end{proposition} 
%\begin{proof}

%By the results of \cite{endclass}, for a lens-shaped p-end, the end fundamental group to be denoted $\Gamma$ acts on a tube in a distanced manner. 
%The dual group $\Gamma^*$ acts on affine subspace and by Theorem A.1, we obtain a properly convex domain 
%$U'$ in an affine space $A$ with boundary $\Lambda^* := U \cap \Bd A$ is a properly convex domain. 
%$U' = \bigcap_{x\in \Lambda^*} H'_x$ where $H'_x$ is the supporting half-space in $A$ with 
%boundary $\Bd H'_x$ containing $x$ and $\Bd H'_x$ transversal to $\Bd A$. 
%We define $U :=  \bigcap_{x\in \Lambda^*} H_x$ where $H_x$ is the unique hemisphere in $\SI^n$ containing $H'_x$ 
%and $\partial H'_x $ in $\partial H_x$. 
%The space of pairs 
%\[\{(x, [\alpha])| x \in \Lambda^*, \alpha \in \bR P^*, \alpha(y) > 0 \hbox{ for } y \in EEEE. \]

%\end{proof} 

%Note that Theorem \ref{thm:lensclass} (iii) hold without the hyperbolicity condition on $\pi_1(E)$.

Now we go to the cases when admissible $\pi_1(E)$ has more than two nontrivial abelian or hyperbolic factors. 
The following theorem shows that each lens-shaped end is totally geodesic and has well-defined $S(v)$ in this case. 
The author obtained the proof of (i-3) from Benoist.

\begin{theorem}[Theorem 6.9 \cite{endclass}]\label{thm:redtot}
Let $\mathcal{O}$ be a noncompact strongly tame properly convex real projective 
$n$-orbifold with radial ends or totally geodesic ends of lens-type and  \index{end!radial!generalized lens}
satisfies {\rm (IE)} and {\rm (NA)}.  
Suppose that the holonomy of $\mathcal{O}$ is strongly irreducible. 
Let $\tilde E$ be a generalized lens-shaped radial p-end of $\tilde{\mathcal{O}}$ in $\bR P^n$ \rlp resp. in $\SI^n$\rrp 
associated with a p-end vertex $v_{\tilde E}$.
%Moreover, if an end is totally geodesic, then it is a generalized lens-shaped and hence is admissible.
Let $\pi_1(\tilde E)$ be the p-end fundamental group corresponding to $\tilde E$ containing a finite index abelian subgroup 
isomorphic to $\bZ^l \times \Gamma_1 \times \cdots \times \Gamma_k$. 
Assume $l \geq 1$.
Then the following statements hold\,{\rm :}
\begin{itemize} 
\item[(i)] For $\SI^{n-1}_{v_{\tilde E}}$, we obtain 
\begin{itemize}
\item[(i-1)]  Under the action of the induced group $\hat h(\pi_1(E))$ of the holonomy group $h(\pi_1(E))$,  
 $\bR^{n}$ splits into $V_1 \oplus \cdots \oplus V_{l_0}$ and $S_{\tilde E}$ is the quotient of the sum 
$C_1+ \cdots + C_{l_0}$ for properly convex or trivial one-dimensional cones $C_i \subset V_i$ for $i=1, \dots, l_0$
\item[(i-2)] The Zariski closure of a finite index subgroup of $\hat h(\pi_1(E))$ is isomorphic 
to the product $G = G_1 \times \cdots \times G_{l_0} \times \bR^{l_0-1}$ 
where $G_i$ is a reductive subgroup of $\SL_{\pm}(V_i)$.  
\item[(i-3)] Let $D_i$ denote the image of $C_i$ in $\SI^{n-1}_{v_{\tilde E}}$.
The number of hyperbolic group factors of $\pi_1(E)$ is $\leq l_0$ and 
each hyperbolic group factor of $\pi_1(E)$ divides
exactly one $D_i$ and acts on other factors trivially.
\item[(i-4)] A finite-index subgroup of $\pi_1(\tilde E)$ has a rank $l_0-1$ free abelian group center corresponding to $\bZ^{l_0-1}$ in $\bR^{l_0-1}$.
\end{itemize}
\item[(ii)] The p-end is  totally geodesic radial p-end of lens-type. % and is contained in a hyperspace $P$. 
$D_i$ corresponds to totally geodesic convex $(n-1)$-ball $D'_i$ disjoint from $v_{\tilde E}$. 
\item[(iii)] $g$ in the center is diagonalizable with positive eigenvalues. 
For a nonidentity element $g$  in the center, the eigenvalue $\lambda_{v_{\tilde E}}$ 
of $g$ at ${v_{\tilde E}}$ is strictly between its largest and smallest eigenvalues. 
\item[(iv)] The p-end is strictly lens-shaped and each $C_i$ corresponds to a cone $C^*_i$ 
over a totally geodesic $(n-1)$-dimensional domain $D'_i$ with $v_{\tilde E}$.  
$C^*_i$ contains a concave open invariant set $U_i$. 
The p-end has a p-end neighborhood that is a strict join of $D'_1,.., D'_{l_0}$ with $v_{\tilde E}$ 
where the strict join $D'$ of $D'_1,.., D'_{l_0}$ forms the boundary. 
They are in a lens part of $\tilde E$ for any lens-type p-end neighborhood, and the top 
and the bottom hypersurfaces of the lens part have the boundary in the boundary of $D'$. 
\item[(v)] $\bigcup S(v_{\tilde E})$ is equal to the union of maximal segments with vertex $v_{\tilde E}$ in 
the union $\bigcup_{i=1}^j v_{\tilde E} *D'_1*\cdots *\check{D'_i}* \cdots * D'_{l_0}$.
%where $S_i(v_{\tilde E})$ is the set of maximal segments with vertex $v_{\tilde E}$ in $\clo(C'_i) \cap U_i$.
\item[(vi)] A concave p-end neighborhood of $\tilde E$ is a proper end neighborhood. 
Also the statement in this case for {\rm (iv)} of Theorem \ref{thm:lensclass} holds.
\end{itemize}
\end{theorem}

%Hence, if $\pi_1(E)$ satisfies the assumptions of Theorems \ref{thm:lensclass} or \ref{thm:redtot}, then 
%the subspace of lens-shaped representations is open. 

\begin{theorem}[Theorem 8.1 \cite{endclass}]\label{thm:qFuch}
Let $\orb$ be a strongly tame properly convex real projective manifold with radial ends or totally geodesic ends of lens-type and satisfies 
 {\rm (IE)} and {\rm (NA)}. Let the holonomy $h(\pi_1(\orb))$ be strongly irreducible. 
Let $\tilde E$ be a properly convex radial p-end of $\tilde{\mathcal{O}}$ in $\bR P^n$ \rlp resp. in $\SI^n$\rrp 
associated with a p-end vertex $v_{\tilde E}$.
Let \[\Hom^s_E(\pi_1(\tilde E), \PGL(n+1, \bR)) \hbox{ {\rm (resp.} } \Hom^s_E(\pi_1(\tilde E), \SL_\pm(n+1, \bR)) )\] be the space of representations of the 
fundamental group of an $(n-1)$-orbifold $\Sigma_{\tilde E}$ with an admissible fundamental group. 
Then 
\begin{itemize}
\item[(i)] the generalized lens-shapedness of an end is equivalent to the strict generalized lens-shapedness of the end, and  \index{end!radial!strict generalized lens}
\item[(ii)] the subspace of generalized lens-shaped  representations in the above space is open. 
\end{itemize}
\end{theorem}

%%% Nov 4th.. 12:08 pm I need to end fund auto for ce conditioned cases.... Need to say this...

\begin{lemma} \label{lem:concave}
Given two concave p-end neighborhoods $U$ and $V$, either they have the same p-end and $U \cap V$ is another concave 
p-end neighborhood or $U \cap V = \emp$ when they have distinct p-ends. 
\end{lemma}
\begin{proof}
Let $\tilde E_1$ and $\tilde E_2$ be the p-end associated with $U$ and $V$ respectively. 
If $U \cap V \ne \emp$, then $S^o(v_{\tilde E_1})$ intersect $S^o(v_{\tilde E_2})$ since 
the lens for $U$ is supported by a totally geodesic hyperspace and so is $V$. 
Thus, the conclusion follows by Theorems \ref{thm:lensclass}(iv) and \ref{thm:redtot}(vi).
\end{proof}

%\begin{remark}\label{rem:realization} 
%Two end neighborhoods are simply-equivalent if one is a subset of the other. 
%We will use the equivalence relation generated by this. 

%If a hyperbolic group is acting on a lens-shaped cone in an affine space, 
%then the cone is uniquely determined up to the antipodal maps. 
%This follows since the attracting and repelling fixed-points are dense in the boundary of the lens
%corresponding to the cone and the holonomy homomorphism irreducible by the result of Benoist \cite{Ben1}. 
%(If there are more than two fixed points, then the space of subspaces containing the two fixed points 
%gives us reducibility.)
%If a holonomy homomorphism realizes a lens-shaped admissible end, then 
%by unique decomposition part of Theorems \ref{thm:lensclass} and \ref{thm:redtot}, 
%every other realized lens-shaped 
%admissible end is actually projectively isomorphic; that is, up to mapping given by
%induced by the direct sum of $\pm \Idd_i|V_i$ for a decomposition $\bR^{n+1} = \oplus V_i$.
%(Usually realized as minus identity for some part of
%the decomposition and identity for other parts of the decomposition.)

%The uniqueness is clear for horospherical ends up to equivalences in $\bR P^n$.
%\end{remark}

%% May 15 2:52pm 2014

\section{Expansion and shrinking of admissible p-end neighborhoods} 

%Add here the things about convex hull neighborhood imply lens....

\begin{definition}\label{defn:limit} 
Let $\Lambda$ denote $\Bd L - \partial L$ for a generalized lens of a radial p-end or a lens of a totally geodesic ends. \index{end!limit set}
We call this set the {\em limit set} of the p-end. 
\end{definition}
Obviously the limit set of a p-end is independent of the choice of lens by Corollary 8.5 of \cite{endclass}.

%See \cite{endclass} for proofs of the following propositions. 
\begin{lemma}[Lemma 8.6 \cite{endclass}]\label{lem:expand}  
Let $\mathcal{O}$ have a noncompact strongly tame SPC-structure $\mu$ with admissible ends. 
%Assume that the holonomy group of $\pi_1(\orb)$ is strongly irreducible.
Let $U_1$ be an admissible p-end neighborhood of a lens-type radial p-end with the vertex $v$ in $\tilde{\mathcal{O}}$
that is foliated by segments from $v$ or a totally geodesic p-end $\tilde E$.
\begin{itemize} 
\item $\torb$ contains a sequence of convex open neighborhoods $U_i$ of $\tilde E$ 
so that $(U_i - U_j)/\Gamma_v$ for a fixed $j$ and $i> j$ is homeomorphic to a product of an open interval with 
the end orbifold. 
\item Given a compact subset $K$ of $\tilde{\mathcal{O}}$, there exists an integer $i_0$ such that 
$U_i$ for $i > i_0$ contains $K$. 
\item We can choose $U_i$ so that $\partial U_i$ is smoothly imbedded and strictly convex with 
$\partial \clo(\partial U_i) \subset \Lambda$ where $\Lambda$ is the limit set contained in $\bigcup S(v)$
if $v$ is the p-end vertex when $\tilde E$ is radial and in $\partial S_{\tilde E}$ if $\tilde E$ is total geodesic. 
\item The Hausdorff distance between $U_i$ and $\tilde{\mathcal{O}}$ can be made as small as possible. 
\end{itemize}
\end{lemma}

See the definition of convex hull of an end in Section \ref{subsub:convh}.
\begin{lemma}[Lemma 8.7 \cite{endclass}] \label{lem:shrink} 
Suppose that $\orb$ is a strongly tame properly convex real projective orbifold with radial or totally geodesic ends of lens-type
and let $\torb$ is a properly convex domain in $\bR P^n$ \rlp resp. in $\SI^n$\rrp covering $\orb$. 
Assume that the holonomy group of $\pi_1(\orb)$ is strongly irreducible.
\begin{itemize} 
\item[(i)] If $\tilde E$ is a horospherical p-end, any p-end neighborhood of $\tilde E$ contains a horospherical p-end neighborhood.
\item[(ii)] If $\tilde E$ is a generalized lens-shaped p-end, any p-end neighborhood whose closure covers a compact end neighborhood and containing the convex hull
of the end contains a lens-shaped p-end neighborhood. %, and there exists a lens-cone containing $\torb$ properly. 
\item[(iii)] If $\tilde E$ is a generalized lens-shaped p-end, % or satisfies the uniform middle eigenvalue condition, 
any p-end neighborhood of $\tilde E$ contains a concave p-end neighborhood. % inside any p-end neighborhood of $\tilde E$.
\item[(iv)] If $\tilde E$ is totally geodesic p-end of lens type, % or satisfies the uniform middle eigenvalue condition, 
any end neighborhood contains a one-sided lens p-end neighborhood with strictly convex boundary in $\torb$. % inside 
%any end neighborhood. %and there exists a convex lens containing $\torb$ invariant under $\pi_1(\tilde E)$ which contains the ideal boundary 
%$S_{\tilde E}$ of $\tilde E$ in $\Bd \torb$.  
\end{itemize} 
\end{lemma}

Proposition 6.7 and Theorem 6.9 of \cite{endclass} imply the following: 
\begin{corollary}\label{cor:independence} 
Let $\mathcal{O}$ be a noncompact strongly tame $n$-orbifold with radial or totally geodesic ends of lens-type and satisfies {\rm (IE)} and {\rm (NA)}.
Assume that the holonomy group of $\pi_1(\orb)$ is strongly irreducible. 
Let $\torb$ is a properly convex domain in $\bR P^n$ \rlp resp. in $\SI^n$\rrp covering $\orb$, 
and let $\tilde E$ be a generalized lens-shaped radial p-end of $\tilde{\mathcal{O}}$ associated with a p-end vertex $v_{\tilde E}$,
or $\tilde E$ is a totally geodesic p-end of lens-type or a horospherical end. Let $U$ be a p-end neighborhood of $\tilde E$.
%Assume that the holonomy group of $\pi_1(\orb)$ is strongly irreducible. 
Then $\clo(U) \cap \Bd \torb$ is independent of the choice of $U$.
\end{corollary}

%\marginpar{notation the universal cover $p_{\orb}$. Unify. } 
\begin{corollary} \label{cor:strictconv} 
Suppose that $\mathcal O$ is a noncompact strongly tame strictly SPC-orbifold with generalized admissible ends. %only horospherical ends, totally geodesic ends of lens-type, 
%or radial ends of generalized lens-type. 
%Assume that the holonomy group of $\pi_1(\orb)$ is strongly irreducible.  
Let $\torb$ is a properly convex domain in $\bR P^n$ \rlp resp. in $\SI^n$\rrp covering $\orb$.
Choose any disjoint collection of end neighborhoods in $\orb$. Let $U$ denote their union. \index{$p_{\orb}$}
Let $p_{\orb}: \torb \ra \orb$ denote the universal cover. 
Then any segment or a non-$C^1$-point of $\Bd \torb$ is contained in the closure of a component of $p_{\orb}^{-1}(U)$ for 
any choice of $U$. 
\end{corollary}
\begin{proof}
%For a radial p-end with a concave p-end neighborhood, 
%the closure of a p-end neighborhood meets $\Bd \torb$ at $\bigcup S(v)$ for the corresponding $p$-end $v$
%by Proposition \ref{prop:I}. 
%
By the definition of a strictly SPC-orbifold, any segment or a non-$C^1$-point has to be in the closure of 
a p-end neighborhood.  Corollary \ref{cor:independence} proves the claim. 
\end{proof}

%% April 10th 2014 4:42pm

\section{Duality of ends} 
The totally geodesic ends of a properly convex real projective orbifolds are properly convex necessarily. %\marginpar{right place}

%\marginpar{Important; state in the beginning? or emph: Maybe move this whole to endclass? Prob. move this.} 
\begin{theorem}[Proposition 6.4 \cite{endclass}]\label{thm:duality} 
Let $\orb$ be a noncompact strongly tame properly convex real projective orbifold 
with horospherical or properly convex radial or totally geodesics ends. 
Let $\torb$ be the convex domain in $\bR P^n$ {\rm (}resp. $\SI^n$\,{\rm )} that covers $\orb$ and $\Gamma$ the projective deck transformation group
in $\PGL(n+1, \bR)$ {\rm (} resp. $\SLpm$\,{\rm ).}
Let $\torb^*$ be the dual convex domain and $\Gamma^*$ the dual group to $\Gamma$. 
Then $\orb^*:= \torb^*/\Gamma^*$ is a noncompact strongly tame properly convex real projective orbifold with radial or totally geodesics ends. 
\begin{itemize} 
\item The set of p-end fundamental groups $\pi_1(\tilde E_i)$ of $\torb$ corresponds to the set of 
p-end fundamental groups $\pi_1(\tilde E_i)^*$ of $\torb^*$. 
\item There is a one-to-one correspondence 
\begin{align} \label{eqn:oneone}
& \{\tilde E| \tilde E \hbox{ is a radial properly convex p-end of } \torb\} \Leftrightarrow \nonumber \\ 
& \{\tilde E| \tilde E \hbox{ is  a totally geodesic p-end of } \torb^*\},
\end{align} 
\item another one 
\begin{align} \label{eqn:oneone2}
& \{\tilde E| \tilde E \hbox{ is a totally geodesic p-end of } \torb\} \Leftrightarrow \nonumber \\ 
& \{\tilde E| \tilde E \hbox{ is  a radial properly convex p-end of } \torb^*\},
\end{align} 
\item and one for the set of horospherical ends of $\torb$ with  and the set of horospherical ends of $\torb^*$. 
%\item The ends of $\orb$ are admissible if and only if so are the ends of $\orb^*$.
\end{itemize}
\end{theorem} 

For correspondences of admissible ends, see Lemma \ref{lem:concavedual}.

%%May 26 2014 6:48pm 
\section{Totally geodesic ends and duality} \label{subsec:totdual}

We discuss somewhat more on totally geodesic ends. 
For totally geodesic ends, by the lens condition, we only consider the ones that have lens neighborhoods in some ambient orbifolds, 
i.e., admissible ones. 
First, we discuss the extension to bounded orbifolds.

\begin{theorem}\label{thm:totgeoext} 
Suppose that $\mathcal O$ is a noncompact strongly tame properly convex real projective orbifold with generalized admissible ends. 
Assume that the holonomy group of $\pi_1(\orb)$ is strongly irreducible. 
Let $E$ be a lens-shaped totally geodesic end, and let $\Sigma_E$ be a totally geodesic hypersurface that is the ideal 
boundary corresponding to $E$. 
 Let $L$ be a lens-shaped end neighborhood of $\Sigma_E$ in an ambient real projective orbifold containing $\orb$. 
 Then $L \cup \orb$ is a properly convex real projective orbifold 
 and has a strictly convex boundary component 
 corresponding to $E$. 
 Furthermore if $\orb$ is strictly SPC and $\tilde E$ is a hyperbolic end, then so is $L \cup \orb$
 which now has one more boundary component and one less totally geodesic ends. 
\end{theorem} 
\begin{proof} 
It is sufficient to prove for $\SI^n$ cases here. 
Let $\torb$ be the universal cover of $\orb$ which we can identify with a properly convex bounded domain in 
an affine subspace. Then $\Sigma_E$ corresponds to a p-end $\tilde E$ and 
 to a totally geodesic surface $S= S_{\tilde E}$. The lens $L$ is covered by 
 a lens $\tilde L$ containing $S$. The p-end fundamental group $\pi_1(\tilde E)$ acts on
 $\torb$ and $\tilde L_1$ and $\tilde L_2$ the two components of $\tilde L - S_{\tilde E}$ in $\torb$ and
 outside $\torb$ respectively. 
 
 %We will lift all these domains to $\SI^n$ and lift the group $\pi_1(E)$ as well. 
 %By an abuse of notation, we will just select the lifts that are adjacent to the chosen lift of $\torb$. 
%\marginpar{I need some ref... here in the paper?} 

% $\pi_1(E)$ satisfies the uniform middle eigenvalue condition since it acts on a lens-domain by Theorem 7.9 in
% \cite{endclass}. 
% \marginpar{wrong ref.. which one? Theorem 4.8?}

%Define $\Lambda = \Bd L - \partial L$. 
%First assume that our group $\pi_1(\tilde E)$  is hyperbolic.

\begin{definition}\label{defn:asymp} 
Let $\bR^n$ denote the affine subspace in $\SI^n$ with boundary $\SI^{n-1}_\infty$. 
Suppose that $\Omega$ is a strictly convex open domain in $\SI^{n-1}_\infty$. 
Given a convex open domain $\Omega_1$ with $\Bd \Omega_1 \supset \clo(\Omega)$ in $\bR^n$, 
the supporting hyperplanes at $p \in \Lambda = \clo(\Omega)-\Omega$ contains the unique hyperplane of codimension-two supporting $\Omega$.
Let \[A_p :=\{ H| H \hbox{ is a supporting hyperspace of $\Omega_1$ at $p$ in $\bR^n$} \} \]
and hence the space $A_p$ of such hyperspaces is homeomorphic to an arc.
An {\em asymptotic supporting hyperplane} at a point $p$ of $\Lambda$ is a supporting hyperplane at $p$
so that there exists no other element of $A_p$ closer to $\Omega_1$ from a point of $\Bd \Omega_1 - \clo(\Omega)$
(using minimal distance between a point and a set).
\end{definition} 

\begin{lemma} \label{lem:commsupp} 
Suppose that $S_{\tilde E}$ is the totally geodesic ideal boundary of 
a lens-type totally geodesic end $\tilde E$ of a strongly tame real projective orbifold $\orb$
and $\pi_1(\tilde E)$ is hyperbolic. 
\begin{itemize} 
\item Given a $\pi_1(\tilde E)$-invariant convex open domain $\Omega_1$ containing $S_{\tilde E}$ in the boundary, 
at each point of $\Lambda$, there exists a unique asymptotic supporting hyperplane. 
\item The hyperspace supporting any $\pi_1(\tilde E)$-invariant convex open set $\Omega$ containing $S_{\tilde E}$ at each point of $\Lambda$ is unique. 
\item Given two $\pi_1(\tilde E)$-invariant convex open domains $\Omega_1$ containing $S_{\tilde E}$ in the boundary
and $\Omega_2$ containing $S$ in the boundary from the other side, $\Omega_1 \cup \Omega_2$ is a convex domain 
and $\clo(\Omega_1) \cap \clo(\Omega_2) = \clo(S_{\tilde E})$ and their asymptotic supporting hyperplanes at each point of $\Lambda$ coincide.
\end{itemize} 
\end{lemma} 
\begin{proof} 
Let $A$ denote the affine subspace that is the complement of $\SI^n$ of the subspace containing $\torb$. 
Because $\pi_1(\tilde E)$ acts on a lens-type domain, the dual group of $h(\pi_1(\tilde E))$ is the holonomy group of lens-type radial p-end. 
By Theorem 7.9 of \cite{endclass}, $h(\pi_1(\tilde E))$ satisfies the uniform middle eigenvalue condition of \cite{endclass}. 

%%% April 15 2:43pm need to see below... correct...
%Suppose first that $\pi_1(\tilde E)$ is hyperbolic. 
If $\Omega_1$ has an asymptotic supporting half-space $H(x)$ for each $x \in \Lambda$ containing $\Omega_1$. 
$H(x)$ is uniquely determined by $\pi_1(\tilde E)$ and $x$ by Lemmas A.7 and  A.8 of \cite{endclass}. 
The uniqueness is in Lemma A.8 of \cite{endclass}. 

%For the second part, $\Omega \cap A$ has two $\pi_1(\tilde E)$-invariant open domains $\Omega_1$ and $\Omega_2$. 
%For $\Omega_2$ attached at the other side, the same reasoning shows that $H(x)$ is a unique asymptotic supporting hyperplanes. 
%Thus, the asymtotic supporting hyperplane at the same point of $\Lambda$ is the same one since they both contain 
%the same subspace of $A$ codimension-two as the results in the Appendix A of \cite{endclass} shows. \marginpar{more details?} 

The third item follows since the asymptotically supporting hyperplane at each point of $\clo(S_{\tilde E}) - S_{\tilde E}$ 
to $\Omega_1$ and $\Omega_2$ have to agree by Lemma A.7 of \cite{endclass}. 
The convexity follows easily from this. 

%Now suppose that $\pi_1(\tilde E)$ is reducible. The dual end $\tilde E^*$ of $\torb^*$ is also reducible and hence 
%is totally geodesic radial end. That is, $\pi_1(\tilde E^*)$ acts on a hyperplane and a fixed point $v_{\tilde E^*}$ outside it. 
%By duality, $\pi_1(\tilde E)$ acts on a hyperplane $P$ and has a fixed point $v$. 
%In this case, each point of $\Lambda$ is on $P$ and the hy

\end{proof}

Suppose that $\pi_1(\tilde E)$ is hyperbolic. 
By Lemma \ref{lem:commsupp}, $\tilde L_2 \cup S \cup \torb$ is a convex domain. 
If $\tilde L_2 \cup \torb$ is not properly convex, then it is a union of 
two cones over $S_{\tilde E}$ from a point $x \in A$ but made by two distinct choices 
of $\pm v_x \in \bR^{n+1}, [v_x] = x$ and the same cone over $\torb$. 
This means that $\torb$ has to be a cone contradicting the irreducibility of $h(\pi_1(\orb))$. 
%However, there exists a triangle $T$ in $\clo(\orb)$ with 
%one edge in $S_{\tilde E}$ and two other edges from $x$ to the end points of the edge. 
%Since all edges in $\Bd \torb$ is in the closures of the end neighborhoods which 
%form a collection of connected closed sets,  it follows that $\partial T \subset \clo(S_{\tilde E})$
%and $T \subset \clo(S_{\tilde E})$, a contradiction. 
Hence, it follows that $\tilde L_2 \cup \torb$ is properly convex. 

Suppose that $\orb$ is strictly SPC and $\pi_1(\tilde E)$ is hyperbolic. 
Then every segment in $\Bd \torb$ or a non-$C^1$-point in $\Bd \torb$
is in the closure of one of the p-end neighborhood.
$\Bd \tilde L_2 - \clo(S_{\tilde E})$ does not contain any segment in it or a non-$C^1$-point. 
$\Bd \torb - \clo(S_{\tilde E})$ does not contain any segment or a non-$C^1$-point outside 
the union of the closures of p-end neighborhoods. 
By Corollary \ref{lem:commsupp} , at each point of $\Lambda := \clo(S_{\tilde E}) - S_{\tilde E}$, 
$\torb \cup \tilde L_2 \cup S_{\tilde E} $ is $C^1$ and $\Lambda$ does not contain a segment. 
This follows because $S_{\tilde E}$ is strictly convex for $\pi_1(\tilde E)$ is a hyperbolic group. (See Theorem 1.1 of \cite{Ben1}.)
Therefore, $L_2 \cup \orb$ is strictly convex relative to the ends.

Suppose now that $\pi_1(\tilde E)$ is a product of hyperbolic and abelian groups.
Then the dual of the totally geodesic p-end is a radial p-end. By Theorem 6.9 of \cite{endclass}, 
the dual radial p-end has a neighborhood that is contained in a strict join with a vertex $x$ with a properly convex open domain $K$ in a hyperplane $V$. 
$\clo(K)$ is a strict join $C_1 \ast \cdots \ast C_k$ for properly compact convex domains $C_i$, for $i=1, \dots, k$ by Theorem \ref{thm:redtot}. 
%Let $K$ denote the dual of $\clo(S_{\tilde E})$ as a subspace of $V$. 
Since $\torb$ contains a one-sided convex p-end neighborhood $D$ of $S_{\tilde E}$. 

By equation \eqref{eqn:dualinc}, the dual $D^*$ of $D$ contains the dual $\torb^*$ of $\torb$.
%Let $x \ast K$ denote one strict join of $x$ and $K$ and $x_- \ast K$ denote the strict join of the other type. 
Since $D^*$ is  the interior of a lens-cone by Lemma  \ref{lem:example}, $D^*$  is contained 
in the union $U$ of two strict joins $x\ast K \cup x_- \ast K$
for the point $x$ dual to the hyperplane containing ideal boundary component $S_{\tilde E}$
and its antipode $x_-$. 
$D^*$ contains $x \ast K$. 
Thus, $\torb^* \subset x\ast K \cup x_- \ast K$. 
The set of supporting hyperspaces at the vertex $x$ is projectively
isomorphic to the dual $K$ of $\clo(S_{\tilde E})$ by Lemma \ref{lem:example}(iii).
Therefore, by Lemma \ref{lem:example}, the dual $\torb$ of $\torb^*$ is contained in the 
the cone $\clo(S_{\tilde E}) \ast a$ for some point $a$ dual to the hyperplane $V$. 

Now, $\tilde L_2$ is a subset of $\clo(S_{\tilde E}) \ast a_-$ sharing boundary $\clo(S_{\tilde E})$ with $\torb$
since we can treat $\tilde L_2$ as $\torb$ in the above arguments. 
%By the $\pi_1(E)$-invariance, the space of supporting hyperspaces of $\tilde L_2$ agrees with 
%that of $\clo(S_{\tilde E}) \ast a'_-$ and so does $\torb$. 
Since both share $S_{\tilde E}$ and are in $S_{\tilde E} \ast a \cup S_{\tilde E} \ast a_-$, 
the convexity of the union $\tilde L_2 \cup \torb$ follows. 
The proper convexity follows also as above. 

\end{proof}

\begin{corollary} \label{cor:lenssupp}
Suppose that $\mathcal O$ is a noncompact strongly tame properly convex real projective orbifold with generalized admissible ends
and $\pi_1(\tilde E)$ is hyperbolic. 
%\begin{itemize} 
%\item[(i)] Let $\tilde E$ be a totally godesic p-end of lens-type. 
%Let $L$ be a lens containing a totally godesic properly convex hypersurface 
%$\tilde E$ so that $\Lambda := \clo(\tilde E) - \tilde E = \Bd L - \partial L$. Then 
%each point of $\Lambda$ has a unique supporting hyperspace of $L$. 
%\item[(i)] 
Let $\tilde E$ be a lens-type radial p-end. 
Let $L$ be a lens in the p-end neighborhood.
Define $\Lambda := \Bd L - \partial L$.
Then each point of $\Lambda$ has a unique supporting hyperplane of $L$. 
%\item[(ii)] Assume as in {\rm (ii)}. Let $\Omega$ be a properly convex $\pi_1(\tilde E)$-invariant 
%domain containing $\Lambda$ in $\Bd \Omega$ so that $\Bd \Omega - \Lambda$ is strictly convex. 
%Then at  each point of $L$, there exists a unique hyperplane supporting $\Omega$.
%\end{itemize} 
\end{corollary}
\begin{proof} 
%(i) is already proved in Lemma \ref{lem:commsupp}. 

This follows  from Lemma \ref{lem:commsupp} and the duality Proposition \ref{prop:duality}. 
$\Lambda$ and the supporting hyperplanes goes to the boundary of a strictly convex domain
and the supporting hyperplane to $L^*$ under the duality map. 
That is, points and the supporting hyperplanes 
change roles here. Then $L^*$ has strictly convex boundary and $\Bd L^* - \partial L^*$ 
is strictly convex by Lemma \ref{lem:commsupp} since $\pi_1(E)$ is hyperbolic. 
Thus, each hyperplane can meet $\clo(L^*)$ at a unique point. 

%(ii) follows as in (i). 
\end{proof}

%% April 1st 10:42pm

%lens tot geo dual to radial of lens type

%%% I need to do the these hyperspace business first.

%%% Aug 20 10:07 I need to correct the following.... 
We sharpen Proposition \ref{prop:duality}. 

\begin{lemma} \label{lem:concavedual}
Given an end $E$ of a strongly tame orbifold $\orb$ and the corresponding end $E^*$ of the dual orbifold $\orb^*$,
$E$ is a radial end of generalized lens-type  if and only if $E^*$ is a totally geodesic end of lens type 
\end{lemma}
\begin{proof}
This is given as Remark 2 in \cite{endclass}. 
%Given a concave p-end neighborhood $U$ of a p-end $\tilde E$, the dual representation $h^*| \pi_1(\tilde E)$ satisfies 
%the uniform middle eigenvalue condition.  By Lemma \ref{lem:shrink}, we obtain a p-end neighborhood of lens-type in $\torb$.  
%Given a totally geodesic p-end of lens type $\tilde E$, $h| \pi_1(\tilde E)$ satisfies 
%the uniform middle eigenvalue condition. By Lemma \ref{lem:shrink}, we obtain a concave p-end neighborhood of $\tilde E$. 
\end{proof}

%%% May 6th 2010 8:09 pm
\subsection{The convex hulls of ends}\label{subsub:convh}

Here we will be working on $\bR P^n$ exclusively from now on. 
One can associate a {\em convex hull} of a p-end $E$ of $\tilde{\mathcal{O}}$ as follows: 
\begin{itemize}
\item For horospherical p-ends, the convex hull of each is defined to be the set of the end vertex actually. 
\item The convex hull of a totally geodesic end $\tilde E$ of lens-type is the closure $\clo(S_{\tilde E})$
the totally geodesic ideal boundary component $S_{\tilde E}$ corresponding to $\tilde E$. 
\end{itemize} 
They equal $\clo(U) \cap \Bd \torb$ for any p-end neighborhood $U$ of $S_{\tilde E}$ by
the following proposition \ref{prop:preI}. 

%%% Oct 29. This needs to be proved.... Use the other paper... parts missing below....
\begin{proposition} \label{prop:preI} 
Let $\orb$ be a strongly tame properly convex real projective orbifold 
with radial ends or totally geodesic ends of lens-type and satisfies {\rm (IE)} and {\rm (NA)}.
Assume that the holonomy group is strongly irreducible. 
For a horospherical p-end, the closure of a p-end neighborhood meets $\Bd \torb$ at the the p-end only.
For a totally geodesic end $\tilde E$ of lens-type, 
a closed p-end neighborhood $L$ of $\tilde E$ % and $\clo(L) \cap \Bd \torb = \clo(S_{\tilde E})$ for the ideal boundary component $S_{\tilde E}$.  
contains the closure of a proper p-end neighborhood of lens-type, and 
$\clo(L) \cap \Bd \torb = \clo(S_{\tilde E})$ for the ideal boundary component $S_{\tilde E}$.  
\end{proposition} 
\begin{proof} 
For a horospherical end, there exists a finite index free abelian group $\bZ^{n-1}$ acting on a compact ellipsoid $E$ with 
a unique fixed point $x \in \partial E$. Then $(\partial E-\{x\})/\bZ^{n-1}$ is compact. We take an open neighborhood $N$
of the fundamental domain in $\torb$. Then $\bigcup_{g \in \bZ^{n-1}} g(N)$ is an open neighborhood of 
$\partial E-\{x\}$ in $\torb$. Thus, $\partial E \cap \Bd \torb =\{x\}$ since $x$ is a fixed point. 

Let $\tilde E$ be a lens-type totally geodesic end.
%By Theorem 8.2 of \cite{endclass}, $\tilde E$ is of strictly lens-type. 
By Corollary 8.7(iv) of \cite{endclass}, $S_{\tilde E}$ has a strict lens p-end-neighborhood $L_1 \subset \torb$ so that 
for its boundary component $A$ we have
$\clo(A) -A \subset \clo(S_{\tilde E}) - S_{\tilde E}$. 
$\Gamma_{\tilde E}$ acts cocompactly in $A$. As above, we can find an invariant open neighborhood of $A$ 
in $\torb$. Since $\clo(A) \cup \clo(S_{\tilde E})$ is homeomorphic to a $(n-1)$-sphere, it follows that 
$\Bd L_1 = \clo(A) \cup \clo(S_{\tilde E})$, and  $\clo(L_1) = L_1 \cup \clo(A) \cup \clo(S_{\tilde E})$.
Since $A \subset \torb$, we obtain $\clo(L_1) \cap \Bd \torb = \clo(S_{\tilde E})$.
Since $S_{\tilde E}$ is the boundary of any p-end neighborhood, the result follows. 
\end{proof}

For a lens-shaped p-end $E$ with a p-end vertex $v$,  the {\em convex hull} $I$ is 
$CH(\bigcup S(v)) \cap \tilde{\mathcal{O}}$. 
We can also characterize it as the intersection of every
$CH(\clo(U_1)) \cap \tilde{\mathcal{O}}$ for a p-end neighborhood $U_1$ of $v$
by (iv) and (v) of Proposition \ref{prop:I}.

%%May 15 4:06pm 2014

Following Propositions \ref{prop:I} and  \ref{prop:preI} imply that the convex hull of an end 
is a well-defined and is independent of neighborhoods. 

\begin{proposition}\label{prop:I} 
%Suppose that $\mathcal O$ is a noncompact strongly tame strictly SPC-orbifold with admissible ends. 
Let $\orb$ be a strongly tame properly convex real projective orbifold with radial ends or totally geodesic ends of lens-type and satisfies {\rm (IE)} and {\rm (NA)}.
Assume that the holonomy group of $\pi_1(\orb)$ is strongly irreducible. 
Let $\tilde E$ be a radial lens-shaped p-end and $v$ an associated 
p-end vertex.
\begin{itemize}
\item[(i)] A segment in the boundary of $\tilde{\mathcal{O}}$ is always contained in the closure of 
a convex hull $CH(\clo(U_1)) \cap \tilde{\mathcal{O}}$ for a p-end neighborhood $U_1$ of $v$
and more precisely it is in the union $\bigcup S(v)$ of the maximal segments in 
$\Bd \tilde{\mathcal{O}}$ ending at $v$ for the corresponding $v$.
Thus, the segment is contained in the closure of any p-end neighborhood of $v$.  
\item[(ii)] $I$ is contained in $CH(\clo(U_1)) \cap \tilde{\mathcal{O}}$ for any p-end neighborhood $U_1$ of $v$.
\item[(iii)] Any segment in $\bigcup S(v)$ corresponds one-to-one manner to a point of the boundary of the open convex 
domain $R_v(\torb)= S_{\tilde E}$ in $\SI^{n-1}_v$.
\item[(iv)] $\Bd I \cap \tilde{\mathcal{O}}$ is contained in the union of a lens part of a lens-shaped p-end neighborhood. 
\item[(v)] $I$ contains any concave p-end-neighborhood of $E$ and 
\[I \cap \torb  = CH(\clo(U)) \cap \tilde{\mathcal{O}}\] for a concave p-end neighborhood $U$ of $v$. 
Thus, $I$ has a nonempty interior. 
\item[(vi)] Each segment from $v$ maximal in $\tilde{\mathcal{O}}$ 
meets the set $\Bd I\cap  \tilde{\mathcal{O}}$ exactly once and
$\Bd I \cap  \tilde{\mathcal{O}}/\Gamma_v$ is an orbifold isotopic to $E$
for the end fundamental group $\Gamma_v$ of $v$. 
\item[(vii)] There exists a nonempty-interior of the convex hull $I$ of a neighborhood of 
the p-end vertex $v$ of $E$ of $\tilde{\mathcal{O}}$ and 
where $\Gamma_v$ acts so that $I \cap \tilde{\mathcal{O}}/\Gamma_v$ is diffeomorphic to the end orbifold times an interval. 
\item[(viii)] $I \cap \tilde{\mathcal{O}}$ has a boundary 
restricting to the covering map is an immersed compact orbifold homotopic to 
the associated end orbifold. 
\end{itemize}
\end{proposition}
\begin{proof}

(i) A segment in $\Bd \tilde{\mathcal{O}}$ is contained in a closure of 
a p-end neighborhood by the strictness of the SPC-structure. Since it meets the interior of $\bigcup S(v)$, 
the segment must be in $\bigcup S(v)$ 
as in the proof of Theorem \ref{thm:lensclass}(ii) and Theorem \ref{thm:redtot} in \cite{endclass}.

By Theorems \ref{thm:lensclass} and \ref{thm:redtot}, the set 
$\bigcup S(v)$ is always contained in the closure of any p-end neighborhood of $v$.
Thus (ii) follows. 

(iii) A segment $s$ from $v$ in $\clo(\tilde{\mathcal{O}})$ either ends in a lens-shaped domain or 
is in $\Bd \tilde{\mathcal{O}}$. In the second case, $s$ is in $\Bd R_v(\torb)$ clearly.

(iv) We define $S_1$ as the set of $1$-simplices with endpoints in segments in $\bigcup S(v)$ and we inductively define
$S_i$ to be the set of $i$-simplices with faces in $S_{i-1}$. 
Then $I$ is a union $\bigcup_{\sigma \in S_1 \cup S_2 \cup \cdots \cup S_n} \sigma$. 
Notice that $\Bd I$ is the union 
$\bigcup_{\sigma \in S_1 \cup S_2 \cup \cdots \cup S_n, \sigma \subset \Bd I} \sigma$
since each point of $\Bd I$ is contained in the interior of a simplex which lies in $\Bd I$ by the convexity of $I$. 
If $\sigma \in S_1$ with $\sigma \subset \Bd I$, then its end point must be in an endpoint of a segment in $\bigcup S(v)$.
%or is a subsegment of an element of $S(v)$. 
If an interior point of $\sigma$ is in a segment in $S(v)$, then the vertices of $\sigma$ are in 
$\bigcup S(v)$ by the convexity of $\clo(R_v(\torb))$. 
Hence, if $\sigma^o$ meets $\tilde{\mathcal{O}}$, then 
$\sigma^o$ is the lens-shaped domain.
Now by induction on $S_i$, $i > 1$, we can verify (iv)
since any simplex with boundary in the union of subsimplices in the lens-domain is in the lens-domain
by convexity.

(v) Since $I$ contains the segments in $S(x)$ and is convex, and so does a concave p-end neighborhood $U$, 
we obtain $\Bd U \subset I$: Otherwise, let $x$ be a point of $\Bd U \cap \Bd I$ where some neighborhood 
in $\Bd U$ is not in $I$. Then a supporting hyperspace at $x$ of the convex set $I$, meets a segment in $S(x)$ in its interior. 
This is a contradiction since $I$ contains the segments entirely. 
Thus, $U \subset I$. 

(vi) $\Bd I \cap  \tilde{\mathcal{O}}$ is a subset of a lens part of a p-end neighborhood by (iii). 
Each point of it meets a maximal segment from $v$ in the end but not in $S(x)$ at exactly one point since 
a maximal segment must leave the lens cone eventually.
Thus $\Bd I \cap  \tilde{\mathcal{O}}$ is homeomorphic to an $(n-1)$-cell and the result follows. 

(vii) This follows from (v) since we can use rays from $x$ meeting $\Bd I \cap  \tilde{\mathcal{O}}$ at unique points 
and use them as leaves of a fibration. 

(viii) This again follows from (vi). 

\end{proof}

%%% April 18 9:59pm
% Maybe a picture here?  Need $S(x)$ correspond to boundary of $\tilde E$...

\begin{figure}
\centerline{\includegraphics[height=7cm]{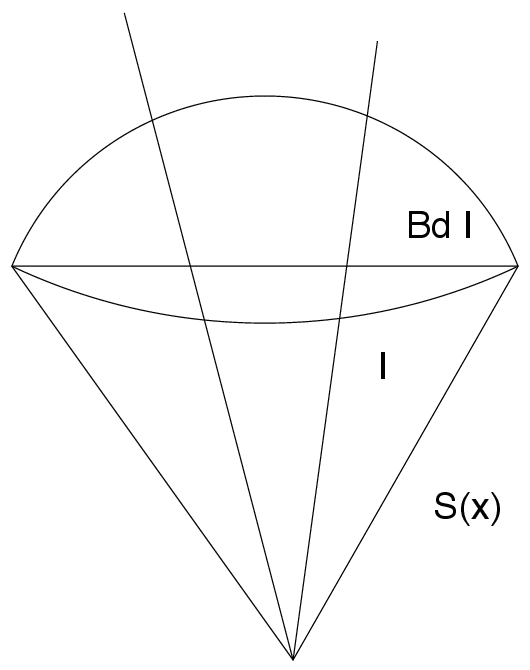}}
\caption{The structure of a lens-shaped p-end.}
\label{fig:lenslem2}
\end{figure}

\begin{definition}\label{defn:lambda}
 Let $\tilde E$ be a lens-type radial end. Let $L$ be the lens-cone neighborhood of ${\tilde E}$, \index{end!mc-p-end neighborhood}
 and let $\Lambda = \clo(D) - D$ for the lens domain $D$ of $L$, i.e., the limit set of $\pi_1(\tilde E)$. 
 Let $CH(\Lambda)$ denote the convex hull.
We define a {\em maximal concave p-end neighborhood} or {\em mc-p-end-neighborhood} $U$ 
 to be the component of $U' - CH(\Lambda)$ containing 
a p-end neighborhood of $\tilde E$ for any choice of a p-end neighborhood $U'$ of $\tilde E$ containing $CH(\Lambda) \cap \torb$. 
The {\em closed maximal concave p-end neighborhood} is 
$\clo(U) \cap \torb$. 
%Let $d_{\torb}$ denote the Hilbert metric of the universal cover $\torb$ of $\orb$. 
An $\eps$-$d_{\torb}$-neighborhood $U'$ of a maximal concave p-end neighborhood is called 
an {\em $\eps$-mc-p-end-neighborhood} $U'$.  \index{end!$\eps$-mc-p-end neighborhood}
\end{definition} 

\begin{lemma} \label{lem:mcc}
Let $D$ be an $i$-dimensional totally geodesic compact convex domain, $i \geq 1$. % in $\clo(U)$ for a p-end neighborhood containing an mc-p-end neighborhood $V$. 
Let $\tilde E$ be a generalized lens-type p-end with the p-end vertex $v_{\tilde E}$. 
Suppose $\partial D \subset \bigcup S(v_{\tilde E})$. Then $D \subset V$ for a maximal concave p-end neighborhood $V$, and 
for sufficiently small $\eps>0$, an $\eps$-$d_{\torb}$-neighborhood of 
$D^o \subset V'$ for any $\eps$-mc-p-end neighborhood $V'$.
\end{lemma}
\begin{proof} 
Assume that $U$ is a generalized lens-cone of $v_{\tilde E}$. 
%Let $\Lambda$ be the limit set of $\tilde E$.  %Then $\Lambda$ is at end points of segments in $S(v_{\tilde E})$. 
Then $\Lambda$ is the set of end points of segments in $S_{v_{\tilde E}}$ with $v_{\tilde E}$ removed. 
Let $P$ be the spanning subspace of $D$ and $v_{\tilde E}$. %Then $P \cap \clo(U)$ is a generalized lens-cone. 
Since $\partial D, \Lambda \cap P  \subset \bigcup S(v_{\tilde E}) \cap P$, 
and $\partial D \cap P$ is closer than $\Lambda \cap P$ from $v_{\tilde E}$, 
it follows that 
$P\cap \clo(U) -D$ has a component $C_1$ containing $v_{\tilde E}$ and a component $C_2$ contains $\Lambda \cap P$.
%whose closures contain $\Lambda \cap P$ and $v_{\tilde E}$ respectively. 
Hence $\clo(C_2) \supset CH(\Lambda) \cap P$ by the convexity of $\clo(C_2)$. 
This implies that $D$ is disjoint from $CH(\Lambda)^o$
or $D$ contains $CH(\Lambda) \cap P$.
Let $V$ be an mc-p-end neighborhood of $U$. 
Since $\clo(V)$ contains the closure of the component of $U - CH(\Lambda)$ whose closure contains $v_{\tilde E}$, 
it follows that $\clo(V)$ contains $D$. %Hence, the result follows. \marginpar{look again...}

Since $D$ is in the mc-p-end neighborhood $V$, the boundary $\Bd \clo(V') \cap \torb$ of the 
$\eps$-mc-p-end neighborhood $V'$ do not meet $D$. Hence $D^o \subset V'$. 
\end{proof} 

%% May 15 6:19pm 2014

\begin{corollary} \label{cor:mcn} 
Let $\orb$ be a properly convex real projective manifold with generalized admissible ends % lens-shaped radial ends, lens-type totally geodesic ends,
%or horospherical ends, 
and satisfies {\rm (IE)} and {\rm (NA)}.
Assume that the holonomy group of $\pi_1(\orb)$ is strongly irreducible. 
Let $\tilde E$ be a generalized lens-type radial end. 
Then 
\begin{itemize}
\item[(i)] Also, a concave p-end neighborhood of $\tilde E$ is always a subset of an mc-p-end-neighborhood of the same p-end. 
\item[(ii)] An mc-p-end-neighborhood is a union of concave end neighborhoods.
\item[(iii)] Each mc-p-end-neighborhood  of $\tilde E$ is a proper p-end neighborhood, and covers  an end-neighborhood with compact boundary in $\orb$. 
\item[(iv)] An $\eps$-mc-p-end-neighborhood of $\tilde E$ for sufficiently small $\eps > 0$ is a proper p-end neighborhood. 
\item[(v)] We can choose $\eps$-mc-p-end neighborhoods of p-ends so that their image end-neighborhoods in $\orb$ 
are mutually disjoint. 
\end{itemize}
\end{corollary} 
\begin{proof}
%Let $\tilde E$ be a p-end.
(i) Since any generalized lens $L$ in 
a generalized lens-cone p-end neighborhood $U$ of $\tilde E$ contains $CH(\Lambda) \cap \torb$ for the limit set $\Lambda$ of $\tilde E$
by Corollary 8.5 of \cite{endclass}. 
Hence, a concave end neighborhood is contained in an mc-p-end-neighborhood. 

%\marginpar{These are hard.... a bit...clarify later}

%An mc-p-end-neighborhood is a union of concave end neighborhoods
%since 

(ii) %Suppose that $CH(\Lambda)$ has a nonempty interior in $\torb$. 
%Then an mc-p-end-neighborhood $V$ has a boundary $S$ in $\torb$. 
Let $V$ be an mc-p-end neighborhood of $\tilde E$.
Then define $S$ to be the set of end points in $\clo(\torb)$ of maximal segments in $V$ from $v_{\tilde E}$ in directions of $S_{\tilde E}$.
That is $\clo(V) \cap \torb = V \cup S$, and $S$ is homeomorphic to $S_{\tilde E}$. 
Thus, $S/\pi_1(\tilde E)$ is a compact set since $S$ is contractible and $S_{\tilde E}/\pi_1(\tilde E)$ is a $K(\pi_1(\tilde E))$-space. 

We can $d_{\torb}$-approximate $S$ by the smooth boundary component $S_{\eps}$ outwards
of a generalized lens using the proof of Proposition 7.6 of \cite{endclass}.
For sufficiently small $\eps >0$, $S_\eps$ is strictly convex by the continuity of the Hessian matrices.
A component $U-S_{\eps}$ is a concave p-end neighborhood. 

(iii) Since a concave p-end neighborhood is a proper p-end neighborhood, we obtain 
$g(V) \cap V = \emp$ or $g(V) = V$ for $g \in \pi_1(\orb)$ by the first item.

Suppose that $g(\clo(V) \cap \torb) \cap \clo(V) \ne \emp$. Then $g(V) = V$ and $g \in \pi_1(\tilde E)$:
Otherwise, $g(V) \cap V =\emp$, and $g(\clo(V) \cap \torb)$ meets $\clo(V)$ in a totally geodesic hypersurface $S$ equal to $CH(\Lambda)^o$
by the concavity of $V$. Hence for every $g \in \pi_1(\orb)$, $g(S) = S$, 
and $g(V) \cup S \cup V = \torb$ since these are subsets of a properly convex domain $\torb$.
Then $\pi_1(\orb)$ acts on $S$ and $S/G$ is homotopy 
equivalent to $\torb/G$ for a finite index torsion free subgroup $G$ of $\pi_1(\orb)$ by Selberg's lemma. 
This cannot be true since the quotients are manifolds with different dimensions. 
%Hence, $\clo(V) \cap \torb$ covers an end neighborhood of $\orb$ with the deck transformation group $\pi_1(\tilde E)$. 

%Now suppose that $CH(\Lambda)$ has an empty topological interior, i.e., it is totally geodesic domain of codimension-one. 
%Let $S$ denote $CH(\Lambda)^o$ a hypersurface. 
%Suppose that $S \subset \torb$. Then $g(V) \cap V =\emp$ or $g(V)= V$ for $g \in \pi_1(\orb)$ as above. 
%As above if $g(\clo(V) \cap \torb) \cap \clo(V) \ne \emp$, then $g(V) = V$ and $g \in \pi_1(\tilde E)$. 
%Hence, $\clo(V) \cap \torb$ covers an end neighborhood of $\orb$ with the deck transformation group $\pi_1(\tilde E)$. 

Now suppose that $S \cap \Bd \torb \ne \emp$. 
Let $S'$ be a maximal totally geodesic domain in $\clo(V)$ supporting $S$. 
Then $S' \subset \Bd \torb$ by convexity by Lemma 7.5 of \cite{endclass}, meaning that $S'=S \subset \Bd \torb$.  
%$S \cap \torb$ is a convex subset of $S$. If $S \cap \Bd \torb$ is not empty, a $2$-dimensional domain $P$ in $\clo(\torb)$
%meeting $S\cap \torb$ and $S\cap \Bd \torb$. Since there is a segment $s'\subset S \cap P$ with $s'\cap S\cap \torb$ 
%containing an end point of $s'$ and $s'\cap S \cap \Bd \torb \ne \emp$ containing the other end point of $s'$.
%We can find some segment $s \subset P$ by a perturbation
%meeting $\Bd \torb$ near $S \cap \Bd \torb$ in the interior but $\partial s \subset \torb$. 
%This contradicts the convexity of $\torb$ as in Chapter 3 of \cite{psconv}. 
%Therefore, $S \subset \Bd \torb$. 
In this case, $\torb$ is a cone over $S$ and the end vertex $v_{\tilde E}$ of $\tilde E$.
For each $g \in \pi_1(\orb)$, $g(V) \cap V \ne \emp$ meaning $g(V)=V$ since $g(v_{\tilde E})$ is on $\clo(S)$. 
Thus, $\pi_1(\orb) = \pi_1(\tilde E)$. 
This contradicts the infinite index condition of $\pi_1(\tilde E)$. 

We showed that $\clo(V) \cap \torb = V \cup S$. 
Thus, an mc-p-end-neighborhood $\clo(V) \cap \torb$ is a proper end neighborhood of $\tilde E$
with compact imbedded boundary $S/\pi_1(\tilde E)$. 
Therefore we can choose positive $\eps$ so that an $\eps$-mc-p-end-neighborhood is a proper p-end neighborhood also.
This proves (iv). 

(v) For two mc-p-end neighborhoods $U$ and $V$ for different p-ends, we have $U \cap V =\emp$
by (iii). % and Lemma \ref{lem:concave}.

We showed that $\clo(V) \cap \torb$ for an mc-p-end-neighborhood $V$ covers an end neighborhood in $\orb$. 
Suppose that $U$ is another mc-p-end neighborhood different from $V$
and  $\clo(U) \cap \clo(V) \cap \torb \ne \emp$. 
Since $U \cap V = \emp$, we have $\clo(U) \cap \clo(V)$ are in the boundary of $U$ and $V$ in a properly convex domain $\torb$,
and $\Bd U \cap \torb$ and $\Bd V \cap \torb$ equal a tangent maximal hyperspace $CH(\Lambda)^o$ in $\torb$ and 
hence they are equal. As above, this is a contradiction. 
Hence $\clo(U) \cap \clo(V) \cap \torb = \emp$. 

Since the closures of mc-p-end neighborhoods with different p-ends are disjoint, the final item follows. 
\end{proof}

For the following, we need a stronger condition of lens-type ends. 
\begin{corollary} \label{cor:disjclosure} 
Let $\orb$ be a properly convex real projective manifold with admissible ends % lens-shaped radial ends, lens-type totally geodesic ends,
%or horospherical ends, 
and satisfies {\rm (IE)} and {\rm (NA)}.
Assume that the holonomy group of $\pi_1(\orb)$ is strongly irreducible. 
Let $\mathcal U$ be the collection of the components of the inverse image in $\torb$ 
of the union of disjoint collection of  end neighborhoods of $\orb$ for all radial or totally geodesic ends of lens-type. 
%\begin{itemize} 
%\item[(i)] Given a concave or mc-p-end-neighborhood $U$ with p-end vertex $v$ for a radial end of lens-type, 
%the set $\clo(U) \cap \Bd \tilde { \mathcal{O}}$ 
%equals $\bigcup_{s \in S(v)} s$ for the set of maximal segments in $\clo(U) \cap \Bd \tilde {\mathcal{O}}$.
%Hence, the set is independent of $U$.  \marginpar{independence again?. consolidate}
%\item[(ii)] Given any p-end neighborhood of radial lens-type, it contains
%a concave p-end-neighborhood $U$, and a concave p-end neighborhood is a proper p-end neighborhood.
%\end{itemize} 
Now replace each of the p-end neighborhoods of radial lens-type 
of collection $\mathcal U$ by a concave p-end neighborhood by Lemma \ref{lem:shrink} {\rm (iii).} 
Then the following statements hold: 
\begin{itemize} 
\item[(i)] Given concave or one-sided lens p-end-neighborhoods $U_1$ and $U_2$ contained in $\bigcup \mathcal U$, 
we have $U_1 \cap U_2 =\emp$ or $U_1= U_2$. 
\item[(ii)] Let $U_1$ and $U_2$ be in $\mathcal U$. Then 
$\clo(U_1) \cap \clo(U_2) \cap  \Bd \tilde{\mathcal{O}} = \emp$ 
or $U_1 = U_2$ holds. 
\end{itemize}
\end{corollary}
\begin{proof} 
%The first item is because for lens-shaped p-ends, the lens part with $\bigcup_{s \in S(v)} s$ removed
%is in the interior of $\tilde {\mathcal{O}}$.
%
%The second is  Theorems 6.8 (iv) and 6.9 (vi) in \cite{endclass}. 
%
Let $\tilde E$ be a p-end of $\torb$. 
Since $L/\pi_1(\tilde E)$ is compact for a lens of a lens cone of a p-end neighborhood of $\tilde E$, 
each lens-cone p-end neighborhood is a proper p-end neighborhood if we take a finite index subgroup of $\pi_1(\orb)$. 
We assume without loss of generality that the lens-cones are proper p-end neighborhoods from now on.

(i) Suppose that $U_1$ and $U_2$ are p-end neighborhoods of radial ends. 
 Let $U'_1$ be the interior of the associated generalized lens-cone of $U_1$ in $\clo(\torb)$ and $U'_2$ be that of $U_2$. 
Let $U''_i$ be the concave p-end-neighborhood of $U'_i$ for $i=1,2$ that covers an end neighborhood in $\orb$ by Lemma \ref{lem:shrink} (iii).
Since the neighborhoods in $\mathcal U$ are mutually disjoint, 
\begin{itemize}
\item $\clo(U''_1) \cap \clo(U''_2) \cap \torb = \emp$ or
\item $U''_1 = U''_2$
\end{itemize} 
since we can choose these to cover disjoint or identical p-end neighborhoods in $\orb$. 

% April 11 12:04am
(ii) Assume that $U''_i \in {\mathcal{U}}$, $i=1, 2$, and $U''_1 \ne U''_2$. 
Suppose that the closures of $U''_1$ and $U''_2$ intersect in $\Bd \torb$.
Suppose that they are both radial p-end neighborhoods. 
%Then we have the first possibility:
Then 
the respective convex hulls $I_1$ and $I_2$ as obtained by Proposition \ref{prop:I} intersect as well. 
Take a point $z \in \clo(U''_1) \cap \clo(U''_2) \cap \Bd \torb$. 
Let $p_1$ and $p_2$ be the respective p-end vertices of $U'_1$ and $U'_2$. 
Then $\ovl{p_1z}\in S(p_1)$ and $\ovl{p_2z} \in S(p_2)$ and these segments
are maximal since otherwise $U''_1 \cap U''_2 \ne \emp$.
The segments intersect transversally at $z$ 
since otherwise we violated the maximality in Theorems \ref{thm:lensclass} and 
\ref{thm:redtot}.
%Then we can deduce that there exist respective p-end vertices $p_1$ and $p_2$
%with maximal segments $\ovl{p_1 z} \in S(p_1)$ and $\ovl{p_2 z} \in S(p_2)$
We obtain a triangle $\tri(p_1p_2z)$ in $\clo(\torb)$ with vertices $p_1, p_2, z$. 
We assume that $\ovl{p_1p_2}^o \subset \torb$. 

If this is not, true, we need to perturb $p_1$ and $p_2$ by a small amount. 
We may not have the geodesic triangle but will have a disk bounded by three arcs. However, the 
disk has an angle $< \pi$ at $z$ since $z$ is not a $C^1$-point of $\Bd \torb$. 
We will denote the disk by $\tri(p_1p_2z)$ still. 

We define a convex curve $\alpha_i := \tri(p_1p_2z) \cap \Bd I_i$ with an end point $z$ for each $i$, $i=1,2$. 
Let $\tilde E_i$ denote the p-end corresponding to $p_i$. 
Since $\alpha_i$ maps to a geodesic in $R_{p_i}(\torb)$, 
there exists a foliation $\mathcal{T}$ of $\tri(p_1p_2z)$ 
by maximal segments from the vertex $p_1$.  
There is a natural parametrization of the space of leaves by $\bR$ 
as the space is projectively equivalent to an open interval using the Hilbert metric of
the interval. We parameterize $\alpha_i$ by these parameters 
as $\alpha_i$ intersected with a leaf is a unique point. 
They give the geodesic length parameterizations under the Hilbert metric  of $R_{p_i}(\torb)$
for $i=1, 2$. 

\begin{figure}
\centerline{\includegraphics[height=6cm]{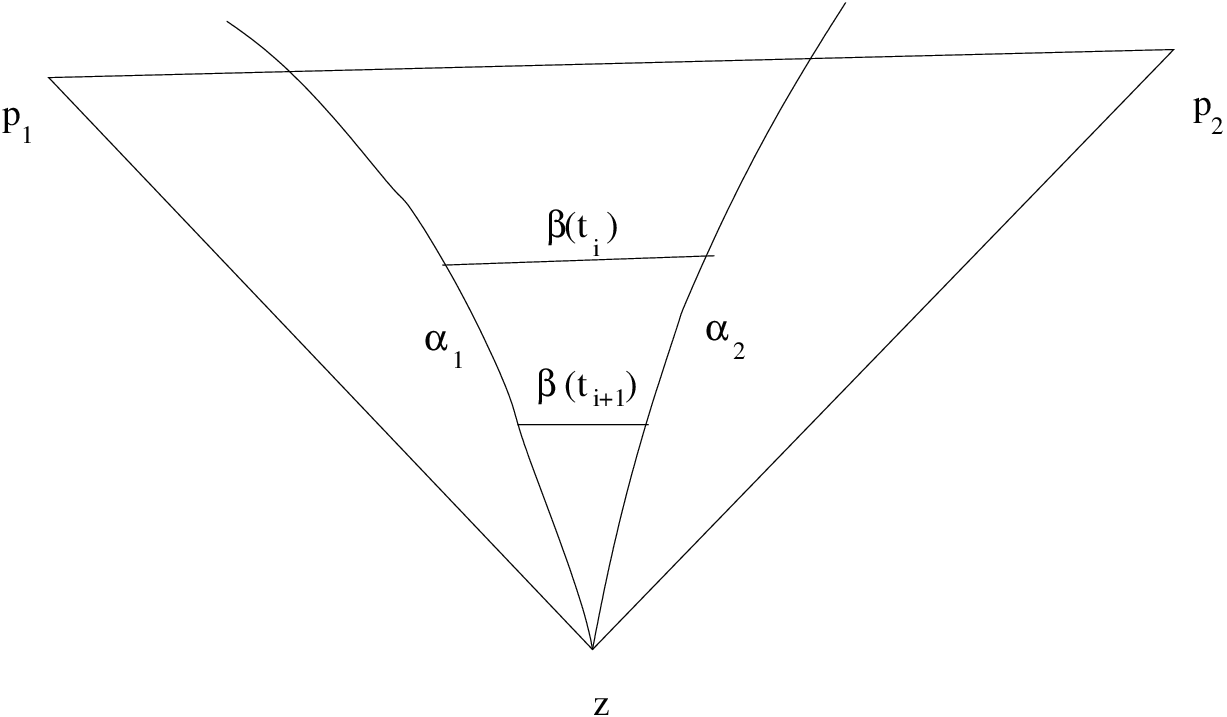}}
\caption{The diagram of the quadrilateral bounded by $\beta(t_i), \beta(t_{i+1}), \alpha_1, \alpha_2$.}
\label{fig:bounded}
\end{figure}

We now show that an infinite-order element of $\pi_1(\tilde E_1)$ is the same as one in $\pi_1(\tilde E_2)$:
By convexity, either $\alpha_2$ goes into $I_1$ and not leave again or $\alpha_2$ is disjoint from $l_1$.
Suppose that $\alpha_2$ goes into $I_1$ and not leave it again. 
Since $I_2/\pi_1(\tilde E_2)$ is compact, there is a sequence $t_i$ so that 
the image of $\alpha_2(t_i)$ converges to a point of $I_1/\pi_1(\tilde E_1)$. 
%Then $\alpha_2$ is recurrent in $R_{p_1}(\tilde E_1)$ by Corollary \ref{cor:rec}.
Hence, by taking a short path between $\alpha_2(t_i)$s, 
there exists an essential closed curve $c_2$ in $I_2/\pi_1(\tilde E_2)$ homotopic to 
an element of $\pi_1(\tilde E_1)$. In fact $c_2$ is in a lens-cone end neighborhood of the end corresponding to $\tilde E_1$. 
This contradicts (NA). 
(The element is of infinite order since we can take a finite cover of $\orb$ so that $\pi_1(\orb)$ is torsion-free 
by Selbert's lemma.)

Suppose now that $\alpha_2$ is disjoint from $l_1$. Then since $\alpha_1$ and $\alpha_2$ have
the same end point $z$ and by the convexity of $\alpha_2$, 
we parameterize $\alpha_i$ so that $\alpha_1(t)$ and $\alpha_2(t)$ is on a line segment in the triangle 
with end points in $\ovl{zp_1}$ and $\ovl{zp_2}$. 
We obtain $d_{\orb}(\alpha_2(t), \alpha_1(t)) \leq C$ for a uniform constant $C$
since one can project to the space of lines through $z$, a one dimensional projective space 
where the end points are fixed and the the image of $\beta(t)$ are so that 
the image of $\beta(t')$ is contained in that of $\beta(t)$ if $t < t'$. 
And the Hilbert-metric length of the segment $\beta(t) := \ovl{\alpha_2(t) \alpha_1(t))}$ is bounded above by the uniform constant.

We have a sequence $t_i \ra \infty$ so that 
\[p_{\orb} \circ \alpha_2(t_i) \ra x, d_{\orb}(p_{\orb}\circ \alpha_2(t_i), p_{\orb}\circ \alpha_2(t_{i+1})) \ra 0, x \in \orb. \]
So we obtain a closed curve $c_{2, i}$ in $\orb$
obtained by taking a short path between the two points. By taking a subsequence, 
the image of $\beta(t_i)$ in $\orb$ geometrically converges to a segment of Hilbert-length $\leq C$.
As $i\ra \infty$, we have $d_{\orb}(p_{\orb} \circ \alpha_1(t_i), p_{\orb} \circ \alpha_1(t_{i+1})) \ra 0$ by extracting a subsequence.  
There exists a closed curve $c_{1, i}$ in $\orb$ again by taking a short path. 
We see that $c_{1, i}$ and $c_{2, i}$ are homotopic in $\orb$ since we can use the image of the disk 
in the quadrilateral bounded by
$\ovl{\alpha_2(t_i) \alpha_2(t_{i+1})}, \ovl{\alpha_1(t_i) \alpha_1(t_{i+1})}, \beta(t_i), \beta(t_{i+1})$ 
and the connecting thin strips between the images of $\beta_{t_i}$ and $\beta_{t_{i+1}}$ in $\orb$. 
This again contradicts (NA). 

%Let $g_1$ be an infinite-order element of $\pi_1(\tilde E_1) \cap \pi_1(\tilde E_2)$. 
%Then $p_1$ and $p_2$ are fixed points of $g_1$ 
%and $\ovl{p_1p_2}$ meets $I_1$ in its interior. 
%Then for the eigenvalue $\lambda_{g_1}(p_i)$ associated with $g_1$ at $p_i$, 
%we have $\lambda_{g_1}(p_1) = \lambda_{g_1}(p_2)$
%since $g_1^n(l_1 \cap \ovl{p_1p_2})$ converges to $p_1$ or $p_2$ as $n \ra \infty$ otherwise.
%By the boundedness of lens part, this is a contradiction to Proposition \ref{prop:I}. 

%One see that $\ovl{p_1p_2}$ is in the boundary direction of $R_p(\tilde E_1)$. 
%Taking $g_1^{-1}$, we see that $\lambda_{g_1}(p_1)^{-1}/\lambda_{g_1}(p_1)^{-1}$ is an eigenvalue 
%occuring at boundary point of  $R_p(\tilde E_1)$.

Now, consider when $U_1$ is a one-sided lens-neighborhood of a totally geodesic p-end and 
let $U_2$ be a concave p-end neighborhood of a radial p-end of $\torb$.
Let $z$ be the intersection point in $\clo(U_1) \cap \clo(U_2)$. 
We can use the same reasoning as above by choosing any $p_1$ in $S_{\tilde E_1}$ 
so that $\ovl{p_1z}$ passes the interior of $\tilde E_1$. Let $p_2$ be the p-end vertex of $U_2$. 
Now we obtain the triangle with vertices $p_1, p_2$, and $z$ as above. Then the arguments are analogous
and obtain an infinite order elements in $\pi_1(\tilde E_1) \cap \pi_1(\tilde E_2)$. 

Finally, consider when $U_1$ and $U_2$ are one-sided  lens-neighborhoods of totally geodesic p-ends respectively.  
Using the intersection point $z$ of $\clo(U_1) \cap \clo(U_2) \cap \torb$
and we choose $p_i$ in $\Bd \tilde E_i$ so that $\ovl{zp_i}$ passes the interior of $S_{\tilde E_i}$ for $i=1, 2$. 
Again, we obtain a triangle with vertex $p_1, p_2, $ and $z$, and find a contradiction as above. 

\end{proof}

%%% April 15 8:20pm
We fully extend the above result.

\begin{corollary}\label{cor:enddisj} 
Let $\orb$ be a properly convex real projective manifold with admissible ends %lens-shaped radial ends, lens-type totally geodesic ends,
%or horospherical ends, 
and satisfies {\rm (IE)} and {\rm (NA)}.
%Let $\orb$ be a properly convex real projective manifold with generalized lens-shaped radial ends, lens-type totally geodesic ends,
%or horospherical ends, 
%and with infinite-index end fundamental groups and with no essential annulus. 
%Let $\mathcal U$ be the collection of the components of the inverse image in $\torb$ 
%of the union of disjoint collection of  end neighborhoods of $\orb$ for all radial or totally geodesic ends. 
Assume that the holonomy group of $\pi_1(\orb)$ is strongly irreducible. 
Let $\mathcal{U}$ be the collection of components of the inverse image in $\torb$ of 
the union of disjoint collection of p-end neighborhoods of $\orb$ for all ends. 
Then for every $U_1, U_2 \in \mathcal{U}$, 
$\clo(U_1) \cap \clo(U_2) = \emp$ or $U_1 = U_2$. 
\end{corollary} 
\begin{proof} 
We now consider horospherical p-ends. Since $\clo(U) \cap \Bd \torb$ is a unique point, 
(iii) of Proposition \ref{prop:affinehoro} implies the result.
\end{proof}

% May 2, 2014 9:04pm 

\section[The strong irreducibility]{The strong irreducibility and stability of the holonomy group of properly convex strongly tame
orbifolds.} 

%A subgroup of $\SLpm$ or $\PGL(n+1, \bR)$ is {\em strongly irreducible} if every finite index subgroup is irreducible. 
%We can replace $\SLpm$ below by $\PGL(n+1, \bR)$ below. 

First, we modify Theorem 6.9 of \cite{endclass} by replacing some conditions. 
\begin{lemma}\label{lem:Th6-9}
Let $\orb$ be a strongly tame 
properly convex real projective manifold with generalized admissible ends %lens-type radial or totally geodesic end. 
and satisfies {\rm (IE)} and {\rm (NA)}.
Let $\tilde E$ be a reducible p-end of $\torb$ of generalized lens-type.
Then there exists a totally geodesic hyperspace $P$ where $h(\pi_1(\tilde E))$ acts on
and $S_{\tilde E}:=P \cap \clo(\torb)$ is a properly convex domain and $S_{\tilde E}^o  \subset \torb$, 
and $S_{\tilde E}/\pi_1(\tilde E)$ is a compact orbifold. 
Also, each element of $g \in \pi_1(\tilde E)$ acts as nonidentity on a subspace properly containing $v$. 
\end{lemma}
\begin{proof}
The proof of Theorem 6.9 of \cite{endclass} shows that $\torb$ is either a join or the conclusion of Theorem 6.9 of \cite{endclass}  holds
and $\pi_1(\tilde E)$ acts on a totally geodesic convex compact domain $D$ of codimension $1$ 
that is the intersection of $P_{\tilde E} \cap \clo(\torb)$ for a $\pi_1(\tilde E)$-invariant subspace $P_{\tilde E}$. 
In the former case, we can show that $\clo(\torb)$ is the join $v_{\tilde E} \ast D$ 
for a compact convex domain $D \subset \Bd \clo(\torb)$ of codimension $1$. 

By (IE), there exists a deck transformation $h$ so that $v_2:= h(v_{\tilde E}) \ne v_{\tilde E}$. 
By geometry of the join, $h(v_{\tilde E})$ is in $D$. This implies that there exists 
another totally geodesic domain $D_1 \subset D$ of lower-codimension so that $\clo(\torb) = v_{\tilde E} \ast v_2 \ast D_2$.
Then by induction, we see that such step must terminate in finite steps. 
This contradicts (IE). 

Let $S_{\tilde E} = D^o$. Then $D^o \subset \torb$. 
The last part follows again from the proof of Theorem 6.9 (ii). 
\end{proof}

Because of the following, we no longer need the assumptions of strong irreducibility of the holonomy 
group of $\pi_1(\orb)$ in this article. 

%\marginpar{I need to emph.}
\begin{theorem}\label{thm:sSPC} 
Let $\orb$ be a noncompact strongly tame properly convex real projective manifold with horospherical, generalized admissible ends %lens-type radial or totally geodesic end. 
and satisfies {\rm (IE)} and {\rm (NA)}. Assume $\partial \orb =\emp$. 
%Let $\orb$ be a closed or noncompact strongly tame properly convex real projective orbifold with generalized admissible ends 
%and each end fundamental group is of infinite index in $\pi_1(\orb)$. 
%Suppose that $\orb$ is strictly convex with respect to ends, or $\orb$ has a horospherical end, 
%or no finite index subgroup of $\pi_1(\orb)$ has an infinite normal subgroup of infinite index.
%%%
Then any finite-index subgroup of the holonomy group is strongly  irreducible and is not 
contained in a proper parabolic subgroup of $\PGL(n+1, \bR)$  {\rm (}resp.  $\SLpm${\rm ).}
\end{theorem} 
\begin{proof} 
We need to prove for $\PGL(n+1, \bR)$ only by Theorem \ref{thm:lifting}.
Let $h:\pi_1(\orb) \ra \PGL(n+1, \bR)$ be the holonomy homomorphism. 
Suppose that $h(\pi_1(\orb))$ is virtually reducible. Then we can choose a finite cover 
$\orb_1$ so that $h(\pi_1(\orb_1))$ is reducible. 

%When $\orb$ is a closed orbifold, this fact is proved in \cite{Ben1}. 
%Now, suppose that $\orb$ has an end.
We denote $\orb_1$ by $\orb$. 
Let $S$ denote a proper subspace where $\pi_1(\orb)$ acts on. 
Suppose that $S$ meets $\torb$. 
Then $\pi_1(\tilde E)$ acts on a properly convex open domain $S\cap \torb$ for each p-end 
$\tilde E$. 
%Then $S \cap \torb$ for any p-end neighborhood gives a submanifold of 
%a closed end orbifold homotopy equivalent to it. 
Thus, $(S\cap \torb)/\pi_1(\tilde E)$ is a compact orbifold homotopy equivalent to one of the end orbifold. 
However, $S \cap \torb$ is $\pi_1(\tilde E)$-invariant 
for each end neighborhood $\tilde E$. This contradicts (IE). 
Therefore, $S \cap \clo(\torb) \subset \Bd \torb$. 

We show that $\clo(\torb) \cap S \ne \emp$. 
Let $\tilde E$ be a p-end. If $\tilde E$ is horospherical, $\pi(\tilde E)$ acts on a great sphere $\hat S$ tangent 
to an end vertex. $S$ has to be a subspace in $\hat S$ containing the end vertex by Proposition 5.1(iii) of \cite{endclass}. 

Suppose that $\tilde E$ is a radial end of generalized lens-type. Then either $S$ passes the end vertex $v_{\tilde E}$ or 
there exists a subspace $S'$ containing $S$ and $v_{\tilde E}$ where $\pi_1(\tilde E)$-invariant. 
%However, $\pi_1(\tilde E)$ acts on $\bR P^n_{v_{\tilde E}}$ irreducibily. 
Hence $S'$ correspond to a proper-invariant subspace in $\SI^{n-1}_{v_{\tilde E}}$
or $S$ is a hyperspace of dimension $n-1$ disjoint from $v_{\tilde E}$. 
By considering a hyperbolic factor, 
it follows that there exists some attracting fixed points 
in the limit sets of $\pi_1(\tilde E)$. Considering when $\pi_1(\tilde E)$ has nontrivial diagonalizable 
elements, we obtain $S \cap \clo(L) \ne \emp$ for a lens $L$, $L \subset \torb$. 
%If $\pi_1(\tilde E)$ is irreducible, then 
%$S$ is a hyperspace disjoint from $v_{\tilde E}$. 
The existence of the attracting fixed points of some elements 
of $\pi_1(\tilde E)$ implies that $S \cap \clo(L) \ne \emp$ for a lens $L$, $L \subset \torb$. 
(This follows from 
Theorem 7.9 of \cite{endclass} and Proposition 1.1 of \cite{Ben5} and the uniform middle eigenvalue condition.)

If $\tilde E$ is totally geodesic of lens-type, we can apply a similar argument using the attracting fixed points. 
Therefore, $S \cap \clo(\torb)$ is a subset $K$ of $\Bd \torb$ of $\dim \geq 0$. 
In fact, we showed that the closure of each p-end neighborhood meets $K$.

%Assume that $\dim K \geq 1$ from now on. 
%Then there exists a segment connecting two points in the closures of distinct neighborhoods. 
%This contradicts the strict convexity by Corollary \ref{cor:enddisj}.
 % since the corollary implies that a segment cannot have points in two closures of distinct p-end neighborhoods.  

%If $\torb$ has a horospherical end, then a horospherical p-end vertex lies in $K$ and this contradicts Proposition \ref{prop:affinehoro}(iii). 

By taking dual orbifold if necessary, 
we assume without loss of generality that there exists a radial end of generalized lens-type with a radial p-end vertex $v$. 

(I) Suppose that a p-end vertex $v$ is in $K$. 
$v$ cannot be a horospherical p-end vertex by Proposition \ref{prop:affinehoro}(iii). 

%Suppose that a radial p-end $\tilde E$ of lens-type exist and the p-end vertex $v$ is in $K$: 
Now $v$ is a p-end vertex of a generalized lens-shaped end $\tilde E$.
Let $V_v$ be the concave p-end neighborhood of $\torb$ of $v$ from a system of mutually disjoint end neighborhood of $\orb$.
Then $h(\pi_1(\tilde E)) \cap h(\pi_1(\orb_1))$ is reducible.

%Proposition \ref{prop:Ben2} $\pi_1(\tilde E)$ is virtually reducible and there exists 
%a sequence of elements $g_i$ in the virtual center of $\pi_1(\tilde E)$
%acting on the space of segments in $K$ from $v$, and for the maximum norm $\lambda_i$ of the eigenvalues associated 
%with $K$ and the next maximum of rest of the eigenvalues $\lambda'_i$ in $P(K)$, we have 
%\begin{equation}\label{eqn:LLp}
%\lambda_i/\lambda'_i \ra +\infty.
%\end{equation} 

%The proof of Theorem 6.9 (ii) in \cite{endclass} shows that 
%there exists a unique codimension $1$ subspace $P_{\tilde E}$ where $\pi_1(\tilde E)$ acts on. 
We obtain a totally geodesic hypersurface $S_{\tilde E} \subset \torb$ by Lemma \ref{lem:Th6-9}.
Let $P_{\tilde E}$ denote the spanning subspace of $S_{\tilde E}$. 
(We are only using the part of the proof where the strong irreducibility of $h(\orb)$ is not used yet there.)
Choose $x \in \torb$ in a direction of $S_{\tilde E}$ from $v$. 
Since $S_{\tilde E}/\pi_1(\tilde E)$ is compact, we choose a sequence $g_i \in \pi_1(\tilde E)$ of central elements so that 
$g_i(x) $ converges to a point of $\clo(V_v)\cap K$ as $i \ra \infty$. 
 We can choose unbounded $g_i$ so that $\{g_i|P_{\tilde E}\cap K\} \ra \Idd_{P_{\tilde E}\cap K}$ on $P \cap K$
by Lemma \ref{lem:Diag} applied to $\SI^{n-1}_v$. 
%f the eigenvalue $\lambda_{v, i}$ of $g_i$ at $v$ is different from other eigenvalues $\lambda'_i$ of $g_i$ in 
%$K$, then there exists a codimension $2$ subspace $P$ so that $\{g_i|P\} \ra \Idd_P$ on $P \cap K$.  ($P$ does not depend of $i$ as $g_i$ are in the center.)
Let $\lambda'_i$ denote the $(n_K:= \dim P_{\tilde E} \cap K)$-th root of the norm of the determinant of $g_i|P_{\tilde E} \cap K$, 
and let $\lambda_{v, i}$ denote the eigenvalue of $g_i$ associated with $v$. 
Since there exists a generalized $g_i$-invariant lens whose closure is disjoint from $v$, we can deduce that  
$\left|\frac{\lambda'_i}{\lambda_{v, i}}\right| \ra \infty $ or bounded below by 
a positive number $C> 0$ as $i \ra \infty$. (See Chapter 7 of \cite{endclass}.)

(i) Suppose first $\left|\frac{\lambda'_i}{\lambda_{v, i}}\right| \ra \infty $.
Hence, $\{g_i(K)\}$ for $i \ra \infty$
geometrically converges  for some choice of signs above $g_i$ to a cone
$v \ast (P_{\tilde E} \cap K)$ or alternatively
 $v\ast (P_{\tilde E}\cap K) \cup (P_{\tilde E}\cap K) \ast v_-$,
 a nonproperly convex set in $\clo(\torb)$ as $g_i| P_{\tilde E} \cap K \ra \Idd_{P_{\tilde E}\cap K}$ by 
applying Lemma 6.11 of \cite{endclass} to the closures of the both components of $K - K\cap P_{\tilde E}$. 
This is a contradiction. 
Only the first case, $K = v \ast (P_{\tilde E}\cap K)$ holds. 

%There exists a complementary subspace $K^c$ in $\SI^n$ invariant by $g'_i$ by the virtual centrality of $g'_i$ and 
%the proof of Theorem 6.9(ii). 
%Let $V_v$ be a concave p-end neighborhood of $\torb$ of $v$ from a system of mutually disjoint end neigbhoroods of $\orb$.
By apply $g_i^n$ to $V_v$, we obtain a subset $U_v$ equal to $v \ast (P_{\tilde E} \cap K)^o$ with nonempty interior. 
$U_v$ is in the relative interior of $\clo(V_v)$ in $\clo(\torb)$ considering segments  in $V_v$ in directions of $S_{\tilde E}$
and the action of $g_i^n$. 

Since $\pi_1(\tilde E)$ is of infinite index in $\pi_1(\orb)$, there is an element $h \not\in \pi_1(\tilde E)$ 
with $h(v)$, $h(v) \ne v$. Since $h(v) \in K$ and $h(v)$ has 
an end neighborhood $h(U_v)$ in $K$ that is also a neighborhood of $v$ in $K$ with nonempty interior,  
Since $v \in U_v$, we obtain $h(U_v) \cap U_v \ne \emp$.
Thus, $h(V_v) \cap V_v \ne \emp$ holds and hence $h(V_v) = V_v$. This is a contradiction to $h \not\in \pi_1(\tilde E)$. 

(ii) Suppose that \[C< \left|\frac{\lambda'_i}{\lambda_{v, i}}\right| < C' \hbox{ for some constants } C, C'> 0.\] 
Then there exists then a sequence of mutually distinct elements $g'_i \in \pi_1(\tilde E)$ so that $\{g'_i| K\} \ra \Idd_K$.  
%a kernel $K_{\tilde E}$ of $\pi_1(\tilde E) \cap \pi_1(\orb_1) \ra \Aut(K)$ of infinite order by Proposition \ref{prop:Ben2}. 

%In both cases (i) and (ii), 
%there exists a complementary subspace $K^c$ in $\SI^n$ invariant by $g'_i$ 
%by equation \eqref{eqn:LLp} and the virtual centrality of $g'_i$. 
%Let $V_v$ be the p-end neighborhood of $\torb$ of $v$ from a system of mutually disjoint end neigbhoroods of $\orb$.
By apply $\{g_i\}$ to $V_v$, we obtain subsets $U_v \subset v \ast (P_{\tilde E} \cap K)^o$ with nonempty interior 
inside the relative interior of $\clo(V)$ in $\clo(\torb)$ as above. 

%Also, $U_v^o:= (v \ast (P\cap K))^o$ is in $U^o\cup V_v^o$ since we can show $\{v\}, P\cap K \subset \clo(\orb)$ 
%as we can apply $g_i$ above. 

Since $\pi_1(\tilde E)$ is of infinite index in $\pi_1(\orb)$, it follows that there is an element $h \in \pi_1(\tilde E)$ 
with $h(v)$, $h(v) \ne v$. Since $h(v) \in K$ and $h(v)$ has 
an end neighborhood $h(U_v)$. % that is a neighborhood of $v$ in $K$ with nonempty interior. 

For a sufficiently large $i$, $g_i(h(U_v^o)) \cap h(U_v ^o)\ne \emp$. 
Since a point of $U_v^o$ has a neighborhood $V_v$, 
this implies $g_i(h(V_v)) \cap h(V_v) \ne \emp$, and
we obtain $g_i(h(V_v))= h(V_v)$ and $g_i \in h^{-1} \pi_1(\tilde E) h$. 
Hence, this implies that there exists an essential annulus, a contradiction to (NA).

%Since $\pi_1(\orb_1)$ acts on $K$, 
%there is a homorphism $\pi_1(\orb_1) \ra \Aut(K)$. In this case, the kernel is of infinite order and of infinite index. 

(II) Suppose that for every radial p-end $\tilde E$, the p-end vertex $v$ is not in $K$. 
Now $v$ cannot be horospherical since the horospherical action fixing $v$ does not preserve a properly convex set except $v$ or horoballs of dimension $n$. 
Hence, there exists a radial p-end $\tilde E$ of generalized lens-type as above. Let $v$ be its vertex. 

%Then $K$ is in the convex hull of 
%the closure of the concave p-end-neighborhood $V_{\tilde E}$ of the end $\tilde E$. 
Suppose that $K$ is of dimension $n-1$. Then $\torb$ is a strict join $v \ast K$, as we can deduce by the $\pi_1(\tilde E)$-invariance of $K$
and the fact $K \subset \Bd \torb$. By (IE), $h(v) \ne v$ for $h \in \pi_1(\orb)$. We obtain $h(v) \in K$ by the proper convexity 
of $S_{\tilde E}$. This contradicts the premise of (II). 

Suppose that $\dim K \leq n-2$. Then each radial p-end $\tilde E$ is reducible. %and $v_{\tilde E} \ast K$ contains 
%a set $U_{v_{\tilde E}}$ in the relative interior of a concave p-end neighborhood $V_{v_{\tilde E}}$
%in $\clo(\torb)$. Since $K$ is properly convex, $\pi_1(\tilde E)$ acts cocompactly on $K^o$. 

%Again using the proof of Theorem 6.9(ii), there exists a hyperspace $P_{\tilde E_1}$ where $\pi_1(\tilde E_1)$ acts on. 
We obtain a totally geodesic hypersurface $S_{\tilde E_1} \subset \torb$ by Lemma \ref{lem:Th6-9}. 
Let $P_{\tilde E_1}$ denote the spanning subspace of $S_{\tilde E_1}$. 
Since $K$ is disjoint from $v_{\tilde E_1}$, let $P_K$ be the subspace containing $K$ and $v_{\tilde E_1}$. 
Choose $x \in S_{\tilde E_1}$.
Since $S_{\tilde E_1}/\pi_1(\tilde E_1)$ is compact, we choose a sequence $g_i \in \pi_1(\tilde E_1)$ of central elements so that 
$g_i(x) $ converges to a point of $C_K \cap P_{\tilde E}$ as $i \ra \infty$. 
 We can choose unbounded $g_i$ so that $\{g_i|C_K\cap P_{\tilde E}\} \ra \Idd_{C_K\cap P_{\tilde E}}$
by Lemma \ref{lem:Diag} applied to $\SI^{n-1}_{v_{\tilde E_1}}$. Let $k$ be the dimension of $C_K \cap P_{\tilde E}$. 
The $(k+1)$-th root $\hat \lambda_i$ of the norm of the determinant of the submatrix of the unit-determinant matrix $g_i$ corresponding to $C_K \cap P_{\tilde E}$. 
Then $\hat \lambda_i > \lambda(g_i)(v_{\tilde E_1})$ for the eigenvalue $\lambda(g_i)(v_{\tilde E_1})$ of $g_i$ at $v_{\tilde E_1}$  by Theorem 7.9 of \cite{endclass}
since there exists a lens whose closure is disjoint from $\{v_{\tilde E_1}\}$.
%Since every $g \in \pi_1(\tilde E_1)$ acts not as an identity on a subspace containing $v_{\tilde E_1}$ properly by Lemma \ref{lem:Th6-9}, 
%it follows that $K \subset P_{\tilde E}$. 
Since $K$ is $h(g_i)$-invariant, we obtain $K \subset P_{\tilde E}$.

Since $S_{\tilde E_1} \subset \torb$ and $S_{\tilde E_1}/\pi_1(\tilde E_1)$ is a compact surface immersed in $\orb$, 
we can assume without loss of generality that $S_{\tilde E_1}$ covers a totally geodesic closed orbifold $S_1$ by taking a finite cover of $\orb$.
Moreover, $K \subset \clo(S_{\tilde E_1})$ by the existence of the sequence $g_i$ acting on $S_{\tilde E_1}$ as above. 
%%% May 16 12:29pm

By condition (IE), choose $h \in \pi_1(\orb)$ so that 
%there exists at least two radial ends $\tilde E$ and $\tilde E_2$ where 
$v_{\tilde E_1} \ne v_{\tilde E_2} = h(v_{\tilde E_1})$. 
%Since $\pi_1(\tilde E_i)$ acts on the totally geodesic hypersurface $S_{\tilde E_i}$ \subset \torb$.
%and $K$ is on segments from $v_{\tilde E_i})$, we obtain $K \subset P_{\tilde E_i}$ and 
%$K \subset \clo(S_{\tilde E_i})$ by the proof of Theorem 6.9 (ii) of \cite{endclass} and Lemma \ref{lem:Diag}. 
%Then we take a triangle with vertices $v_{\tilde E_1}, v_{\tilde E_2}, x$ for $x \in K$
%as in Corollary \ref{cor:disjclosure} and concave p-end neighborhoods of $v_{\tilde E_1}$ 
%and $v_{\tilde E_2}$ and obtain a contradiction to (IE) similarly.  
Also, by the same argument as above, we have $\clo(h(S_{\tilde E_1}))\supset K$. 

Then as in the proof of Corollary \ref{cor:disjclosure}, we obtain that 
$d_{\torb}(S_{\tilde E_1}, h(S_{\tilde E_1})) \leq C$ for a constant $C>0$, and 
there is an infinite order element of 
$\pi_1(\tilde E_1) \cap h \pi_1(\tilde E_1) h^{-1}$. This is a contradiction to (NA).

Thus, we have shown that $\pi_1(\orb)$ is irreducible. Since the argument works for every finite cover 
of $\orb$, $\pi_1(\orb)$ is strongly irreducible.

%Suppose that $\tilde E$ is a totally geodesic end of lens-type. 
%We can dualize the above argument again by Theorem \ref{thm:duality}
%where a totally geodesic end of lens-type corresponds to a radial end of generalized lens-type. 
%Hence, $h(\pi_1(\orb))$ is strongly irreducible. 

%Each p-end fundamental group of $\pi_1(\orb_1)$ must act on $K$. 
%By the admissibility, the closure of the corresponding p-end neighborhood must meet $K$ as above. 
%There are infinitely many closures of the distinct such neighborhoods meeting $K$. 
%If $K$ is a singleton, we obtain the contradiction since every end vertex is identical.
%Two horospherical p-end vertices cannot be same since they are supported by a unique hyperplane 
%at the vertex. Two generalized lens p-end neighborhoods have two concave p-end neighborhoods.
%\marginpar{wrong here...}
%They must intersect and this implies that the p-end has to be the same. 
%These are contradictions to (IE). 

Since a parabolic group acts on a nontrivial flag in $\bR^{n+1}$ by definition, a parabolic group is always reducible. 
This shows that $\pi_1(\orb)$ is not parabolic. 

%Supposed that $h(\pi_1(\orb))$ is unipotent. Then every p-end has to be horospherical. 
%Since $h(\pi_1(\orb))$ is nilpotent, there is a Zariski closure $N$ in $\PGL(n+1, \bR)$ 
%so that $N/h(\pi_1(\orb))$ is compact.  Then by Proposition 1 of \cite{BenNil}, we can choose a finite 
%index subgroup $\Gamma$ of $h(\pi_1(\orb))$ so that $N/\Gamma$ is compact and $N$ acts transitively and simply on $\torb$.
%Hence $\torb/h(\pi_1(\orb))$ is compact. %The strict convexity implies that $\Gamma$ is hyperbolic by \cite{Ben1}, contradiction. 
\end{proof} 

For a matrix $g$ in $\GL(n+1, \bR)$, we denote the induced projective automorphism by $S(g):\SI^n \ra \SI^n$.
We state the elementary lemma to clarify. 
\begin{lemma}\label{lem:Diag} 
Let $\bZ^l$ be in the center of the holonomy group of 
a properly convex closed real projective $(n-1)$-orbifold with admissible fundamental group. 
Let $r: \bZ^l \ra D$ be the inclusion homomorphism to a group of diagonal matrices
on $\bR^n = \bigoplus_{i=1}^m V_i$ where $S(r(g))| S(V_i)= \Idd_{S(V_i)}$ for all $g \in \bZ^l$. 
For any given $V':= V_{i_1} \oplus \cdots \oplus V_{i_j}$ for a proper set $\{i_1, \dots, i_j\}$, 
we can find a sequence of elements $g_i \in \bZ^l$ so that 
$\lambda_i/\lambda'_i \ra \infty$ for the largest norm $\lambda_i$ of the eigenvalue of $r(g_i)$ on $V'$ 
and $\lambda'_i$ on the complement of $V'$ and $S(r(g_i))| S(V') \ra \Idd_{S(V')}$. 
\end{lemma}
\begin{proof} 
This follows by Theorem 1.1 of  \cite{Ben3}. 
\end{proof}

%May 3 7:04pm

%For the following, we do not need the strong irreducibility of $\pi_1(\orb)$. 
%\begin{prop} \label{prop:redend} 
%Let $\orb$ be a properly convex real projective manifold with horospherical, generalized admissible ends %lens-type radial or totally geodesic end. 
%and satisfies {\rm (IE)} and {\rm (NA)}.
%Let $\tild E$ be reducible radial p-end. Then there exists a totally geodesic hypersurface $S$ imbedded in $\torb$
%where $\pi_1(\tilde E)$ acts on and $S/\pi_1(\tilde E)$ is compact. 
%\end{prop}
%\begin{proof}

%\end{proof} 

%\marginpar{in the right place? Repeated?}
%\begin{remark}\label{rem:propconv}
%For some orbifolds, $(n-1)$-dimensional convex real projective structures on end orbifolds are always properly convex:
%If $\mathcal{O}$ is closed and $\pi_1(\mathcal{O})$ is virtually trivial center, then this holds. For example, closed $2$-orbifolds 
%of negative Euler characteristic or orbifolds admitting hyperbolic structure is always properly convex 
%by Proposition 1.3 of Benoist \cite{Ben3} (see Vinberg \cite{Vin} for earlier such results and Theorem 3.2 of \cite{CL} also). 
%For orbifolds with ends, we can prove this in some cases at the moment
%without no generality is in sight. For closed $2$-orbifolds with zero Euler characteristic, this is not true.
%For example, consider tori and those covered by them such as a turn-over with singularities of orders $3, 3, 3$. 
%\end{remark}

%Thus, we associate the subgroup $\Gamma(E)$ or $\Gamma(U_1)$ of $\pi_1(\mathcal{O})$ to the end $E$ corresponding to $U_1$.

%%le 26 mai 2014. 

\chapter{The strict SPC-structures and relative hyperbolicity} %{ The SPC-structures, the strict convexity and 
%the relative hyperbolicity of the fundamental groups of strongly tame orbifolds.} 
\label{sec:SPC}
%\chaptermark{The strict SPC-structures and relative hyperbolicity}

%Outline here.
In this section, we will be working with $\bR P^n$ exclusively. 

From now on, we will assume that properly convex strongly tame real projective orbifolds with generalized admissible 
ends have strongly irreducible holonomy groups by Theorem \ref{thm:sSPC}.

\section{The Hilbert metric on $\mathcal{O}$.} 

A Hilbert metric on an orbifold with an SPC-structure is defined as a distance metric 
given by cross ratios. (We do not assume strictness here.)

%If necessary, we reselect each component of $U$  so that the minimum distances between two 
%components of $\tilde U$ is at least $1$ under the Hilbert metric. 

%Let $U'$ be the union of neighborhoods of ends of ${\mathcal{O}}$ 
%and let $\tilde U'$ be its inverse image in $\tilde {\mathcal{O}}$ as in the above section. 
%We choose a mutually disjoint collection of neighborhoods for each p-end and the closure of 
%${\mathcal{O}}-U'$ is a compact suborbifold $K$ with boundary 
%$\partial U'$. We will choose sufficiently large $K$. Let $K_1$ denote a fundamental domain of $K$ in $\torb$.  

%We define $\tilde{\mathcal{O}}/\tilde U'$ to be obtained by collapsing each compo
%We define $\tilde {\mathcal{O}}/\tilde U'$ as the space obtained by collapsing each component of $\tilde U'$ to a point. 
%(This is might be different from the usual definition in topology but we use it here for convenience.)
%That is, each component of $\tilde U'$ is collapsed to a point to be denoted as a p-end. 
%More precisely, we remove the interior of  $\tilde U'$ and produce the path metric using the Hilbert metric 
%and collapse each component of the boundary to a point. 
%We still call this the {\em electric Hilbert metric} and denote it by $d_U$; however, we will not use 
%this approaches of B. Farb. 

Given an open properly convex domain $\Omega$, 
we note that given any two points $x, y$ in $\Omega$, there is a geodesic arc \index{Hilbert metric}
$\ovl{xy}$ with endpoints $x, y$ so that its interior is in $\Omega$. 

%% September 29 9:58 pm
\begin{proposition}\label{prop:shortg} 
Let $\Omega$ be a properly convex open domain. 
Let $P$ be a subspace meeting $\Omega$,
and let $x$ be a point of $\Omega  - P$\,{\rm :}
\begin{itemize}
\item[(i)] There exists a shortest path $m$ from $x$ to $P \cap \Omega$ that is a line segment. 
\item[(ii)] The set of shortest paths have end points in a connected compact subset $K$ of $P\cap \Omega$. 
\item[(iii)] For any line $m'$ containing $m$ and $y \in m'$, the segment in $m'$ from $y$ to the point of $P \cap \Omega$ is 
one of the shortest segments.
\item[(iv)] When $P$ is a complete geodesic in $\Omega$, outside the compact set $K$, 
the distance function from $P-K$ to $x$ is strictly increasing or strictly decreasing. 
\end{itemize}
\end{proposition}
\begin{proof} 
The distance function $f: P \cap \Omega \ra \bR$ defined by $f(y) = d(x,y)$ is 
a proper function where $f(x) \ra \infty$ as $x \ra z$ for any boundary point $z$ of $P \cap \Omega$ in $P$. 
Hence, there exists a shortest segment with an endpoint $x_0$ in $P \cap \Omega$. 
(i) follows. 

Let $\gamma$ be any geodesic in $P \cap \Omega$ passing $x_0$. 
We need to consider the $2$-dimensional subspace $Q$ containing $\gamma$ and $x$.
The set of end points of shortest segments of $\Omega$ in $Q$ is a connected compact subset containing $x_0$ by 
Proposition 1.4. of \cite{marg}.
Hence, by considering all geodesics in $P \cap \Omega$ passing $x_0$, (ii) follows. 

(iii) Suppose that there exists $y \in m'$, so that the the shorted geodesic $m''$ to $P \cap \Omega$ is not in $m'$. 
Consider the $2$-dimensional subspace $Q$ containing $m'$ and $m''$. Then this is a contradiction by
Corollary 1.5 of \cite{marg}. 

(iv) Again follows by considering a $2$-dimensional subspace containing $P$ and $m$. 

%Therefore, there is a connected compact interval of minimum of $f$ 
%in $l$. Take a geodesic $m$ from $y$ to one of the points of the interval. 

%For any point on $m$,  the shortest segment to $l$ is again inside $m$ since otherwise, we get a shorter segment from $x$. 
%For a point $z$ on $m'$ but not on $m$, suppose that there is a shortest geodesic $m''$ from $z$ to $l$: 
%Then $m'', m', l$ lies on  a two-dimensional subspace $P$ and our discussion reduces to $P\cap \tilde {\mathcal{O}}$. 
%The methods of Proposition 1.4 in \cite{marg} apply and we are done. 
\end{proof}

%Given a subspace $P$, the set of foot of shortest geodesic segments from $x$ to $\Omega \cap P$.  
An endpoint in $P$ of a shortest segment is called a {\em foot of the perpendicular} from $x$ to $\gamma$. \index{foot} 
%The foot is not unique but 
%the set of feet from $x$ in $\tilde K$ to a geodesic $\gamma$ in $\tilde K$ is a connected subsegment in 
%$\tilde {\mathcal{O}}$. Again, this follows from \cite{marg} since we only need a two-dimensional subspace.

%Given a subspace $P$, the set of foot of shortest geodesic segments from $x$ to $\Omega \cap P$ 

\section{Strict SPC-structures and the group actions}

%\begin{definition}\label{defn:SPC}
%We will only study {\em stable irreducible properly convex real projective structures} or SPC-structures on $\mathcal{O}$, 
%i.e., properly convex structures with irreducible holonomy representations and convex ends. 
%We also need a condition that straight arcs in the boundary of $\tilde{\mathcal{O}}$ and the non-$C^1$-points 
%must be contained in the closure of every p-end neighborhood of a p-end
%and as a consequence any triangle with interior in $\tilde{\mathcal{O}}$ and boundary in 
%$\Bd \tilde{\mathcal{O}}$ must be inside the closure of 
%a p-end-neighborhood.  
%This means that there is  no open properly convex triangle in $dev(\tilde O)$ 
%with boundary in the boundary of $dev(\tilde O)$ with one vertex an end vertex. %and at least two of the interior of 
%the edges meet the limit sets of a neighborhood of the end. 
%The SPC-structure satisfying the condition is said to be the {\em strict SPC}-structures. 
%\end{definition}
%We also assume that each component of $U_1$ of $\dev(\tilde O)$ has a convex neighborhood $U'_1$
%so that $U'_1/\Gamma_1$ for the deck transformation group of $U_1$ is homeomorphic 
%to the orbifold times an open interval. 
%% Convex end neighborhood?? Think about this.

%In the Fuchsian case it is also a {\em convex end-neighborhood}.

%% March 1, 5:00 pm

An {\em elliptic element} of $g$ is a nonidentity element of $\pi_1(\mathcal{O})$ \index{elliptic element} \index{elliptic element} 
fixing an interior point of $\tilde{\mathcal{O}}$. 
Since $\pi_1(\mathcal{O})$ acts discretely on the space $\torb$ with a metric, 
an elliptic element has to be of finite order. 
%Let $\torb$ be a properly convex domain in $\bR P^n$ or $\SI^n$. 
%An element of $\pi_1(\orb)$ is elliptic if and only if it is of finite order since $\torb$ is contractible. 

%\begin{lemma}\label{lem:triangle} 
%An abelian group $\Gamma$ of rank greater than equal to $2$ in $\pi_1(\mathcal{O})$ is 
%in an end fundamental group.
%\end{lemma}
%\begin{proof}

%\end{proof} 
\begin{lemma} \label{lem:hor} 
Let $\orb$ be a strongly tame strict SPC-orbifold with admissible ends. 
Let $\tilde E$ be a p-end of $\torb$.
\begin{itemize}
\item[(i)] Suppose that $\tilde E$ is a horospherical p-end.
Let $B$ be a horoball at a p-end vertex $p$ corresponding to $\tilde E$. 
There exists a homeomorphism $\Phi_{\tilde E}: \Bd B - \{p\} \ra \Bd \torb -\{p\}$ given 
by sending a point $x$ to the end point of maximal convex segment containing $x$ and $p$
in $\clo(\torb)$. 
\item[(ii)] Suppose that $\tilde E$ is a radial p-end of lens-type. 
Let $U$ be a lens-shaped radial p-end neighborhood with the p-end vertex $p$ corresponding to $\tilde E$. 
There exists a homeomorphism $\Phi_{\tilde E}: \Bd U \cap \torb  \ra \Bd \torb - \clo(U)$ given 
by sending a point $x$ to the other end point of the maximal convex segment containing $x$ and $p$
in $\clo(\torb)$.
\end{itemize}
Moreover, each of the maps denoted by $\Phi_{\tilde E}$ commutes with elements of $h(\pi_1(\tilde E))$. 
\end{lemma} 
\begin{proof} 
(i) By Proposition 5.1 (i) of \cite{endclass}, $\Phi_B$ is well-defined.  
The same proposition implies that  $\Bd B $ is smooth at $p$ and $\Bd \torb$ has a unique supporting hyperplane. 
Therefore the map is onto. 

(ii) The second item follows from Theorems \ref{thm:lensclass} and \ref{thm:redtot}
since they imply that the segments in $S(p)$ are maximal ones in $\Bd \torb$ from $p$.

\end{proof}

We now study the fixed points in $\clo(\torb)$ of elements of $\pi_1(\orb)$. 
A {\em great segment} is a geodesic arc in $\SI^n$ with antipodal p-end vertices. \index{great segment} 
It is not properly convex.

\begin{lemma}\label{lem:fix} 
Let $\orb$ be a strict SPC-orbifold with admissible ends. 
Let $g$ be an infinite order element of a p-end fundamental group. 
Then every fixed point of $g$ in $\clo(\torb)$ is in the closure of a p-end-neighborhood.
\end{lemma}
\begin{proof}
Suppose that the radial p-end $\tilde E$ is lens-shaped. 
The direction of each segment in the interior of the lens cone 
with an endpoint $v_{\tilde E}$ is fixed by only the identity element of 
$\pi_1(\tilde E)$ since $\pi_1(\tilde E)$ acts properly discontinuously on $S_{\tilde E}$. 
% since each segment in the interior of the lens cone 
%with an endpoint $v_{\tilde E}$ is fixed by only the identity element of 
%$\pi_1(\tilde E)$ by Lemma \ref{lem:hor}. 
Thus, the fixed points are on the rays in the direction of the boundary of $\tilde E$. 
They are in one of $S(v_{\tilde E})$ for the p-end vertex $v_{\tilde E}$ corresponding to $\tilde E$.
Hence, the fixed points of 
the holonomy homomorphism of $\pi_1(\tilde E)$ is in the closure of the lens-cone
with end vertex $v_{\tilde E}$ and nowhere else
in $\clo(\tilde{\mathcal O})$. 

If $\tilde E$ is a horospherical, then the p-end vertex $v_{\tilde E}$ is not contained in any segment $s$ in $\Bd \torb$
by Proposition \ref{prop:affinehoro}.  
Hence $v_{\tilde E}$ is the only point $S \cap \Bd \torb$ of any invariant subset $S$ of $\pi_1(\tilde E)$
by Lemma \ref{lem:hor}. 
Thus, the only fixed point of $\pi_1(E)$ in $\Bd \torb$ is $v_{\tilde E}$.

Suppose that $E$ is a totally geodesic p-end of lens-type, 
and a fixed point $s \in \Bd \torb$. % is not in $\clo(S_{\tilde E})$ for 
%the totally geodesic p-end hypersurface $S_{\tilde E}$ corresponding to $\tilde E$. 
Since $\tilde E$ is a properly convex real projective orbifold that is closed, 
we obtain an attracting fixed point 
$a$ and a repelling fixed point $r$ of $g|\clo(S_{\tilde E})$ by \cite{Ben2}. %\marginpar{Need a lemma?} 
Then $a$ and $r$ are attracting and repelling fixed points of $g|\clo(\torb)$ by 
the existence of the lens neighborhood and Theorem 7.9 in \cite{endclass}. 

We claim that $\ovl{as}$ and $\ovl{rs}$ are in $\Bd \torb$. 
Let $P$ denote the two-dimensional subspace containing $r, s, a$. 
Suppose that one of the segment intersects $\torb$ in $x$. 
Then we take a open convex ball-neighborhood of $x$ in $P \cap \torb$.
Suppose that $x \in \ovl{rs}^o$.
Then using the sequence $g^n(B)$, we obtain a great segment in $\clo(\torb)$ by choosing $n \ra \infty$
by Theorem 7.11 of \cite{endclass}. 
This is a contradiction. If $x \in \ovl{as}^o$, we can use $g^{-n}$(B) as $n \ra \infty$, again giving us
a contradiction.  

Since $\tilde E$ has a lens type one-sided neighborhood $U$, $\clo(U) \cap \Bd \torb$ is in $\clo(S_{\tilde E})$ by Proposition \ref{prop:preI}. 
By the strict convexity of $\torb$, we see that $a, r,$ and $s$ have to be in $\clo(S_{\tilde E})$.
\end{proof}

See Crampon and Marquis \cite{CM} and Cooper-Long-Tillman \cite{CLT3} for similar work to the following.  
%%% April 2, 2014 5:14pm
%% Avril 2 10:34pm 

\begin{proposition} \label{prop:loxopara}
Suppose that $\mathcal O$ is a noncompact strongly tame strict SPC-orbifold with  admissible ends. 
%with generalized lens-type radial end, totally geodesic 
%lens-type ends, or horospherical end.
%Each abelian subgroup of $\pi_1(\orb)$ of rank $\geq 2$ is in a p-end fundamental group.
Then each nonidentity and infinite-order element $g$ of $\pi_1(\mathcal{O})$ has three mutually exclusive possibilities: 
\begin{itemize}
\item $g|\clo(\torb)$ has exactly two fixed points in $\Bd \tilde{\mathcal O}$ none of which is in the closures of the p-end neighborhoods, 
\item $g$ is in a p-end fundamental group, and $g|\clo(\torb)$ 
\begin{itemize} 
\item has all fixed points in $\Bd \tilde{O}$ in the closure of a concave p-end neighborhood of a lens-shaped radial p-end, 
\item has all fixed points  in $\Bd \tilde{O}$ in $\clo(S_{\tilde E})$ for the ideal boundary component $S_{\tilde E}$ 
of a totally geodesic p-end $\tilde E$ of lens-type, or
\item has a unique fixed point  in $\Bd \tilde{O}$ at the horospherical p-end vertex. 
\end{itemize} 
\end{itemize}
\end{proposition}
\begin{proof}
Let $g \in \pi_1(\mathcal{O})$. 
Suppose that  $g$ has a fixed point at a horospherical p-end vertex $v$ for a p-end $\tilde E$. 
We can choose the horoball $U$ at $v$ that 
maps into an end-neighborhood of $\orb$. 
Since $g(U) \cap U \ne \emp$ by the geometry of a horoball having a smooth boundary at $v$, 
%and $v \in \Bd \torb$ an imbedded hypersurface, 
$g$ must act on the horoball since the horoball is either sent to a disjoint one or 
sent to the identical one, and hence $g$ is in the p-end fundamental group: 
A horoball $U$ has a unique hyperspace that also supports $\tilde{\mathcal{O}}$. Thus, $g(U) \cap U \ne \emp$
for any horoball p-end neighborhood $U$. Thus, $g(U)= U$ for a horoball p-end neighborhood.
Since $\Bd U - \partial U$ is a unique point and $\partial U \subset \torb$ where $g$ acts freely, 
the p-end vertex is the unique fixed point of $g$ in $\Bd \torb$ by Lemma \ref{lem:fix}.  

Similarly, suppose that $g \in \pi_1(\orb)$ fixes a point of the closure $U$ of a concave
p-end neighborhood of a p-end vertex $v$ of lens-type. 
$g(\clo(U))$ and $\clo(U)$ meet at a point. 
By Corollary \ref{cor:disjclosure}, $g(\clo(U))$ and $\clo(U)$ share the p-end vertex 
and hence $g(U) = U$ as $g$ is a deck transformation. 
Therefore, $g$ is in the p-end fundamental group of the p-end of $v$. 
Lemma \ref{lem:fix} implies the result. 

%For a lens-shaped p-end vertex $v$, $g \in \Gamma_v$ has fixed points in the closure of the lens-shaped neighborhood. 
%The existence of the homeomorphism $\Phi_p$ of Lemma \ref{lem:hor} implies that 
%the only fixed points can exists on the directions in $\partial E$ and hence they are in $\bigcup S(v)$.

Suppose that $g \in \pi_1(\orb)$ fixes a point of $\clo(S_{\tilde E})$ for a totally geodesic ideal boundary $S_{\tilde E}$ 
corresponding to a p-end $\tilde E$. 
Again by Corollary \ref{cor:disjclosure} and Lemma \ref{lem:fix} imply the result for this case. 

%We take the dual group $\pi_1(\tilde E^*)$ and the dual domain $\torb^*$ where 
%$\tilde E$ corresponds to a radial p-end $\tilde E^*$ of $\torb^*$. Let $v$ denote the corresponding p-end
%of $\torb^*$. By Lemma \ref{lem:hor}, $g^{\ast -1}$ has fixed points only in $\bigcup S(v)$ and $g^\ast \in \pi_1(\tilde E^*)$. 
%Suppose that $g^{\ast -1}$ acts on a supporting hyperplane $S$ of $\torb$. Then $S \cap \clo(\torb)^*$ 
%is a compact convex set and hence has a point $q$ fixed by $g^{\ast -1}$ by the Brouwer fixed point theorem. 
%Then $q$ is a point of $\bigcup S(v)$. 
%The supporting hyperplanes to points of $\bigcup S(v)$ are dual to points of $\clo(S_{\tilde E})$. 
%Since a fixed point of $g$ in $ \torb$ corresponds to a supporting hyperplane of $\torb^*$ 
%acted by $g^{\ast -1}$, which must be supporting $\torb^*$ at a point of $\bigcup S(v)$, 5
%it follows that the fixed point lies in $\clo(S_{\tilde E})$. \marginpar{We will change Lemma \ref{lem:example} to lens....} 
%(See Lemma \ref{lem:example}.)

%For a horospherical end vertex $v$, only $v$ is the fixed point in $\Bd \tilde{\mathcal O}$ similarly. 

%If $g$ fixes a point of the closure $U$ of a pseudo-reduced ends, then we take a join neighborhood $N$. 
%Then by a geometric consideration of a horoball side, $g(N)^o \cap N \ne \emp$. This means $N=g(N)$ and $g$ is in 
%the associated end fundamental group. 

%%% March 4th.. 9;53 pm There is a problem here...
Suppose that an element $g$ of $\pi_1(\mathcal{O})$ is not homotopic to 
any element of a p-end fundamental subgroup.
Then by above, $g$ does not fix any of the above types of points. 
We show that $g$ has exactly two fixed points in $\Bd \torb$. 

Suppose that $g  \in \pi_1(\orb)$ fixes a unique point $x$ in the closure of $\Bd \tilde{\mathcal{O}}$ 
and $x$ is not in the closure of p-end neighborhoods as above.
Then $x$ is a $C^1$-point by the strict convexity. 
Suppose that we have two eigenvalues with largest absolute values 
$>1$ and the smallest one with $< 1$. 
If the eigenvalue is not positive real, $\clo(\tilde{\mathcal{O}})$ contains a nonproperly convex subset
as we can see by an action of $g^n$ on a generic point of $\torb$. 
Thus, we obtain attracting and repelling subspaces easily with these and there are at least 
two fixed point. This is a contradiction. 
Therefore, $g$ has only eigenvalues of unit norms. 

Take a line $l(t)$ converging to $x$ as $t \ra 0$ where $l(t)$ is a projective function of $t \in \bR$. 
Then a two-dimensional subset $P$ contains $l(t)$ and $g(l(t))$ for all $t$. 

%Since $g$ acts on the supporting half-space of $\torb$ at $x$ given by $x_1 > 0$ for a projective coordinate function $x_1$,  
%and$g$ has only norms of eigenvalue equal to $1$, we obtain $x_1 \circ g = x_1$. 

We may assume that $\torb$ is a precompact subset of an affine subspace $H$ and $x$ is the origin for some affine coordinate system of $H$. 
Find a sequence of dilatation $s_r$ fixing $x$ and acting on $H$, 
i.e., sending vectors $v$ to $r v$ for $r > 0$ in the vector space $H$ with the origin $O$ identified as $x$. 
We now use a projective coordinate function $x_1$ on $H$ so that $x_1(l(t))=t$. 
$g$ acts on a hyperspace $L$ supporting  $\clo(\tilde{\mathcal{O}})$ at $x$. 

We choose an affine coordinate 
on $H$ where $x$ is the origin and $L$ has value $0$ for a coordinate function.
Then $s_r$ acts on $P$,  and two points $s_r(l(1/r))$ and $s_r g(l(1/r))= s_r g s_r^{-1} (s_r(l(1/r)))$ 
is connected by line $l_r$. 
Since $s_r g s_r^{-1}$ converges to a linear map in the compact open topology of $H$, the differential of $g$, 
 with eigenvalues of norm $1$ only,  the slope of $l_r$ converges to $0$ as
 $r \ra \infty$.  (One can use an easy estimation.)
 
 Since $x$ is a $C^1$-point, $s_r(\torb)$ converges to the half space $H$ as $r \ra \infty$. 
Since $P \cap \clo(\tilde{\mathcal{O}})$ is a convex subset, and $s_r(P \cap \Bd(\tilde{\mathcal{O}}))$ 
geometrically converges to $H$, $s_r l(1/r)$ and $s_r g l(1/r))$ are converging to an interior of $H$, 
it follows that the Hilbert distance
\begin{equation}\label{eqn:sr} 
d_{\torb}(l(1/r), g(l(1/r))) = d_{s_r(\torb)}(s_r l(1/r), s_r g l(1/r)) \ra 0 \hbox{ as } r \ra \infty
\end{equation}
by equation \eqref{eqn:HiHa}. 
%If $C$ is a proper cone of the half space, then we

%Suppose that $d(l(t), g(l(t))$ is bounded below by a small positive constant. 
Then drawing a segment $s(t)$ between $l(t)$ and $g(l(t))$, we obtain a closed circle in $\orb$ in 
the homotopy class corresponding to $g$. That is, $s(t_i)$ maps closed curves $c(t_i)$.
Since the Hilbert length of $c(t_i)$  as $i \ra \infty$ goes to zero, and there is a uniform lower bound on 
the non-nullhomotopic closed curve lengths in the complement of the union of end neighborhoods, 
$c(t_i)$ has to be inside an end neighborhood of $\orb$ for sufficiently large $i$. 
In this case $g$ is in a p-end fundamental group, $x$ is in the closure of a p-end neighborhood.
This is a contradiction. 
%\marginpar{Mension similar things in Cooper, and Crampon...}

We conclude that $g\in \pi_1(\orb)$ fixes at least two points $a$ and $r$ in $\Bd \tilde{\mathcal{O}}$. 
We choose the two fixed points to have the positive real eigenvalues that are largest and smallest absolute values of 
the eigenvalues of $g$. (As above, the largest and smallest norm eigenvalues must be positive for $\torb$ to be properly convex.) 

No fixed point of $g$ in $\Bd \torb$ is in the closures of p-end neighborhoods. 
By strict convexity, the interior of $\torb$ contains an open line segment $l$ connecting $a$ and $r$. 

%since otherwise we can then form a segment $s$ in $\Bd \tilde{\mathcal{O}}$ containing one of the end points in $l$ 
%by acting on a segment in $\torb$ meeting $l$ transversally. This contradicts the strict convexity since the end point of $l$ are both not
%in closures p-end neighborhood. 

%In fact, any invariant line in the interior of $\tilde{\mathcal{O}}$ connects the fixed points corresponding 
%to largest and smallest positive real eigenvalues. 
Let $S$ denote the subspace spanned by $a, r, t$. 
Suppose that there is a third fixed point $t$ in $\Bd \tilde{\mathcal{O}}$.
It is not in the closures of p-end neighborhoods as we assumed that $g$ is not in the p-end fundamental group. 
Then the line segment connecting it to the $a$ or $r$ must be 
in $\Bd \tilde{\mathcal{O}}$: otherwise, we can form a segment $s$ in $\torb \cap S$ transversal to the segment.
Then $\{g^k(s)\}$ geometrically converges to a segment in $\Bd \torb$ containing $a$ or $r$ as $k \ra \infty$ or $k \ra -\infty$
by the properness of the action. 
Thus, the existence of $t$ contradicts the strict SPC-property. 

Hence, there are exactly two 
fixed points of $g$ in $\Bd \tilde{\mathcal{O}}$ of the  positive real eigenvalues that are largest and smallest absolute values of 
the eigenvalues of $g$.

\end{proof}

%%% May 16 5:28pm 

\begin{proposition} \label{prop:loxopara2}
Suppose that $\mathcal O$ is a noncompact strongly tame strict SPC-orbifold  with generalized admissible ends %lens-type radial end, totally geodesic 
%lens-type ends, or horospherical end.
%Each abelian subgroup of $\pi_1(\orb)$ of rank $\geq 2$ is in a p-end fundamental group.
Let $\tilde E$ be an end. 
Then for a p-end $\tilde E$, $(\Bd \torb - K)/\pi_1(\tilde E)$ is compact where
$K= \bigcup S(\tilde E)$ for radial p-end $\tilde E$ of lens-type,  
$K= \clo(S_{\tilde E})$ for totally geodesic p-end $\tilde E$, or
 $K= \{v_{\tilde E}\}$ for horospherical p-end $\tilde E$.
\end{proposition}
\begin{proof} 
Suppose that $\tilde E$ is radial type of lens or horospherical type.
By Lemma \ref{lem:hor}, the homeomorphism $\Phi_{\tilde E}: S_{\tilde E} \ra \Bd \torb - K$
gives us the result.   

Suppose that $\tilde E$ is a totally geodesic p-end of lens-type. Let $\torb^*$ denote the dual domain.
%$\Gamma^*$ the dual group of $\Gamma = \pi_1(\orb)$. 
Then there exists a dual radial p-end $\tilde E^*$ corresponding to $\tilde E$.
Hence, $(\Bd \torb^* - K')/\pi_1(\tilde E^*)$ is compact for $K'$ equal to the closure of p-end neighborhoods of $\tilde E^*$ in the radial case or 
the vertex in the horospherical case.

Recall Proposition \ref{prop:duality}. 
Let $\Bd^{\Ag} \torb$ be the augmented boundary with the fibration $\Pi_{\Ag}$, and 
let $\Bd^{\Ag} \torb^*$ be the augmented boundary with the fibration map $\Pi^*_{\Ag}$.
Let $K'':= \Pi_{\Ag}^{-1}(K)$ and $K''':= \Pi_{\Ag}^{\ast -1}(K')$. 
There is a duality homeomorphism by Proposition \ref{prop:duality}
\[ \mathcal{D}: \Bd^{\Ag} \torb - K''\ra \Bd^{\Ag} \torb^* - K'''.\]

Now $(\Bd^{\Ag} \torb^* - K''')/\pi_1(\tilde E^*)$ is compact since $\Bd \torb^*  - K'$ has a compact fundamental domain 
and the space is the inverse image in $\Bd^{\Ag} \torb^*$. 
By (iv) of Proposition \ref{prop:duality}, $(\Bd^{\Ag} \torb - K'')/\pi_1(\tilde E)$ is compact also. 
Since the image of this set under the map induced by a proper map $\Pi_{\Ag}$ is $(\Bd \torb - K)/\pi_1(\tilde E)$, it is is compact.
\end{proof}

We call the following construction {\em shaving the ends}. \index{end:shaving} 

\begin{proposition}\label{prop:remconch} 
Given a noncompact strongly tame SPC-orbifold $\mathcal{O}$ and its universal cover $\tilde{\mathcal O}$, 
there exists a collection of mutually disjoint open concave p-end neighborhoods for p-ends 
of lens-type.
%Remove the union of a $\pi_1(\orb)$-equivariant collection of 
%some these with hyperbolic end fundamental groups from $\tilde{\mathcal O}$.
We remove a finite union of concave end-neighborhoods of some radial ends. 
Then 
\begin{itemize}
\item we obtain a convex domain as the universal 
cover of a strongly tame orbifold ${\mathcal{O}}_1$ with additional strictly convex smooth boundary components
that are closed $(n-1)$-dimensional orbifolds.
\item Furthermore, if $\mathcal{O}$ is strictly SPC with respect 
to all of its ends, and we remove only hyperbolic ends, then ${\mathcal{O}}_1$ is strictly SPC with respect to the remaining ends. 
\end{itemize} 
\end{proposition}
\begin{proof} 
If $\mathcal{O}_1$ is not convex, then there is a triangle $T$ in $\tilde{\mathcal{O}}_1$
with three segments $s_0, s_1, s_2$ so that $T-s_0^o \subset \tilde{\mathcal{O}}_1$ 
but $s_0^o - \tilde{\mathcal{O}}_1 \ne \emp$. (See Theorem A.2 of \cite{psconv} for details.)
Since $\tilde{\mathcal{O}}$ is an open manifold, $s_0^o - \tilde{\mathcal{O}}$ is a closed subset of
$s_0^o$. Then a boundary point of $x \in s_0^o - \tilde{\mathcal{O}}_1$ 
is in the boundary of one of the removed concave-open neighborhoods or is in $\Bd \tilde{\mathcal{O}}$ itself. 
The second possibility implies that 
$\mathcal{O}$ is not convex as $\torb_1 \subset \torb$. These are contradictions. 
The first possibility implies that there exists a segment in the interior of a concave p-end neighborhood $U$ with endpoints in $\Bd U \cap \torb$.
This is geometrically not possible. Also, since $\torb_1 \subset \torb$, we have the convexity. Since $\torb$ is properly convex, 
so is $\torb_1$. 

Now we go to the second part. 
We suppose that $\mathcal{O}$ is strictly SPC. 
Let $\mathcal{H}$ denote the set of p-end vertices with hyperbolic p-end fundamental groups 
that were removed in the equivariant manner. For each $p \in \mathcal{H}$, denote by $U_p$ the concave 
neighborhood removed. 

Any segment in the boundary of the developing image of $\mathcal O$ is a subset of 
the closure of a p-end neighborhood of a p-end vertex. 
For the p-end-vertex $p$ of a p-end $\tilde E$, the domain $R_p(\torb) \subset \SI^{n-1}_p$  is strictly convex
if $\pi_1(\tilde E)$ is hyperbolic. 
Since $\Bd R_p(\torb)$ contains no straight segment by hyperbolicity in \cite{Ben1}, 
only straight segments in $\clo(U) \cap \Bd \torb$ for the concave p-end neighborhood  $U$ of $\tilde E$
are in the segment in $S(p)$. 
Thus, their interiors are disjoint from $\Bd \tilde{\mathcal{O}}_1$, and hence 
$\Bd \torb_1$ contains no geodesic segment in $\bigcup_{p \in {\mathcal{H}}} \clo(U_p) \cap \Bd \torb$

Since we removed concave end neighborhoods of the lens-type ends with the hyperbolic end fundamental groups, any straight segment in 
$\Bd \tilde{\mathcal{O}}_1$ lies in the closure of a p-end neighborhood of a remaining p-end vertex. 

A non-$C^1$-point of $\Bd \torb_1$ is not on the boundary of the concave p-end neighborhood $U$ for a hyperbolic 
p-end $\tilde E$ nor in $\Bd \torb - \clo(U)$.
However,  $\clo(U) \cap \Bd \torb_1$ contains the limit set  $\Lambda = L - \partial L$
for the lens part $L$ in a lens-neighborhood. % i.e., $\Lambda = L - \partial L$. 
$\torb$ has the same set of supporting hyperplanes as $L$ at points of $\Lambda$
since they are both $\pi_1(\tilde E)$-invariant convex domains by Corollary \ref{cor:lenssupp}.
However, the supporting hyperplanes at $\Lambda$ of $L$ are also supporting ones for $\torb_1$ by Corollary \ref{cor:lenssupp}
since we removed the outside component $U$ of $\torb - L$. Thus, $\torb_1$ is $C^1$ at points of $\Lambda$. 
Since these are true for all removed concave p-end neighborhood $U$, 
$\mathcal{O}_1$ is strictly SPC.     %\marginpar{We need to add "with respect to" always.} 

\end{proof}

%\marginpar{The proof above must be carefully read again later}

\subsection{Bowditch's method}

There are results proved by Cooper, Long, and Tillman \cite{CLT3} and Crampon and Marquis \cite{CM} 
similar to below. However, the ends have to be horospherical in their work. 
We will use Bowditch's result \cite{Bowditch} to show 

\begin{theorem}\label{thm:relhyp}
Let $\mathcal{O}$ be a noncompact strongly tame strict SPC-orbifold with generalized admissible ends % lens-type radial end, totally geodesic 
%lens-type ends, or horospherical end
$E_1, \dots, E_k$ and satisfies {\rm (IE)} and {\rm (NA)}. 
Assume $\partial \orb =\emp$. 
% and 
%every abelian subgroup of $\pi_1(\orb)$ of rank $\geq 2$ is contained in 
%the end fundamental group. %Assume that the end fundamental groups have holonomies 
%realized by lens-type ends or horospherical ends. 
%Then 
%$\pi_1(\mathcal{O})$ is a relatively hyperbolic group with respect to 
%admissible end groups $\pi_1(E_1), ..., \pi_1(E_k)$ where $1, ..., k$ are so that 
%each $E_i$ is horospherical or  lens-shaped. 
Let $\tilde U_i$ be the inverse image $U_i$ in $\torb$ for a mutually disjoint collection of
 neighborhoods $U_i$ of the ends $E_i$ for each $i=1, \dots, k$. 
Then 
\begin{itemize} 
\item $\pi_1(\mathcal{O})$ is relatively hyperbolic with respect to 
the end fundamental groups \[\pi_1(E_1), ..., \pi_1(E_k).\] 
Hence $\orb$ is relatively hyperbolic with 
respect to $U_1 \cup \cdots \cup  U_k$. 
\item If $\pi_1(E_{l+1}), .., \pi_1(E_k)$ are hyperbolic for some $1 \leq  l \leq k$
{\rm (}possibly some of the hyperbolic ones{\rm )},
then $\pi_1(\mathcal{O})$ is relatively hyperbolic with respect to the end fundamental group 
$\pi_1(E_1), \dots, \pi_1(E_{l})$. 
\end{itemize}
\end{theorem}
\begin{proof}
We show that $\pi_1(\mathcal{O})$ is relatively hyperbolic with respect to 
the end fundamental groups $\pi_1(E_1), ..., \pi_1(E_k)$. 
%Since the group is roughly isometric to $\tilde{\mathcal{O}}$, 
By Proposition \ref{prop:remconch}, we have the second statement.

Choose a collection of $\pi_1(\orb)$-invariant concave p-end neighborhoods of $\torb$ whose union covers $\bigcup_{i=l+1}^k U_i$
for concave end neighborhoods $U_{l+1}, \dots, U_k$. 
We use $\tilde{\mathcal{O}}$ and remove the union of a collection of $\pi_1(\orb)$-invariant concave p-end neighborhoods of $\torb$
by Proposition \ref{prop:remconch}.
Now, $\torb_1$ covers a strict SPC-orbifold $\orb_1$ with admissible end.

Thus, we obtain $\Bd \tilde{\mathcal{O}}_1$. 
%(This corresponds to removing some collection of concave end neighborhoods $U_1, \dots, U_l$ from $\orb$. )
\begin{itemize}
\item We now collapse each set of form $\clo(U_i) \cap \Bd \torb_1 = S_{\tilde E}$  for a concave p-end neighborhood $U_i$ 
to a point and 
\item collapse $\clo(S_{\tilde E})$ for each totally geodesic end $\tilde E$ of lens-type to a point.
\end{itemize} 
By Corollary \ref{cor:disjclosure}, these sets are mutually disjoint balls
Let ${\mathcal C}_B$ denote the collection, and let $C_B:= \bigcup {\mathcal C}_B$. 

We claim that for each closed set $J$ in $\Bd \torb_1$, the union of $C_J$ of 
elements of ${\mathcal C}_B$ meeting $J$ is also closed: 
Let us choose a sequence $\{x_i \}$ for $x_i \in C_i$, $C_i \cap J \ne \emp$, $C_i \in {\mathcal C}_B$. 
Suppose that $x_i \ra x$. Let $y_i \in C_i \cap J$. Let $v_i$ be the end vertex of $C_i$ if it is radial. 
Then define $s_i:= \ovl{x_i v_i}\cup \ovl{v_i, y_i} \subset C_i$ if $C_i$ is radial 
or else $s_i:= \ovl{x_iy_i} \subset C_i$. Choose a subsequences so that 
$\{s_i\}$ geometrically converges to a limit containing $x$. 
The limit $s_\infty$ is a singleton, a segment or a union of two segments. 
By the strict convexity of $\torb$, we obtain $s_\infty$ is a subset of an element of ${\mathcal C}_B$
and $s_{\infty}$ meets $J$. Thus, $x \in s_{\infty} \subset C_i$ for $C_i \cap J$, and $C_J$ is closed. 

We denote this quotient space  $\Bd \tilde{\mathcal{O}}_1/\sim$ by $B$.
By Proposition \ref{prop:metriz}, $B$ is a metrizable space. 

%$B$ is homeomorphic to a sphere again since the components are disjoint 
%closed balls by Corollary \ref{cor:disjclosure}. (For the disjointness, we needed the condition of no-essential annuli.)

We show that $\pi_1(\mathcal{O})$ acts on the metrizable space $B$ 
as a geometrically finite convergence group. 
By Theorem 0.1 of Yaman \cite{Yaman} following Bowditch \cite{Bowditch}, this shows that  \index{relative hyperbolicity} 
$\pi_1(\orb)$ is relatively hyperbolic with respect to $\pi_1(E_1), \dots, \pi_1(E_k)$.
The definition of conical limit points and so on are from the article.

(I) We first show that the group acts properly discontinuously on the set of triples in $B=\partial \tilde{\mathcal{O}}_1/\sim$. 
Suppose not. Then there exists a sequence of nondegenerate triples $\{(p_i, q_i, r_i)\}$ of points in $\Bd \torb_1$ 
converging to a distinct triple $\{(p, q, r)\}$ so that $p_i = \gamma_i(p_0), q_i = \gamma_i(q_0),$ and $r_i= \gamma_i(r_0)$
where $\gamma_i$ is a sequence of mutually distinct elements of $\pi_1(\mathcal{O})$.
We assume here that $p_0$, $q_0$, and $r_0$ are representatives of distinct points of $B$
and so are $p, q,$ and $r$. 
By multiplying by some uniformly bounded element $R_i$ in $\PGL(n+1, \bR)$, we obtain 
that $R_i\circ \gamma_i$ for each $i$ fixes $p_0, q_0, r_0$ and restricts to a diagonal matrix with 
entries $\lambda_i, \delta_i, \mu_i$ on the plane with coordinates so that $p_0= e_1, q_0=e_2, r_0=e_3$.

Then we can assume that \[\lambda_i\delta_i\mu_i =1, \lambda_i \geq \delta_i \geq \mu_i > 0\] by restricting to the plane and 
up to choosing subsequences and renaming. %we can assume that $\lambda_i \geq \delta_i \geq \mu_i > 0$. 
Thus $\lambda_i \ra \infty$ and $\mu_i \ra 0$ since otherwise these two sequences are bounded and 
this contradicts the discreteness of the holonomy homomorphism.

Let $P_0$ denote the $2$-dimensional subspace containing $p_0, q_0$, and $r_0$. 
By strictness of convexity, as we collapsed each of the p-end balls, 
the interiors of the segments $\ovl{p_0q_0}$, $\ovl{q_0r_0}$, and $\ovl{r_0p_0}$ 
are in the interior of $\tilde{\mathcal{O}}_1$. 

We claim that one of the sequence  $\lambda_i/\delta_i$ or the sequence $\delta_i/\mu_i$ are bounded:
Suppose not. %Then these sequences are unbounded. 
Then $\lambda_i/\delta_i \ra \infty$ and $\delta_i/\mu_i \ra \infty$. 
We choose generic segments $s_0$ and $t_0$ in $\torb_1$
with a common end point $q_0$ and the respective other end point $\hat s_0$ and $\hat t_0$ 
so that \[\bdd(\hat s_0, q_0), \bdd(\hat t_0, q_0) \geq \delta \hbox{ for a uniform } \delta > 0.\] 
We choose $s_0$ and $t_0$ so that their directions from $q_0$ differ from that of $\ovl{p_0q_0}$ and $\ovl{q_0r_0}$
at least by $\delta' > 0$. 
Then the sequence $R_i \circ \gamma_i(s_0\cup t_0)$ converges to the segment with end point $p_0$ passing $q_0$ in the middle geometrically.
The segment is a great segment. 
%We see that there exists a segment of length $\pi$ in $\clo(\torb)$. 
(See Section \ref{subsec:metrics}.)
Since $R_i$ is bounded, this implies that there exists such a segment in $\clo(\torb_1)$. 
This is a contradiction to the proper convexity of $\torb_1$. 

Suppose now that the sequence $\lambda_i/\delta_i$ is bounded: Now the sequence of segments $\ovl{p_iq_i}$ 
converges to $\ovl{pq}$ whose interior is in $\tilde{\mathcal{O}}_1$. 
Then we see that $\ovl{pq}$ must be in the boundary as these points must 
be the limit points of points of sequence of $\gamma_i(s)$ for some compact subsegments 
$s \subset \ovl{p_0q_0}^o$
by the boundedness of the above ratio and the proper-discontinuity of the action. 
This contradicts the strict convexity as we assumed that 
$p, q,$ and $r$ represent distinct points in $B$. 
If we assume that $\delta_i/\mu_i$ is bounded, then 
we obtain a contradiction similarly.

This proves the proper discontinuity of the action on the space of distinct triples. 

%%% May 16 6:30pm 2014
%An end group $\Gamma_x$ for end vertex $x$ is a parabolic subgroup fixing $x$ 
%since the elements in $\Gamma_x$ fixes only the contracted set in $B$ and
%since there are no essential annuli. 

(II) By Propositions \ref{prop:loxopara} and \ref{prop:loxopara2}, 
each group of form $\Gamma_x$ for a point $x$ of $B=\tilde{\mathcal{O}}_1/\sim$ is a bounded parabolic subgroup
in the sense of Bowditch \cite{Yaman}. 

%Also, 
%$B-\{x\}$ is homeomorphic to $\tilde E$ since each maximal segment form $x$ 
%meets $\Bd \tilde{\mathcal{O}}_1 - \clo(U)$ for the end neighborhood $U$ of $x$
%exactly once and the collapsing does not affect the topology.
%$(B - \{x\})/\Gamma_x$ is easily seen to be homeomorphic to the end orbifold and 
%therefore, compact. Hence, $\Gamma_x$ are the bounded parabolic subgroups. 

(III) Let $p \in \Bd \torb$ be a point that is not in a horospherical endpoint or a singleton 
corresponding an lens-shaped p-end of radial or totally geodesic type of $B$. 
%or a pseudo-reduced
%p-ends of $B$. 
We show that $p$ is a conical limit point. This will complete our proof by Theorem 0.1 of \cite{Yaman}. 

We find a sequence of holonomy transformations $\gamma_i$ and distinct points $a, b \in \partial B$ 
so that $\gamma_i(p) \ra a$ and $\gamma_i(q) \ra b$ locally uniformly for $q \in \partial B - \{p\}$. 
To do this, we draw a line $l(t) $ in $\torb$ from a point of the fundamental domain to $p$
where as $t \ra \infty$, $l(t) \ra p$ in the compactification. 
Since $l(t)$ is not eventually in a p-end neighborhood, there is a sequence $\{t_i\}$ going to $\infty$ 
so that $l(t_i)$ is not in any of the p-end neighborhoods in $\tilde U_{1} \cup \cdots \cup  \tilde U_{k}$. 
Let $p'$ be the other endpoint of the complete extension of $l(t)$ in $\tilde{\mathcal{O}}$. 
We can assume without generality that 
$p'$ is not in the closure of any p-end neighborhood by choosing the line $l(t)$ differently if necessary. 

Since $(\tilde{\mathcal{O}}_1- \tilde U_1 - \cdots - \tilde U_k)/\Gamma$ is compact, 
we have a compact fundamental domain $F$  of 
$\tilde{\mathcal{O}}_1 - \tilde U_1 - \cdots - \tilde U_k$ with respect to $\Gamma$.
Note that $\bdd(F, \Bd \torb) > C_0$ for some constant $C_0 > 0$. (We remark that this is note not true for $\Bd \torb_1$.) 

Now we will look at the convex open domain $\torb$ and use the Hilbert metric $d_{\torb}$. 
We find points $z_i \in F$ so that $\gamma_i(l(t_i))= z_i$ for a deck transformation $\gamma_i$.
Then by taking subsequences, we assume without loss of generality that $\gamma_i^{-1}(p) \ra a$ for a point $a$ and
$\gamma_i^{-1}(p')$ also converges to another point $b$, $b \ne a$. 
Take a point $q \in X - \{p, p' \}$ and find a geodesic $m$ from $q$ to $l$ 
with the property that every point on $m$ a shortest geodesic from the point to $l$ lies in $m$ by Proposition \ref{prop:shortg}. 
Let $q'$ be the intersection of $m$ to $l$. 
Then $\gamma_i^{-1}(q')$ converges to $b$ by a Hilbert metric consideration. 

The sequence of segments $\gamma^{-1}_i(\ovl{qq'})$ has the property that 
every point on it is a shortest geodesic from the point to $\gamma^{-1}_i(l)$ lies in $\gamma^{-1}_i(\ovl{qq'})$.
(See Figure \ref{fig:figbow}.)
%%% Avril 3 12:17pm
Because of this property, for any sequence of points $x_i \in \gamma^{-1}_i(\ovl{qq'})$, we
have that the shortest geodesic from $x_i$ to $l_i:=\gamma^{-1}_i(l)$ lie in $\gamma^{-1}_i(\ovl{qq'})$. %\marginpar{$\torb_1$ or $\torb$ careful}

We show that the sequence $\gamma^{-1}_i(\ovl{qq'})$ is exiting; that is, 
for every compact subset $K$ of $\torb_1$, there exists $I_1$ such that 
$\gamma^{-1}_i(\ovl{qq'}) \cap K = \emp$ for $i > I_1$.
Suppose not. 
Then we can choose $x_i \in \gamma^{-1}_i(\ovl{qq'})$
and $x_i \ra x \in \tilde{\mathcal{O}}$
so that the corresponding sequence of $d_{\torb}$-distances in $\gamma^{-1}_i(\ovl{qq'})$ is converging to $\infty$.
However, as $x_i \ra x \in \tilde{\mathcal{O}}$, the sequence $d_{\torb}(x_i, l_i)$ is bounded 
since $l_i$ passes $F$ for all $i$.
Since $d_{\torb}(x_i, l_i)$ is the arclength from $x_i$ to the point of $l_i$ in $\gamma^{-1}_i(\ovl{qq'})$, 
we have $d_{\torb}(x_i, l_i) \ra \infty$. This is contradiction.
%Therefore, for every compact subset $K \subset \torb_1$, there exists $I_1$ such that
%$\gamma^{-1}_i(\ovl{qq'}) \cap K = \emp$ for $i > I_1$. 
%That is $\{\gamma^{-1}(\ovl{qq'})\}$ is an exiting sequence for any compact subset of $\tilde{\mathcal{O}}$. 
(By choosing a continuous parameters of shortest geodesics from a small neighborhood of $q$, 
we obtain a local uniform convergence.)

By choosing a subsequence, the sequence $\{\gamma^{-1}_i(\ovl{qq'})\}$ converges to a point or a segment in the boundary 
$\Bd \tilde{\mathcal{O}}_1$. If the limit is a segment, then it is contained in one of the collapsed subsets containing $b$
since $\torb$ is strictly convex relative to p-ends.
This shows that $p$ is a conical limit point. 

%For the final item, the hypothesis implies that 
%we remove the concave p-end neighborhoods of selected hyperbolic p-ends 
%and obtain tame orbifold $\mathcal{O}_2$ with generalized lens-type radial end, totally geodesic 
%lens-type ends, or horospherical end $E_1, \dots, E_l$. 
%Then by Proposition \ref{prop:remconch},  $\mathcal{O}_2$ is strictly SPC and strongly tame
%with respect to the remaining ends  $E_1, \dots, E_l$. Hence, the above proof still applies.
%(Note that we can use Theorem \ref{thm:drutu} as well.)

\end{proof}

\begin{figure}[h]
\centerline{\includegraphics[height=7cm]{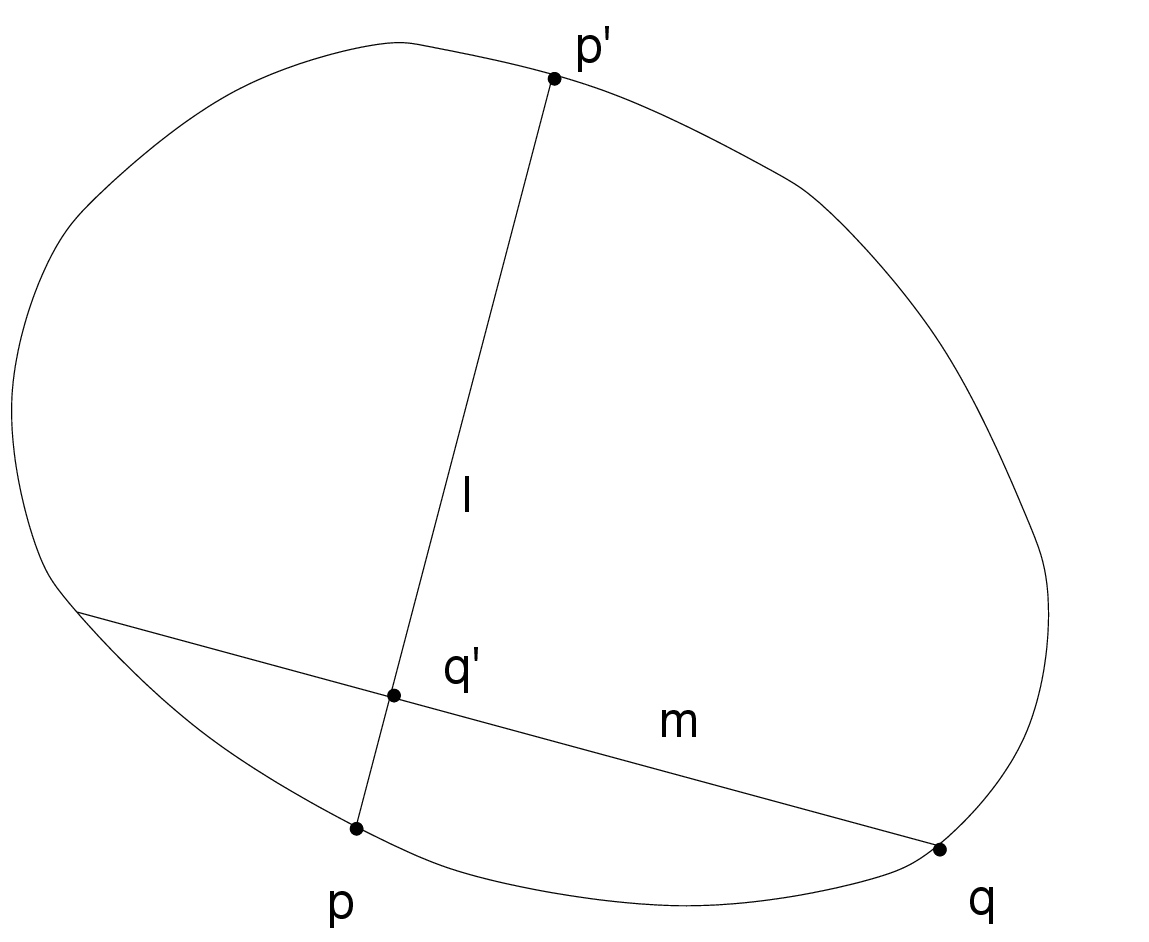}}
\caption{The shortest geodesic $m$ to a geodesic $l$.}
\label{fig:figbow}
\end{figure}

%%% August 20, 3:21pm

\subsection{The Theorem of Drutu}

The author obtained a proof of the following theorem from Drutu. 
See \cite{Du} for more details.  

\begin{theorem}[Drutu] \label{thm:drutu}
Let $\mathcal{O}$ be a noncompact strongly tame orbifold with admissible and satisfies {\rm (IE)} and {\rm (NA)}.
Let  $\pi_1(E_1), ..., \pi_1(E_m)$ be end fundamental groups where $\pi_1(E_{n+1}), ..., \pi_1(E_m)$ for $n\leq m$ 
are hyperbolic groups.
Then $\pi_1(\mathcal{O})$ is a relatively hyperbolic group with respect to 
$\pi_1(E_1), ..., \pi_1(E_m)$ if and only if 
$\pi_1(\mathcal{O})$ is one with respect to 
 $\pi_1(E_1), ..., \pi_1(E_n)$.
\end{theorem} 
\begin{proof} 
With the terminology in the paper \cite{Du},   $\pi_1(\mathcal{O})$ is a relatively hyperbolic group with respect to 
the admissible end fundamental groups $\pi_1(E_1), ..., \pi_1(E_m)$ if and only if 
 $\pi_1(\mathcal{O})$ with a word metric is asymptotically tree graded (ATG) with respect to 
 all the left cosets $g\pi_1(E_i)$ for $g \in \pi_1(\mathcal{O})$ and $i=1,.., m$. 

We claimed that 
 $\pi_1(\mathcal{O})$ with a word metric is asymptotically tree graded (ATG) with respect to 
 all the left cosets $g\pi_1(E_i)$ for $g \in \pi_1(\mathcal{O})$ and $i=1,.., m$ if and only if
$\pi_1(\mathcal{O})$ with a word metric is asymptotically tree graded with respect to 
 all the left cosets $g\pi_1(E_i)$ for $g \in \pi_1(\mathcal{O})$ and $i=1,.., n$. 
 
 Conditions ($\alpha_1$) and ($\alpha_2$) of Theorem 4.9 in \cite{Du} are satisfied still when we drop end fundamental groups
  $\pi_1(E_{n+1}), ..., \pi_1(E_m)$ or add  them. % $\pi_1(E_{n+1}), ..., \pi_1(E_m)$ 
  (See also Theorem 4.22 in \cite{Du}.)
  
  %%% March 8th 1:00 pm
  For the condition $(\alpha_3)$ of Theorem 4.9 of \cite{Du}, it is sufficient to consider only hexagons.
 According to Proposition 4.24 of \cite{Du} one can take the fatness constants as large
as one wants, in particular $\theta$ (measuring how fat the hexagon is) much larger than $\chi$ prescribing
how close the fat hexagon is from a left coset.

 If $\theta$ is very large, left cosets containing such hexagons in their
neighborhoods can never be cosets of hyperbolic subgroups 
since hyperbolic groups do not contain fat
hexagons. So the condition $(\alpha_3)$ is satisfied too whether one adds  $\pi_1(E_{n+1}), ..., \pi_1(E_m)$
or drop them.

\end{proof} 

%%% May 3 9:36pm 2014

\subsection{Converse}

We will prove the converse to Theorem \ref{thm:relhyp}:

For each totally geodesic end of lens-type, we add the totally geodesic ideal boundary component 
and outside open lens-neighborhood to form $\orb'$ by Theorem \ref{thm:totgeoext}.  
%Let $\torb^e$ denote the universal covering domain corresponding to it. We say $\torb^e$ to be an {\em extended universal cover}. 
%From now on, we let $\torb = \torb^e$.  
%Now $S_{\tilde E}$ for totally geodesic p-end $\tilde E$ is in $\torb$. 
We need this for a technical reasons.

Note here that we do not assume the ``full" admissibility of the ends. 
The group property will imply that the ends have to be admissible as well.
\begin{theorem} \label{thm:converse}
Let $\mathcal{O}$ be a noncompact strongly tame properly convex real projective orbifold with generalized admissible ends 
% horospherical or generalized lens-type radial ends 
%or totally geodesic ends  of lens type 
and satisfies {\rm (IE)} and {\rm (NA)}.  Assume $\partial \orb =\emp$. 
Suppose that $\pi_1(\mathcal{O})$ is a relatively hyperbolic group with respect to 
the admissible end groups $\pi_1(E_1), ..., \pi_1(E_k)$ where $E_i$ are horospherical 
for $i=1, ..., m$ and  generalized lens-shaped for $i=m+1, ..., k$ for $0 \leq m \leq k$.  
Then $\mathcal{O}$ is strictly SPC with respect to the admissible ends $E_1, \dots, E_k$. 
%Let $\sim$ denote the equivalence class of $\tilde{\mathcal{O}}$ collapsing each component of 
%ends to a point. 
%Equivalently suppose that the quotient space 
%$\tilde{\mathcal{O}}/\sim$ is a $\delta$-hyperbolic space for some $\delta > 0$. 
%\begin{itemize}
%\item 
%\item Then again if $\mathcal{O}$ is SPC, then $\mathcal{O}$ is strictly SPC with respect to the corresponding collection of admissible ends.
%\end{itemize}
\end{theorem}
\begin{proof} 
%{thm:totgeoext} 

Since an $\eps$-mc-p-end-neighborhood is always proper by Corollary \ref{cor:mcn}, for any $i$, 
we can choose the end neighborhood $U_i$ of any  generalized lens-type end $\tilde E_i$
to be the image of an $\eps$-mc-p-end-neighborhood. 
We can choose all such neighborhoods to be mutually disjoint by Corollary \ref{cor:mcn}. 
Let $\tilde U$ denote the union of the inverse images of end neighborhoods $U_1, ..., U_k$. 
%We define $\tilde U'$ to be the developing image 
%of the union of inverse images of $U_1, \dots, U_k$ in $\torb$.
%We will define $\tilde U^h$ be the union $\tilde U_1 \cup \cdots \cup \tilde U_n$.
%Let $U_{m+1}, .., U_k$ be the lens-shaped ends. 
%From $\orb$, we remove a concave neighborhood from $U_i$ for each radial end $E_i$, $i > m$, and we can let 
%$S_i$ be the boundary of the concave neighborhood 
%and we remove the domain outside $S_{\tilde E}$ in the lens-domain for each totally geodesic end $\tilde E$ of lens-type
%to form $\orb_2$ with the universal cover $\torb_2$.
%(Actually it is no longer a neighborhood of the end now.)
%Then the inverse image $\tilde U_i$ is a union of boundary components. 
%We assume this from now on. 
%We will denote by $\Omega_1$ the open domain $\tilde  {\mathcal{O}}_1$.
%We will use the Hilbert metric $d_{\torb}$ of $\tilde{\mathcal{O}}^e$ restricted to $\tilde {\mathcal{O}}$ from now on. 
%Also, $U_i$ for $i=n+1, \dots, k$ denotes $\clo(U_i)\cap \tilde  {\mathcal{O}}_2$ from now on. 
%The quotient is a compact convex surface in ${\mathcal{O}}_2$. 
%Note that the boundary of the convex hull of lens-shaped neighborhoods are now contained in lens parts of the lens-cone neighborhoods. 
%The set of lenses does not cover $\torb$ or $\tilde {\mathcal{O}}_2$ or $\tilde{\mathcal{O}}_2 - \tilde U$. 

Suppose that $\mathcal{O}$ is not strictly convex. 
We divide into two cases: 
First, we assume that there exists a segment in $\Bd \torb$ not contained in 
the closure of a p-end neighborhood. Second, we assume that there exists a non-$C^1$-point in $\Bd \torb$ not 
contained in the closure of a p-end neighborhood. 

(I) We assume the first case now.
We will obtain a triangle with boundary in $\Bd \tilde{\mathcal{O}}$ and not contained in the convex hull of p-ends:
Let $l$ be a nontrivial maximal segment in $\Bd \tilde{\mathcal{O}}$ not contained in the closure of a p-end neighborhood. 
First, $l$ does not meet the closure of a horospherical p-end neighborhood by Proposition \ref{prop:affinehoro}. 
By Theorems \ref{thm:lensclass} and \ref{thm:redtot} if $l^o$ meets the closure of 
a lens-shaped p-end neighborhood, then $l^o$ is in the closure. 
Also, suppose that $l^o$ meets $S_{\tilde E}$ for a totally geodesic p-end $\tilde E$. Then $l^o \cap \Bd S_{\tilde E} \ne \emp$. 
Then applying $\pi_1(\tilde E)$, we obtain a great segment in $\Bd \torb$ since $S_{\tilde E}/\pi_1(\tilde E)$ 
is a compact properly convex orbifold. (See \cite{Ben1}.) 
Therefore, $l$ meets the closures of p-end neighborhoods possibly only at its endpoints. 
%Hence, $l$ is maximal in $\Bd \torb$ as well. 

Let $P$ be a $2$-dimensional subspace containing $l$ and meeting $\tilde {\mathcal{O}}$ outside $\tilde U$.
%We define $\overline{\tilde{U'_P}}$ as the union of the closures of components of $\tilde U'\cap P$. 
By above, $l^o$ is in the boundary of $P \cap \tilde {\mathcal{O}}$. 
Draw two segments $s_1$ and $s_2$ in $P \cap \tilde {\mathcal{O}}$ 
from the end point of $l$ meeting at a vertex $p$ in the interior of $\tilde {\mathcal{O}}$. 
Since $l^o$ is not contained in the closure of a single component of $\tilde U$,
$(\torb - \tilde U) \cap P$ has a sequence of points $x_i$ converging to a point $x$ of $l^o$. 
%we can choose $x_i$ so that it lies outside $\tilde U$
%We choose a sequence of points $x_i$ on a line $m$ in $P$ converging to a point $x$ in the interior of $l$. 
Then $d_{\torb} (x_i, s_1\cup s_2) \ra \infty$ by consideration in the Hilbert metric by looking 
at all straight segment from $x_i$ to a point of $s_1$ or $s_2$ and the maximality of $l$ in $\Bd \torb$.

 %as only endpoints of $l$ are possibly in the closures of p-end neighborhoods. 
Recall that there is a compact fundamental domain $F$ of $\tilde {\mathcal{O}} -\tilde U$ under the action of $\pi_1(E)$.
%Since $l^o$ is not contained in any lens-shaped p-end neighborhood, we further require that $x_i$ is not
%in a single convex lens-shaped neighborhood. 
Now, we can take $x_i$ to the fundamental domain $F$ by $g_i$. 
We choose $g_i$ to be a sequence of mutually distinct elements of 
$\pi_1(\mathcal{O})$. % as $x_i$ is in a distinct images $g_i(F)$. 
We choose a subsequence so that we assume without loss of generality that 
$\{g_i(T)\}$ geometrically converges to a convex set, 
which could be a point or a segment or a nondegenerate 
triangle. Since $g_i(T) \cap F \ne \emp$, and the sequence $\partial g_i(T)$ exits any compact 
subsets of $\tilde {\mathcal{O}}$ always while $d_{\torb}(g_i(x_i), \partial g_i(T)) \ra \infty$
and $g_i(T)$ passes $F$, 
we see that a subsequence of $g_i(T)$ converges to a nondegenerate triangle, say $T_\infty$. 

By following Lemma \ref{lem:geogroup}, $T_\infty$ is so that $\partial T_\infty$ is in $\bigcup S(v_{\tilde E})$ 
for a radial generalized lens-type p-end $\tilde E$.  

%Since $T_\infty$ meets $\torb_2 - \tilde U^o$, it follows that 
%$\partial T_\infty$ cannot be a subset of $\clo(S_{\tilde E})$ for a totally geodesic ideal boundary 
%component $S_{\tilde E}$. 

Now, $T_\infty$ is so that $\partial T_\infty \subset \clo(U_1)$ for a p-end neighborhood $U_1$ of a generalized 
lens-type end $\tilde E$. Then for sufficiently small $\eps> 0$, the $\eps$-$d_{\torb}$-neighborhood of $T_\infty\cap \torb$ 
is a subset of $U_1$ as $U_1$ was chosen to be an $\eps$-mc-p-end-neighborhood
(see Lemma \ref{lem:mcc}).
However as $g_i(T) \ra T_\infty$ geometrically, 
for any compact subset $K$ of $\torb$, $g_i(T) \cap K$ is a subset of $U_1$ for sufficiently large $i$. 
But $g_i(T) \cap F \ne \emp$ for all $i$ and the compact fundamental domain $F$ of $\torb - \tilde U$, 
disjoint from $U_1$. This is a contradiction.

%Since $g_i(T) \cap F \ne \emp$, 
%We show that by choosing $x_i$ carefully, we can choose the limit $T_\infty$ 
%so that $\partial T_\infty$ is not inside the closure an mc-p-end-neighborhood $U_1$ of a p-end. 
%If $\partial T_\infty$ is in the closure of an mc-p-end-neighborhood, then $\partial T_\infty \subset \clo(U_1)$. 
%This implies that $m$ enters an $\eps$-mc-neighborhood inside $\tilde U$ for sufficiently small $\eps > 0$ 
%of the same end and not leave it. 
%However, $l$ and $m$ do not have this property. 

(II) Now we suppose that $\Bd \torb$ has a non-$C^1$-point $x$ outside the closures of p-end neighborhoods. 
Then we go to the dual $\torb^*$ and the dual group
$\Gamma^*$ where $\torb^*/\Gamma^*$ is a strongly tame properly convex orbifold with horospherical ends, totally geodesic ends of lens-type
or radial ends of generalized lens-type by Theorem \ref{thm:duality} and Lemma \ref{lem:concavedual}. 
%% Aug 20 10:07 correct this ... Write lemma

 Then we have a one-to-one correspondence of the set of p-ends of $\torb$ to the set of p-ends of $\torb^*$, 
and we obtain that $x$ corresponds to a convex subset of $\dim \geq 1$ in $\Bd \torb$ containing a segment $l$ 
not contained in the closure of p-end neighborhoods using the map $\mathcal{D}$ in Proposition \ref{prop:duality}. 
Thus, the proof reduces to the case (I). 
%We obtain a triangle $T \subset \clo(\torb^*)$ so that $\partial T_\infty$ is not in $S_{\tilde E}$ or $\bigcup S(v_{\tilde E})$ 
%for a p-end $\tilde E$.  By the following Lemma \ref{lem:geogroup}, this cannot happen.
%Since $\Gamma^*$ is isomorphics to with the isomorphic p-end fundamental groups
%$\pi_1(\tilde E_i^*)$, the group is relatively hyperbolic. 
%By following Lemma \ref{lem:geogroup}, this cannot happen.
%We conclude that $\mathcal O$ is strictly convex relative to the end neighborhoods. 

%%% Aug 20 9:51pm

%We claim that two boundary edges of $T_\infty$ are not in an end and thus, at most one boundary edge of $T_\infty$ is in an end: 
%Suppose that a lens-shaped cone part of the closure of an end contains the two edges. By convexity of the end, 
%the third edge is also in the same end and is in the boundary of the closure of the same end. This contradicts the above. 

%Finally, we show that at most one closure of the convex hull of an end can contain boundary edges of $T_\infty$: 
%Suppose that two ends contain two boundary edges of $T_\infty$. Then their disjoint end neighborhoods
%$N_1$ and $N_2$ contains the vertex $v$ of $T_\infty$. Then the distance goes to zero as we
%follow their boundnary components of $N_1$ and $N_2$ in $T_\infty$ toward $v$
%in the metric of the triangle. However, two disjoint ends have lower bound in the distances. 

(III) Finally, we show that the ends are admissible. 
Given a totally geodesic p-end $\tilde E$, it is of lens-type by assumption. 
A horospherical p-end $\tilde E$ is admissible. 
%$h(\pi_1(\tilde E))$ satisfies the uniform middle eigenvalue condition 
%by lens type condition on the holonomy of the end by the assumption.
%Hence, $\torb$ contains a lens-type one-sided p-end neighborhood of the ideal boundary component $S_{\tilde E}$ of $\tilde E$
%by Theorem 3.6 of \cite{endclass} and Lemma \ref{lem:shrink}. % \marginpar{need to consolidate the argument} 
%Hence, the admissibility of $\tilde E$ follows.

We will now study a generalized lens-type radial end $\tilde E$ and show that it is of lens-type as well: 
Given a radial end $\tilde E$, $h(\pi_1(\tilde E))$ satisfies the uniform middle eigenvalue condition 
by generalized lens type condition on the holonomy of the end by the assumption and Theorem 7.9 of \cite{endclass}.
Since the concave-end neighborhood exists by Corollary 8.7 in \cite{endclass}, 
we know that $\bigcup S(v_{\tilde E}) \subset \Bd \torb$. 
We obtain the compact convex hull $CH(\bigcup S(v_{\tilde E})) \subset \clo(\torb)$.  %\marginpar{A bit weak here...}

If $I:= \partial CH(\bigcup S(v_{\tilde E})) - \bigcup S(v_{\tilde E})$ is a subset of $\torb$, then 
%Lemma 8.7 and 
Proposition  7.6 of \cite{endclass} shows that if the lens neighborhood of $I$
is in $\torb$, then we are done. 

Suppose not. Then there exists an $i$-dimensional simplex $\sigma$, $\sigma^o \subset I$ for $i \geq 1$ so that 
$\Bd \torb \cap \sigma^o \ne \emp$. Then it must be that $\sigma \subset \Bd \torb$ by Lemma 7.5 of \cite{endclass}. 
%we can perturb a segment in $\sigma$ to $t$ in $\clo(\torb)$ so that $t$ meets $\Bd \torb$ only at a point of $t^o$. 
%This contradicts the convexity of $\torb$. (See Theorem A.2 of \cite{psconv} for example.)
Hence, there exists a segment $k$ in $I \cap \Bd \torb$ and in $\sigma$ with end points in $\bigcup S(v_{\tilde E})$. 
The two segments $s_1, s_2 \in S(v_{\tilde E})$ with endpoints equal to the endpoints of $l$, 
we obtain a triangle $T$ with $\partial T \subset \Bd \torb$ by the convexity of $\clo(\torb)$. 

By Lemma \ref{lem:geogroup}, $\partial T \subset \bigcup S(v_{\tilde E})$. However, by geometry, 
this implies $k, T \subset \bigcup S(v_{\tilde E})$. This contradicts $I$ being disjoint from $\bigcup S(v_{\tilde E})$.

By Theorem \ref{thm:sSPC}, we obtain that our orbifold is strictly SPC. 
\end{proof}

%For later purposes, we do not assume the admissibility for all ends for the following lemma. 

Now, we come to a lemma with a very long proof. 
\begin{lemma}\label{lem:geogroup} 
%Suppose that $\mathcal{O}$ is a noncompact strongly tame 
%properly convex $n$-orbifold with horospherical or ends, totally geodesic ends  of lens-type, or radial ends of generalized lens-type, 
%and every end fundamental group is of infinite index in $\pi_1(\mathcal{O})$.
%Assume that $\mathcal{O}$ has no essential tori or essential annuli, and 
%Let $\orb_1$ be constructed as above from $\orb$, and 
Assume the premise of Theorem \ref{thm:converse}.
%Suppose that 
%$\pi_1(\mathcal{O})$ is relatively hyperbolic with respect to the end fundamental groups.
Then for every triangle $T$ in $\tilde{\mathcal{O}}$ with $\partial T \subset \Bd \torb$, we obtain
$\partial T \subset \bigcup S(\tilde E)$ for a p-end $\tilde E$ of radial type.
%or $\partial T \subset \clo(S_{\tilde E})$ for a totally geodesic end $\tilde E$. 
\end{lemma}
\begin{proof}
%We will use the notations of the proofs in above Theorem \ref{thm:converse} but we do not know 
%the lens-shapedness of the nonhorospherical ends. We will not use $\mathcal{O}_2$ here or the concave end neighborhood.

We expand $\orb$ to $\orb^e$ by adding lens neighborhoods to totally geodesic ideal boundary components 
by Theorem \ref{thm:totgeoext}.  We will use the Hilbert metric of $\orb^e$ restricted to $\orb_1$. 
From now on, we will denote by $\orb$ the extended orbifold and $\torb$ will denote $\torb^e$ also. 
%Let $\torb^e$ denote the universal covering domain corresponding to it. We say $\torb^e$ to be an {\em extended universal cover}. 
%From now on, we let $\torb = \torb^e$.  
Now $S_{\tilde E}$ for totally geodesic p-end $\tilde E$ is in $\torb$. 

Suppose that some of the ends $E_i$, $i=1, \dots, k$, are hyperbolic. Then remove the concave end neighborhoods for these ends
to obtain $\orb_1$. The universal cover $\torb_1$ is an open domain in $\torb$. 
We have that $\pi_1(\mathcal{O})$ is still relatively hyperbolic with respect to the rest of the end fundamental groups by
Theorem \ref{thm:drutu}. Let $E_i, i=1, \dots, m$ denote the remaining ends, nonhyperbolic ones.

Let $T'$ be a triangle with $\partial T' \subset \Bd \torb$ and $\partial T'$ is not a subset of 
$\bigcup S(v_{\tilde E})$ of a radial type end or $\clo(\bigcup S_{\tilde E})$ of 
a totally geodesic end $\tilde E$ of lens-type.
Clearly $T'$ is in $\clo(\torb_1)$ and $\partial T' \subset \Bd \torb_1$ since if the interior of a segment meets 
$\bigcup S(v_{\tilde E})$, then the segment must be in $\bigcup S(v_{\tilde E})$ by Theorems \ref{thm:lensclass} and \ref{thm:redtot}. 

Let $U$ be the union of all concave p-end neighborhoods for radial p-ends and lens p-end neighborhoods  
for totally geodesic p-ends and horospherical p-ends mutually disjoint from one another
and covering a union of disjoint end neighborhoods of $\torb$. 
In this case, $S_{\tilde E}$ for the totally geodesic end $\tilde E$ is a subset of $\torb$. 

%%% 12 Aug 2013 4:53pm
Now we will consider various possibilities for the triangle: 
By assumption, one of the component $T'-\tilde U'$ is not compact. (Otherwise, we are done.) 
Denote by $L$ a noncompact component of $T'-\tilde U'$.

%This is a noncompact surface in $\tilde {\mathcal{O}}_2-U'$ in general. Take a minimal leaf $L$.
%Then $L$ is again in a triangle $T'$ meeting $\tilde U'$ and $L= T'-\tilde U'$. 
%Thefore, the sequence of diameters of components of $\tilde U' \cap T'$ are uniformly bounded
%since otherwise, we obtain a noncompact component by minimality. 

(I) Suppose that there is at least one $g \ne \Idd$ so that $g(L) = L$. 
Then clearly $g(T')=T'$ as well. 
Let $v$ be a vertex of $T'$.
Then $L/\langle g \rangle$ corresponds to an annulus mapping into $\mathcal{O}_1$. 
Let $l$ be a maximal geodesic in $L^o$ so that  $l$ and $g(l)$ bound a fundamental domain of the annulus. 
Then notice that $T^o$ contains a geodesic $\tilde \alpha_0$ connecting a point $v$ of $l$ to $g(v)$ of $g(l)$ mapping to 
a closed curve $\alpha_0$. By a similarity based at $v$, we form a parameter of closed 
curves $\alpha_t$ for $v(t) \in l$ and $t \in \bR$. The vertex of $\alpha_t$ is the image of $v(t)$.

Then % the $d_{\torb}$-lengths of $\alpha_t$ are bounded above. 
the $d_{\torb}$-lengths of $\alpha_t$ are uniformly bounded above since $g$ acts on a triangle with a diagonalizable 
matrix and we can compute the $d_{\torb}$-lengths of $\alpha_t$ by its vertex $v(t)$. Assume that $\alpha_t$ is periodic with 
fundamental interval $I \subset \bR$ always. 
Either 
\begin{itemize}
\item[(i)] $\alpha_t(I) \subset U$ for an end neighborhood $U$ of $\tilde E$ and $t \geq c$ and $t \leq c'$ for some $c$ and $c'$. 
\item[(ii)] $\alpha_t(I) \subset U$ for an end neighborhood  $U$ 
for some $t$ and $\alpha_t(I) \not\subset U$ for $t \geq c'$ or $t \leq c''$ for some $c', c''$.  
\item[(iii)] $\alpha_t(I) \not\subset U$ for an end neighborhood  $U$ for all $t$. 
\end{itemize}
In the first case, $T'$ must be in a p-end neighborhood $U'$ of $\tilde E$ so that $U'-U$ covers a compact set in $\orb$. 
%by considering the orbits under $\langle g \rangle$
%since $T'$ is convex. 
Since $\partial T' \subset \Bd \torb$, and $\clo(U') \cap \Bd \torb = \bigcup S(\tilde E)$, 
we obtain $\partial T' \subset \bigcup S(\tilde E)$.
We are finished in this case. 
%\marginpar{This is not clear at all May 16th}

In the second case, $g$ is freely homotopic to the end fundamental group. 
We can assume that a closed curve freely homotopic to $\alpha_t(I)$ cannot be a subset of another p-end neighborhood 
since otherwise we obtain an essential annulus.  %Thus, $L$ is one of at most two components of $T'- U'$.

Using the residual finiteness of the linear group $\pi_1(\torb)$, we  
take a finite index subgroup of $\pi_1(\orb)$ and a power of $g$ so that we can assume that  
$g$ is a generator of $\langle g \rangle$. 
Also, assume that $\pi_1(\orb)$ is torsion-free by taking a finite cover using the Selberg lemma.

Assume that $\alpha_t$ is not 
in an end neighborhood entirely for $t > c$. 
%If $\alpha_t$ is not in an end neighborhood for $t> c$, then it cannot enter another end neighborhood entirely; 
%otherwise, we obtain a free homotopy between elements of the fundamental groups of the two ends. 
In this case, there is a subsequence $t_i$ so that $\alpha_{t_i}$ converges to a closed curve not 
contained in any end neighborhood.  We assume that $\tilde \alpha_{t_0}$ is a subset of an end neighborhood. 
We can assume that $\bigcup_i \alpha_{t_i}$ has a convex hull containing $L$ 
since the sequence is of bounded diameter ones bounding a region with $\tilde \alpha_{t_0}$ covering an annuli and hence
eventually containing any compact subset of $L$. 

%Since $\pi_1(\orb)$ is relatively hyperbolic, there exists a 
%Let $G_g$ denote a maximal infinite cyclic group containing $g$. 
%Since $h(\pi_1(\orb))$ is a discrete linear subgroup, there are only finitely many elements $h$ so that 
%$h^m = g$ for $m \in \bZ$. By taking a finite cover of $\orb$ if necessary, we may assume that $g$ is a primitive element. 
%We denote by 
%\[M := \max \{d_{\torb}\hbox{-length}(\alpha_{t_j})| j=1, 2, \dots \}. \] %\cup \{ d_{\torb}^M(h(\tilde \alpha_{t_j}), \tilde \alpha_{t_j})| h \in G_g\}. \]

We see that $\alpha_{t_i}$ and $\alpha_{t_j}$ are homotopic for 
$i, j, i> j$ sufficiently large.
%We will now fix $j$ and consider $i$ as a variable. 
Let $\tilde \alpha_{t_i}:\bR \ra \torb$ denote the lift of $\alpha_{t_i}$ in $T^o$ where $g$ acts on. 
The $d_{\torb}$-length of $\alpha_{t_i}$ is uniformly bounded above since in $T^o$ the $d_{\torb}$-lengths are the $d_{T^o}$-lengths
and can be estimated by the action on the projective link of the vertex of $T$. 
 \begin{itemize} 
\item the minimum distance
%$d_{\orb}(\alpha_{t_i}, \alpha_{t_j}) \geq 2M+1$ if $j > J_0$ for some $J_0$. 
$d_{\torb}(\tilde \alpha_{t_i}, \tilde \alpha_{t_j}) \geq 2M+1$ for infinitely many pairs $i, j$. 
%$M := \max \{d_{\torb}\hbox{-length}(\alpha_{t_j})| j=1, 2, \dots \} .$ 
\item Hence, for every $\eps$, $0 < \eps < 1/2$, there exists infinitely many pair $i, j$ and a deck transformation $h_{i, j}$ for each $i, j$ so that 
for every $s\in \bR$, there exists $s' \in \bR$ so that
$d_{\torb}(h_{i, j}(\tilde \alpha_{t_i}(s')), \tilde \alpha_{t_j}(s) ) \leq \eps$,
and conversely for every $s'\in \bR$, there exists $s \in \bR$ satisfying this equation. 
\end{itemize} 
%That is, the oriented curves are also ``very close'' to each other and 
Since they have bounded lengths, 
$[\alpha_{t_i}]$ and $[\alpha_{t_j}]$ have the same homotopy class for infinitely many pairs $i, j$. 
%Assume that $\eps$ is sufficiently small and $j > J_1$, 
Thus, $g$ and $h_{i, j}$ commute with each other. 
Since $h_{i, j}$ sends $\tilde \alpha_{t_i}$ to points of minimal distance $\geq 2M$, 
it follows that $h_{i, j}$ is not in $\langle g \rangle$.
By (NA), $h_{i, j}, g \in \pi_1(\tilde E)$ for a p-end $\tilde E$. 

We  take a sufficiently large lens or lens-cone p-end neighborhood $D$ of $\tilde E$ by Lemma \ref{lem:expand}. 
%\marginpar{the change of $D$. needed in \cite{endclass}. Refts not done yet}
%we can assume that 
%\begin{itemize}
%$\tilde \alpha_{t_{i_0}} \subset D$ for at least one $i_0$.
%\item $d_{\torb}(\tilde \alpha_{t_{i_0}}, \partial D) \geq C$ for a constant $C > 1$, 
%and 
%\item $d_{\torb}(\alpha_{t_{i_0}}, \alpha_{t_i}) \leq 1/2$. 
%\end{itemize} 
Since $\alpha_{t_i}$ converges to a closed curve as $i \ra \infty$, 
we may assume that $h_{i, j}(\tilde \alpha_{t_i}) \subset D$ for $i > I_0$ where we assume $I_0 > i_0$.
$h_{i, j} \in \pi_1(\tilde E)$ implies that $\tilde \alpha_{t_i} \subset D$ since $D$ is $\pi_1(\tilde E)$-invariant. 
$\partial T$ is a subset of the closure of $\bigcup_{k \in \bZ} h_{i, j}^k (\tilde \alpha_{t_i})$. 
Thus, $\tilde E$ is a radial or totally geodesic p-end of lens-type. 
Since $D$ is a lens-cone or a lens, we obtain 
\[T \subset D \hbox{ and } \partial T \subset D \cap \Bd \torb = \bigcup S(v_{\tilde E}).\]
Also, $\tilde E$ is not totally geodesic since otherwise $T^o \cap \torb = \emp$. In this case, we are done. 
%the vertices of $T'$ has to be in $\clo(D) \cap \Bd \torb = \bigcup S(v_{\tilde E})$.  
%$T'$ is a subset of a convex hull of its p-end.
%This is what we wished to show. 

(II)  Now, we suppose that we are in case (iii) or 
there exists no $g \ne \Idd$ so that $g(L) = L$ from now on for each component $L$ of $T' - \tilde U'$. 
Now, we obtain a triangle in the asymptotic limit of $\torb$. 

We will obtain an asymptotic limit of $T$ in an asymptotic limit of $\torb$ and show that 
we cannot have such an object unless $\partial T'$ is in $\bigcup S(v_{\tilde E})$ for an radial p-end $\tilde E$. 
%\marginpar{maybe better outline here?}

%Then by taking a finite index subgroup if necessary, we see that $T$ is so that for each vertex $v$ of $T$
%so that $g(v)$ is not on $T$ or $g(v) = v$. 
%Suppose that $L$ covers a compact surface $S$. Then the boundary components of $S$ either is from 
%the boundary of a compact disk in $\tilde U'\cap T'$ or a noncompact disk $\tilde U' \cap T'$. 

%Suppose that there is a noncompact component $\alpha$ of $\tilde U' \cap T'$. 
%If the component is in a horospherical end, then a boundary point of $T'$ is in a horospherical end.
%This was ruled out by Proposition \ref{prop:affinehoro}.
%If a component is in a properly convex end, then since $\alpha$ maps to a closed curve in an end, 
%there is a nontrivial $g \in \pi_1(\mathcal{O})$ so that $g(\alpha) = \alpha$ and $g(L)=L$. 
%This was ruled out above. 

%Suppose now that every component of $\tilde U' \cap T'$ is compact. 
%In this case, $S$ with the compact components of $\tilde U'\cap T'$ filled is a compact real projective surface
%covered by a triangle. In this case, $\chi(S) = 0$, and $S$ is  a tori or a Klein bottle. We assume that 
%$\mathcal{O}$ has no such immersed tori or Klein bottles. 
%As in the compact case, there are no noncompact components of $\tilde U' \cap T'$ for a noncompact $L$.

%Suppose now that $L$ does not cover a compact surface. 
%$\pi_1(\mathcal{O})$ acts on the union of images of $L$, i.e., $\bigcup_{g \in \pi_1({\mathcal{O}})} g(L)$. 

Suppose that $L$ meets infinitely many horospherical p-ends
and the $d_{\torb}$-diameters of $L$ intersected with these are not bounded. Then we can show that $L$ or a leaf $L'$ in 
its closure of $\bigcup_{g \in \pi_1(\mathcal{O})} g(L)$ meet a horoball and its vertex in its closure.
However, in the first case, this gives an arc in $L$ or $L'$ with one horospherical p-endpoint equal to an interior point 
of an edge or a vertex of a triangle containing $L$ or $L'$. Both cases are ruled out by Proposition \ref{prop:affinehoro}.

Thus, $d_{\orb}$-diameters of horospherical p-end neighborhoods intersected with $L$ are bounded above uniformly. 
Therefore, by choosing a horospherical end neighborhood sufficiently far inside each horospherical 
end neighborhood, we may assume that $L$ does not 
meet any horospherical p-end neighborhoods. That is we choose a horoball $V'$ inside a one $V$
so that \[d_{\torb}(V', \partial V) \geq \frac{1}{2} \sup \hbox{$d_{\torb}$-diam}\{ V \cap T' \}_{V \in \mathcal{V}, T'\in \mathcal{T}}\]
where $\mathcal{V}$ is the collection of horospherical neighborhoods that we were given 
in the beginning and $\mathcal{T}$ is the collection of all triangles $T'$ with boundary in $\Bd \torb$ --(*).

%We can easily assume that $L$ is not a subset of a convex hull of an end since we can choose such a triangle as above. 

We will use the theory of tree-graded spaces and asymptotic cones \cite{DuSa} and \cite{OsSa}. 
We remove neighborhoods of sufficiently small 
horospherical p-end neighborhoods from $\tilde {\mathcal{O}}_1$ and call the result $\tilde {\mathcal{O}}'$. 
%We use the induced path metric on $\tilde{\mathcal{O}}'$ from the Finsler metric on $\tilde{\mathcal{O}}$.
We will be using the restricted path-metric $d_{\torb'}$ on $\torb'$ restricted from infinitesimal Finsler metric associated with $d_{\torb}$.
(See \cite{Kobpaper}.) 
Then $\tilde {\mathcal{O}}'$ is quasi-isometric with $\pi_1({\mathcal{O}})$: 
This follows since $\tilde{\mathcal{O}}'$ has a compact fundamental domain
and hence there is a map from it to  $\pi_1({\mathcal{O}})$ decreasing distances up to a positive 
constant. Conversely, there is a map from $\pi_1({\mathcal{O}})$ to  $\tilde{\mathcal{O}}'$
with same property. 

%Also, $\tilde {\mathcal{O}}'$ is quasi-isometric to a convex metric space.

%Recall that $\pi_1(U_1)$ for a p-end neighborhood $U_1$ 
%is isomorphic to a finite extension of $\bZ^l \times \Gamma_1 \times \cdots \times \Gamma_k$ 
%for hyperbolic groups $\Gamma_i$.

%%% Dec 17, 12:01 pm

Let us recall definitions in Section 3.1 of Drutu-Sapir \cite{DuSa}. 
An {\em ultrafilter $\omega$} is a finite additive measure on $P(\bN)$ of $\bN$ so that
each subset has either measure $0$ or $1$ and all finite sets have measure $0$. 
If a property $P(n)$ holds for all $n$ from a set with measure $1$, we say that $P(n)$ holds {\em $\omega$-almost surely}. \index{ultrafilter}

Let $(X, d_X)$ be a metric space.
Let $\omega$ be an ultrafilter over the set $\bN$ of natural numbers. 
For a sequence $(x_i)_{i \in \bN}$ of points of $X$, its {\em $\omega$-limit} is $x \in X$ if for every neighborhood $U$ of $x$
the property that $x_i \in U$ holds $\omega$-almost surely. If $X$ is Hausdorff, the limit is unique and if $X$ is compact, 
every sequence has a convergent sequence. 

An {\em ultraproduct} $\prod X_n/\omega$ of a sequence of sets $(X_n)_{n \in \bN}$ is the set of 
the equivalence classes of sequences $(x_n)$ where $(x_n) \sim (y_n)$ if $x_n  = y_n$ holds for $\omega$-almost surely. 

Given a sequence of metric spaces $(X_n, d_n)$, consider the ultraproduct $\prod X_n$ and an observation point $e=(e_n)$. 
Let $D(x, y) = \lim_{\omega} d_n(x_n, y_n)$. Let $\prod_e X_n/\omega$ denote the set of equivalence classes of finite distances from $e$.
The {\em $\omega$-limit $\lim^{\omega} (X_n)_e$} is the metric space obtained from $\prod_e X_n/\omega$ by identifying  \index{asymptotic cone} 
all pair of points $x, y$ with $D(x, y) =0$. 

Given an ultrafilter $\omega$ over the set $\bN$ of natural numbers, an observation point $e=(e_i)^\omega$, and 
sequence of numbers $\delta= (\delta_i)_{i\in \bN}$ satisfying $\lim_\omega \delta_i = \infty$, the $\omega$-limit 
$\lim^\omega (X, d_X/\delta_i)_e$ is called the {\em asymptotic cone} of $X$. (See \cite{Gr1}, \cite{Gr2} and Definitions 
3.3 to 3.8 in \cite{DuSa}.) We denote it by $Con^\omega(X, e, \delta)$. 

For a sequence $(A_n)$ of subsets $A_n$ of $X$, we denote by $\lim^\omega(A_n)$ the subset of $Con^\omega(X, e, \delta)$ 
that consists of all elements $(x_n)$ where $x_n \in A_n$ $\omega$-almost surely. 
The asymptotic cone is always complete and $\lim^\omega(A_n)$ is closed.

%For our space $\tilde \mathcal O_1$ with the above metric. 
%Oct 6 1:57pm

We will fix an ultrafilter $\omega$ in $\bN$ and $\delta$ from now on. We set an observation point $e$ to be a constant sequence $e \in \torb'$.
By Theorem 9.1 of Osin and Sapir \cite{OsSa} and Theorem 5.1 of Drutu and Sapir \cite{DuSa} and Theorem \ref{thm:relhyp}, 
$\pi_1(\tilde {\mathcal{O}})$ is asymptotically tree-graded with respect to the p-end fundamental groups $\pi_1(E_i)$, $i=1, \dots, m$. 
Thus $\tilde {\mathcal{O}}'$ is asymptotically tree graded with respect to the p-end neighborhoods. 
Choose an ultrafilter $\omega$, and 
let $\tilde {\mathcal{O}}_\infty$ denote the asymptotic cone of $\tilde {\mathcal{O}}'$ 
with the $\omega$-limit of the p-end neighborhoods as pieces.  
Here a piece is a closed subset satisfying certain properties in \cite{DuSa}. 
Let $\mathcal P$ denote the set of pieces. %(See Drutu and Sapir \cite{DuSa} for terminology used here.)

The basic heuristic strategy is to show that a triangle in an asymptotically tree-graded space must be inside one of the p-end neighborhoods. 

But the existence of a triangle $T'$ in $\clo(\torb)$ 
gives us a subspace $T$ isometric with $T^{\prime o}$ in the asymptotic limit $\torb_\infty$. 
We obtain this by considering all sequences $(x_i)_{i \in \bZ_+}$ for $x_i \in T^{\prime o}$. 
The geodesics here are precisely the straight lines since the Hilbert metric $d_T$ on $T^{\prime o}$ scaled by a constant 
$d_T/\delta_i$ is isometric with $d_T$ by an isometry $f_i$. Hence $T^{\prime o}$ with $d_T/\delta_i$ has 
$\omega$-limit $T^{\prime o}$ with $d_T$.  
($d_T$ is called a hex metric \cite{Harpe}. This fact was fist observed by Cooper, Delp and so on, as we understand.) \index{Hilbert metric!hex metric}

Each point of $T'$ has a p-end neighborhood of uniformly bounded $d_{\torb}$-distance from it as $\mathcal O - U$ is 
compact by the action of $\pi_1(\tilde {\mathcal{O}})$.
Hence, each point of $T$ is in an element of a piece. (See Definition 3.9 of \cite{DuSa}.) 

%%
%Each p-end neighborhood has a convex p-end neighborhood in a bounded distance away;
%that is, 
%For each p-end neighborhood $U$, there exists a convex p-end neighborhood $V$, $U \subset V$ 
%so that $\{d(\partial V, x)| x \in U\} \leq C$ for some uniform constant $C_V > 0$ depending on $V$. 
%Since $\tilde{\mathcal{O}'}$ is quasi-isometric to a convex metric space, 
%Each piece is also convex since we can consider the piece has the asymptotic limit of $V$ instead. 
%Thus, a piece intersected with $T'$ is a convex subset, possibly with nonempty interior. 
%Let the set of the intersections of pieces with $T'$ be denoted by ${\mathcal P}'$.

%It is not possible that there is only one piece in ${\mathcal P}'$ meeting $L'$ since this means that there is a sequence 
%of points $x_i \in L$ so that the distance from $x_i$ to ends is going to $\infty$ except for one sequence of ends $E_i$
%we have $d(x_i, E_i)$ grows less faster than the scaling factors in the asymptotic cone constructions. 
%This is not true by (*).

%% May 16 9:04pm

%%Oct 6 3:00 pm
\begin{figure}[h]
\centerline{\includegraphics[height=12cm]{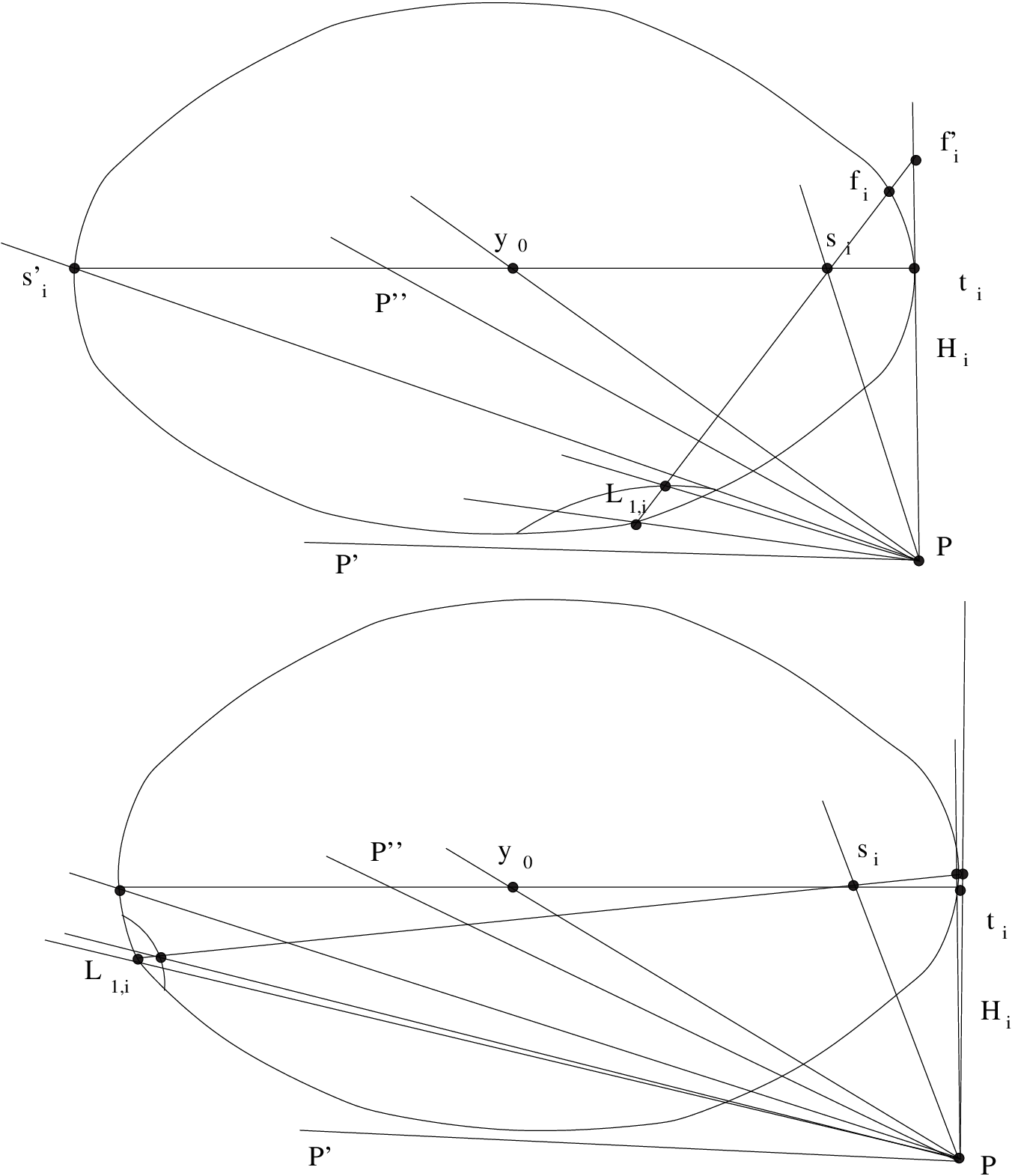}}
\caption{The diagram for Theorem \ref{thm:converse}. }
\label{fig:figmetric}
\end{figure}

%\marginpar{Change $s'_i$ to $s'_i$.}

%%% Oct 30, 2013: 10:27pm I need to consider the totally geodesic ends also....

(III) We will now show that by the asymptotic tree graded property, %the triangle $T$ cannot exist in the asymptotic limit
$\partial T'$ must be contained in the closure of a p-end-neighborhood. % or in the closure of a totally geodesic 
%ideal boundary to begin with.  Hence, we obtain a desired contradiction. 

%\marginpar{These are really thick to read... Change this...April 22}

%First suppose that there are no horospherical p-end neighborhood. 
%\marginpar{$L_{1, i}$ to become the closure of boundary of concave neigh... Change.
%Need to draw the cone as well}

For $i$ in the index set of p-ends, we define $L_{1, i}$  to be the subset of $\tilde U$
\begin{itemize} 
\item $\clo(U(v_{\tilde E})) \cap \torb_1$ where 
$U(v_{\tilde E})$ is the concave p-end neighborhood of $\tilde E$
when $\tilde E$ is a radial p-end of lens-type, 
\item The closure of the outer lens $L$ of 
$S_{\tilde E}$ if $\tilde E$ is a totally geodesic end of lens type, or %\marginpar{We are in thicken $\torb'$. needed?}
\item a horoball $U_{\tilde E}$ for a horospherical end $\tilde E$.
\end{itemize} 
Here $i$ denotes the labeling of the end vertices. 
We call it the lower boundary component of a lens of the end $\tilde E$. 
Let us choose $\hat L_{1, j}$ for $j=1, \dots, m$ from each representative in $L_{1, i}$ equivalent under $\pi_1(\torb)$.

By Proposition 7.26 in \cite{DuSa}, each piece is 
a $\omega$-limit of $(g_{i_n} \hat L_{1, j_i})$ 
where $(g_{i_n})$ has the $\omega$-limit in 
the $\omega$-limit $\pi_1(\orb)^\omega$ of $\pi_1(\orb)$ 
and \[\lim_{\omega} \frac{d_{\torb'}(e, g_{n_i} \hat L_{1, j_i})}{\delta_i} < \infty.\]

We can assume that $T'$ does not intersect with horoballs as above for $L$
as in (*).
Hence, $T'\cap L_{1, i}^o = \emp$ since $T$ is obtained as the $\omega$-limit of triangles outside 
the concave p-end neighborhoods. 

Suppose that the asymptotic limit $T$ of $T'$ 
is contained in a piece. 
Since $T$ has a sequence of compact domains $(\hat K_i)$ exhausting it, we obtain 
by taking sequences $(K_{i, j} \subset T')$ of compact sets converging to these and taking a diagonal sequences
the following: 
\begin{itemize} 
\item[(i)] There is a sequence of compact domains $K_i$ of points in $T^{\prime o}$ 
\[d_{\torb'}(y_0, \partial K_i)= \lambda_i \ra \infty \hbox{ for } y_i \in K_i\] 
so that $\lambda_i/\delta_i \ra \infty$ and 
%\item[(ii)] $y_i= y_0$
\item[(ii)] Since every point of $T'$ becomes very close to a 
the lower boundary component $L_{1, i}$ of a lens of a lens-shaped p-end neighborhood $U_{i, 1}$
under the normalization by $1/\delta_i$, 
we have  \[\max\{d_{\torb'}(x, L_{1, i})| x \in K_i\} \leq  \mu_i\] for the closure of $L_{1, i}$, 
and $\mu_i/\delta_i \ra 0$.
%\item[(iv)] 
\end{itemize} 
This implies also
\begin{equation} \label{eqn:iii}
\lambda_i/\mu_i \ra \infty.
\end{equation}
(Here $L_{1, i}$ is of form $g_{n_i} \hat L_{1, {n_i}}$ above for some sequence $g_{n_i} \in \pi_1(\orb)$.)

%We can choose $K_i$ to be convex since a convex hull of itself satisfies the needed property
%by Proposition \ref{prop:shortg}. 
%We obtain that since $\lambda_i/\delta_i \ra \infty$, 
%$K_i$ eventually contains $y_0$. Therefore, we may let $y_i =y_0$ for convenience. 
For any sequence of segments $y_0 \in m_i = m \cap  K_i$ with $\partial m_i \subset \partial K_i$
in $T$ and the distance $\lambda_i$ of $m_i$ from $y_0$ to the end points of $m_i$,
 we have $\lambda_i/\delta_i \ra \infty$. %We can assume that $m_i \subset m$ for a fixed maximal line in $\torb$ though $y_0$.

%We claim that this cannot happen as $\lambda_i$ is less than or equal to $\mu_i$ times some constant: 

We need a hypothesis: 
\begin{description}
\item[(H)] Suppose that the sequence $\{L_{1, i}\}$ is $\bdd$-bounded away from one of the end points 
of $m$ in $\partial T$. 
\end{description}
Let $s_i \in m_i$ be a point of $\partial m_i \subset \partial K_i$.
Let $t_i$ be the end point of the extending line $m_i$ further away from $s_i$ from $y_0$ 
and let $s'_i$ be the other end point of $m_i$.
%We can fix $m_i$ to be in a maximal line $m$ without the loss of generality 
%since the argment above will work in this way also. 

We take a supporting hyperspace $H_i$ at the point $t_i$.
Then we take a subspace $P_i$ of $H_i$ of codimension $2$ in $\bR P^n$ 
disjoint from $\tilde{\mathcal{O}}$, which exists by the proper convexity of $H_i \cap \clo(\tilde{\mathcal{O}})$. 
Then we take hyperspaces containing $P_i$ and points $s', t_i, s_i, y_0$ 
and take the logarithms of the cross ratios to find the distance $d_{\torb}(y_0, s_i)=\lambda_i$
(See Figure \ref{fig:figmetric}.)

We also take a geodesic $\tilde l_i$ from $s_i$ to $L_{1,i}$ of length $\mu_i$ 
and denote by $f_i$ the end point of the extension of 
$s_i$ not in the boundary of $L_{1,i}$. Then we take the logarithm of the cross ratio of $f_i, s_i$ 
and the two other points in the closure of 
$L_{1, i}$ to obtain the distance $\mu_i$. We can replace $f_i$ with $f'_i$ 
the intersection of the line containing $\tilde l_i$ with 
$H_i$. The corresponding logarithm $\mu'_i$ is smaller than or equal to $\mu_i$.

Here to compute $\mu_i$, we need a point $\hat l_{1, i}$ on $L_{1, i}$ and one $\hat m_i$ 
in $\bigcup S(v_i)$ for the pseudo-convex p-end 
vertex $v_i$ corresponding to $L_{1, i}$ if $L_{1, i}$ from a radial p-end of lens-type. 
We need a point $\hat l_i$ on 
the totally geodesic ideal boundary and one $\hat m_i$ in the boundary of the outer part lens if $L_{1, i}$ 
is of totally geodesic type. 
For horospherical case, $\hat l_i$ is in the boundary of the horosphere and $\hat m_i$ is in $\Bd \torb'$. 

%\marginpar{(We need to choose large $i$ so that $s_i$ is sufficiently close to $t_i$ for this to work.)}

\begin{description}
\item[(H1)] For now, assume that $\orb$ has no horospherical ends. 
\end{description} 
So, now we can just use $d_{\torb}$ instead of $d_{\torb'}$. 

We will fix $P_i$ sufficiently far way from $\torb$. 
Let $A$ be an affine space containing $\torb$ containing $P_i$. 
(We will see things from $P_i$, and the half-hyperspaces for two points for $\clo(L_i)$ are bounded away from 
those of two points $y_0, s'$ uniformly as $L_{1, i}$ is bounded away from $t_i$.)

We have \[\lambda_i = |\log[s', t_i, s_i, y_0]|,  
\mu_i = |\log[\hat m_i, f_i, s_i, \hat l_i]| \geq \mu'_i := |\log[\hat m_i, f'_i, s_i, \hat l_i]|.\]
 We define $P(v)$ to be the unique half-hyperspace in $A$ containing in the boundary $P$ and $v \not\in P$. 
 Let $H'_i$ be the half-hyperspace in $A$ containing $P$ and $t_i$. 
Then we have
\[ \mu'_i := \log[P(\hat m_i), P(f'_i), P(s_i), P(\hat l_i)].\]
There exists a half-hyperplane $P'$ containing $P$ and furthermost among $P(\hat m_i)$  away from $H_i$ from the view point of $P$.
%namely the one supporting $\torb$. 
There exists a half-hyperplane $P''$ containing $P$ that $P(\hat m_i)$ can be realized as
closest to $H_i$ since $L_{1, i}$ is bounded away from $t_i$ uniformly. 
Then \[ \mu'_i \geq |\log[P', H_i, P(s_i), P(\hat m_i)]| \geq |\log[P', H_i, P(s_i), P'']|\] 
since we chose the minimal possibility in the right term by taking the half-planes to extreme possibility
to lower values first by $P'$ and then by $P''$ second. 
(Here for sufficiently large $i$, $P(s_i)$ is closer to $H_i$ than $P''$ as $s_i \ra t_i$ 
since $L_{i, 1}$ is bounded away from $t_i$.)
We have \[\lambda_i =\log[P(s'), H_i, P(s_i), P(y_0)].\] 
Now parameterize $s_1$ as a linear function $s_1(t)$ of new variable
so that $s_1(t) \ra t_i$ as $t \ra 1$. 
Now define functions of $t$ as 
\[f_1(t) := \log[P', H_i, P(s_i(t)), P''] \hbox{ and } 
f_2(t) :=\log[P(s'), H_i, P(s_i(t)), P(y_0)].\] 
Here $s'$ and $y_0$ are fixed. 
Then we see easily that only constant 
terms are different as rational functions of $t$ with a pole of same order at $1$.
We obtain $f_2(t) \leq C f_1(t)$ for $1 - \eps < t < 1$. 

Thus, we obtain $\lambda_i \leq C \mu'_i \leq C \mu_i$ for sufficiently large $i$ 
for some fixed positive constant $C$. 
Therefore, considering all directions of $m$, we obtain that $K_i$ has a diameter 
bounded by $Cd_{{\torb}}(x, L_{1, i})$ for a positive constant $C'$.
This is a contradiction to equation \eqref{eqn:iii}. 
%(Here, ``bounded away" means that bounded away in the elliptic metric of 
%$\bR P^n$. More details needed here... Upper bound on the total angle needed and so on)

%Now, we can suppose that $L_{1, i}$ has points converging to an end point of $\partial T$ or that $L_{1, i}$ contains an end point of $m_i$.
%Let $s_i \in \tilde s_i$ be a point of $\partial s_i \subset \partial K_i$ away from that end from $y_0$. 
%Then a similar argument will show a contradiction. 

%%% July 20 12:08 AM to do below
Now suppose that we have some horospherical p-ends. That is, we drop the hypothesis (H1). 
We follow the same argument as above and we show that a triangle $T'$ in $\clo(\torb)$, $T^{\prime o} \subset \torb$,  
must satisfy $\partial T' \subset \bigcup S(v_{\tilde E})$ for
a radial end $\tilde E$. 
We obtain a sequence $\{L_{1, i}\}$ as above. Since there are only three types, we assume that each 
$L_{1, i}$ share a common type for each sequence. 

 If $L_{1, i}$ are from still radial or totally geodesic p-ends, then 
in the metric $d_{{\torb}'}$ we can do the same arguments where we obtain lower bound 
in terms of $d_{\torb}$. We have  $d_{\torb} \leq d_{{\torb}'}$.
Since on totally geodesic $T'$ the metrics $d_{{\torb}'}$ and $d_{\torb}$
are the same, we again obtain that $K_i$ has a diameter 
bounded by $Cd_{{\torb'}}(x, L_{1, i})$ for a positive constant $C'$. % \marginpar{$T$ or $T'$? notation prob. April 22 1:41pm}

If each $L_{1, i}$ is the boundary of a horosphere, then we use $d_{\torb}$ and 
the larger $d_{{\torb}'}$ to obtain the same proof. 
(In this case, the hypothesis (H) is always true. 
%In this case, 
the sequence of the $\bdd$-diameters of $\{\clo(L_{1, i})\}$ goes to zero since the components of the inverse 
images of the end neighborhoods are locally finite. )
Hence, $T$ cannot be in the limit of a sequence $\{L_{1, i}\}$ of horospherical type. 

Now we drop the hypothesis (H). 
Suppose now that the sequence $\{L_{1, i}\}$ is not $\bdd$-bounded away from the both ends of $m$ in $\partial T'$. 
%Since it cannot be that $L_{1, i}$ converges to some set containing $m$ as this means 
%an interior of an edge of $T$ meets the closure of a p-end neighborhood. 
%($L_{1, i}$ cannot be a horospherical clearly.) 

Suppose that there exists a sequence $(L_{1, i_k})$ so that each end point $\delta_j m$, $j=1, 2$, of $m$ has 
a sequence $\{x^j_{i_k}\}$, $j=1, 2$, in $(L_{1, i_k})$ where $\{x^j_{i_k}\} \ra \delta_i m$ as $k \ra \infty$ in the $\bdd$-metric. 
Suppose that $L_{1, i_k}$ is from a radial p-end of lens-type. 
%Then we project $m$ to a geodesic arc $m_{i_k}$ in $S_{\tilde E_{i_k}}$.
%Taking a maximal geodesic containing $m_{i_k}$ in $S_{\tilde E_{i_k}}$, we obtain end points in 
%$\clo(S_{\tilde E_{i_k}})$ and correspond to two points in $\bigcup S(v_{\tilde E_{i_k}})$.
We take the union of the two segments in $\bigcup S(v_{\tilde E_{i_k}})$ to form an arc $s_{i_k}$
with $\partial s_{i_k} =\{x^1_{i_k}, x^2_{i_k}\}$. 
%The union of the end points of $s_{i_k}$ geometrically converges to that of $m$ as $k \ra \infty$. 

%from the p-end vertex of $L_{1, i_k}$. 
Suppose that $L_{1, i_k}$ is totally geodesic of lens-type.
Then there is an arc $m_{i_k}$ in a lens obtained by taking the intersection of $m$ with the lens neighborhood 
in $L_{1, i_k}$. Then we take a maximal segment $s_{i_k}$ in $\clo(S_{E_{i_k}})$ near $m_{i_k}$ with 
endpoints $x^j_{i, k}$, $j=1, 2$ and $\{x^j_{i, k}\} \ra \delta_j m$, $j = 1, 2$ in the $\bdd$-metric. 
%converging to the end point of $m$ as $k \ra \infty$. 

We take a subsequence so that $\{s_{i_k}\}$ geometrically converges to a union $s$ of one or two segments. 
By strictness of the SPC-structure on $\orb$, it follows $s$ is a subset of $\bigcup S(v_{\tilde E})$ or 
$\clo(S_{\tilde E})$. Since the end point of $m$ is in $s$, we have $\delta_1 m, \delta_2 m \in S({v_{\tilde E}})$ for 
a radial end $\tilde E$. ($\tilde E$ cannot be totally geodesic.)

%We may assume that the Hilbert length of $m_{i_k}$ goes to infinity because of 
%our convergence property. %Then there exists subarcs $m'_{i_k}$ in $m_{i, k}$ where we can obtain
%an essential closed curve $m''_{i_k}$ representing an element of $\pi_1(\tilde E_{i, k})$. 
%Since $S_{\tilde E_{i_k}}/\pi_1(\tilde E_{i_k})$ is compact and there are only finitely many end orifolds, 
%the compact arcs $m_{i_k}$ geometrically converges to an infinite geodesic as $k \ra \infty$ in 
%the real projective orbifold $\Sigma:=\bigcup_{j=1}^m \Sigma_{E_j}$ for some $m$ up to a choice of a subsequence.

%Hence, we obtain a sequence of segments $m_{i_k, j} \subset m_{i_k}$ so that 
%the two end points $\delta_1 m_{i_k, j}, \delta_2 m_{i_k, j}$
%$m_{i_k, j}$ projects to a pair of points of $\Sigma$ whose sequence of distances between them goes to zero as $j \ra \infty$.
%Let $\hat{\delta}_1 m_{i_k, j}, \hat{\delta}_2 m_{i_k, j} \in m$ be the corresponding points in $m$ by the projection 
%from $v_{\tilde E_{i_k}}$. 

%By compactness of $\Sigma$, there exists sufficiently large $k'$ and $k''$ so that 
%$m'_{i, k'}$, a closed curve obtained as above, is freely homotopic to $m'_{i, k''}$, a closed curve obtained as above (in fact geometrically close). 
%This means that $\pi_1(\tilde E_{i, k'})$ sharing a nonzero element with $\pi_1(\tilde E_{i, k'})$.
%This contradicts no-annulus condition.  

%Therefore, we may assume that $L_{1, i}$ is independent of $i$. 
%Suppose that $L_{1, i}$ is from a radial p-end of lens type. Then 

Now, an edge of $T'$ meets $\bigcup S(v_i)$ for the corresponding p-end vertex $v_i$.
By Theorems \ref{thm:lensclass} (ii) and \ref{thm:redtot}, the edge is $S(v_i)$. 
By changing the directions of the maximal segment $m$ in $T'$, it follows that 
$\partial T' \subset S(v_i)$ since the edges where $m$ ends have to be in $S(v_j)$ for the same $j$
by the above reasoning. 
This contradicts the assumption.
%and that $T'$ is in the convex hull of the p-end by convexity.
%Hence, $\partial T'$ is in the closure of a p-end neighborhood and we have a contradiction to the assumption here. 

%Suppose that $L_{1, i}$ corresponds to a lens-type and totally geodesic p-end. 
%In this case, $T$ has to be on the ideal boundary component so that $T$ is in $\clo(\torb)$.
%These are  contradictions on the assumption on $T$. 
%If $L_{1, i}$ equals $\clo(S_{\tilde E})$, then we can clearly apply the same argument to show that this 
%does not happen. 

Therefore, we conclude that the asymptotic limit triangle $T$ in $\torb_\infty$ 
cannot be contained in one piece in the asymptotic limit. 

It is not possible for exactly two components ${\mathcal P}'$ contain $T$: 
Suppose $T = C_1 \cup C_2$ for closed $C_1, C_2$ and $C_1 \cap C_2$ is a single point. 
However, removing a point cannot separate $T$ into two components. 

As a consequence, let $p_1, p_2$ be two points of the interior of $T$ not in a single piece.
Then there exists a point $p_3$ in general position so that no two of $p_1, p_2, p_3$ are contained 
in a common piece. 

By taking a geodesics between two of $p_1, p_2, p_3$, we obtain a simple triangle $\Delta'$.
This contradicts the definition of tree-graded spaces. 
(See Definition 1.10 of \cite{DuSa}.) 

We conclude that there exists no triangle such as $T'$.

\end{proof}

\begin{remark} 
We think that there is a proof for $n=3$ using trees as Benoist have done in closed $3$-dimensional cases in \cite{Ben4}. 
\end{remark}

We recapitulate the results: 
\begin{corollary}\label{cor:remhyp}
Assume that $\mathcal{O}$ is a noncompact strongly tame SPC-orbifold with generalized admissible ends % lens-type radial ends, 
%lens-type totally geodesic ends, and horospherical ends and 
and satisfies {\rm (IE)} and {\rm (NA)}. Let $E_1, \dots, E_n, E_{n+1}, \dots, E_k$ be the ends of $\orb$
where $E_{n+1}, \dots, E_k$ are some or all of the hyperbolic ends. 
Assume $\partial \orb =\emp$. 
Then $\pi_1(\mathcal{O})$ is a relatively hyperbolic group with respect to 
the admissible end groups $\pi_1(E_1), ..., \pi_1(E_n)$ 
if and only if $\mathcal{O}_1$ is strictly SPC with respect to admissible ends $E_1, \dots, E_n$. 
\end{corollary}
\begin{proof} 
If $\pi_1(\mathcal{O})$ is a relatively hyperbolic group with respect to 
the admissible end groups $\pi_1(E_1), ..., \pi_1(E_n)$, 
%with all the hyperbolic end fundamental groups not included, 
then $\pi_1(\mathcal{O})$ is a relatively hyperbolic group with respect to 
the admissible end groups $\pi_1(E_1), ..., \pi_1(E_k)$ by Theorem \ref{thm:drutu}. 
By Theorem \ref{thm:converse}, it follows that $\mathcal{O}$ is strictly SPC with 
respect to the ends $E_1,.., E_k$. 
%Let $\mathcal O_2$ denote the $\mathcal O_1$ with the equivariant collection of 
%the $(n-1)$-balls corresponding to all the hyperbolic ends 
%no longer considered ends. Then $\mathcal O_2$ covers a tame orbifold with radial ends.
By Proposition \ref{prop:remconch}, we obtain that $\mathcal O_1$ is strictly SPC with respect to 
$E_1, \dots, E_n$. 

For converse, if $\mathcal O_1$ is strictly SPC with respect to $E_1, \dots, E_n$, 
then $\mathcal O_1$ is strictly SPC with respect $E_1, \dots, E_k$. 
By Theorem \ref{thm:relhyp}, $\pi_1(\mathcal{O})$ is a relatively hyperbolic group with respect to 
the admissible end groups $\pi_1(E_1), ..., \pi_1(E_k)$. 
The conclusion follows by Theorem \ref{thm:drutu}. 
\end{proof}

%%% May 27 3:06pm

\subsection{Strict SPC-structures deform to strict SPC-structures.}

\begin{theorem}\label{thm:relhyp1}
Let ${\mathcal{O}}$ denote a noncompact strongly tame SPC-orbifold with admissible ends 
and satisfies {\rm (IE)} and {\rm (NA)}.  Assume $\partial \orb =\emp$. 
Let \[E_1, \dots, E_n, E_{n+1}, \dots, E_k\] be the ends of $\orb$
where $E_{n+1}, \dots, E_k$ are some or all of the hyperbolic ends. 
\begin{itemize}
%\item An SPC-structure on $\mathcal{O}$ with admissible ends is strictly SPC if and only if $\pi_1(\mathcal{O})$ is relatively hyperbolic
%with respect to the subgroups $\pi_1(E_i)$ for ends of $\mathcal{O}$.
\item Given a deformation through SPC-structures with generalized admissible ends of a strict SPC-orbifold with respect to admissible ends 
$E_1, \dots, E_k$ to an SPC-structure, the SPC-structure remains strictly SPC with respect to $E_1, \dots, E_k$. 
%\item An SPC-structure on $\mathcal{O}$ with admissible ends with all the hyperbolic ends not regarded as ends is strictly SPC after removing 
%the concave neighborhoods of the hyperbolic ends if and only if $\pi_1(\mathcal{O})$ is relatively hyperbolic
%with respect to the subgroups $\pi_1(E_i)$ for ends of $\mathcal{O}$ with corresponding hyperbolic end fundamental groups dropped.
\item Given a deformation through SPC-structures with generalized admissible ends of a strict SPC-orbifold with respect to $E_1, \dots, E_n$ 
to an SPC-structure with generalized admissible end,  the SPC-structure remains strictly SPC 
with respect to admissible ends $E_1, \dots, E_n$. 
\end{itemize}
\end{theorem}
\begin{proof}
%The first and the third item are  proved by Theorem \ref{thm:relhyp}, Theorem \ref{thm:converse}, 
%and Corollary \ref{cor:remhyp}.
%The fourth item is  proved by the fact that the groups do change under small deformations. 

%The first item is proved as follows: The forward direction is proved by Theorem \ref{thm:relhyp}. 
%\marginpar{Converse, I am stuck... }

For the second item, $\mathcal{O}_1$ being strictly SPC with respect to $E_1, \dots, E_n$ 
implies that $\pi_1(\mathcal{O})$ is relatively hyperbolic 
with respect to the end fundamental groups $\pi_1(E_1), \dots, \pi_1(E_n)$ by Corollary \ref{cor:remhyp}. 
Then the small deformation does not change the group property. Thus, after deformation 
$\mathcal{O}$ is strictly SPC with respect to $E_1, \dots, E_n$ by the fourth item
by Corollary \ref{cor:remhyp}. %Then $\mathcal{O}$ is strictly SPC with respect to E_1, \dots, E_k$.

The first item is simpler to show by Theorem \ref{thm:relhyp} and Theorem \ref{thm:converse}.

\end{proof}

\part{The openness and the closedness of the deformations of convex real projective structures}

\chapter{The openness of the convex structures}\label{sec:openness}

In this section also, we will only need $\bR P^n$ versions. 
Given a real projective orbifold with radial ends or totally geodesic ends of lens-type, each end has an orbifold structure of dimension $n-1$ 
and inherits a real projective structure. 

Let $\mathcal U$ and $s_{\mathcal U}: {\mathcal{U}} \ra  (\bR P^n)^{e_1} \times (\bR P^{n \star})^{e_2}$
 be as in Section \ref{subsec:endreal}.
\begin{itemize}
\item We define $\Def^s_{E, \mathcal U, s_{\mathcal U}, ce}(\mathcal{O})$ to be the subspace of 
$\Def_{E,\mathcal U, s_{\mathcal U}}(\mathcal{O})$ 
of real projective structures with generalized admissible ends determined by $s_{\mathcal U}$,
and stable irreducible holonomy homomorphisms in $\mathcal U$. 
%and define $\Def^s_{E, ce}(\mathcal{O})$ to be the subspace of $\Def_{E}(\mathcal{O})$ 
%with real projective structures with admissible ends and stable irreducible holonomy homomorphisms.
%We suppose that these subspaces for $\mathcal{O}$ satisfy the end fundamental group condition. % by Remark \ref{rem:realization}.

\item We define 
$\CDef_{E, s_{\mathcal U}, ce}(\mathcal{O})$ to be the subspace consisting of SPC-structures with generalized admissible ends
 in $\Def_{E, s_{\mathcal U}, ce}(\mathcal{O})$. 
 %We have that $\CDef_{\bR P^n, E, s_{\mathcal U}, ce}(\mathcal{O}) \subset \Def_{\bR P^n, E, s_{\mathcal U}, ce}(\mathcal{O})$. 

\item %Suppose that $\mathcal{O}$ satisfies more generally the convex end fundamental group conditions. Then 
We define 
$\CDef_{E, u, ce}(\mathcal{O})$ to be the subspace of 
$\Def_{E, u, ce}(\mathcal{O})$ consisting of SPC-structures with generalized admissible ends. 
%Also, we have $\CDef_{E, ce}(\mathcal{O}) \subset \Def^s_{E, ce}(\mathcal{O})$.

\item We define 
$\SDef_{E, s_{\mathcal U}, ce}(\mathcal{O})$ to be the subspace of consisting of strict SPC-structures with admissible ends
 in $\Def_{E, s_{\mathcal U}, ce}(\mathcal{O})$. 
 %We have that $\CDef_{\bR P^n, E, s_{\mathcal U}, ce}(\mathcal{O}) \subset \Def_{\bR P^n, E, s_{\mathcal U}, ce}(\mathcal{O})$. 
%Suppose that $\mathcal{O}$ satisfies more generally the convex end fundamental group conditions. 
\item We define 
$\SDef_{E, u, ce}(\mathcal{O})$ to be the subspace of 
$\Def_{E, u, ce}(\mathcal{O})$ consisting of strict SPC-structures with admissible ends. 
%Also, we have $\SDef^s_{\bR P^n, E, ce}(\mathcal{O}) \subset \Def^s_{\bR P^n, E, ce}(\mathcal{O})$.
\end{itemize}

%\begin{theorem}\label{thm:conv} 
%Suppose that $Def_{Aff, E}$
%In $Def_{Aff, E, s_{\mathcal U}}(\mathcal{O})$, the subspace of properly convex affine structures is open. 
%Suppose that $\mathcal{O}$ has no essential homotopy annulus. 
%In $\Def_{\bR P^n, E, s_{\mathcal U}, ce}(\mathcal{O})$, the subspace  $\CDef_{\bR P^n, E, s_{\mathcal U}}(\mathcal{O})$ 
%of strict SPC-structures is open. 
%\end{theorem}

\begin{theorem}\label{thm:conv} 
%Suppose that $Def_{Aff, E}$
Let $\mathcal{O}$ be a noncompact strongly tame real projective $n$-orbifold with generalized admissible ends 
and satisfies {\rm (IE)} and {\rm (NA)}. Assume $\partial \orb =\emp$. 
%Suppose that $\mathcal{O}$ satisfies the end fundamental group condition or more generally the convex end fundamental group conditions,
%Suppose that $\mathcal{O}$ has no essential annulus or essential torus.
%In $Def_{Aff, E}(\mathcal{O})$, the subspace of properly convex affine structures is open. 
In $\Def^s_{E, u, ce}(\mathcal{O})$, the subspace  $\CDef_{E, u, ce}({\mathcal{O}})$ of SPC-structures with generalized admissible ends is open, and 
so is $\SDef_{E, u, ce}({\mathcal{O}})$.
\end{theorem}

\begin{theorem}\label{thm:conv2} 
Let $\mathcal{O}$ be a noncompact strongly tame real projective $n$-orbifold with generalized admissible ends 
and satisfies {\rm (IE)} and {\rm (NA)}. Assume $\partial \orb =\emp$. 
For an open $\PGL(n+1, \bR)$-conjugation invariant \[\mathcal{U} \subset \Hom^s_E(\pi_1(\orb), \PGL(n+1, \bR)),\]
%$\Def^s_{E, \mathcal U, s_{\mathcal U}, ce}(\mathcal{O})$ 
and a $\PGL(n+1, \bR)$-equivariant section 
$s_{\mathcal{U}}: {\mathcal{U}} \ra (\bR P^n)^{e_1} \times (\bR P^{n \star})^{e_2}$, 
$\CDef_{E, s_{\mathcal U}, ce}(\mathcal{O})$ is open 
 in $\Def^s_{E, s_{\mathcal U}, ce}(\mathcal{O})$,
 and so is $\SDef_{E, s_{\mathcal U}, ce}(\mathcal{O})$.
 \end{theorem}

For orbifolds such as these the deformation space of convex structures 
may only be a proper subset of space of the characters. 

By Theorem \ref{thm:conv} and Theorem \ref{thm:A}, we obtain:
\begin{corollary}\label{cor:conv} 
Let $\mathcal{O}$ be a noncompact strongly tame real projective $n$-orbifold with generalized 
admissible ends and satisfies {\rm (IE)} and {\rm (NA)}. Assume $\partial \orb =\emp$. 
%Suppose that $\mathcal{O}$ satisfies the end fundamental group condition or 
%more generally the convex end fundamental group conditions.
%and suppose that $\mathcal{O}$ has no essential homotopy annulus. 
Then \[\hol: \CDef_{E, u, ce}(\mathcal{O}) \ra \rep^s_{E, u, ce}(\pi_1(\mathcal{O}), \PGL(n+1, \bR))\] is a local homeomorphism. 
Furthermore, if $\mathcal{O}$ has a strict SPC-structure with admissible ends and and satisfies {\rm (IE)} and {\rm (NA)}, then so is
\[\hol: \SDef_{E, u, ce}(\mathcal{O}) \ra \rep^s_{E, u, ce}(\pi_1(\mathcal{O}), \PGL(n+1, \bR)).\]
\end{corollary}

\begin{corollary} \label{cor:conv2}
Let $\mathcal{O}$ be a noncompact strongly tame real projective $n$-orbifold 
with generalized admissible ends and satisfies {\rm (IE)} and {\rm (NA)}. Assume $\partial \orb =\emp$. 
Let $\mathcal{U}$ and $s_{\mathcal{U}}$ be as above. 
Suppose that $\mathcal{U}$ has its image ${\mathcal{U}}'$ in $\rep^s_{E, u, ce}(\pi_1(\mathcal{O}), \PGL(n+1, \bR))$. 
Then 
\[\hol: \CDef_{E, s_{\mathcal U}, ce}(\orb) \ra \mathcal{U'}\] is a local homeomorphism, and so is
\[\hol: \SDef_{E, s_{\mathcal U}, ce}(\orb) \ra \mathcal{U'}.\]
\end{corollary}

Here, in fact, one needs to prove for every possible continuous section. 

Koszul  \cite{Kos} proved these facts for closed  affine manifolds and expanded by Goldman \cite{Gconv}
for the closed real projective manifolds. See \cite{dgorb} and also Benoist \cite{Ben3}.

%% May 17 4:53pm 

\section{The proof of the convexity theorem}

%% Vey function... property. Lower bound on hessian... See by differentiation of the formula 

Recall that a {\em convex open cone} $V$ is a convex cone of $\bR^{n+1}$ containing the origin $O$ in the boundary. 
%Any segment containing with a point of the cone can be extended to an open ray in the cone containing the point.
Recall that a {\em properly convex open cone} is a convex cone so that its closure does not contain a pair of $v, -v$ for 
a nonzero vector in $\bR^{n+1}$. Equivalently, it does not contain a complete affine line in its interior.

A dual convex cone $V^*$ to a convex open cone is a subset of $\bR^{n+1 *}$ given by
the condition $\phi \in V^*$ if and only if $\phi(v)> 0$ for all $v \in \clo(V) -\{O\}$.

Recall that $V$ is a properly convex open cone if and only if so is $V^*$ and $(V^*)^* = V$ 
under the identification $(\bR^{n+1 *})^* = \bR^{n+1}$. 
Also, if $V \subset W$ for a properly convex open cone, then $V^* \supset W^*$.

For properly convex open subset $\Omega$ of $\bR P^n$, its dual $\Omega^*$ in $\bR P^{n*}$ 
is given by taking a cone $V$ in $\bR^{n+1}$ corresponding to $\Omega$
and taking the dual $V^*$ and projecting it to $\bR P^{n *}$. The dual $\Omega^*$ is a properly convex open domain if so was $\Omega$.

Recall the Koszul-Vinberg function for a properly convex cone $V$ and the dual properly convex cone $V^*$ 
\begin{equation}\label{eqn:kv}
f_{V^*}: V \ra \bR_+ \hbox{ defined by }  x \in V \mapsto f_{V^*}(x)= \int_{V^*} e^{-\phi(x)} d\phi
\end{equation}
where the integral is over the euclidean measure in $\bR^{n+1 *}$. 
This function is strictly convex if $V$ is properly convex. $f$ is homogeneous of degree $-(n+1)$. 
%the Hessian $\partial_i \partial_j \log(f)$ gives us a metric. 
Writing $D$ as the affine connection, we will write the Hessian $D d\log(f)$. 
The hessian is positive definite and norms of unit vectors are 
strictly bounded below in a compact subset $K$ of $V - \{O\}$.
(See Chapter 6 of \cite{wmgnote}.) 
The metric $D d\log(f)$ is invariant under the group $\Aff(V)$ of affine transformation acting on $V$. 
(See Theorem 6.4 of \cite{wmgnote}.)
In particular, it is invariant under dilatation maps $x \mapsto s x$ for $s > 0$.

%\[ \partial_i \partial_j \log \int_{V^*} e^{-\phi(x)} d\phi = (1/f) \int_{V^*}  \int_{V^*} \phi_i \phi_j e^{-\phi(x)} d\phi .\]
%The latter term used as a coefficients of bilinear form 
%\[ (1/f)  \int_{V^*} \sum_{i=1}^n (\phi_i v_i)^2e^{-\phi(x)} d\phi ,\]
%which is uniformly bounded on $K$. 

%repeated....
%\begin{theorem}\label{thm:conv} 
%Suppose that $Def_{Aff, E}$
%In $Def_{Aff, E, s_{\mathcal U}}(\mathcal{O})$, the subspace of properly convex affine structures is open. 
%In $Def_{\bR P^n, E, s_{\mathcal U}, ce}(\mathcal{O})$, the subspace of properly convex real projective structures is open. 
%\end{theorem}

%%% July 20 10:37 pm

%%% Sept 30 10:02 To do... strict and not strict.. State integrals.. The proofs of the following not complete with integrals.. 
%% More precise integrals...

A {\em Hessian metric} on an open subset $V$ of an affine space 
is a metric of form $\partial^2 f /\partial x_i \partial x_j$ for affine coordinates $x_i$ \index{Hessian metric}
and a function $f: V \ra \bR$ with a positive definite Hessian defined on $V$. 
A Riemannian metric on an affine manifold is a {\em Hessian metric} if the manifold is affinely covered by a cone 
and the metric lifts to a Hessian metric of the cone. 

Let $\mathcal{O}$ have an SPC-structure $\mu$ with admissible ends. Clearly $\tilde{\mathcal{O}}$ is a properly convex open domain. 
Then an affine suspension of $\mathcal{O}$ has an affine Hessian metric defined by $D d\phi$ for a function $\phi$ defined on 
the cone in $\bR^{n+1}$ corresponding to $\tilde{\mathcal{O}}$ by above. 
%We will find a neighborhood $L$ in $\mathcal U$ that consists of properly convex affine structures with 
%parallel ends. First, we require structures in $L$ to be $C^2$ close to $\mu$.

For a subset $K$ of $\SI^n$ or $\bR P^n$, we denote by $C(K)$ the cone in $\bR^{n+1}-\{O\}$ the inverse image of $K$ under
the projection. 
Recall that a {\em parameter of real projective structures} $\mu_t, t\in [0, 1]$ on a strongly tame orbifold $\orb$.
is a collection so that the restriction $\mu_t|K$ to each compact suborbifold $K$ is continuous parameter; 
In other words, the associated developing map $\dev_t|\hat K$ for every compact subset $\hat K$ of $\torb$ 
is a family in the $C^r$-topology continuous for the variable $t$. (See \cite{dgorb} and Canary \cite{Canary}.)

%%% Define aff susp... April 4th 8:42pm 
%Given a domain with a real projective structures

\begin{proposition}\label{prop:openess}
Let $\mathcal{O}$ be a strongly tame orbifold with ends and satisfies {\rm (IE)} and {\rm (NA)}.
%Suppose that $\mathcal{O}$ satisfies the end fundamental group condition or more generally 
%the convex end fundamental group conditions,
%Suppose that $\mathcal{O}$ has no essential annulus or torus.
Suppose that $\mathcal{O}$ has an SPC structures $\mu_0$ with generalized admissible ends
%generalized lens-type radial ends, lens-type totally geodesic end, 
%and horospherical end and 
%and a real projective structure $\mu_1$ and 
and the suspension of $\mathcal{O}$ with $\mu_0$ has a Hessian metric. 
The ends of $\mathcal{O}$ are given $\cR$-type or $\cT$-types. 
If 
\begin{itemize}
\item $\mu_0$ is SPC, and a parameter of real projective structure $\mu_{0, t}$, $t\in [0, 1]$, 
with generalized admissible ends  and $\mu_{0,0} =\mu_0$ 
where the $\cR$-types or $\cT$-types of ends are preserved,  %radial or totally geodesic ends
%\item there exists a parameter of real projective structures $\mu_{0, t}$ 
%with generalized admissible ends so that $\mu_{0, 0} = \mu_0$ and $\mu_{0, 1}=\mu_1$, 
%\item $\mu_{0, t}$ restrict to each compact suborbifold of $\orb$ is a continuous in the $C^2$-topology, 
%and 
%\item the holonomy of the end fundamental fundamental group of $\mu_1$  is either horospherical, radial lens-type, 
%or totally geodesic lens type,  
\end{itemize}
then for sufficiently small $t$, the affine suspension
$C(\torb)$ for $\torb$ with $\mu_t$ also has a Hessian metric invariant under dilations and the 
affine suspensions of the holonomy homomorphism for $\mu_{0, t}$. 
%and $\orb$ with $\mu_t$ has generalized admissible ends. 
%Finally, if we replace the condition of lens-type radials ends to be relaxed to having concave radial ends for $\mu_0$ for radial ends, 
%then the affine suspension of $\mu_1$ has a Hessian metric and has generalized lens-type radial ends, totally geodesic ends or horospherical ends. 
\end{proposition} 
\begin{proof}
%Define $\torb'$ to be the universal cover correspoding to $\orb$ with $\mu_1$. 
We will keep $\torb$ fixed and only change the structures on it. 
Note that the subsets here remain fixed and the only changes are on the real projective structures, 
i.e., the atlas of charts to $\bR P^n$. 
%(For every $\eps$, $\eps > 0$, 
%if $\mu_0$ and $\mu_t$ are sufficiently close in the $C^2$-topology of the developing maps, given a compact subset $K$ of 
%$\tilde{\mathcal{O}}$, there is a corresponding $K'$ in $\tilde{\mathcal{O}}$ 
%whose image under $\dev_1$ is $\eps$-close in the Hausdorff topology to $\dev_0(K)$.)  %\marginpar{needed?}

%Suppose we modify the projective structure $\mu_0$ to $\mu_1$ on $\mathcal{O}$ in $C^2$-topology. 
Let $\tilde{\mathcal{O}}$ in $\bR P^n$ denote the universal covering domain corresponding to $\mu_0$. % and let 
%$\tilde{\mathcal{O}}'$ denote one corresponding to $\mu_1$. 
Again $\dev_0$ being an embedding identifies the first with subsets of $\bR P^n$
but $\dev_t$ is not known to be so. We shall prove this below. 

%April 21 8:34pm p.86

We will make a simplifying assumption: 
\begin{itemize} 
\item[(H)] For $\mu_0$, every radial end of generalized lens-type is an end of radial type. 
\end{itemize}

(A) The first step is to understand the deformations of the end-neighborhoods: 

Let $\tilde E'$ be a p-end of $\torb$ and it corresponds to a p-end of $\torb'$ as well. 
%By Theorem \ref{thm:A} and \ref{thm:projective}, given $\mu_0$ and $\mu_1$ with above assumptions, 
There exists a $C^r$-parameter of real projective structures $\mu_{0, t}$ 
with radial or totally geodesic ends of lens-type so that $\mu_{0, 0} = \mu_0$. 
We can also find a parameter of developing maps $\dev_t$ associated with $\mu_{0, t}$
where $\dev_t|K$ is a continuous with respect to $t$ for each compact $K \subset \torb$. 
To begin with, we assume that $\tilde E'$ keeps being of a radial p-end of lens type or horospherical type.

Here, we do not allow $\cR$-type ends to change to $\cT$-type ends and vice versa
as this will make us to violated the local injectivity property from the deformation space to 
a space of characters. (See Theorem \ref{thm:A}.)
Thus, we need to consider only three cases to prove openness: % how the type changes for ends. 
\begin{itemize} 
\item[(I)] A radial p-end changes to a radial p-end in the cases: 
\begin{itemize} 
\item A radial p-end of lens-type becoming a radial p-end of generalized lens-type.
\item A horospherical p-end becoming a radial p-end of generalized lens-type or horospherical type. 
%(which we will show to be of generalized lens or horospherical type).
%\item[(II)] A radial p-end of horospherical type changes to one of horospherical  or generalized lens-type. 
\end{itemize} 
\item[(II)] A totally geodesic p-end of lens type deforms to a totally geodesic p-end of lens type. 
\item[(III)] A horospherical p-end deforms to a horospherical p-end or to a totally geodesic p-end.
\end{itemize}
Here, the premise assumes that these hold for the corresponding 
holonomy homomorphisms of the fundamental groups of ends.
We will show that the above happens in actuality as well.

%% March 9th 4:39pm
We will now work on one end at a time: 
Let us fix a p-end $\tilde E$ of R-type of $\tilde{\mathcal{O}}$. 
Let $v$ be the p-end vertex of $\tilde E$ for $\mu_0$ and $v'$ that for $\mu_1$. 
We denote by $v = v_0$ and $v'= v_1$. 
Assume that $v_t$ is the p-end vertex of $\tilde E$ for $\mu_t$. 
Let $\dev_t$ and $h_t$ denote the developing map and the holonomy homomorphism of $\mu_t$. 
Assume first that the corresponding p-end for $\mu$ is of radial or horospherical type. 
By post-composing the developing map by a transformation near the identity,
we assume that the perturbed vertex $v_t$ of the corresponding p-end $\tilde E$
is mapped to $v_0$, i.e., $v = \dev_t(v_t)$.

(I) %We can assume that $v = \dev_t(v_t)$ by post-composing with isometries near identity. 
If $\tilde E$ is of radial p-end of horospherical or lens-type for $\mu_0$, then $\tilde E$ is always a radial p-end of 
horospherical a generalized lens-type for $\mu_t$.
%for $\mu_t$ by Theorem 7.9 \cite{endclass} and 
%by the assumption.

%Let $K$ be a compact convex subset of  $\tilde{\mathcal{O}}$ and $K'$ the perturbed one in  $\tilde{\mathcal{O}}'$
%and $\tilde E'$ be the corresponding p-end.
%Also, a compact subset $R_v(K) \subset R_v(\tilde E)$ is changed to a compact $R_v(K') \subset R_v(\tilde E')$. 

%We use a projective transformation of $\bR P^n$ very close to the identity map, 
%to make sure that $R_{v'}(K')\subset R_{v'}(\tilde E')$ for the set of rays $R_{v'}(\tilde E')$ of the corresponding end 
%$\tilde E'$ of the perturbed $\mathcal{O}$ and the p-end vertex $v'$.  

%%% July 22 3:23 pm to do below%
%\marginpar{$\Bd_S$ define? spherical length} 
%%% 

%%% Oct 31 5:25 pm 
%We will assume that $\tilde E$ is of lens-type if $\tilde E$ is of generalized lens-type. 
%(We will study the radial p-ends of generalized lens-type later.)

Let $\Lambda$ denote the limit set in the tube of the radial p-end $\tilde E$ for $\torb$ if $\tilde E$ is of lens-type radial p-end, 
or $\{v_{\tilde E}=v\}$ if $\tilde E$ is a horospherical type for $\mu_0$. 

%Let $r_v(A)$ denote the space of segments of $\torb$ in $R_v(\torb)$ passing through the set $A$
%from the p-end vertex $v$.
\begin{itemize}
\item Let $R_{v, t}(\torb)$ denote the space of rays in $\torb$ mapping to ones from $v$ under $\dev_t$,
\item $r_{v, t}(A_t)$ denote the union of segments of $\Omega_{s_0, t}$ in $R_{v, t}(\torb)$ passing through the set $A_t \subset \torb$
mapping to ones from the p-end vertex $v$ under $\dev_t$.  
\end{itemize}

%Let $t$ be large real number.
Then for $\mu_0$, 
a smooth and strictly convex hypersurfaces $\partial \Omega_{s_0} \subset \torb$, $s_0 \in \bR_+$, 
as obtained by Lemma \ref{lem:expand} with 
\[ \clo(\partial \Omega_{s_0}) - \partial \Omega_{s_0}  \subset \Lambda.\]
Also, each radial geodesic is transversal to $\partial \Omega_{s_0}$. 
$\partial \Omega_{s_0}$ bounds a properly convex domain $\Omega_{s_0}$. 
Here $\bigcup_{s_0\in S} \Omega_{s_0} = \torb$ where $S$ is an infinite index set.

%%% May 17 6:05pm 

For a sufficiently small $t$ in $\mu_{0, t}$, we obtain a domain 
$U_t \subset \torb$ that is a concave neighborhood or a horospherical one 
%a totally geodesic lens-type end neighborhood 
with $U_0 \subset \Omega_{s_0}$. 
Now $U_t$ can be compactified to $\hat U_t $ so that $\dev_t| U_t$ for the developing map 
$\dev_t$ extends to an imbedding $\widehat{\dev_t}| \hat U_t$ to a concave end neighborhood 
or a horoball. % or a one-sided lens neighborhood of a totally geodesic end.
In the first two cases, there exists a point $v' \in \hat U_t$.   
%Define $S'(v)_t \subset \hat U_t$ as the preimage of $\bigcup S(v)_t$ in the image lens-cone $\dev_t(U_t)$
%provided $U_t$ is a radial concave p-end neighborhood. 
Let $S(v)_t$ denote the set of segments from $v$ in $\hat U_t $ in the corresponding to the boundary of 
$S_{\tilde E}$ of $\mu_t$. 
Again $\Bd U_t \cap \torb$ is assumed to be strictly concave for all sufficiently small $t$ if $U_0$ was a concave p-end neighborhood. 

Let the tube $B_t$ be determined by $\dev_t(U_t)$; i.e., $B_t $ is the union of great segments with end points in $v, v_-$ in the direction of 
$\dev_t(U_t)$. 

%% April 13 2:16pm
Define $\Lambda_t$ be the limit set in $\bigcup S(v)_t$ for generalized radial p-end cases and 
$\Lambda_t =\{v\}$ for the horospherical case. 

We will denote by $S_{\tilde E, t}$ the universal cover of the end orbifold corresponding to $\tilde E$ for $\mu_{0, t}$. 
Since $\tilde E$ is a radial p-end and $S_{\tilde E}$ is properly convex or complete affine 
for $\mu_0$, the admissible holonomy condition on $\mu_t$ implies that $\tilde E$ is a generalized lens-type end
or $S_{\tilde E, t}$ is a complete affine space. (See \cite{Ben3} for properly convex cases.) 
The surface $S_{\tilde E, t}$ is always a convex real projective $(n-1)$-orbifold. 

We assume that the $C^r$-change $r \geq 2$ of $\mu_{0, t}$ from $\mu_0$ be sufficiently small so that 
we obtain a region $\Omega_{s_0, t} \subset \torb$ with $\partial \Omega_{s_0, t}$ strictly convex
and transversal to radial rays under $\dev_t$. Here, $\Omega_{s_0, 0} = \Omega_{s_0}$. 
(The strict convexity follows since the change of affine connections are small as the argument of Koszul \cite{Kos}.)
Choose a compact domain $F$ in $\partial \Omega_{s_0}$. 
Let $F_t$ denote the corresponding deformed set in $\partial \Omega_{s_0, t}$. 
For sufficiently small $t$, $0 < t <1$, $\dev_t(F_t)$ is a subset of the tube $B_t$ determined by $\dev_t(U_t)$
since $B_t$ and $\dev_t(F_t)$ depend continuously on $t$. 
By transversality to the segments mapping to ones from $v$ under $\dev_t$, it follows that 
$\dev_t|\partial \Omega_{s_0, t}$ gives us a smooth immersion to a convex domain $S_{\tilde E, t}$ 
that equals the space of maximal segments $B_t$ with vertices $v$ and $v_-$. 
In this case, $\dev_t| \partial \Omega_{s_0, t}$ is a diffeomorphism to $S_{\tilde E, t}$ by \cite{Kobpaper} if 
$S_{\tilde E, t}$ is properly convex and by Proposition \ref{prop:affinehoro}(i) if $h_t(\pi_1(\tilde E))$ is horospherical. 
Since $\pi_1(\tilde E)$ is of generalized lens type or horospherical type, it follows that 
\[\dev_t(\clo(\partial \Omega_{s_0, t}) - \partial \Omega_{s_0, t}) \subset \dev_t(\Lambda_t)\]
by Corollary 8.5 of \cite{endclass}.
%\marginpar{(By a general limit set discussion in \cite{endclass}. To be done)}

Suppose that $\tilde E$ is of radial lens-type.
Since $\partial \Omega_{s_0, t}$ is convex, each point of $\partial \Omega_{s_0, t} \cup S(v)_t$ 
has a neighborhood that maps under the completion $\widehat{\dev}_t$ to a convex open ball. 
Thus, $\partial \Omega_{s_0, t} \cup S(v)_t$ 
bounds a compact ball $\Omega_{s_0, t} \cup \bigcup S(v)_t$ by Lemma \ref{lem:locconv}. 
%since the local convexity implies the global convexity. (This can be shown by Theorem A.2 of \cite{psconv} for example.) 

Suppose that $\tilde E'$ is horospherical type. 
$\partial \Omega_{s_0, t} \cup \{v\}$ bounds a convex domain $\Omega_{s_0, t}$
by the local convexity of the boundary set $\partial \Omega_{s_0, t} \cup \{v\}$ and Lemma \ref{lem:locconv}. 

We showed that the $\tilde E$ is still a radial p-end of lens-type or horospherical here for $\mu_1$. 

Now, we will show how these regions change. 
\begin{itemize}
\item Let $K$ be a compact convex subset of  $\Omega_{s_0, 0}$ with smooth boundary,
which we can choose to be sufficiently large, and $K_t$ the perturbed one in  $\Omega_{s_0, t}$
and $\tilde E$ be the corresponding p-end. We can form a compact set inside $\Omega_{s_0, t}$ consisting 
of segments from the p-end vertex to $K$ in the set of radial segments. 
For $\mu_{0, t}$ from $\mu_0$ changed by a sufficiently small manner, 
a compact subset $R_v(K) \subset R_v(\torb)$ is changed to a compact convex domain 
$R_{v, t}(K_t) \subset R_{v, t}(\torb)$. 
\end{itemize}

The p-end $\tilde E$ has either a concave p-end neighborhood or a horospherical p-end neighborhood. 
If $\tilde E$ has a concave p-end neighborhood, then since $\Omega'_{s_0, t}$ is strictly convex, 
we can obtain a lens. Thus, $\tilde E$ is admissible.

Let $K_t$ be a large compact convex domain in $\torb$ with $\mu_t$. For sufficiently small $t$, 
$\Omega_{s_0, t} \cap r_{v}(K_t)$ is a convex domain since $\partial \Omega_{s_0, t}$ is strictly convex and 
transversal to segments from $v$ and hence embeds to a convex domain under $\dev_t$. 
We may assume that $\dev_t(\Omega_{s_0, t} \cap r_{v}(K_t))$ is sufficiently close to $\dev_0(\Omega_{s_0} \cap r_v(K))$ 
as we changed the real projective structures sufficiently small in the $C^2$-sense and hence the holonomy 
of the generators of the p-end fundamental group $\pi_1(\tilde E)$ is changed by a small amount
if the change from $\mu_0$ to $\mu_{0, t}$ is sufficiently small. 
%The set $\clo( \Omega'_t)$ bounded by $\Omega'_t$ in $r'(E')$ is again convex 
%as $\Omega'_t$ is strictly convex. 
%(Here, by construction, $\tilde E$ is lens-type or horospherical radial end for $\mu_{0, t}$.)

An {\em $\eps$-thin} space is a space which is an $\eps$-neighborhood of its boundary for small $\eps> 0$.
%Let $r_v(K)$ denote the union of great segments with vertex $v$ passing through points of $K$. 
By Corollary \ref{cor:smvar}, we may assume that $\clo(R_v(\torb))$ and $\clo(R_{v, t}(\torb))$ as subsets of $\SI^{n-1}_v$
are $\eps$-$\bdd$-close convex domains in the Hausdorff sense for sufficiently small $t$.
Thus, $\dev_t(r_v(\torb)-r_v(K_t))$ is an $\eps$-thin space and so is $\dev_0(r_v(\torb)-r_v(K))$ for sufficiently small changes of $\torb$ and $K$. 
Given an $\eps > 0$, we can choose $K$ and $K'_t$ and small 
deformation of the real projective structures so that 
\begin{align} 
\clo(\dev_t( \Omega_{s_0} \cap (r_v(\torb)-r_v(K)))) \subset N_\eps(\clo( \dev_0(\Omega_{s_0} \cap r_v(K))) \nonumber \\ 
\clo(\dev_t(\Omega_{s_0, t} \cap (r_v(\torb)-r_v(K_t))))\subset N_\eps(\clo(\dev_0(\Omega_{s_0, t} \cap r_v(K_t)))). 
\end{align} 
The reason is that the supporting hyperplanes of $\clo(\Omega_{s_0})$ at points of $\partial r_v(K) \cap \clo(\Omega_{s_0})$
are in arbitrarily small acute angles from geodesics from $v$ and similarly for those of $\clo(\Omega_{s_0, t})$ for sufficiently small $t$. 
Therefore the Hausdorff distance between $\clo(\dev_0(\Omega_{s_0}))$ and $\clo(\dev_t(\Omega_{s_0, t}))$ can be made
as small as desired as long as we choose $\mu_{0, t}$ sufficiently close to $\mu_0$.
%Hence, the same can be said for $\mu_1$. 

%Now suppose that $\tilde E$ has only a concave p-end neighborhood of $\mu_0$. 
%Then we showed above that $\mu_1$ has a concave p-end neighborhood since the holonomy of the end is required 
%to be of generalized lens type and Theorem 7.9 of \cite{endclass}. 

(II) Now suppose that $\tilde E$ is totally geodesics p-end, and we suppose that $\tilde E$ is totally geodesic for $\mu_t$.
We can take dual domains of corresponding p-end neighborhoods and obtain a radial lens-type p-end or a horospherical p-end
by Theorem \ref{thm:duality}. 
We take $\Omega_{s_0}$ be the convex domain obtained as in Lemma \ref{lem:expand}.
Then $\partial \Omega_{s_0}/\pi_1(\tilde E)$ is a strictly convex compact $(n-1)$-orbifold. 
Suppose that $\mu_{0, t}$ is sufficiently close to $\mu_0$. Then $\partial \Omega_{s_0, t}$ is also
cocompact under the $\pi_1(\tilde E)$-action associated with $\mu_1$ and strictly convex. 
We have $\partial \clo(\partial \Omega_{s_0, t}) = \partial \clo(S_{\tilde E, t})$ for a totally geodesic ideal boundary 
component $\tilde S_{\tilde E, t}$ by Theorem 8.2 of \cite{endclass} since $h_t(\tilde E)$ is of lens-type and hence 
satisfies the uniform middle eigenvalue condition. % \marginpar{strong irr cond needed?}
Therefore the union of $\partial \Omega_{s_0, t}$ and % the totally geodesic ideal boundary component
$\clo(\tilde S_{\tilde E, t})$ bounds a properly convex compact $n$-ball in $\bR P^n$.
Hence, we obtain a lens-type end or horospherical end for $\tilde E$ and $\mu_t$ by Lemma \ref{lem:locconv}. 
%(need lemma on local conv implies conv)
%since the holonomy of $\mu_{0, t}$ is of lens-type. 
%as we can show by duality using Lemma \ref{lem:shrink} and Lemma \ref{lem:expand}. 
%Since the holonomy of $\pi_1(\tilde E)$ satisfies the lens-condition.

Moreover, we may assume without loss of generality that
$\bigcup_{s_0 \in I} \Omega_{s_0, 0}= \torb$ for some infinite index $I$.

%Given a compact fundamental domain $F$ of 
%$\partial \Omega_t$, for $\eps> 0$, we have $\bigcup_{g\in G_k} g(F)$ for a finite subset $G_k$ of $\pi_1(\tilde E)$ 
%so that $\partial \Omega_t - \bigcup_{g\in G_k} g(F)$ is of $\eps$-distance from $\partial \clo(\tilde E)$. 
%By taking $G_k$ sufficiently large, we may assume that the a smooth $(n-1)$-ball $B$ in $\partial \Omega_t$ 

%we may also assume that $\partial \Omega'_t - \bigcup_{g\in G_k} g(F')$ is of $\eps$-distance 
%from $\partial \clo(\tilde E')$. \marginpar{Why? explain?} 

Let us choose a sufficiently small $\eps > 0$. 
Let $B$ be a compact $(n-1)$-ball in $\partial \Omega_{s_0}$ so that 
\[ \bdd^H(\dev_0(\partial \Omega_{s_0} - B), \dev_0(\partial \clo(S_{\tilde E}))) < \eps. \]
Given a supporting hyperplane $W_x$ of a point $x$ of $\dev_0(\partial \Omega_t)$, there exists 
a supporting closed half-sphere $H_x$
containing it as the the boundary. Let $V_{\tilde E}$ as the hyperplane containing $\dev_0(S_{\tilde E})$. 
We define the {\em shadow} $S$ of $\partial B$ as the set 
\[ \bigcap_{x \in \partial B} H_x \cap V_{\tilde E}. \]
Then we can choose sufficiently large $B$ so that $\bdd^H(S, \dev_0(\clo(S_{\tilde E}))) \leq \eps$. 
We can also assure that $W_x$ meets $V_{\tilde E}$ in angles in $(\delta, \pi-\delta)$ for some $\delta > 0$
by compactness of $\partial B$ and the continuity of map $x \mapsto W_x$. 

Suppose that we change the structure from $\mu_0$ to $\mu_t$ with a small $C^2$-distance. 
Then $B$ will change to $B'_t$ with $W_x$ change by small amount. The new shadow $S'_t$ will 
have the property $\bdd^H(S'_t, \clo(S_{\tilde E, 1})) \leq \eps$ for a sufficiently small $C^2$-change of $\mu_t$ from $\mu_0$. 
Hence, we obtain \[ \bdd^H(\dev_t(\partial \Omega_{s_0, t} - B'_t), \dev_t(\partial \clo(S_{\tilde E, t}))) < \eps \] 
for a sufficiently small $C^2$-change of $\mu_t$ from $\mu_0$. 
Therefore by Corollary \ref{cor:smvar} the Hausdorff distance between $\clo(\Omega_{s_0})$ and $\clo(\Omega_{s_0, t})$ can be made
as small as desired as long as we choose $\mu_{0, t}$ sufficiently close to $\mu_0$.
(Note that we can have a change to a horospherical end here.)
%%%% July 23, 12:13: I still need to above and do for all tot geo cases... 

Here, the admissibility of the ends of orbifolds follows since by construction we obtain a lens shaped one-sided 
neighborhood for the totally geodesic p-end $\tilde E$. 

(III) Lemma \ref{lem:horob} studies this case. Here, similarly to the case (II), we can
obtain that $\clo(\Omega_{s_0})$ and $\clo(\Omega'_{s_0, t})$ can be made as small as one desires.

%%% Oct. 7 2:23 pm

%Finally, suppose that the holonomy group of $\tilde E'$ changes
%One can change 
%\begin{itemize} 
%\item[(III)] from a radial lens-type to a horospherical type, 
%item[(IV)] from a totally geodesic lens type to horospherical type, \marginpar{need to prove this.} 
%\item[(V)] from a horospherical type to a radial lens-type and 
%\item[(VI)] from a horospherical type to a totally geodesic lens type.
%\end{itemize} 
%The cases (III) and (V) are covered above since the arguments are not different. 
%(IV) and (VI) are dual to (III) and (V) respectiviely. So this follows from Theorem \ref{thm:duality}. 

%%% April 4th 2014 8:38pm
(B) We change the Hessian function on the cone associated with 
the universal covers. We need to obtain one for the deformed end neighborhoods and another one the outside 
of the union of end neighborhoods and patch the two together.  

With $\torb$ with $\mu_t$, we obtain a special affine suspension on $\orb \times \SI^1$ with the affine structure $\hat \mu_t$. 
Let $C(\torb)$ be the cone over $\torb$. Then this covers the special affine suspension. 
Let $\tilde \mu_t$ denote the affine structure on $C(\torb)$ corresponding to $\hat \mu_t$. 
For each $\mu_t$, it has an affine structure $\tilde \mu_t$, different from the induced one from $\bR^{n+1}$ as for $t = 0$. 
We require that scalar multiplication \[s\cdot v = sv, v \in C(\torb), s \in \bR\] for any affine structure $\tilde \mu_t$. 
Also, given a subset $K$ of $\torb$, we denote by $C(K)$ the corresponding set in $C(\torb)$. 
This set is independent of $\tilde \mu_t$ but will have different affine structures nearby. 

For $\mu_0$, $\dev_0(\torb)=\torb$ is a domain in $\SI^n$. 
Recall the Koszul-Vinberg function $f: C(\torb) \ra \bR_+$ homogeneous of degree $-n-1$. (See Lemma \ref{lem:KV}.) \index{Koszul-Vinberg function} 
By Lemma \ref{lem:expand}, 
the Hausdorff distance between $\clo(\Omega_{s_0})$ and $\tilde{\mathcal{O}}$ can be made as small as 
desired given $s_0$. % if we choose $\mu_1$ sufficiently close to $\mu_0$. 
By the third item of Lemma \ref{lem:convHaus},
the Hessian functions $f'_t$ defined by equation \eqref{eqn:kv} 
on $C(\Omega_{s_0, t})^o$
is very close to the original Hessian function $f$ 
in compact subsets of $C(\Omega_{s_0}^o)$ in the $C^2$ topology
by Lemma \ref{lem:KV}. By construction, $f'_t$ is homogeneous of degree $-n-1$. 

%%% 7:52 pm April 13, 2014 I need to work out notations.... page 88.
The holonomy groups $h(\pi_1(\mathcal{O}))$ and $h(\pi_1(\tilde E))$ being in $\GL(n+1, \bR)$ preserve $f$ and $f'$ under 
deck transformations respectively.
%\marginpar{Very cluttered here: shorten...} 

Now do this for all p-ends and we obtain functions $f'_t$ on $C(U_t)$ of the 
$\pi_1(\orb)$-invariant mutually disjoint union $U_t$ of p-end neighborhoods of p-ends of 
$\torb$ %of form $\Omega_{s_0, t_0} \subset\Omega_{s_0, t}$
for $\mu_t$ and sufficiently small $t_0$.
%We obtain $f'_t$ by restricting the Koszul-Vinberg function of $C(\Omega_{s_0, t})$
%to $C(U_t)$. 

Let $U$ be the corresponding $\pi_1(\orb)$-invariant union of  proper p-end neighborhoods of $\torb$ for $\mu_0$. 
For each component $U_i$ of $U$, we construct $f'_t$ on $C(U_i)$ using $\Omega_{s_0}$ so that $f'_t$ satisfies the above properties. 
We call $f'_t$ the union of these functions. 

Let $V$ be a $\pi_1(\orb)$-invariant neighborhood of the complement of $U$ in $\torb$. 
Given $\eps > 0$, there exists $\delta >0$ so that 
we have an $\eps$-$C^2$-map close to the identity map on a compact fundamental domain of the set $V$ 
to $V := \torb - U$ since a developing map of $\mu_t$ is $\delta$-close to that of $\mu_0$ in the $C^2$-sense. 
We obtain a diffeomorphism $k_t: C(V) \ra C(V)$ close to the identity on a compact set in the $C^2$-sense
so that $k_t(sv) = sk_t(v)$ for $s > 0$ and $v \in V$. 
We transfer $f$ to $C(V)$ by this map. 
Denote the result by $f''_t := f \circ k_t$ where $f''_t(sv) = s^{-n-1}f''_t(v)$ for $s > 0$ and $v \in V'$.
%as well by choosing $k_t$ so that $k_t(sv) = sk_t(v)$ for $s > 0$ and $v \in V$. 
%Also, $f''_t$ is homogeneous of degree $-n-1$. 
(For example, this can be done by deforming the function $f$ only on a section of the radial flow 
and extending.)

%%% July 21, 2013 3:06pm
%Final patching up
%We obtain $f'$ sufficiently close to $f$ on $U \cap V$
%and $f'(tv)=t^{-n}f'(v)$ and $f(tv)=t^{-n}f(v)$ for all $t > 0$.
\begin{itemize}
\item Let $\partial_s V$ be a copy of $\partial V\times \{s\}$ inside the regular neighborhood of $\partial V$ in $U_t$
parameterized as $\partial V \times [-1, 1]$ for $s \in [-1, 1]$.
\item We assign $\partial V= \partial_0 V$. 
\item Let $\partial_{[s_1, s_2]} V$ be the image of $\partial V_t\times \{[s_1, s_2]\}$ inside 
the regular neighborhood of $\partial V$ in $V\cap U'$ for a neighborhood $U'$ of $\clo(U) \cap \torb$.  
\end{itemize} 
We find a $C^\infty$ map $\phi_t: C(U')\cap C(V) \ra \bR_+$ so that $\phi_t(sv)=\phi_t(v)$ for every $s >0$
and $f'_t(v)=\phi_t(v) f''_t(v)$ and $\phi_t$ is very close to the constant value $1$ function.
%Using a partition of unity adopted to $C(U')$ and $C(V')$, we form $\phi$ which differs from $1$ only in $C(W)$
%a compact neighborhood $W:= \partial_{[t_1, 1]} V'$ in $U'$ for $t_1 < 1$.  
By making $f'_t/f''_t$ near $1$ and the derivatives of $f'_t/f''_t$ up to two near $0$ as possible, 
we obtain $\phi_t$ that has derivatives up to order to two as close to $0$ in a compact subset as we wish:
%For example, define $\phi' = f'/f''$ in $W$ and be $1$ outside it. 
This is accomplished by taking a partition of unity functions $p_{1, t}, p_{2, t}$ 
so that 
\begin{itemize}
\item $p_{1, t}=1$ on $C(W)$ for 
\[W:= \partial_{[0, s_1]} V \cup (U'-V)  \hbox{ for } s_1 < 1,\]
\item $p_{1, t} = 0$ on $C(\torb -N)$ for a neighborhood $N$ of $W$ 
in $\partial_{(-1, 1)} V_t \cup (U'-V)$, and 
\item $p_{1, t}+p_{2, t}=1$ identically.
\end{itemize}  
We assume that \[1-\eps < f'_t/f''_t < 1+\eps \hbox{ in } C(U' \cap V),\]
and $f'_t/f''_t$ has derivatives up to order two sufficiently close to $0$
by taking $f'_t$ and $f''_t$ sufficiently close in $C(U') \cap C(V)$ by taking sufficiently small $t$. 
We define \[\phi_t = (f'_t/f''_t-(1-\eps))p_{1, t}+ \eps p_{2, t} + (1-\eps)\] 
as $f'_t$ and $f''_t$ are homogeneous of degree $-n-1$. 
Then $1-\eps < \phi_t < 1+\eps$ and derivatives of $\phi_t$ up to order two are sufficiently close to $0$
as we can see easily from computations. 
Thus, using $\phi_t$ we obtain a Hessian function $f'''_t$ obtained from $f'_t$ and $\phi_t f''_t$ on $C(W)$ 
and extending them smoothly.
We can check the welded function from $f'_t$ and $\phi_t f''_t$ has the desired Hessian properties
since the derivatives of $\phi_t$ up to order two can be made sufficiently close to zero. 
Now we do this for every p-end of $\tilde{\mathcal{O}}$.
%This shows that $\orb$ with $\mu_1$ is properly convex by the existence of the Hessian function. (See Koszul \cite{Kos}).

%(C) For the final step, we show that the ends are admissible. 
Finally, suppose that we only required the radial p-ends to be of generalized lens type. 
Then the arguments above changes for only the case (I) of a generalized radial p-end
where a radial p-end of generalized lens-type become a radial p-end of generalized lens-type. 
Then at $\mu_0$, we obtain $\Omega_{s_0}$ 
that contains $\torb$ and $\Omega_{s_0}$ in an $\eps$-$\bdd$-neighborhood of $\torb$.
(Here we do not need $\Omega_{s_0}$ to be a subset of $\torb$. ) 
We can obtain such a neighborhood by the methods of Section 7.2 (cf. Lemma 7.8) of \cite{endclass}.
As above, by the third item of Lemma \ref{lem:convHaus},
the Hessian functions $f'_t$ defined by equation \eqref{eqn:kv} 
on $C(\Omega_{s_0, t})$ deformed from $C(\Omega_{s_0, t})$ using the above methods in (I)
is very close to the original Hessian function $f_t$ 
in compact subsets of $C(\Omega_{s_0, t} \cap \torb)$ in the $C^2$ topology
by Lemma \ref{lem:KV}. Now we operate as above using possibly not necessarily convex neighborhoods of the radial p-ends
to obtain the Hessian metric for $C(\torb)$. 

The $-(n+1)$-homogeneity gives us the invariance of the Hessian metric under the dilatations and the affine lifts of the holonomy groups. 
(See Chapter 6 of \cite{wmgnote}.)
\end{proof}

\begin{figure}[h]
\centerline{\includegraphics[height=5cm]{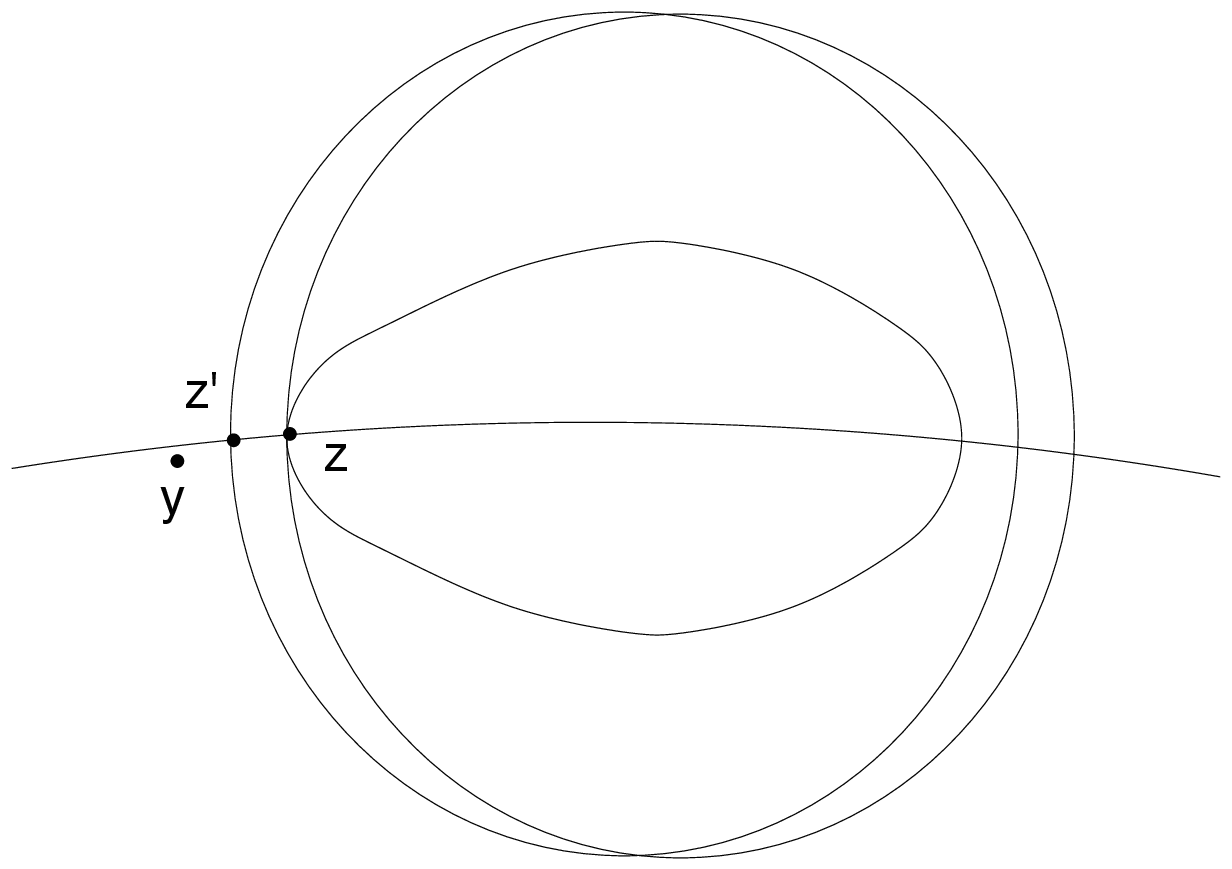}}
\caption{The diagram for Lemma \ref{lem:convHaus}. }
\label{fig:figdual}
\end{figure}

\begin{lemma}\label{lem:KV} 
Let $V$ be a properly convex cone and let $V^*$ be a dual cone. 
Suppose that a Koszul-Vinberg function 
$f_{V^*}(x)$ is defined on a compact neighborhood $B$ of $x$ contained in a convex cone $V$. 
Let $V_1$ be another properly convex cone
containing the same neighborhood. %Choose any integer $s \geq 1$. 
Let $V^*$ has projectivization $\Omega$ and the dual $V_1^*$ of $V_1$ has a projectivization $\Omega_1$. 
For given any integer $s \geq 1$ and $\eps > 0$, there exists $\delta >0 $ so that 
if  the Hausdorff distance between $\Omega$ and $\Omega_1$ is $\delta$-close, 
%if we have $\Omega \subset N_\delta(\Omega_1)$, $\Omega_1 \subset N_\delta(\Omega)$,
%$\Omega - N_\delta(\partial \Omega) \subset \Omega_1$, and $\Omega_1 - N_\delta(\partial \Omega_1) \subset \Omega$, 
then $f_{V^*}(x)$ and $f_{V_1^*}(x)$ are $\eps$-close in $B$ 
in the $C^s$-topology.
\end{lemma}
\begin{proof}
By Lemma \ref{lem:convHaus}, we have 
\begin{align} 
\Omega^* \subset N_\delta(\Omega_1^*), \Omega_1^* \subset N_\delta(\Omega^*), \nonumber \\
(\Omega - N_\delta(\partial \Omega))^* \subset \Omega_1^*, \hbox{ and } \nonumber \\ 
(\Omega_1 - N_\delta(\partial \Omega_1))^* \subset \Omega^*.
\end{align} 
%We now choose $x$ and the neighborhood of $x$ so that they are in 
We choose sufficiently small $\delta > 0$ so that 
\[ B \subset \Omega - N_\delta(\partial \Omega), \Omega_1 - N_\delta(\partial \Omega_1).\]
%\[B \subset N_\delta(\Omega^*), N_\delta(\Omega_1^*), (\Omega - N_\delta(\partial \Omega))^*, (\Omega_1 - N_\delta(\partial \Omega_1))^*.\]
Since the integral is computable from an affine hyperspace meeting $V^*$ and $V_1^*$ in bounded precompact 
convex sets and $e^{-\phi(x)}$ and its derivatives for $\phi$ in the domains are uniformly bounded, 
the integrals and their derivatives are estimable from each other, the result follows by
taking the Hausdorff distances of $\Omega^*$ and $\Omega_1^*$ sufficiently small. 
(See the proof of Theorem 6.4 of \cite{wmgnote}.) 
\end{proof}

Recall the standard elliptic metric $\bdd$ of $\bR P^n$. We also have the elliptic metric $\bdd$ on $\bR P^{n *}$, \index{elliptic metric} 
denoted by the same letter. 
Define the {\em thickness} of a properly convex domain $\Delta$ is given as 
\[\min\{ \max\{\bdd(x, \Bd \Delta)|x \in \Delta\}, \max\{d(y, \Bd \Delta^*)|y \in \Delta^*\}\} \]
for the dual $\Delta^*$ of $\Delta$. 

\begin{lemma}\label{lem:convHaus} 
Let $\Delta$ be a properly convex open domain in $\bR P^n$ and $\Delta^*$ its dual in $\bR P^{n *}$. 
Let $\eps$ be a positive number less than the thickness of $\Delta$. Then the following hold: 
\begin{itemize}
\item $N_\eps(\Delta) \subset (\Delta^* - \clo(N_\eps(\Bd \Delta^*))^*$.
%\item We can replace $\Delta$ and $\Delta^*$ in the above equation. 
\item If two properly convex open domains $\Delta_1$ and $\Delta_2$ are of Hausdorff distance $< \eps$ for 
$\eps$ less than the thickness of each $\Delta_1$ and $\Delta_2$, 
then $\Delta^*_1$ and $\Delta_2^*$ are of Hausdorff distance $< \eps$.
\item Furthermore, if 
$\Delta_2 \subset N_{\eps'}(\Delta_1)$ and $\Delta_1 \subset N_{\eps'}(\Delta_2)$ for $0< \eps'< \eps$, 
then we have $\Delta_2^* \supset \Delta_1^* - \clo(N_{\eps'}(\Bd \Delta_1^*))$
and $\Delta_1^* \supset \Delta_2^* - \clo(N_{\eps'}(\Bd \Delta_2^*))$.
\end{itemize} 
\end{lemma} 
\begin{proof} 
Using the double covering map $\SI^n \ra \bR P^n$ and $\SI^{n *} \ra \bR P^{n *}$ of unit spheres in 
$\bR^{n+1}$ and $\bR^{n+1 *}$,
we take components of $\Delta$ and $\Delta^*$. It is easy to show that the result for properly convex open 
domains in $\SI^n$ and $\SI^{n*}$ is sufficient. 

For elements $\phi \in \SI^{n*}$, and $x \in \SI^{n}, \SI^n$, 
we say $\phi(x) < 0$ if $f(v) < 0$ for $\phi =[f], x =[v]$ for $f \in \bR^{n+1 *}, v \in \bR^{n+1}$. 

For the first item, let $y \in N_\eps(\Delta)$. Suppose that $\phi(y) < 0$ for 
\[\phi \in \clo((\Delta^* - \clo(N_\eps(\Bd \Delta^*))) \ne \emp.\] 
Since $\phi \in \Delta^*$, the set of positive valued points of $\SI^n$ under $\phi$ is an open hemisphere $H$
containing $\Delta$ but not containing $y$. 
The boundary $\Bd H$ of $H$ has a closest point $z \in \Bd \Delta$ 
of distance $\leq \eps$. The closest point $z'$ on $\Bd H$ is in $N_\eps(\Delta)$ 
since $y$ is in $N_\eps(\Delta) - H$ and $z'$ is closest to $\Bd \Delta$.
The great circle $\SI^1$ containing $z$ and $z'$ are perpendicular to $\Bd H$
since $\ovl{zz'}$ is minimizing lengths.  
Hence  $\SI^1$ passes the center of the hemisphere.
One can push the center of the hemisphere on $\SI^1$
until it becomes a supporting hemisphere to $\Delta$. The corresponding $\phi'$ is in $\Bd \Delta^*$ and 
the distance between $\phi$ and $\phi'$ is less than $\eps$. This is a contradiction. 
Thus, the first item holds (See Figure \ref{fig:figdual}.)

%Conversly, suppose $ y$ is not in $N_\eps(\Delta)$. Then $\phi(y) < 0$ for some $\phi \in \clo(\Delta^*)$. 
%Moreover, $\phi \not\in N_\eps(\Bd \Delta^*))$ 
%Thus, $y$ is not in the second set. 

%For the second item, we invoke the fact that the dual of the dual space is the original space.

For the final item, we have that 
\[\Delta_2 \subset N_{\eps'}(\Delta_1),  \Delta_1 \subset N_{\eps'}(\Delta_2) \hbox{ for } 0< \eps' <  \eps.\] 
Hence, $\Delta_2 \subset (\Delta_1^* - \clo(N_{\eps'}(\Bd \Delta_1^*))^*$:
Thus, $\Delta_2^* \supset \Delta_1^* - \clo(N_{\eps'}(\Bd \Delta_1^*)$,
which proves the third item, 
and so $N_\eps(\Delta_2^*) \supset \Delta_1^*$ and conversely.
The second item follows.
\end{proof}

%%% July 23 3:00pm 2013

\begin{proof}{[The proof of Theorem \ref{thm:conv}] }
Suppose that $\mathcal{O}$ has an SPC-structure $\mu$ with generalized admissible ends. We will show that a sufficiently 
close structure $\mu_s$ that has generalized admissible ends is also SPC. 
Let $h': \pi_1(\orb) \ra \SL_\pm(n+1, \bR)$ be the lift of the holonomy homomorphism corresponding to $\mu_s$. 

Let $\mathcal{O}:= C(\torb)/h_s(\pi_1(\orb))$ with $C(\torb)$ as the universal cover. 
Let $\torb_s$ denote $\torb$ with $\mu_s$. One applies special affine suspension to obtain 
an affine orbifold $\orb \times \SI^1$. (See Section \ref{sub:asusp}.)
The universal cover is still $C(\torb)$ and has a corresponding affine structure $\tilde \mu_s$. 
We denote $C(\torb)$ with the lifted affine structure of $\tilde \mu_s$ by $C(\torb)_s$. 
Recall the projective completion $\hat{C}(\torb)_s$ of $C(\torb)$. This is a completion of $C(\torb)$
the path metric induced from the pull-back of the standard Riemannian metric on $\bR^{n+1}$ by the developing map 
$D_s$ of $\tilde \mu_s$.
The developing maps always extend to one on $\hat(C)(\torb)_s$ which we denote by $D_s$ again. 
(See \cite{cdcr1} and \cite{psconv} for details.) 
%Let $\hat{C}(\torb)$ denote the completion of $C(\torb)$ under the standard 
%Euclidean metric of $\bR^{n+1}$. 

By Proposition \ref{prop:openess}, an affine 
suspension $\tilde \mu_s$ of $\mu_s$ also have a Hessian function $\phi$. 
The Hessian metric $D d\phi$ is invariant under affine automorphism groups of $C(\torb)$ 
by construction. 
We prove that $\tilde \mu_s$ is properly convex, which will show $\mu_s$ is properly convex:

Suppose that $\tilde \mu_s$ is not convex. Then there exists a triangle imbedded in $\hat{C}(\torb)_s$
with points in the interior of an edge in the limit set $\Lambda_s:=\hat{C}(\torb)_s - C(\torb)_s$.
We can move the triangle so that the interior of an edge $l$ has a point $x_\infty$ in $\Lambda_s$
and $D_s(l)$ does not pass the origin.
We form a parameter of geodesics $l_t$, $t \in [0, \eps]$ in the triangle so that $l_0=l$ and $l_t \subset C(\torb)$ 
is close to $l$ in the triangle. 
 (See Theorem A.2 of \cite{psconv} for details.)

%Now, we can lift such a triangle and the segment to
%$\hat{\mathcal{O}}_s$ and they correspond under the radial projection. 
%We use the same notation for the lift and the image. 

%There is a minimum length $k$ under $d'$ in any arc with end points in different components of 
%the suspension $U'$ of the a neighborhood  of the union of ends in $\mathcal{O}_s$
%but is not homotopic into the neighborhood. 
%Let $\tilde U$ denote the inverse image in $C(\torb)$ of the union $U$ of mutually disjoint p-end neighborhoods  of $\torb$.
%Suppose that $l_0 - \tilde U$ has infinitely many components of arcs that are not homotopic into $\tilde U$ with 
%their endpoints fixed. Then $l_0-\tilde U$ is infinitely long
%and the length of $l_t$ goes to $\infty$ as $t\ra 0$ as $l_t$ approximates $l$ in for any given sufficiently large closed
%subarc of $l_0$ as closely as possible. Moreover, we can assume that the direction of $l_t$ are constant when 
%developed. 

Let $p, q$ be the endpoints of $l$. Then the Hessian metric is $D^s d \phi$ for a function $\phi$ defined on 
$C(\torb)_s$. And $d \phi |p$ and $d \phi |q$ are bounded, where $D^s$ is the affine connection of $\mu_s$.
This should be true for 
$p_t$ and $q_t$ for sufficiently small $t$ uniformly. 
Let $u$, $u \in [0, 1]$, be the affine parameter of $l_t$, i.e., $l_t(s)$ is a constant speed line in $\bR^{n+1}$ when developed. 
We assume that $u \in (\eps_t, 1-\eps_t)$ parameterize $l_t$ for sufficiently small $t$
where $\eps_t \ra 0$ as $t \ra 0$
and $d l_t/ds = v$ for a parallel vector $v$.
The function $D^s_vd_v \phi( l_t(u))$ is uniformly bounded since its integral $d_v \phi( l_t(u))$ is strictly 
increasing by the strict convexity and converges to certain values as $u \ra \eps_t, 1-\eps_t$. 

Since  \[ \int_{\eps_t}^{1-\eps_t} D^s_vd_v \phi( l_t(s))du = d\phi(p_t)(v) - d\phi(q_t)(v), \]
the function $\sqrt{D^s_vd_v \phi( l_t(u))}$ is also integrable and have a bounded integral by Jensen's inequality. 
This means that the length of $l_t$ is bounded.  

Let $U$ be a union of disjoint end-neighborhoods of $\orb$; 
$U$ correspond to an inverse image $\tilde U$ in $\torb$ and 
to $C(\tilde U)_s$ the inverse image in $C(\torb)_s$. 
The minimum distance 
between components of $U$ is bounded below since the metric is invariant under dilatations
$x \mapsto tx$ in $C(\torb)_s$. 
If $l$ meets infinitely many components of $C(\tilde U)_s$, then the length is infinite because of this. 

As $t \ra 0$, the number is thus bounded, 
$l$ can be divided into finite subsections, each of which meets 
one component of $C(\tilde U)_s$. 
Any subarc of each with end points in the boundary of a component $C_{1, s}$ 
of $C(\tilde U)_s$ is homotopic into a component $C_{1, s}$
with endpoints fixed. 

%Now we go back to $\tilde{\mathcal{O}}$. 
%As above denote by $\tilde U$ the inverse image  in $\tilde{\mathcal{O}}$ of the union of end neighborhoods. 
Let $\hat l$ be the subsegment of $l$ in $C(\torb)_s$ containing $x_\infty$ in 
the ideal boundary and meeting only one component $C(U_1)_s$ of $C(\tilde U)_s$
with $\Bd s \in \Bd C(U_1)_s$. Let $\hat l_t$ be the subsegment of $l_t$ so that 
the endpoints of $\hat l_t$ converges to those of $\hat l$ as $t \ra 0$. %\marginpar{$s_t$ needed?}
Let $p'$ and $q'$ be the endpoint of $\hat l$. 

Suppose that $C(U_1)_s \subset C(\torb)_s$ corresponds to a p-end neighborhood $U'_1$ of radial of lens-type 
or horospherical type in $\torb_s$.
and $x_\infty$ is on a line corresponding to the p-end vertex of the radial lines of
$U'_1$. % and $s$ should coincide with radial lines. 
%Since $p'$ is in $U_1$ and the set of rays in $U_1$ form 
%a convex subset of the projective space of rays, it follows that $q'$ is outside $U_1$. 
%This is clearly a contradiction. 
We project to $\SI^n$ from by the projection $\Pi:\bR^{n+1} -\{O\} \ra \bR P^n$. 
Now suppose that $\Pi(x_\infty)$ is in the middle of the radial line. Then the interior of the triangle 
is transversal to the radial lines. Since our end orbifold is convex, there cannot be such a line 
with a single interior point in the ideal set. 

If $C(U_1)_s$ is the inverse image in $C(\torb)_s$ of 
a p-end neighborhood $U_1$ in $\torb_s$ of totally geodesic lens-type, then clearly there is no such a segment $l$ similarly. 

This is again a contradiction. Therefore, $\tilde{\mathcal{O}}$ is convex. 

Finally,  for sufficiently small deformations, the convex real projective structures are properly convex. 
If not, then there are sufficiently small deformed convex real projective structures which are not properly convex and 
hence their holonomy homomorphism is reducible. By taking limits, the original one has to be reducible as well. 
However, we assumed that it was  irreducible. Since the subspace of reducible representation is closed, 
we see that there is an open set of irreducible properly convex projective structures near the initial one $\mu$.

Suppose now that $\mathcal{O}$ with $\mu$ is strictly SPC with admissible ends. 
The relative hyperbolicity of $\tilde{\mathcal{O}}$ with respect to the p-ends is stable under small deformations
since it is a metric property invariant under quasi-isometries by Theorem \ref{thm:relhyp1}.

%This completes the proof of Theorem \ref{thm:conv}.
The irreducibility and the stability follow since these are open conditions. 
Also, the ends are admissible. 

%This completes the proof of Theorem \ref{thm:conv} 
\end{proof}

Theorem \ref{thm:conv2} also follows similarly. 
Hence Corollary \ref{cor:conv} and \ref{cor:conv2} follow by Theorem \ref{thm:A}.

%% May 17 11:04pm 2014 p. 119

\chapter{The closedness of convex real projective structures}\label{sec:closed}

We recall  $\rep_E^s(\pi_1(\mathcal{O}), \PGL(n+1, \bR))$ the subspace of stable irreducible characters
of $\rep_E(\pi_1(\mathcal{O}), \PGL(n+1, \bR))$
which is shown to be an open subset of a semialgebraic set in Section \ref{sub:semialg}, 
and denote by $\rep_{E, u, ce}^s(\pi_1(\mathcal{O}), \PGL(n+1, \bR))$ the subspace of stable irreducible characters
of $\rep_{E, u, ce}(\pi_1(\mathcal{O}), \PGL(n+1, \bR))$, an open subset of a semialgebraic set. 

In this section, we will need to discuss $\SI^n$ but only inside a proof.

%An end is {\em permanently properly convex} 
%if the fundamental group is isomorphic to a finite extension of $\bZ^l \times \Gamma_1 \times \dots \times \Gamma_k$ where 
%$l, l+1 =k \geq 1$ and $\Gamma_i$s are hyperbolic groups or if the group itself is hyperbolic. 
%In this case, the end orbifold is always properly convex by Theorem \ref{thm:  }. 

%\begin{proposition}\label{prop:sSPCd} 
%\end{proposition}

%%% 4:00pm Avril 5
%% Closedness THis is too general....
\begin{theorem} \label{thm:closed1} 
Let $\mathcal{O}$ be a noncompact strongly tame SPC $n$-orbifold with %convex end fundamental group condition
%and 
generalized admissible ends and satisfies {\rm (IE)} and {\rm (NA)}. Assume $\partial \orb =\emp$. 
Assume that every finite index subgroup of $\pi_1(\mathcal{O})$ has no nontrivial nilpotent normal subgroup.
Then the following hold\,{\rm :} 
\begin{itemize}
\item The deformation space $\CDef_{E, u, ce}(\mathcal{O})$ of SPC-structures on $\mathcal{O}$ with generalized 
admissible ends maps under $\hol$ 
homeomorphically to a union of components of $\rep_{E, u, ce}^s(\pi_1(\mathcal{O}), \PGL(n+1, \bR))$.
\item %Assume that $\pi_1(\mathcal{O})$ has no nontrivial nilpotent normal subgroup.
The deformation space $\SDef_{E, u, ce}(\mathcal{O})$ of SPC-structures on $\mathcal{O}$ with admissible ends maps under $\hol$
homeomorphically to the union of components of  $\rep_{E, u, ce}^s(\pi_1(\mathcal{O}), \PGL(n+1, \bR))$.
\end{itemize}
%Furthermore each of these is the union of  components of $\rep_{E, u, ce}(\pi_1(\mathcal{O}), \PGL(n+1, \bR))$. 
\end{theorem}
\begin{proof} 
Define $\widetilde{\CDef}_{E, u, ce}(\mathcal{O})$ to be the inverse image of $\CDef_{E, u, ce}(\mathcal{O})$ in 
$\widetilde{\Def}_{E, u, ce}(\mathcal{O})$. 
We show that \[\hol: \widetilde{\CDef}_{E, u, ce}(\mathcal{O}) \ra \Hom_{E, u, ce}^s(\pi_1(\mathcal{O}), \PGL(n+1, \bR))\]
is a homeomorphism onto a union of components. This will imply the results. 

Suppose that the map is not injective. Then there exists 
a homomorphism $h: \pi_1(\mathcal{O}) \ra \PGL(n+1, \bR)$ and properly convex
open domains $\Omega_1$ and $\Omega_2$ where $h(\pi_1(\mathcal{O}))$ acts 
properly on so that $\Omega_1/h(\pi_1(\mathcal{O}))$ and $\Omega_2/h(\pi_1(\mathcal{O}))$ are both diffeomorphic to $\mathcal{O}$
by diffeomorphisms inducing $h$. 

Suppose that $\Omega_1$ and $\Omega_2$ are distinct in $\bR P^n$. 
We claim that $\Omega_1$ and $\Omega_2$ are disjoint:
Suppose not. Then let $\Omega'$ be the intersection $\Omega_1 \cap \Omega_2$ 
where $\Gamma:=h(\pi_1(\mathcal{O}))$ acts. 
Each p-end fundamental group also acts on $\Omega'$ also. We can form a topological space 
$\Omega'/\Gamma$ with end neighborhood system. 
Since $\Omega_1, \Omega_2,$ and $\Omega'$ are all $n$-cells, 
the set of p-ends of $\Omega_1$, the set of those of $\Omega_2$, and the set of those of $\Omega'$ are 
in one-to-one correspondences since the end groups uniquely determines the end vertex and ideal totally geodesic boundary 
inside $\Omega_1, \Omega_2,$ and $\Omega'$ respectively by the uniqueness condition for each end holonomy group. 
%$\clo(\Omega_1) \cup \clo(\Omega_2)$. 
(We need to see the orbits of points of $\Omega'$ under the end fundamental group.)
Also, using concave p-end neighborhoods for radial p-ends, lens p-end neighborhoods for totally geodesic p-ends, and horoball p-end 
neighborhoods of p-ends, 
we verify easily that a p-end neighborhood of $\Omega_1$ exists if and only if a p-end neighborhood of $\Omega_2$ 
exists and their intersection is a p-end neighborhood of $\Omega'$. 
By taking a torsion-free finite-index subgroup $\Gamma'$ of $\Gamma$ using Selberg's Lemma, 
$\Omega'/\Gamma'$ is a closed submanifold in $\Omega_1/\Gamma'$ and in $\Omega_2/\Gamma'$. 
Thus, $\Omega_1/\Gamma', \Omega_2/\Gamma',$ and $\Omega'/\Gamma'$  are all homotopy 
equivalent relative to the union of disjoint end-neighborhoods. 
 The map has to be onto in order for the map to be a homotopy equivalence as we can show using 
 relative homology theories, 
 and hence, $\Omega'=\Omega_1=\Omega_2$. 

Suppose now that $\Omega_1$ and $\Omega_2$ are disjoint. 
Each corresponding pair of the p-end neighborhoods share p-end vertices or have antipodal p-end vertices. 
Since $\Omega_1$ and $\Omega_2$ are disjoint but each pair of the p-ends have same p-end holonomy groups.  
Now $\clo(\Omega_1)\cap \clo(\Omega_2)$ or $\clo(\Omega_1) \cap {\mathcal{A}}(\clo(\Omega_2)$  
is a compact properly convex subset $K$ of dimension $< n$ 
and is not empty since the fixed points of the p-ends are in it. 
 The minimal hyperspace containing $K$ is a proper subspace and is invariant under $h(\pi_1(\mathcal{O}))$. 
 This contradicts the irreducibility. 
 
 Hence, this proves that $\hol$ is injective. 
 $\hol$ is an open map by Theorem \ref{thm:conv} and Theorem A. 
 Actually, this show that a strongly tame SPC-orbifold of given end types is uniquely determined by each holonomy group.
 
 %Suppose that a convex domain is not properly convex. Then by \cite{ChCh}, such a domain contains 
 %a great sphere of dimension $\geq 0$ in its boundary. If a representation $h$ acts on this set, then 
% $h$ is reducible. 
 
%%% May 18 12:28
 To show that the image is of $\hol$ is closed, the subset of 
\[ \Hom_{E, u, ce}^s(\pi_1(\mathcal{O}),\PGL(n+1,\bR))\] corresponding to 
elements in $\widetilde{\CDef}_{E, u, ce}(\mathcal{O})$ is closed. 
 Let $(\dev, h_i)$ be a sequence of development pairs
 so that we have $h_i \ra h$ algebraically. 
 Let $\Omega_i =\dev(\torb)$ denote the properly convex domains. 
 The limit $h$ is a discrete representation 
 by Lemma 1.1 of Goldman-Millson \cite{GM}.
 The sequence $\clo(\Omega_i)$ also converges to a compact convex set
 $\Omega$ up to choosing a subsequence where $h(\pi_1(\mathcal{O}))$ acts on
 as in \cite{CG}. 
 If $\Omega$ is not properly convex or have the empty interior, 
 $h$ is reducible. 
 Thus, $\Omega^o$ is not empty and is properly convex. 
 As in \cite{CG}, since $\Omega^o$ has a Hilbert metric, 
 $h(\pi_1(\mathcal{O}))$ acts on $\Omega^o$ properly discontinuously. 
  
 The condition of the generalized lens-shapedness is a closed condition in the 
 \[ \Hom_{E, u, ce}^s(\pi_1(\mathcal{O}),\PGL(n+1,\bR))\]
 as we defined above.  
 The ends of the orbifold $\orb':=\Omega^o/h(\pi_1(\orb))$ are 
 generalized admissible since the holonomy conditions of the ends insure this 
 by Theorems 7.9 and 7.11 of \cite{endclass}. 
 We can deform $\orb'$ using the openness of $\hol$ by Theorem \ref{thm:conv}. 
 We can find a deformed orbifold 
 $\orb''$ that has a holonomy $h_i$ for some large $i$. 
 Hence, $\orb''$ is diffeomorphic to $\orb$ since they share the same open domain as universal cover 
 by the uniqueness above for each holonomy group. 
 By openness of $\hol$ for $\orb'$, $\orb''$ is diffeomorphic to $\orb'$. 
 Hence, $\orb'$ is diffeomorphic to $\orb$.  
 
 Therefore, we conclude that $\widetilde{\CDef}_{E, u, ce}(\mathcal{O})$
goes to a closed subset of \[\Hom_{E, u, ce}^s(\pi_1(\mathcal{O}), \PGL(n+1, \bR)).\]
These imply the first item. 
 
 Now, we go to the second item. %The injectivity and the openness are proved as above for $\CDef_{E, u, ce}(\orb)$. 
 Define $\widetilde{\SDef}_{E, ce}(\mathcal{O})$ to be the inverse image of $\SDef_{E, ce}(\mathcal{O})$ in 
$\widetilde{\Def}_{E, ce}(\mathcal{O})$. 
We show that \[\hol: \widetilde{\SDef}_{E, ce}(\mathcal{O}) \ra \Hom_{E, u, ce}^s(\pi_1(\mathcal{O}), \PGL(n+1, \bR))\]
is a homeomorphism onto a union of components. 
 %Now, we study $\SDef_{E, u, ce}(\mathcal{O})$. 
 Theorem \ref{thm:conv} shows that $\hol$ is a local homeomorphism 
 to an open set. The injectivity of $\hol$ follows the same way as in the above item. 
 
 We now show the closedness. 
By Theorem \ref{thm:relhyp}, $\pi_1(\orb)$ is relatively hyperbolic with respect to the admissible end fundamental groups. 
Let $h$ be the limit of a sequence of holonomy representations $h_i:\pi_1(\orb) \ra \PGL(n+1, \bR)$. 
As above we obtain $\Omega$ as the limit of $\clo(\Omega_i)$ where $\Omega_i$ is the image of 
the developing map associated with $h_i$. $\Omega$ is properly convex and $\Omega^o$ is not empty.  
 Since $h$ is irreducible and acts on $\Omega^o$ properly discontinuously, it follows that 
 $\Omega^o/h(\pi_1(\mathcal{O}))$ is an orbifold $\mathcal{O}'$ homotopy equivalent to $\mathcal{O}$
 and with generalized admissible ends as above. 
 By Theorem \ref{thm:converse}, $\mathcal{O}'$ is a strict SPC-orbifold with admissible ends. 
 The rest is the same as above. 
\end{proof}

\begin{remark}[Thurston's example]
We remark that without the end controls we have, there might be counter-examples as we can 
see from the examples of geometric limits differing from algebraic limits
for sequences of hyperbolic $3$-manifolds. (See Anderson-Canary \cite{AC}.) 
\end{remark}

%\marginpar{Is this the first time?}
Recall that an affine subspace of $\SI^n$ is an open subspace mapping to an affine subspace of $\bR P^n$. 
Let us choose an affine subspace $A$ and give a coordinate system on $A$. 
Let $B_R$ denote the ball of radius $R$ with the center at the origin. 

Recall the dual sphere $\SI^{n \ast}$ as the space $\bR^{n+1 \ast} -\{O\}/\sim$ where 
two vectors are equivalent iff they are positive scalar multiples of each other. 
Given a properly convex open domain $\Omega \subset \SI^n$, 
the {\em dual domain} $\Omega^*$ in $\SI^{n \ast}$ is the set 
\[\{[f] | f\in \bR^{n+1, \ast}, f(x) > 0 \hbox{ for all } x \in \clo(C_\Omega) - \{O\} \}.\]
$\Omega^*$ is properly convex and open also.

%%% April 14 2014 4:39pm

We can drop the superscript $s$ from the above space. 
Hence, the components consist of stable irreducible characters. This is a stronger result. 

\begin{corollary} \label{cor:closed1} 
Let $\mathcal{O}$ be a noncompact strongly tame SPC $n$-dimensional orbifold with admissible ends and satisfies {\rm (IE)} and {\rm (NA)}. 
Assume $\partial \orb =\emp$. 
Assume that no finite-index subgroups $\pi_1(\mathcal{O})$ has a nontrivial nilpotent normal subgroup.
%and every end fundamental group has an infinite index. 
%If the ends of $\mathcal{O}$ are permanently properly convex, then 
Then $\hol$ maps the deformation space $\CDef_{E, u, ce}(\mathcal{O})$ of SPC-structures on $\mathcal{O}$ homeomorphic to 
a union of components of 
\[\rep_{E, u, ce}(\pi_1(\mathcal{O}), \PGL(n+1, \bR)).\]  % no superscript s here!
The same can be said for $\SDef_{E, u, ce}(\mathcal{O})$.
\end{corollary}
\begin{proof} 
We will show that the image of $\widetilde{\CDef}_{E, u, ce}(\mathcal{O})$ under $\hol$ in 
\[\Hom^s_{E, u, ce}(\pi_1(\mathcal{O}), \PGL(n+1, \bR))\]
is closed and consists of stable irreducible characters. %Since the image is open by Theorem \ref{thm:conv}, we are done. 

It is sufficient to show the closedness of the subspace of \[\Hom_{E, u, ce}(\pi_1(\mathcal{O}), \PGL(n+1, \bR))\] of holonomy 
homomorphisms of elements of $\CDef_{E, u, ce}(\mathcal{O})$. 
This again follows from the closedness of the subspace of \[\Hom_{E, u, ce}(\pi_1(\mathcal{O}), \SLpm)\] of lifted holonomy 
homomorphisms of elements of $\CDef_{E, u, ce}(\mathcal{O})$. 

Using Theorem \ref{thm:lifting}, 
let $h_i : \pi_1(\orb) \ra \SL_\pm(n+1, \bR)$ be a sequence of holonomy homomorphisms of real projective structures
corresponding to liftings of elements of $\CDef_{E, u, ce}(\mathcal{O})$.
Let $\Omega_i$ be the sequence of associated properly convex domains in $\SI^n$ and $\Omega_i/h_i(\pi_1(\orb))$ is 
diffeomorphic to $\orb$ and has the structure that lifts an element of $\CDef_{E, u, ce}(\mathcal{O})$. 
We assume that $h_i \ra h$ algebraically, i.e., for a fixed set of generators $g_1, \dots, g_m$ of $\pi_1(\orb)$, 
$h_i(g_j) \ra h(g_j) \in \SLpm$ as $i \ra \infty$. We will show that $h$ is a lifted holonomy homomorphism of 
an element of $\CDef_{E, u, ce}(\mathcal{O})$, and hence $h$ is stable and strongly irreducible. 

We take a dual domain $\Omega_i^* \subset \SI^{n \ast}$. Then the sequence $\{\clo(\Omega_i^*)\}$ also converges geometrically 
to convex compact set $K^*$.
%We show that $K$ is a properly convex domain with nonempty interior.
If $K$ has an empty interior and properly convex, then we can show easily that $K^*$ has a nonempty interior since for any 
$1$-form $\alpha$ positive on the cone $C_K$, any sufficiently close $1$-form is still positive on $C_K$. 
Also, if $K^*$ has an empty interior and properly convex, $K$ has a nonempty interior. 
%The converse is also true. 

%%% Oct 31, 9:20 pm
(I) The first step is to show that at least one of $K, K^*$ has nonempty interior.

Suppose that both $K$ and $K^*$ are not properly convex and have empty interior.
If there exists a radial p-end for $\Omega_i$ and the type does not become horospherical, 
then the mc-p-end neighborhood must be in $K$ since this is true for 
all structures in $\Omega_i$ and holonomy homomorphisms  in 
\[\Hom_{E, u, ce}(\pi_1(\orb), \SLpm)\] 
where for each p-end $\tilde E$, there exists a distanced compact set $L$ away from $v_{\tilde E}$, and the lens-cone
$v_{\tilde E} \ast L -\{v_{\tilde E}\} \subset K$ has a nonempty interior. 
If there is a totally geodesic end of lens-type for $\orb$ and the type does not 
become horospherical, then the dual 
$\Omega_i^*$ and $K^*$ have nonempty interiors. These are contradictions. 

Suppose now that $\Omega_i$ has only horospherical ends or the sequence of the end holonomy groups 
converge to horospherical ones. Put by choosing $\Omega_i^*$ if necessary, 
we can assume that there is a p-end vertex $v_i \in \Bd \Omega_i$ where 
$\pi_1(\tilde E)$ is the fixed associated p-end fundamental group. 
Moreover $v_i = v$ for a fixed vertex $v$ by conjugating by a bounded sequence of projective automorphisms. 
The generators $g_{j} \in \pi_1(\tilde E)$ for $j=1, \dots, m$ for some $m$ are convergent. 
There exists a fundamental domain $F_i$ in $S_{\tilde E_i}=R_{v_i}(\Omega_i)$. 
Since $\{h_i(g_j)\}$ is convergent for each $j$, we can choose $F_i$ so that $\{F_i \subset \SI^{n-1}_v\}$ is geometrically convergent. 
We choose a great segment $s_i$ with vertex $v_i$  in a direction of $F_i$. 
%Let $s_i$ be any segment from $v_i$ of $\bdd$-length $\pi$.
\begin{equation} \label{eqn:li0}
\hbox{For } l_i:= \Omega_i \cap s_i, \, \bdd\hbox{-length}(l_i) \ra 0:
\end{equation}
If not, then the convex hull $C(h_i(\pi_1(\tilde E_i))(s_i)) \subset \Omega_i$ 
contains balls of fixed radius since $h_i(g_{j})$ are convergent to an element of 
a parabolic group in a copy $P$ of $\PO(n, 1)$
for each $j$ as $i \ra \infty$. 

Let $h(\pi_1(\tilde E))$ be the algebraic limit $h_i(\pi_1(\tilde E))$. 
Then $P \cap h(\pi_1(\tilde E))$ is a lattice in $P$. 

%We can find a basis of $\bR^{n+1}$ so that $P$ is of form 
%\[
%\left(
%\begin{array}{ccc}
%1  & 0  & 0  \\
%v^T  & \Idd_{n-1}  & 0  \\
%||v||^2/2  & v   &  1  
%\end{array}
%\right)
%\]
%where $v \in \bR^{n-1}$ correspond to $[1,0, \dots, 0]$. 
%$h_i(\pi_1(\tilde E))$ is a lattice in a copy of $P$. Since we showed above that $h$ is faithful, 
%$h(\pi_1(\tilde E))$ is a lattice in $P$. There is a flag $S^o \subset S^1 \subset \cdots \subset S^{n-1}$ of 
%subspaces $S^i$ of dimension $i$ where $P$ acts on. Each open hemisphere in $S^i - S^{i-1}$ for $i=0, \dots, n-1$. 
%$P$ acts transitively on each hemisphere. Being a lattice, the orbit of $h(\pi_1(\tilde E))$ has 
%as the convex hull each hemisphere. 

\begin{lemma} \label{lem:horob2} 
Let $v$ be a fixed point of $P$ and let $L$ be a lattice in $P$.
Let $H$ be a $P$-invariant hemisphere with $v$ in the boundary, and let $l$ be the maximal perpendicular line 
with endpoints $v$ and $v_-$. 
Then %there exists a $P$ invariant hemisphere $H$ with a maximal open line $l$ in $H$ with endpoints $v$ and $v_-$
there exists a finite subset $F$ of $L$ 
so that for any point $x \in l$ and a $\bdd$-perpendicular hyperspace at $x$ bounding a closed hemisphere $H_1$  
$I_x :=\bigcap_{g\in F} g(H_x) \cap H$ is a properly convex domain, and as $x \ra v$, 
$I_x$ geometrically converges to $\{v\}$. 
\end{lemma}
\begin{proof} 
If $F$ is large enough, then $\{g(H_x), H\}$ is in a general position. 
The last fact follows by considering the set of outer normal vectors of $\{g(\partial H_x)\}$ in an affine space where $H$ is a half-space. 
\end{proof}

%Let $g$ be any element of $\pi_1(\orb)$ so that $h_i(g)(v_i) \ne v_i$ by (IE). 
%Then let $s_i$ contain $v_i$ and $h_i(g)(v_i)$. Let $\bdd$-length of $l_i$ 

Let $H$ denote the $P$-invariant hemisphere containing $K$. We assume that $\Omega_i \subset H$ 
and recall that radial p-end vertices are fixed to be $v$. 
We assume that the direction of $l$ for $h(\pi_1(\tilde E))$ is in $F_i$ always. 

%Since the sequence of $\bdd$-lengths are going to zero, and this is true for any $s_i$ in the direction of $F_i$, 
%and $h_i(g_{j})$ is converging to a bounded parabolic element in a copy of $\PO(n, 1)$ as $i \ra \infty$, 
%it follows that $\{\Omega_i\}$ is converging to a point:
%We can prove this using Lemma \ref{lem:horob}: Let $I$ be a large integer. 
Let $\eps_i$ be the maximum $\bdd$-length of 
a maximal segment $s'_i$ in $\Omega_i$ from $v_i$ in direction of $F_i$ for $i \geq I$. 
Let $F'_i$ denote the set of endpoints of the maximal segments in $\Omega_i$ in direction of $F_i$. 
%Let $s''_i$ have the maximal $\bdd$-length from $v_i$ among such $s_i$s. 
Then $\eps_i \ra 0$ by the above argument. 
A hyperplane perpendicular to $l$ at $x'_i\in l$ bounds a closed hemisphere $H'_i$ containing $F'_i$. 
%We choose $x_i$ to be closed possible one to $v$. 
$\bdd(v, x_i) = \delta_i$ satisfy $\{\delta_i \} \ra 0$
since otherwise equation \eqref{eqn:li0} does not hold. 
By Lemma \ref{lem:horob2}, there is a finite set $F \subset \pi_1(\tilde E)$
so that $\hat K_i := \bigcap_{g\in F} h_j(g)(H'_i) \cap H$ is properly convex for sufficiently large $j$ 
since $h_j(g) \ra h(g), g\in F$. 
This set contains $\clo(\Omega_i)$ since $H'_i \supset \clo(\Omega_i)$. 
As $x'_i \ra 0$, it follows that $\hat K_i \ra \{v\}$ since $x'_i \ra v$. 
(We just need to show that the normal vectors of $h_i(g)(\partial H'_i), g \in F$ being 
sufficiently larger different from that of $H$. Since $h(g)(\partial H'_i), g \in F$ is very close 
to these for sufficiently large $i$, we are done.)

%We take a small horoball or a lens-cone $B_i$ for $h_i(g_{j})$, $j=1, \dots, m$, 
%that is the convex hull $C(h_i(\pi_1(E))(F'_i))$ in $\clo(\Omega_i)$.  
%so that the $\bdd$-lengths of the maximal segments in the directions of $F_\infty$ are bigger than $\eps_i$. 
%those of any maximal $s'_i$ from $v_i$ in direction of $F_i$ for $i \geq I$ and
%perturbing $B$ to show that $\Omega_i$ is contained in a small diameter ball $B_i$ perturbed from $B$. 
%Then we can take larger $I$. 
%Then $\Omega_i \subset B_i$. We know that $\{B_i\}$ geometrically converges to a singleton. 
Therefore, we conclude that $K$ is a singleton. 

In case, $K$ is a singleton, $K^*$ must be a hemisphere by duality of $\Omega_i$ and $\Omega_i^*$. 
We now conclude that $K$ or its dual $K^*$ has a nonempty interior. 

%May 18 1:40pm .... Need to fix this proof... Not done yet...

Thus, by choosing $h_i^*$ and $h^*$ if necessary, we may assume without loss of generality that $K$ has nonempty interior. 
We will show that $K$ is a properly convex domain and this implies that so is $K^*$. 

(II) The second step is to show $K$ is properly convex.

Assume that $h(\pi_1(\orb))$ acts on a convex open domain $K^o$.
We may assume that $K^o \subset \mathbb{A}$ for an affine subspace $\mathbb{A}$
and $\Omega_i \subset \mathbb{A}$ as well by acting by an orthogonal $\kappa_i \in \SL_\pm(n+1, \bR)$ converging to $\Idd$. 
We can accomplish this by moving $\Omega_i$ into $\mathbb{A}$. 
Also, we may assume that $\mathbb{A}$ contains a unit ball $B_1 \subset \Omega_i$ for all $i$. 
Choose $x_0 \in B_1$ as the origin in the affine coordinates. 

Let $g_1, \dots, g_m$ denote the set of generator of $\pi_1(\orb)$. 
Then by extracting subsequences, we may assume without loss of generality that 
$h_i(g_j) $ converges to $h(g_j)$ for each $j=1, \dots, m$. 

First, 
\begin{equation}\label{eqn:gjC0} 
\bdd(h_i(g_j)(x_0), \Bd \Omega_i) \geq C_0 \hbox{ for a uniform constant } C_0:
\end{equation}  
If not, then there is a sequence of a $\bdd$-length constant segment $s_i$ with an origin $x_0$ 
is sent to the segment $h_i(g_j)(s_i)$ in $\Omega_i$ with 
end point $h_i(g_j)(x_0)$ and lying on the shortest $\bdd$-length segment from $h_i(g_j)(x_0)$ to $\Bd \Omega_i$. 
Thus, the sequence of the $\bdd$-segment $h_i(g_j)(s_i)$ is going to zero. 
This implies that $h_i(g_j)$ is not in a compact subset of $\SL_{\pm}(n+1, \bR)$, a contradiction. 

By estimation from equation \eqref{eqn:gjC0}, 
a uniform constant $C$ satisfies
\begin{equation}\label{eqn:bdbel}
d_{\torb_i}(x_0, h_i(g_j)(x_0)) < C.
\end{equation}
By Benz\'ecri (see C.24 of Goldman \cite{wmgnote}), there exists 
a constant $R_B > 1$ and 
$\tau_i \in  \SL_\pm(n+1, \bR)$ 
so that 
\[ B_1 \subset \tau_i(\Omega_i) \subset B_R. \] 
Now, $\tau_i h_i(\pi_1(\orb)) \tau_i^{-1}$ acts on $\tau_i(\Omega_i)$. 
Then as in the proof of Theorem 7.1 of Cooper-Long-Tillman \cite{CLT3}, 
we obtain that $\tau_ih_i(g_j)\tau_i^{-1}$ for $j=1, \dots, n$ 
are in a compact subset of $\SL_{\pm}(n+1, \bR)$ independent of $i$. 
(Theorem 7.1 of \cite{CLT3} is not enough but their proof is sufficient to show this.)

Therefore, up to choosing subsequences,  we 
have $\tau_i(\Omega_i)$ geometrically converges to a properly convex domain $\hat K$ in $B_R$ containing $B_1$ 
and $\tau_i h(g_j) \tau_i^{-1}$ converges to a holonomy homomorphism $h': \pi_1(\orb) \ra \SL_\pm(n+1, \bR)$. 
And the image of $h'$ acts on a properly convex domain $\hat K$. 

Suppose that the sequence $\{\tau_i\}$ is not bounded. Then $\tau_i = k_i d_i k'_i$ where $d_i$ is diagonal with respect to a standard 
basis of $\bR^{n+1}$ and $k_i, k'_i \in O(n+1, \bR)$ by the KTK-decomposition of $\SLpm$. 
Then the sequence of the maximum modulus of the eigenvalues of $d_i$ are not bounded above. 
We assume without loss of generality $k_i \ra k$ and $k'_i \ra k'$ in $O(n+1, \bR)$. Thus, 
$k'_i h_i(g_j) k^{\prime -1}_i$ converges to $k' h(g_j) k^{\prime -1}$ for $k' \in O(n+1, \bR)$. 
Since $k_i d_i k'_i h_i(g_j) k^{\prime -1} d_i^{-1} k^{-1}_i$ is convergent, and the sequence of 
the norms of $d_i$ is divergent, $\{ d_i k'_i h(\pi_1(\orb) k^{\prime -1}_i d_i^{-1}\}$ converges to a  reducible group
acting on $k^{-1}(\hat K)$. (Because the sequences of some of the entries must become zero under the conjugation by 
a sequence of unbounded diagonal matrices.)
By following Lemma \ref{lem:endspres}, and Theorem \ref{thm:sSPC}, the limit
of $\{ d_i k'_i h(\pi_1(\orb) k^{\prime -1}_i d_i^{-1}\}$
cannot be reducible. 
Therefore the sequence of the norms of $d_i$ is uniformly bounded.

Thus the norm of $d_i$ is uniformly bounded. We may assume without loss of 
generality that $\tau_i$ converges to an element $\tau \in \SLpm$.
We assumed above that $h_i \ra h$. By the above, 
$\clo(\Omega_i) \ra \tau(\hat K)$ and %where $\hat K \subset \SI^n$ is a properly convex compact domain and
$h(\pi_1(\orb))$ acts on $\tau(\hat K)$. 

By the following Lemma \ref{lem:endspres}, we obtain that $\tau(\hat K^o)/h(\pi_1(\orb))$ is a strongly tame  SPC-orbifold
with generalized admissible ends. This completes the proof for $\CDef_{E, u, ce}(\mathcal{O})$.

%Now projecting to $\bR P^n$ we denote by $h_i$, $h$ the projected homomorphism to $\PGL(n+1, \bR)$ and 
%$\clo(\Omega_i) \subset \bR P^n$ and $K \subset \bR P^n$ where $\clo(\Omega_i)$ and $K$ are projected images 
%and $\clo(\Omega_i) \ra K$. 

By the condition on admissibility of the ends, we see that 
\begin{lemma} \label{lem:endspres} 
Assume that no finite index subgroup of $\pi_1(\orb)$ contains a normal infinite nilpotent subgroup. 
Let $h_i \in \Hom(\pi_1(\orb)), \SLpm)$, $\Omega_i$ a properly convex open domain in $\SI^n$, and
let $\Omega_i/h_i(\pi_1(\orb))$ be an $n$-dimensional noncompact strongly tame 
SPC-orbifold with admissible ends and satisfies {\rm (IE)} and {\rm (NA)} for each $i$. 
Assume that the end fundamental groups of $h_i(\pi_1(\orb))$ have fixed types. 
Suppose that $h_i \ra h$ algebraically and $\clo(\Omega_i) \ra K$ for a compact properly convex domain 
$K \subset \SI^n$, $K^o \ne \emp$. 
Then
\begin{itemize}
\item $K^o/h(\pi_1(\orb))$ is an SPC-orbifold  with generalized admissible ends 
to be denoted $\orb_h$ diffeomorphic to $\orb$. 
\item For each p-end $\tilde E$ of the universal cover $\torb_h$ of $\orb_h$, 
$K^o$ has a subgroup $h(\pi_1(\tilde E))$ and $h(\pi_1(\tilde E))$-invariant open set $U_{\tilde E}$
corresponding fixed vertex $v_{\tilde E}$ or a totally geodesic domain $S_{\tilde E}$.  
\item $U_{\tilde E}/h(\pi_1(\tilde E))$ is real projectively diffeomorphic to an end neighborhood, 
is either horospherical or of lens type totally geodesic end neighborhood, or else 
is projectively isomorphic to a concave end neighborhood. 
\item If $\pi_1(\orb)$ is relatively hyperbolic, then $K^o/h(\pi_1(\orb))$ is a strongly tame strict SPC-orbifold with admissible ends. 
\item Also, the $\cR$- or $\cT$-types of ends are preserved. \end{itemize} 
\end{lemma} 
\begin{proof} 
By Goldman-Millson \cite{GM},  $h(\pi_1(\orb))$ is discrete since each finite index subgroup of it has no normal infinite nilpotent subgroup. 
Hence, $K^o/h(\pi_1(\orb))$ is an orbifold  to be denoted $\orb_h$. 

For $h$, let $[h]$ denote the corresponding homomorphism to $\PGL(n+1, \bR)$. 
Since $[h] \in \rep_{E, u, ce}(\pi_1(\mathcal{O}), \PGL(n+1, \bR))$ holds, 
each p-end fundamental group $h(\pi_1(\tilde E))$ acts on a horoball $H \subset \SI^n$, 
a generalized lens-cone, or a totally geodesic hypersurface $S_{\tilde E}$ with a lens-neighborhood $L$.
In the first case, we can choose a sufficiently small horoball inside $K^o$ and in $H$ since the supporting hyperplanes at the vertex of $H$
 must coincide by the invariance under $h(\pi_1(\tilde E))$. 
In the third case, we can find a one-sided lens-neighborhood of $S_{\tilde E}$ in $K^o$ 
since $S_{\tilde E} \subset \Bd K$ as 
the corresponding sets are always in $\Bd K_i$ for each $i$. 

The types do not change for $h_i$. The limiting types are not changed. 

Suppose that $h(\pi_1(\tilde E))$ acts on 
a generalized lens-cone $L$. Then $h(\pi_1(\tilde E))$ has a unique fixed point $v$ in $K$
since $h \in \rep_{E, u, ce}(\pi_1(\mathcal{O}), \PGL(n+1, \bR))$.
Then $v \not \in K^o $ since otherwise the elements fixing it has to be elliptic.
Then it satisfies the uniform 
middle eigenvalue condition. By Theorems 7.9 and 7.10 of \cite{endclass}, the action is distanced and 
we can find a concave p-end neighborhood. 
%By Theorem \ref{thm:converse}, 
By Theorem \ref{thm:sSPC},
$\orb_h$ is a noncompact strongly tame SPC-orbifold with generalized admissible ends. 

Near $[h]$ there is an open neighborhood in $\rep_{E, u, ce}(\pi_1(\mathcal{O}), \PGL(n+1, \bR))$ 
where $h_i$ in it is realized by a strongly tame SPC-orbifold with admissible ends diffeomorphic to $\orb_h$
by Theorem \ref{thm:conv} and lifting by Theorem \ref{thm:lifting}. 
Hence, $\Omega_i/h_i(\pi_1(\orb))$ is diffeomorphic to $\orb_h$ for sufficiently large $i$. 
Hence, $\orb_h$ is diffeomorphic to $\orb$. 

By Theorem \ref{thm:sSPC}, $h$ is stable and strongly irreducible. 

The final item follows by Theorem \ref{thm:converse}. 

\end{proof}

%%% July 24, 2013 2:23 pm

To prove for $\SDef_{E, u, ce}(\mathcal{O})$, we need additionally the second last item of Lemma \ref{lem:endspres}. 
This completes the proof of Corollary \ref{cor:closed1}.  
\end{proof}

%\begin{remark} 
%For the closedness, we need not check wheter the boundary quotients are properly discontinuous since 
%the end conditions are assumed in our spaces considered. 
%\end{remark}

%Therefore, we obtain that $K$ is properly convex, and hence 
%$K^o/h(\pi_1(\orb))$ is a properly convex obifold with horospherical and totally geodesic lens-type ends or 
%ends with concave end-neighborhood. By Theorem \ref{thm:converse}, the orbifolds is strictly SPC. 
%Proposition 3.3 of \cite{endclass}, we obtain that all ends are of lens-type or horospherical. 

%% April 5, 2014 6:44pm

%%% April 22 8:37pm 2014

\chapter{Examples}\label{sec:examples}

% Main examples: Other method
% Example: tetrahedral example

%\section{Main examples}

%In this section, we work with $\bR P^n$ only. 

For dimension $2$, the surfaces with principal boundary component can be made into 
surfaces with ends since we can select the fixed points for each boundary components 
and produce radial ends. These are rather trivial examples for our theory.   %\marginpar{Why here?} 

%If an orbifold $\orb$ admits a complete hyperbolic structure, 
%$(\tilde{\mathcal{O}}, \tilde U)$ admits a hyperbolic metric for $U$ equal to the cusp neighborhoods. 
The complete hyperbolic structure on an orbifold gives  a strict SPC-structure on the orbifold with horospherical ends. 
Sometimes these deform to ones not from the complete hyperbolic ones and with different types of ends. 

Let $P$ be a properly convex polytope in an affine subspace of $\bR P^n$ with faces $F_i$ where 
each side $F_i \cap F_j$ of codimension two is given an integer $n_{ij} \geq 2$.
A {\em reflection orbifold or Coxeter orbifold based on $P$} is given by a group of reflections $R_i$ associated with faces satisfying relations 
$(R_i R_j)^{n_{ij}} = I$ for every pair $i, j$ with $F_i \cap F_j$ is a side of codimension $2$. 
Vinberg showed that $R_i$ acts on a convex domain in $\bR P^n$ with a fundamental domain $P$ 
and thus $P$ with a number of vertices removed 
can be given a structure of an orbifold with interior of faces $F_i$ silvered and the interior of edge 
is given dihedral group structure and so on. 

These orbifolds have radial end always. If $P$ is a hyperbolic polytope, the orbifold is relatively
hyperbolic with respect to ends and the ends are virtually reducible having virtually free abelian fundamental group. 

Then the theories of this paper are applicable to such orbifolds
and the deformation theorems obviously hold. The proper convexity also holds during deformations 
according to Vinberg's work also (see \cite{dgorb}).
When $P$ is a cube with all edge orders $3$, then we obtain a complete hyperbolic orbifold with horospherical ends.
Computations done by G. Lee show that there are nontrivial deformations from the hyperbolic structure to real projective structures  
where ends deform from horospherical ends to totally geodesic radial ends. 
(See also \cite{CHL}.)

The example of S. Tillmann is an orbifold on a $3$-sphere with singularity consisting of \index{handcuff orbifold} 
two unknotted circles linking each other only once under a projection to a $2$-plane 
and a segment connecting the circles (looking like a linked handcuff) with vertices removed and all arcs 
given as local groups the cyclic groups of order three. (See Figure \ref{fig:cubs}.)
This is one of the simplest hyperbolic orbifolds in 
Heard, Hodgson, Martelli, and Petronio \cite{heard} labelled $2h\underbar{\,}1\underbar{\,}1$. 
The orbifold admits a complete hyperbolic structure since we can start from a complete hyperbolic 
tetrahedron with four dihedral angles equal to $\pi/6$ and two equal to $2\pi/3$ at a pair of opposite edge $e_1$ and $e_2$.
Then we glue two faces adjacent to $e_i$ by an isometry fixing $e_i$ for $i= 1, 2$. 
The end orbifolds are two $2$-spheres with three cone points of orders equal to $3$ respectively. 
These end orbifolds always have induced convex real projective structures in dimension $2$, 
and real projective structures on them have to be convex. Each of these is either the quotient of a properly convex 
open triangle or a complete affine plane as we saw in Proposition 3.3 of \cite{endclass}.

Tillman found a one-dimensional solution set from the complete hyperbolic structure by explicit computations. 
His main questions are the preservation of convexity and realizability as convex real projective structures on the orbifold. 
%That small deformation remains strictly SPC as given by our theory.

The another main example can be obtained by doubling a complete hyperbolic Coxeter orbifold based on a convex polytopes. 
%Here $P$ is doubled to an $n$-sphere with cod
%However, by Theorem \ref{thm:conv}, this is proved. 
We take a double $D_T$ of the reflection orbifold based on 
a convex tetrahedron with orders all equal to $3$. This also admits a complete hyperbolic 
structure since we can take the two tetrahedra to be the regular complete hyperbolic tetrahedra
and glue them by hyperbolic isometries. The end orbifolds are four $2$-spheres with three singular points of orders $3$. 
Topologically, this is a $3$-sphere with four points removed and six edges connecting them 
all given order $3$ cyclic groups as local groups.

\begin{theorem}\label{thm:tetrbounded} 
%Suppose that a $3$-dimensional orbifold is triangulated into one or two tetrahedra with edges in the singular 
%locus and the vertices are all removed. 
%Suppose that this orbifold has no essential annulus or torus or equivalently it admits a complete hyperbolic 
%structure. 
Let $\mathcal{O}$ denote the hyperbolic $3$-orbifolds $2h\underbar{\,}1\underbar{\,}1$ or $D_T$. 
%the reflection orbifold based on a triangle with edge order equal $3$. 
We assign the $\cR$-type to each end. 
Then  $\SDef_{E}({\mathcal{O}})$ equals $\SDef_{E, u, ce}(\orb)$ 
%is a union of components of $\Def_{E}({\mathcal{O}})$
and $\hol$ maps $\SDef_{E}({\mathcal{O}})$
as an onto-map to a component of characters \[\rep_{E}(\pi_1({\mathcal{O}}), \PGL(4, \bR))\] containing 
a hyperbolic representation which 
is also a component of 
\[\rep_{E, u, ce}(\pi_1({\mathcal{O}}), \PGL(4, \bR)).\] %% no superscript s here!!!
In this case, the components are cells of dimension $1$ and dimension $4$ respectively 
for $2h\underbar{\,}1\underbar{\,}1$ and the double $D_T$.
\end{theorem}
\begin{proof} 
%In these cases, the orbifolds are either a double of a reflection orbifold with all edge orders equal to $3$
%and the orbifold found by Tillmann, i.e., a orbifold based on $S^3$ with two points removed 
%and two arcs of singularity $3$.

The end orbifolds have Euler characteristics equal to zero and all the singularities are of order $3$.
Since the singularity is in the end neighborhood, it follows that
$\SDef_{E}({\mathcal{O}})$ equals $\SDef_{E, u}(\orb)$.  %\marginpar{make the stuff into lemmas?}
Each of the ends has to be either horospherical or radial of lens and totally geodesic type by Proposition 3.3 of \cite{endclass}. 
Let $\partial \orb$ denote the union of end orbifolds of $\orb$. 

%Let $\mathcal{O}$ be as in the premise. 
In \cite{jkms}, we showed that the real projective structures on the ends determined the real projective structure on $\mathcal{O}$. 
First, there is a map $\SDef_{E}({\mathcal{O}}) \ra \CDef(\partial {\mathcal{O}})$ 
given by sending the real projective structures on $\mathcal{O}$ 
to the real projective structures of the ends. (Here if $\partial {\mathcal{O}}$ has many components, then
$\CDef(\partial {\mathcal{O}})$ is the product space of the deformation space of all components.)
Let $J$ be the image. 

Let $\mu$ be an element of $\SDef_{E}({\mathcal{O}})$. The universal cover $\torb$ is a properly convex domain in $\SI^3$. 
The singular geodesic arcs in $\torb$ connect one p-end vertex to the other. The developing image of $\torb$ is a convex open domain
and the developing map is a diffeomorphism. 
Their developing images form geodesics meeting at vertices transversally. 
A choice of six edges will map to a $1$-skeleton of a convex tetrahedron in $\SI^n$. 
The geometry of the situation 
forces us that in the universal cover $\torb$, there exists two convex tetrahedra $T_1'$ and $T_2'$ with vertices removed 
in $\torb$. They are adjacent and their images under $\pi_1(\orb)$ tessellate $\torb$. 

%Each convex real projective structure on $\mathcal{O}$ gives us admissible ends. 
%First, the end orbifolds are two-spheres with singularity of index $3, 3, 3$. 
%Such an orbifold admits a complete affine structure or is a quotient of a properly convex triangle
%as it cannot be a quotient of a half-space with a distinguished foliation by lines. 
%Take a finite-index free abelian group $A$ of rank two.
%Consider an end, i.e., a cone over such an orbifold with end vertex $v$. 
%In the complete affine structure case, we have a horoball end neighborhood
%since $A$ acts without fixed point and on a properly convex domain $\tilde{\mathcal{O}}$ and on the supporting 
%hyperspace at the end vertex $v$. 
%In the triangle case, $A$ acts with an element $g'$ with an eigenvalue $>1$ and an eigenvalue $<1$. 
%This means that there are fixed points $v_1$ and $v_2$ other than $v$
%in a direction of the vertices of the triangle in the cone. Then $g'$ has four fixed point and 
%an invariant subspace $P$ disjoint from $v$. It follows that the end fundamental group acts on $P$ 
%as well. We have totally geodesic ends and by Lemma \ref{lem:lensend}, the end is lens-shaped. 

The end orbifold is so that if given an element of the deformation space, then the geodesic triangulation is uniquely obtained.
Hence, there is a proper map from $\SDef_{E}({\mathcal{O}})$ to the space of invariants of the triangulations 
as in \cite{jkms}, i.e, the product space of cross-ratios and Goldman-invariant spaces.
(The projective structures are bounded if and only if the projective invariants are bounded.)
Thus by the result of \cite{jkms}, there is an inverse to the above map
$s: J \ra \SDef_{E}({\mathcal{O}})$ that is a homeomorphism. 

 For $2h\underbar{\,}1\underbar{\,}1$, $J$ is connected by Tillmann's computations. 

Now consider when $\mathcal{O}$ is the orbifold obtained from doubling a tetrahedron with edge orders $3, 3, 3$.
We consider an element of $\SDef_{E}(\orb)$. Since it is convex, 
we triangulate $\mathcal{O}$ into two tetrahedra and this gives a triangulation for 
each end orbifold diffeomorphic to $S_{3,3,3}$, each of which gives us triangulations into two triangles.
We can derive from the result of Goldman \cite{Gconv} and Choi-Goldman \cite{CG} that given projective invariants 
$\rho_1, \rho_2, \rho_2, \sigma_1, \sigma_2$ for each of the two triangles satisfying $\rho_1 \rho_2 \rho_3 = \sigma_1 \sigma_2$,
we can determine the structure on $S_{3, 3, 3}$ completely. 

For $S_{3,3,3}$ with a convex real projective structure and divided into two geodesic triangles, 
we compute these invariants $\rho_1, \rho_2, \rho_2, \sigma_1, \sigma_2$ for one of the triangles 
\begin{eqnarray} 
 & s^2+s \tau_1+1,s^2+s \tau_2+1,   s^2+s \tau_3+1,  &\nonumber \\ 
&   t \left(s^2+s \tau_2+1\right),  \frac{1}{t}\left(s^2+s \tau_1+1\right)  \left(s^2+s \tau_3+1\right)   & \label{eqn:inv1}
   \end{eqnarray} 
and for the other triangle the corresponding invariants are 
\begin{eqnarray}
 & \frac{1}{s^2}(s^2+s \tau_1+1),  & \frac{1}{s^2}(s^2+s \tau_2+1), \frac{1}{s^2}(s^2+s \tau_3+1), \nonumber \\
  & \frac{t}{s^3}\left(s^2+s \tau_2+1\right), & \frac{1}{s^3 t}\left(s^2+s \tau_1+1\right) \left(s^2+s \tau_3+1\right) \label{eqn:inv2}
   \end{eqnarray}
   where $s, t$ are Goldman parameters and $\tau_i = 2 \cos 2\pi/p_i$ for the order $p_i$, $p_i = 3$. \index{projective invariant}

The set $J$ is given by projective invariants of the $(3,3,3)$ boundary orbifolds 
satisfying some equations that the cross ratio of an edge are same from one boundary orbifold 
to the other and that the products of Goldman $\sigma$-invariants equal $1$ for some 
quadruples of Goldman $\sigma$-invariants. 
%We show that $J$ is connected by 
%computations by writing cross ratios and Goldman $\sigma$-invariants from coordinates $s_i, t_i$, $i=1,2,3, 4$, of 
%the deformation spaces of end orbifolds and the equations are given by cross ratios 
%and products of $\sigma$-invariants as in \cite{jkms}. 
By the method of \cite{jkms} developed by the author, we obtain the equations that $J$ satisfies.  
The equation is solvable: 
\[s_1 = s_2 = s_3 = s_4=s, t_1 t_2 t_3 t_4 = C(s) \hbox{ for a constant } C(s)> 0 \hbox{ depending on} s.\]
(See the Mathematica file \cite{schoimath}.)
Thus $J$ is homeomorphic to a $4$-dimensional cell. (The dimension is one higher than 
that of the deformation space of the reflection $3$-orbifold based on the tetrahedron. Thus 
we have examples not arising from reflection ones here as well.)

Conversely, we can assign invariants at each edge of the tetrahedron and the Goldman $\sigma$-invariants at the vertices
if the invariants satisfy the equations.
This is given by starting from the first convex tetrahedron and gluing one by one 
using the projective invariants (see \cite{jkms} and \cite{schoimel}):
Let the first one by always be  the standard tetrahedron with vertices 
\[[1,0,0,0],[0,1,0,0],[0,0,1,0], \hbox{ and } [0,0,0,1]\] and we let $T_2$ a fixed adjacent tetrahedron with vertices 
\[[1,0,0,0],[0,1,0,0],[0,0,1,0] \hbox{ and } [2,2,2,-1].\] Then projective invariants 
will determine all other tetrahedron triangulating $\tilde{\mathcal{O}}$. 
Given any deck transformation $\gamma$, $T_1$ and $\gamma(T_1)$ will be connected by 
a sequence of tetrahedrons related by adjacency and their pasting maps are completely
determined by the projective invariants, where cross-ratios do not equal $0$. 
Therefore, 
as long as the projective invariants are bounded, the holonomy transformations of the generators are bounded. 
%(The term ``holographic" will be explained in our new paper \cite{endclass}. 
(This method was spoken about in 
our talk in Melbourne, May 18, 2009 \cite{schoimel}.)

From this, we see that $\SDef_{E}({\mathcal{O}})$ is connected. 
The results now follow from Theorem \ref{thm:conv} and Corollary \ref{cor:conv}. 

\end{proof}

%We mention that this will work with any $3$-orbifold with end orbifolds of Euler characteristic zero 
%and with an ideal triangulation where each edge is in the singular locus. 
We remark that the above theorem can be generalized to orders $\geq 3$ with hyperideal ends
with similar computations. 
%This will be done in another paper \cite{endclass}.
See \cite{schoimath} for examples to modify orders and so on. 

\begin{figure}
\centerline{\includegraphics[height=8cm]{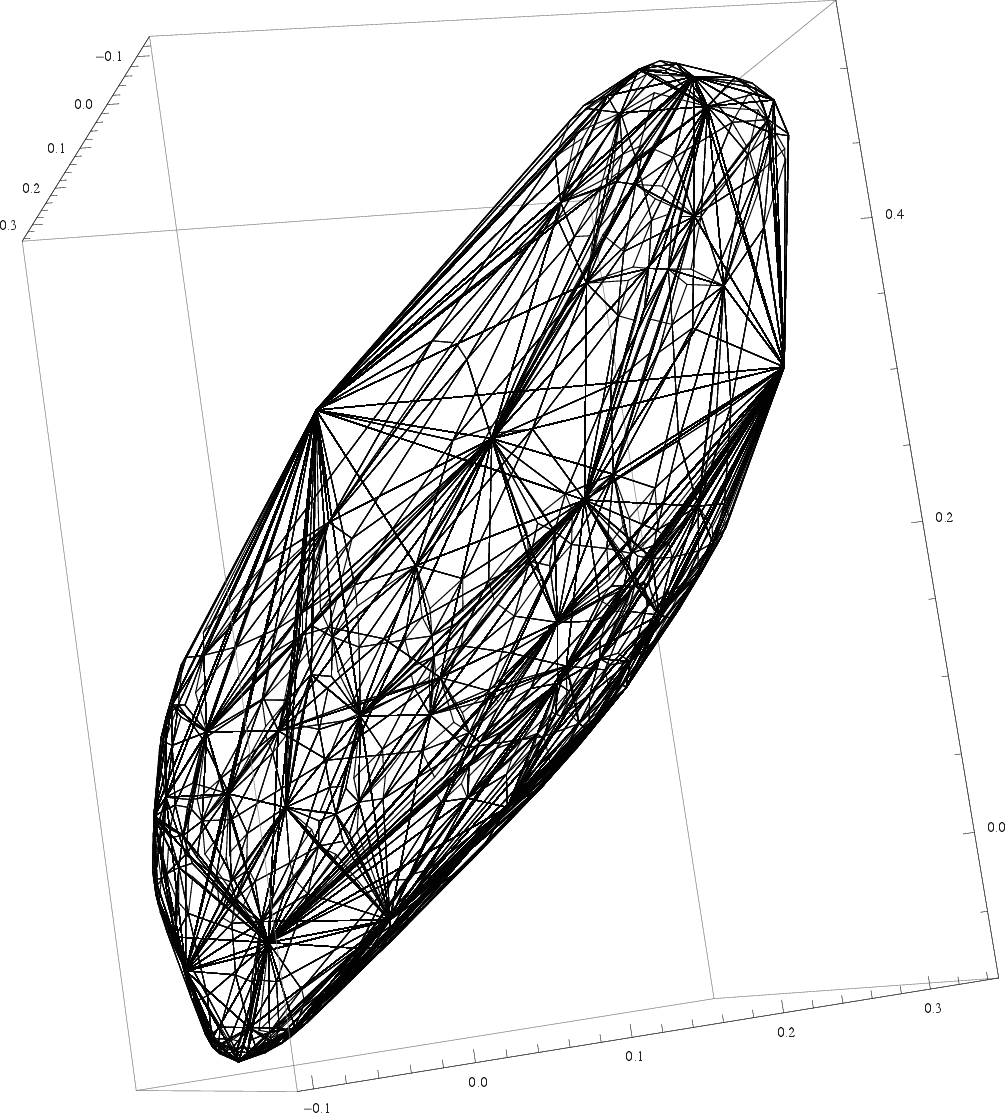}}
\caption{A convex developing image example of a tetrahedral orbifold of orders $3,3,3,3,3,3$.}
\label{fig:good}
\end{figure}

%\begin{figure}
%\centerline{\includegraphics[height=8cm]{Figbadexm}}
%\caption{A developing example of tetrahedral orbifold of orders $7,3,3,5,5$ which is probably not convex. We can show 
%numerically that a neighborhood of a fixed point is in a convex hull of the orbits. Such an orbifold cannot be convex.}
%\label{fig:bad}
%\end{figure}

%Example: Tillman's example....

%% Reflection group cases; Use J Lee's work....
%%May 18 6:25pm

\appendix

\chapter{Projective abelian group actions on convex domains}

\begin{lemma} \label{lem:gconv} 
Let $t_0 \in I$ for an interval $I$. 
Suppose that we have a parameter of compact convex domains $\tri_t \subset \SI^{n-1}$ for $t < t_0$, $t \in I$, and a compact 
convex domain $\tri_{t_0}$ in $\SI^{n-1}$ and a transitive group action $\Phi_t:L \times \tri^o_t \ra \tri^o_t, t \in I$ by 
a connected abelian group $L$ of rank $n$ for each $t \in I$. 
Suppose that $\Phi_t$ depends continuously on $t$ and $\Phi_t$ is given by a continuous parameter of 
homomorphisms $h_t: L \ra \SL_{\pm}(n, \bR)$. 
Then $\tri_t \ra \tri_{t_0}$ geometrically.
\end{lemma} 
\begin{proof} 
We may assume without loss of generality that $\bigcap_{t\in I} \tri^o_t \ne \emp$ by taking a smaller $I$.
Choose a generic point $x_0 \in \bigcap_{t\in I} \tri^o_t$. 
Any point $x \in \tri^o_{t_0}$ equals $\Phi_{t_0}(g, x_0)$ for $g \in L$.  
Therefore, $\Phi_t(g, x_0) \in \tri^o_t \ra \Phi_{t_0}(g, x_0)$ as $t \ra t_0$. 
Hence, every point of $\tri_{t_0}$ is the limit of a path $\gamma(t) \in \tri^o_t$ for $t < t_0$. 

Conversely, given any parameter of points $x(t) \in \tri^o_t$ for $t\in I$, we obtain 
that $x(t) = \Phi_t(g_t, x_0)$ for $g_t \in L$. 
Let $L \cong \bR^{n-1}$ have coordinates $(a_1, \dots, a_{n-1})$. 
We define $x_t := \Phi_{t_0}(g_t, x_0)$. 
$\Phi_t(g, \cdot): \SI^{n-1} \ra \SI^{n-1}$ is represented 
as a matrix \[h_t(g) = \exp ( H_t( \sum_{i=1}^{n-1} a_i(g) e_i))\]
where $\{H_t: \bR^{n-1} \ra \mathfrak{sl}(n, \bR)\}$ is 
a uniformly bounded collection of linear maps 

We claim that $\Phi_t$ is an equicontinuous family of functions: 
%This followsfrom the abelian Lie algebra $\mathfrack{a}$ of dimension $n$. 
Let $v$ be a generic vector $\bR^{n-1}$ of norm $1$. 
By dividing by the largest norm matrix entries we obtain 
a matrix $m_t(g)$ with entries $\leq 1$. Since $\bR P^{n-1}$ is bounded, 
a computation shows that the family of functions $\{m_t| t \in I\}:\bR^{n-1} \times \bR^n \ra \bR^{n}$ have derivatives 
uniformly bounded above as the entries are rational functions of exponential functions with bounded coefficients and 
these rational functions are bounded above by $1$. 

Hence $x(t)$ and $x_t$ have the same set of limit points as $t \ra t_0$. 
Since $x_t \in \tri_{t_0}$, $x(t)$ has limit points in $\tri_{t_0}$ only. 

The Hausdorff convergence topology 
and these two facts give us that $\tri_t \ra \tri_{t_0}$ geometrically, which is an elementary fact.

\end{proof} 

For a matrix $A$, we denote by $|A|$ the maximum of the norms of entries of $A$. 

\begin{lemma}\label{lem:dense} 
Let $h:\bZ^l \ra \SL_\pm(n, \bR)$ be a representation to unipotent elements. 
Let $g_1, \dots, g_l$ denote the generators. 
Then given $\eps> 0$
there exists a positive diagonalizable representation $h': \bZ^l \ra \SL_\pm(n, \bR)$ with matrices 
satisfying  $|h'(g_i) - h(g_i)| < \eps, i=1, \dots, l$. 
\end{lemma} 
\begin{proof} 
First assume that every $h(g_i)$, $i=1, \dots, l$, has matrices that upper triangular matrices with diagonal elements equal to $1$
since the Zariski closure is in a nilpotent Lie group. 

Let $\eps > 0$ be given. 
We will inductively prove 
that we can find $h'$ as above with eigenvalues of $h'(g_1)$ are all positive and mutually distinct. 
For $n=2$, we can simply change the diagonal elements to positive numbers not equal to $1$. 
Then the group imbeds in $\Aff(\bR^1)$. We choose positive constant $a_i$ so that $|a_i - 1| < \eps$. 
Let $g_i$ be given as $x \mapsto a_i x + b_i$. 
The commutativity reduces to equations $a_j b_i = a_i b_j$ for all $i, j$. 
Then the solution are given by $b_i = a_1^{-1} a_i b_1$. 

Suppose that the conclusion
is true for dimension $k-1$. We will now consider a unipotent homomorphism $h: \bZ^l \ra \SL_\pm(k, \bR)$.
Since $h(g_i)$ is upper triangular, let $h_1(g_i)$ denote the upper-left $(k-1)\times (k-1)$-matrix. 
We find a homomorphism $h'_1: \bZ^l \ra \SL_\pm(k-1, \bR)$ a positive diagonalizable representation
and the eigenvalues of $h'_1(g)$ are positive and mutually distinct. 
Also assume $|h'_1(g_i) - h_1(g_i)| < \eps/2$ for $i=1, \dots, l$.  
We change 
\[h(g_i) =  \left( 
\begin{array}{cc}
h_1(g_i)  & b(g_i)   \\
  0   & 1
\end{array}
\right) 
\hbox{ to } 
h'(g_i) = \left( 
\begin{array}{cc}
\frac{1}{\lambda(g_i)^{\frac{1}{k-1}}} h'_1(g_i)  & b'(g_i)   \\
  0   & \lambda'(g_i)   
\end{array}
\right)
\]
for some choice of $h_1'(g), b'(g), \lambda'(g_i) > 0$ for $i=1, \dots, l$. 
For commutativity, we need to solve for $b'(g_i)$, $i=1, \dots, l$, 
\[\frac{1}{\lambda'(g_i)^{\frac{1}{k-1}}} h'_1(g_i) b'(g_j) = \frac{1}{\lambda'(g_j)^{\frac{1}{k-1}}} h'_1(g_j) b'(g_i), 
1 \leq i, j \leq l. \]
The solution is given by 
\[b'(g_i) = \frac{\lambda(g_1)^{\frac{1}{k-1}}}{\lambda(g_i)^{\frac{1}{k-1}}} h'_1(g_1)^{-1} h'_1(g_i) b'(g_1), i=2, \dots, l\]
We choose $b'(g_1)$ arbitrarily near $b(g_1)$. %the upper-right $(k-1) \times 1$-submatrix of $h(g_1)$. 
Here, $\lambda(g_1)$ has to be chosen generically to make all the eigenvalues distinct and sufficiently near $1$ 
so that $|h'(g_i) - h(g_i)|< \eps$, $i=1, \dots, l$.  
We can check the solution by the commutativity. 
Hence, we complete the induction steps. 
\end{proof} 

Given a connected abelian group $A$ with positive real eigenvalues only, 
we can form for each $g \in A$, the Jordan block decomposition of $\bR^{n+1}$ into
subspaces with the same real eigenvalues. 
%An elementary Jordan block subspace $V^1_g$ of $g$ is preserved by 
%each element $h$ of $A$ and $h$ has to be of form a scalar times an identity matrix plus a nilpotent 
%matrix commuting with $h$ and $g$. 
%An elementary Jordan block subspace of $h$ either meet $V^1_g$ trivially 
%or must be contained in it. Hence, we can decompoose $\bR^n$ into elementary Jordan block subspaces
%minimal in the sense that if $V$ is one, then all other Jordan block subspace of $h \in A \ne \Idd$ 
%either contains it or meets trivially with it. 
%Let $V_1, \dots, V_l$ denote these. 
We direct-sum all the elementary Jordan blocks that have the same eigenvalue under every $g \in A$:
For each $g$, we let $\lambda_i(g)$ denote the eigenvalue of $g$ associated with $V_i$, $i=1, \dots, l$, 
i.e., $g- \lambda_i(g)\Idd_{V_i} | V_i$ is nilpotent. 
Two elementary Jordan block subspaces $V_i$ and $V_j$ are {\em equivalent} if $\lambda_i(g) = \lambda_j(g)$ for all $g \in A$. 
We direct sum the Jordan block subspaces in a Jordan decomposition that are equivalent to one $V_i$. We call this subspace
a {\em common Jordan block space}. 

Since $h$ can be connected to the identity, $h \in A$ cannot switch 
common Jordan blocks of $g$.
A {\em scalar group} is a group acting by $s\Idd$ for $s \in \bR$ and $s> 0$. 
A {\em scalar unipotent group} is a subgroup of the product of a scalar group with a unipotent group. 
%%% May 27 6:43pm

A {\em join} of two convex real projective $m$-dimensional and $l$-dimensional orbifolds 
$S_1$ and $S_2$ is obtained as follows: 
we take the convex open domains $\Omega_1$ and $\Omega_2$ covering $S_1$ and $S_2$ 
respectively. We projectively embed the two 
in two affine subspaces of the complementary subspaces $U_1$ and $U_2$ of $\SI^{n-1}$ (resp in $\bR P^{n-1}$), 
for $n=m+l+1$
and obtain the interior of the join $\Omega_1 \ast \Omega_2$.
We take the quotient space of $(\Omega_1 \ast \Omega_2)^o$ 
by direct summing two holonomy groups and 
adding a diagonalizable projective automorphism fixing all points of $U_1$ and $U_2$. 
We can of course generalize this to join of $k$, $k \geq 2$, real projective orbifolds. 
In this case, we obtain a center a free abelian group of rank $k-1$.

One can think of the following lemma as a classification of convex real projective orbifolds
with abelian fundamental groups. Benoist \cite{BenNil} investigated thoroughly in this topic also. 
\begin{lemma} \label{lem:realeign} 
Let $\Gamma$ be an abelian group acting on a convex domain $\Omega$ of $\SI^{n-1}$ \rlp resp. 
$\bR P^{n-1}$ \rrp\,  cocompactly and properly discontinuously.
Then the Zariski closure $L$ of a finite index subgroup $\Gamma'$ of $\Gamma$ is 
so that $L/\Gamma'$ is compact and with positive eigenvalues.
Furthermore, $\Omega$ is an orbit of the abelian Lie group $L$. 
$\Omega/\Gamma'$ is a join of a closed properly convex manifold and a finite number of closed complete affine manifolds. 
\end{lemma}
\begin{proof}
We will prove for the case $\Omega \subset \SI^{n-1}$. The other case is implied by this. 
%If $\Omega$ is properly discontinous, then there exists a finite index subgroup $\Gamma'$ with positive eigenvalues only 
%by Proposition \ref{prop:Ben2}.
If $\Omega$ is properly convex, then Proposition \ref{prop:Ben2} gives us a diagonal matrix group $L$ 
acting on a simplex. 

Assume that $\Omega$ is not properly convex. 
We assume that $\Gamma$ is torsion-free using Selberg's Lemma. 
Since $\Gamma$ is abelian, the syndetic closure $L'$ of $\Gamma$ is an abelian Lie group and $L'/\Gamma$ is compact. 
We take a connected component $L$ of $L'$ and let $\Gamma' = L \cap \Gamma$. 
$L/\Gamma'$ is a manifold and $\Omega/\Gamma'$ is a closed manifold. 
Since they are both $K(\Gamma', 1)$-spaces, it follows that $L$ and $\Omega$ have the same dimension 
and $L/\Gamma'$ is compact also. (see \cite{DW}.)

If $\Omega$ is a complete affine space, then Proposition T of \cite{HG} proves our result. 
$L$ is a unipotent radical there. So the eigenvalues are all $1$. 

Suppose that  $\Omega$ is not complete affine but not properly convex. Then there exists a great sphere 
$\SI^{i-1}$ in the boundary of $\Omega$ where $L$ acts on 
and is the common boundary of $i$-dimensional affine spaces foliating $\Omega$. 
(see \cite{ChCh}.) There is a projective projection $\Pi_{SI^{i-1}}: \SI^{n-1} - \SI^{i-1} \ra \SI^{n-i-1}$. 
Then the image of $\Omega_1$ of $\Omega$ is properly convex. Since $L$ acts on it, 
$\Omega_1$ is an $(n-i-1)$-dimensional simplex by Proposition \ref{prop:Ben2}. 
Thus, $\Omega$ is the inverse image $\Pi_{\SI^{i-1}}^{-1}(\Omega_1)$. 
Since $\Gamma$ acts on $\Omega$, it follows that $L$ acts on $\Omega$. 
Since $\dim L =\dim \Omega$ and $\Gamma$ acts properly with compact fundamental domain, 
$L$ acts properly on $\Omega$. (See Section 3.5 of \cite{Thbook}.)

Let $N$ denote the kernel of $L$ going to a Lie group $L_1$ acting on $\Omega_1$.
\[ 1 \ra N \ra L \ra L_1 \ra 1.\]
Since $L_1$ is diagonalizable by Proposition \ref{prop:Ben2}, $L_1$ acts simply transitively on $\Omega_1$. 
$\dim L_1 = n-i$. Thus, $\dim N = i$ and the abelian group $N$ acts on each $i$-dimensional 
affine space $A_l$ that is a leaf. Since the action of $N$ is proper, $N$ acts on $A_l$ simply transitively and 
$N$ is also unipotent by Proposition T of Hirsch-Goldman \cite{HG}. 
Hence, we deduce that each element of $L$ has only positive eigenvalues and so do $\Gamma'$. 
$L$ is also the Zariski closure of $\Gamma'$ by Saito \cite{Saito}. (See \cite{DW}.)

%Since $\Gamma'$ is abelian and have positive eigenvalues, 
%the generators have unique elements in the Lie algebra $sl(n+1, \bR)$ mapping to them 
%under the exponential map. It is easy to see that the subgroup $L$ can be chosen by the smallest abelian subspace $L'$ in $sl(n+1, \bR)$ containing them. 
%Clearly, $L$ has positive eigenvalues only. Since $\Gamma$ is abelian and correspond to a lattice in $L'$, $L/\Gamma$ is compact. 
%Clearly, $L$ is a Zariski closure of $\Gamma$ as well. (See Corollary 1.3 of \cite{DW}.)

By finding the common Jordan block subspaces, 
we decompose $L$ into subspaces $V_1 \oplus V'_1 \oplus \cdots \oplus V'_k$ where $L$ acts on $V_1$ as a diagonalizable linear group and 
$L$ acts on $V'_j$ as elements of an abelian scalar unipotent group for $j=1, \dots, k$. 
That is, $g\in L$ is a direct sum of matrices of form 
\[  \left( \begin{array}{cccc}  
\lambda_j(g) & 0 & 0 & 0 \\ 
*     & \lambda_j(g) & 0 & 0  \\ 
* & * & \lambda_j(g) & 0 \\ 
* & * & * & \lambda_j(g) 
\end{array} \right) \]
where $\{\lambda_1, \dots, \lambda_k\}$ is a set of mutually distinct homomorphisms $L \ra \bR_+$
by the Lie-Kolchin theorem.

A scalar unipotent linear group of a vector space always acts on a half-space of a vector space. 
The $L$-action on $S(V'_j)$ can be seen as direct sum of affine groups on affine subspaces in some affine coordinates of an affine subspace of $S(V'_j)$.  
It is a unipotent action since $\lambda$ becomes $1$ in the affine coordinate form. Thus, $L$ is in a nilpotent group. 
Also, the $L$-action on invariant subsets of $S(V_1)$ can be conjugated to an affine translation group using exponential maps.

%Since the eigenvalues are all positive, the orbits in $S(V_1)$ and $S(V'_j)$ are 
%homeomorphic to cells. Thus, the orbits of $L$ in $\Omega$ have to be cells also since they are fiber products of 
%orbits in $S(V_1)$ and $S(V'_j)$. 

%% May 18 8:10pm 2014

Since $S(V_1)$ and $S(V'_j)$ are $\Gamma'$-invariant subspaces of $\SI^{n-1}$, 
we can show that $\Omega$ is contained in the strict join of the interiors of convex domains $K_1 \subset S(V_1)$ and $K'_j \subset S(V'_j)$
that are images of the projections of $\clo(\Omega)$. (Note $\dim K'_j = \dim S(V'_j) \geq 1$ since $\Omega$ is open.)
%that are closures of orbits of $L$. 
By the description of the $L$-action and $\Gamma$-action on $S(V_i)$, it follows that $L$ acts on $K_1$ and $K'_j$ also. 
The induced action of $\Gamma'$ on each of $K_1^o$ and $K^{\prime o}_j$ is cocompact since otherwise $\Omega/\Gamma'$ cannot be compact
by the join-coordinate description of the join. 
$L$ acts as a unipotent group on $S(V'_j)$ for each $j$ and hence $L$ acts on an affine subspace $A$ meeting $K^{\prime o}$ 
as a group of affine transformations. Then $K^{\prime o} \subset A$ since there is a  
a homotopy equivalence $K^{\prime o}/\Gamma'$ and $K^{\prime o} \cap A/\Gamma'$. 
In fact each $K^{\prime o}_j$ is an open hemisphere in $S(V'_j)$ and is 
an orbit of $L$ by Theorem 4.1 of \cite{GH} since  $\Gamma'$ acts cocompactly.

Since $L$ is diagonal on $K_1$, it acts transitively on $K_1^o$ also. 
$K_1^o$ is an open properly convex simplex since $\Gamma'$ acts as a diagonalizable group. 
Thus, the orbit of $L$ on each of $K_1$  equals $K_1^o$ since otherwise we can check that $K_1^o/L$ is not compact
by affine coordinate description above. 

We can introduce a coordinate system on $K_1 \ast K'_1 \ast \cdots \ast K'_k$ so that $L$ acts as unipotent affine transformation group
since we have such coordinates for $K_1$ and $K'_j$ and we can take logarithms of the join coordinates. 
Again, by Theorem 4.1 of \cite{GH}, $L$ acts transitively on the interior of the strict join $K_1 \ast K'_1 \ast \cdots \ast K'_k$. 
Thus, it follows that $L$ has an open orbit in $\Omega$ and $\Omega$ equals the interior of $K_1 \ast (K'_1 \ast \cdots \ast K'_k)$. 
This show that $\Omega$ is an orbit of $L$ as well.  
(See also \cite{BenNil}.)
Finally, $\Omega/\Gamma'$ is a join of $K_1/\Gamma'$ and $K'_i/\Gamma'$ for $i=1, \dots, k$. 
\end{proof}

%%% April 29, 2014 11:31pm

A convex real projective structure $\mu_0$ on an orbifold $\Sigma$ is {\em virtually immediately deformable to} a properly convex real structure
if there exists a parameter $\mu_t$ of real projective structures on a finite cover $\hat \Sigma$ 
of $\Sigma$ so that $\hat \Sigma$ with induced structures $\hat \mu_t$ is 
properly convex for $t > 0$. 

\begin{proposition}\label{prop:vip} 
Assume that $M$ is covered by a cell. 
Then a convex real projective structure on closed $(n-1)$-orbifold $M$ 
with infinite free abelian holonomy is always virtually  immediately deformable to 
a properly convex real projective. 
Furthermore, any join of such orbifolds with properly convex orbifolds are also virtually immediately deformable to 
a properly convex real projective orbifold. 
\end{proposition} 
\begin{proof} 
Let $\bZ^l$ denote the fundamental group of a finite cover $M'$ of $M$. 
Let $h \in \Hom(\bZ^l, \SL_\pm(n, \bR))$ be the restriction of the holonomy homomorphism to $\bZ^l$. 
Nearby every $h$, there exists a positively 
diagonalizable holonomy $h': \bZ^l \ra \SL_{\pm}(n, \bR)$ by Lemmas \ref{lem:dense} and \ref{lem:realeign}.
By the deformation theory of \cite{dgorb}, $h''$ is realized as a holonomy of 
a real projective manifold $M''$ diffeomorphic to $M'$. Also, the universal cover of $M'$ is an orbit of an abelian Lie group $L$ 
where $h'(\bZ^l)$ is a lattice by Lemma \ref{lem:realeign}. 
Now $h''(\bZ^l)$ is a lattice in a positively diagonalizable abelian Lie group $L'$. 
Since $L/h'(\bZ^l)$ is diffeomorphic to $L'/h''(\bZ^l)$ and they have real projective structures by \cite{BenNil}. 
$L'/h''(\bZ^l)$ with the real projective structure is projectively diffeomorphic to $M''$.
From this, we deduce that the universal cover $M''$ is 
an orbit of $L'$. Hence, $M''$ is properly convex since $L'$ is a positive real diagonal group. 

The final part just follows by deforming the first factor orbifold of the join. 
\end{proof} 

%\begin{corollary} \label{cor:realeign} 
%Let $\Gamma$ be an abelian group acting on a convex domain $\Omega$ of $\SI^{n-1}$ \rlp resp. 
%$\PGL(n, \bR)$\rrp\,  cocompactly and properly discontinuously. 
%Then the Zariski closure $L$ of a finite index subgroup $\Gamma'$ of $\Gamma$ is 
%so that $L/\Gamma'$ is compact and with positive eigenvalues  
%Furthermore, $\Omega$ is an orbit of $L$. 
%\end{corollary}
%\begin{proof}
%By Proposition \ref{prop:vip}, $\Gamma$ has a finite index subgroup $\Gamma'$ 
%that is an algebraic limit of holonomy groups of properly convex structures. Hence, $\Gamma'$ has only positive eigenvalues. 
%\end{proof} 

\begin{proposition}\label{prop:permconv} 
Let $\Sigma$ be a closed $(n-1)$-dimensional convex projective orbifold 
with the real projective structure $\mu$,
and let $\Omega$ in $\bR P^{n-1}$ \rlp resp. in $\SI^{n-1}$ \rrp\, denote a universal cover of $\Sigma$. 
%Suppose that $\mu_0$ is virtually immediately deformable to a properly convex real projective structure. 
Let $\mu_t, t \in [0, 1]$ be a parameter of convex real projective structures on $\Sigma$ where $\mu_0$ is properly convex or complete affine
and $\mu = \mu_1$. 
Let $\pi_1(\Sigma)$ be isomorphic to a finite extension of $\bZ^l \times \Gamma_1 \times \cdots \times \Gamma_k$ 
for a hyperbolic group $\Gamma_i$ for each $i$ and $1 \leq k -1 \leq l$. 
Let $h$ denote the holonomy homomorphism 
of $\pi_1(\Sigma)$ to $\PGL(n, \bR)$ \rlp resp. $\SL_{\pm}(n, \bR)$\rrp\, corresponding to the original real projective structure. 
Then 
\begin{itemize} 
\item[(i)] If $k \geq 1$ and $k-1 \leq l \leq k$, then $\Omega$ is  a properly convex domain. 
\item[(ii)] If $k = 0, l \geq 1$, then $\Omega$ is 
a convex domain $d(\Delta)$ that is an orbit of a connected abelian Lie group $\Delta$ in 
the Zariski closure of a finite index abelian subgroup of $h(\pi_1(\Sigma))$.
\item[(iii)] %Suppose that $\mu_0$ is virtually immediately deformable to a properly convex real projective structure. 
If $l \geq k, k \geq 1$, then %$\Omega$ is either projectively diffeomorphic to a properly convex domain 
$\Omega$ is real projectively diffeomorphic to the interior of the  strict join of 
\begin{itemize} 
\item an orbit $d(\Delta)$ for a connected abelian Lie group $\Delta$ that is in the Zariski closure of 
the center of a finite-index subgroup of $h(\pi_1(\Sigma))$ and 
\item a strict join $K_1 \ast \dots \ast K_k$ where a hyperbolic factor $\Gamma_i$ of $\Gamma$ acts divisibly on a properly convex domain  $K_i^o$.
\end{itemize} 
\item[(iv)] %Let $h$ denote the holonomy homomorphism of $\pi_1(\Sigma)$ corresponding to the original real projective structure. 
Suppose that $h(\pi_1(\Sigma))$ acts on a properly convex domain in $\bR P^{n-1}$ \rlp resp. in $\SI^{n-1}$ \rrp.  Then $\Omega$ is properly convex. 
\end{itemize}
\end{proposition} 
\begin{proof} 
We will prove for the case $\Omega \subset \SI^{n-1}$. The $\bR P^{n-1}$-version follows from this. 
%The deformability assumption is as follows:
%\begin{itemize}
%\item Let $\mu_t$ be the parameter of real projective structures deforming to $\mu_1$ for $t \in (0, 1]$
%where $\mu_t$ is properly convex for $t > 0$.  
%\end{itemize}
%The general case will follow from this. 
If $\mu_0$ is complete affine, then we deform for $t < 0$ and assume $\mu_t$ is properly convex 
by Proposition \ref{prop:vip} and shift our parameter. 

%We can define a convex domain $\Omega_t$ for properly convex $\mu_t$ so that $\

%Let $\Gamma$ denote the finite index subgroup of $\pi_1(\Sigma)$ isomorphic to the above product. 

We can choose a developing map $D_t:\tilde \Sigma  \ra \SI^{n-1}$ and 
a holonomy homomorphism $h_t$ as above for $\mu_t$
and $h_t$ is a continuous family of representations $\Gamma \ra \SL_{\pm}(n, \bR)$.  
This follows for each element $g \in \Gamma$, $h_t(g)$ is determined by developing maps 
in the space of $C^r$-topology. (See Chapter 6 of \cite{msj}.) 
Let $\Omega_t$ denote the image of $D_t$, a convex open domain. 

(i) Let $\Gamma$ denote the torsion-free subgroup of $\pi_1(\Sigma)$ isomorphic to $\bZ^l \times \Gamma_1 \times \cdots \times \Gamma_k$. 
We will denote the corresponding subgroups of $\Gamma$ by the same notations in  
$\bZ^l \times \Gamma_1 \times \cdots \times \Gamma_k$. 
%We may assume without loss of generality that each $\Gamma_i$ is torsion-free.

%Using Theorem \ref{thm:lifting}, we lift $\Gamma$ in $\SL_{\pm}(n, \bR)$ and $\Omega$ to $\SI^{n-1}$. 
By Koszul \cite{Kos}, the set $A_C$ of $t \in [0, 1]$ where $\Omega_t$ is properly convex is an open subset of $I$.
We will now show that the set is closed. Let $t_0$ be the supremum of a component of $I$ containing $0$. 
Suppose that there exists a parameter of holonomy representations 
$h_t$, $t \in [0, 1]$ acting on a convex domain $\Omega_t$ so that $\Omega_t$ is properly convex for $t < t_0$. 
%and $\Omega_{t_0} = \Omega$.  
Then for $t < t_0$, $h_t(\Gamma_i)$ acts on a properly convex domain $K_{i, t}$ for $i=1,\dots, k$ in $\clo(\Omega_t)$ 
and there are $l-k+1$ number of discrete points $k_{j, t}$ in $\clo(\Omega_t)$, for $j = 1, \dots, l-k+1$, 
by Proposition \ref{prop:Ben2} fixed by $\Gamma$ (up to choosing $\Gamma$ even smaller). 

Choose a subsequence $t_l \in A_C$, $t_l \ra t_0, t_l < t_0$ as $l \ra \infty$. 
By choosing a subsequence of $t_k$, we may assume without loss of generality that 
$K_{i, t_l} \ra K_{i, t_0}$ for $i = 1, \dots, k$ and $k_{j, t_l} \ra k_{j, t_0}$ for $j =1, \dots, l- k+1$ 
as $l \ra \infty$, to be called {\em vertices} still,
where $K_{i, t_0}$ is a properly convex compact domain where $\Gamma_i$ acts on and $k_{j, t_0}$ is a point fixed by $\Gamma$. 

\begin{itemize}
\item Let $k_{L, t_l}$ denote the convex hull of $\{k_{j, t_l}| j \in L\}$ for any subcollection $L$ of $\{1, \dots, l-k+1\}$
and 
\item $K_{M, t_l}$ denote the convex hull of $\{ K_{j, t_l}| j \in M\}$ for a subcollection $M$ of $\{1, \dots, k\}$. ($t_l < 1$.)
\end{itemize} 
By Lemma \ref{lem:gconv} and \cite{Ben3}, 
we may assume without loss of generality 
 that $k_{L, t_l} \ra k_{L, t_0}$ as $l \ra \infty$ and 
$K_{i, t_l} \ra K_{i, t_0}$
 geometrically as $l \ra \infty$ for a convex compact sets $k_{L, t_0}$ and $K_{i, t_0}$. 

By Corollary 1.2 of \cite{Ben3}, $\Gamma_i$ divides the properly convex $K_{i, t_0}^o$ irreducibly for each $i = 1, \dots, k$. 

Let $P_{M, t_0}$ denote the subspace spanned by $K_{M, t_0}$, and let $Q_{L, 1}$ the subspace spanned by $k_{L, t_0}$. 

If $l=0$ and $k=1$, then (i) follows from Corollary 1.2 of \cite{Ben3}. 
Now suppose $l = 1$ and $k= 2$. Then the subspaces $P_{1, t_0}$ and $P_{2, t_0}$ are disjoint. 
Otherwise, $P_{1, t_0}=P_{2, t_0}$ by the irreducibility of $\Gamma_1$ and $\Gamma_2$. 
Since $\Gamma_1$ acts trivially on $P_{2, t_0}$ by the limit argument, this is a contradiction. 
Now suppose $l= k-1$ and $k \geq 2$. Suppose that the subspaces $P_{M, t_0}$ and $P_{M', t_0}$ 
containing $K_{M, t_0}$ and $K_{M', t_0}$ meet and $M$ and $M'$ are minimal disjoint pair of such sets. 
We may assume that the collection of subspaces $\{P_{i, t_0}|i \in M\}$ in $P_{M, t_0}$ are independent 
and so are $\{P_{j, t_0}|j \in M'\}$ in $P_{M', t_0}$. 
%Then $M$ and $M'$ can be assumed to be disjoint since we can reduce $M$ and $M'$ to $M - M\cap M'$ and $M'- M \cap M'$ 
%and $P_{M - M\cap M', t_0} \cap P_{M'-M\cap M', t_0} \ne \emp$. 
%Furthermore, assume that $M$ and $M'$ are minimal so that 
Since $P_{M, t_0} \cap P_{M', t_0} \ne \emp$, $\Gamma$ acts reducibly on $P_{M, t_0}$. 
Since for the irreducible factors $P_{i, t_0} \cap P_{j, t_0}=\emp$ for $i\ne j$, 
and $P_{i, t_0}$ in $P_{M, t_0}$ are the only irreducible subspaces, 
it follows that $P_{M, t_0} \cap P_{M', t_0}$ is a join of a number of them. 

This contradicts the minimality unless $P_{M, t_0} = P'_{M', t_0}$. 
This is a contradiction as above. 
Therefore, we showed that $\{P_1, \dots, P_k\}$ are independent subspaces of $\SI^{n-1}$. 

Since $\clo(\Omega_{t_0})  = K_{1, t_0} \ast \cdots \ast K_{k, t_0}$, we obtain 
that $\Omega_0$ is properly convex. 

Now suppose that $l \geq 1$ and $k \geq  1$ and $k-1 \leq l \leq k$. 
Suppose that $l=k$.
Then there exists at most one vertex $k_{1,t_l}$ for each $i$. 
Let $k_{1, t_0}$ denote the limit.  Then again we show as above $\{Q_{1, t_0}, P_{1, t_0}, \dots, P_{k, t_0}\}$ are independent. 
As above, $\Omega_{t_0}$ is properly convex. 
This completes the proof of the closedness of the set of $t$ where $\Omega_t$ is properly convex. 
Thus, we showed that $\Omega_t$ is always properly convex for $t \in [0, 1]$. 

\begin{lemma} \label{lem:disj} Let $t\in [0, 1]$. 
We let $\mu_t$ be convex real projective 
structures on $\Sigma$ virtually immediately deformable to properly convex ones. 
Suppose that $M \ne \emp$. 
We assume that the following $(\ast)$ is true for $t=0$. 
\begin{description} 
\item[$(\ast)$] The subspace $Q_{L, t}$ spanned by $k_{L, t}$ is disjoint from the subspace $P_{M, t}$ spanned by $K_{M, t}$. 
\end{description} 
Then the above $(\ast)$ is true for all $t \in [0, 1]$. 
\end{lemma}
\begin{proof} 
The set $I$ of $t \in [0, 1]$ where $Q_{L, t} \cap P_{M, t} = \emp$ is open with $0 \in I$. 

Now let $I$ be a set so that $Q_{L, t}$ is disjoint from $P_{M, t}$. Let $t_0$ be a supremum of a component of $I$ containing $0$.
First, we need to only consider a fixed finite index subgroup by the compactness of intervals. 
$K_{M, t}$ is properly convex as shown in the proof above. By Lemma \ref{lem:gconv} 
$k^o_{L, t}$ is an orbit of a Lie group $A$ containing the center of holonomy of $\Gamma$. 
If $P_{M, t_0} \cap Q_{L, t_0} \ne \emp$, %and assume $P_{M, t_0} \ne Q_{L, t_0}$, then 
$P_{M, t_0} \cap Q_{L, t_0}$ equals the join of partial sets $P_{M', t_0} = Q_{L', t_0}, M' \subset M, L' \subset L$ 
by the reducibility of $\Gamma$. 
Then $P_{i, t_0}$ for $i \in M'$ meets $Q_{L, t_0}$. Again, this violates the irreducibilty of $\Gamma_i$.
The open and closedness argument proves the result. 
\end{proof}

(ii) Here, $\Gamma$ is an abelian group. $h_t(\Gamma)$ are diagonalizable with positive eigenvalues for $t < 1$. Thus, 
the eigenvalues of $h_t(\Gamma)$ are real positive for all $t$.  Lemma \ref{lem:realeign} implies the result.

(iii) Let $B_C$ denote the set of $t\in [0, 1]$ where the conclusion (iii) holds.
We show that $B_C$ is an open set:
Let $t \in B_C$. We have $\clo(\Omega_{t}) = d(L)_t \ast K_{1, t} \ast \cdots \ast K_{k, t}$
%where $L_t$ is an abelian connected Lie group in the Zariski closure of the center of the holonomy of $\Gamma$. 
%Here $L_t$ is a subgroup of $L$ chosen to act effectively on the convex domain $d(L_t)$.
Here, $d(L)$ is characterized as an orbit of $L$ and a $L$-invariant complement to a subspace $P_{t}$ containing 
$K_{1, t} \ast \cdots \ast K_{k, t}$. Since $L$ acts semisimply on $P_{t}$, 
there exists $\eps > 0$ such that for $|t'-t| < \eps$, there exists a complement $S_{t'}$ to $P_{t'}$. 
The subspace $S_{t'}$ and $T_{t'}$ containing $K_{1, t'} \ast \cdots \ast K_{k, t'}$
are disjoint. 
%$S_{t'}$ is spanned by an orbit $d(L)_{t'}$ by Lemma \ref{lem:realeign} since otherwise 
%we cannot have 
%Then let $\hat \Omega_t$ be the interior of $d(L_t) \ast K_{1, t} \ast \cdots \ast K_{k, t_0}$. 
Then considering the accumulation points of orbits of points under $\Gamma$
that includes $K_{i, t}$ and $d(L)_t$ shows that 
$\Omega_t$ contains the interior of $d(L)_t \ast K_{1, t} \ast \cdots \ast K_{k, t}$. 
$\Omega_t$ is a subset of the join since we can project the fundamental to the factors and act by $\Gamma$. 
By the dimension consideration, $S_{t'}$ is spanned by an orbit $d(L)_{t'}$ by Lemma \ref{lem:realeign}. 

Now we show the closedness of $B_C$. 
Suppose that there exists a parameter of holonomy representations 
$h_t$, $t \in I$ acting on convex domain $\Omega_t$. Let $t_0$ be the supremum of the connected component of $I$ containing $0$. 
and $\Omega_{t_0} = \Omega$.  We use the notation of (i). 

As in (i), there can only be a pair of antipodal points in $k_{L, t_0}$ for a subcollection $L = \{1, \dots, l-k+1\}$. 
Then $K_{M, t_0}$ is properly convex by (i).  
Let $S_1$ denote the minimal great sphere containing $k_{L, t_0}$ and $S_2$ the one containing $K_{M, t_0}$ for $M = \{1, \dots, k\}$. 
Now $S_1$ and $S_2$ does not intersect by Lemma \ref{lem:disj}.
%as we can use the reasoning similar to (i): 
%In this case, we find $g_i$ such that $g_i|S_1$ converges to the identity by $g_i|S_2$ is not bounded and vice versa. (We omit details.)
$\clo(\Omega) $ is a subset of a strict join of $k_{L, t_0}$ and $K_{M, t_0}$ again using limit sets and the fundamental domain as above. 
 %since this is true for each $\clo(\Omega_{t_l})$. 
Thus, 
$\dim k_{L, t_0} + \dim K_{M, t_0} + 1= n$ since otherwise $\clo(\Omega)$ would not be a domain in $\SI^{n-1}$.
Since $S_1$ and $S_2$ are disjoint and are complimentary,  there exists a natural projection $\SI^{n-1} - S_1 \ra S_2$. 
Then restricting $\Gamma'$ to  $k_{L, t_0}^o$, we obtain a cocompact action since $k_{L, t_0}^o$ is the image of the projection
of $\Omega$ and $\Omega/\Gamma'$ is compact. 
For each $l$, the interior of $k_{L, t_l}^o$ is an orbit of a connected abelian group 
that is the Zariski closure of the subgroup corresponding to $\bZ^l$. 
Hence, $k_{L, t_0}^o$ is an orbit of a connected abelian Lie group and is convex as in (ii). 
%As above $K_{1, t_0} \ast \cdots \ast K_{k, t_0}$ is properly convex and not antipodal to any vertex. 
Thus $\clo(\Omega)$ is the strict join of these two sets. 

(iv) Suppose that $\pi_1(\Sigma)$ acts on a properly convex domain in $\SI^{n-1}$. 
$d(\Delta)$ acts on the domain. %Hence, $\Delta$ is a diagonalizable group.
Therefore, the previous three items imply the result. 

%Finally, we can drop the hypothesis (H1). 
%There exists some $t_0$ where $\Omega_{t_0}$ is not properly convex. After $t_0$, 
%the proof follows the above since we can separate out the orbit $d(\Delta_{t_0})$ for a connected abelian group
%$\Delta_{t_0}$, a Zariski closure of the holonomy subgroup of the center,  as 
%a factor of a strict join and study the properly convex factor separately. 
\end{proof} 
%In case (i), the group is said to be {\em Benoist}. 

An immediate corollary is:
\begin{corollary} \label{cor:convhol} 
Suppose that a real projective orbifold $\Sigma$  is a closed $(n-1)$-orbifold that is convex with admissible fundamental group and 
with the structure deformable to properly convex structures. 
%Let $\Sigma$ have a deformed real projective structure $\mu'$ so that 
Suppose that the holonomy group 
$h'(\pi_1(\Sigma))$ in $\PGL(n, \bR)$ \rlp resp. $\SL_{\pm}(n, \bR)$\rrp\, 
acts on a properly convex domain or a complete affine space $D$ and the associated developing map  maps into $D$.
Then $\Sigma$ is also properly convex or complete affine with developing map that is a diffeomorphism to $D$.
\end{corollary}

\chapter{A topological result} 

\begin{proposition}\label{prop:metriz}
Let $X$ be a compact metrizable space. Let ${\mathcal C}_X$ be a countable collection of compact sets. 
The collection has the property that 
$C_K := \bigcup_{C \in {\mathcal C}_X, C\cap K \ne \emp} C$ is closed for any closed set $K$.
We define the quotient space $X/\sim$ as follows $x \sim y$ iff $x, y \in C$ for an element $C \in {\mathcal C}_X$. 
Then $X/\sim$ is metrizable. 
\end{proposition} 
\begin{proof} 
We show that $X/\sim$ is Hausdorff, $2$-nd countable, and regular and use the Urysohn metrization theorem. 
We define a countable collection $\mathcal{B}$ of open sets of $X$ as follows:
We take an open subset $L$ of $X$ that is an $\eps$-neighborhood of an element of ${\mathcal C}_X$ or a point of 
a dense countable set $Y$ in $X - \bigcup {\mathcal C}_X$ 
for $\eps \in \bQ, \eps > 0$. 
We form $L - \bigcup_{C \cap \Bd L \ne \emp, C \in {\mathcal C}_X} C$ 
for all such $L$ containing an element of ${\mathcal C}_X$
or $Y$. This is an open set since $\Bd L$ is closed and by the premise, 
and are neighborhoods of elements of ${\mathcal C}_X$ and $Y$.
%We also form $L 
Furthermore, each element of $\mathcal{B}$ is a saturated open set under the quotient map. 
Now, the rest is straightforward. 
\end{proof}

%% References....

%\end{document}